\documentclass[a4paper,11pt]{amsart}

\usepackage[margin=3cm,top=2.5cm,bottom=2cm]{geometry}

\usepackage{amsmath, amssymb, amsthm, mathrsfs}
\usepackage{enumitem,microtype} 
\usepackage[dvipsnames]{xcolor}

\usepackage[colorlinks=true,linkcolor=red!70!black,citecolor=MidnightBlue]{hyperref}

\usepackage[utf8]{inputenc} 
\usepackage[T1]{fontenc}
\usepackage{lmodern}

\newcommand{\R}{\mathbf R}
\newcommand{\C}{\mathbf C}
\newcommand{\T}{\mathbf T}
\newcommand{\N}{\mathbf N}
\newcommand{\Z}{\mathbf Z}

\renewcommand{\AA}{\mathcal A}
\newcommand{\BB}{\mathcal B}
\newcommand{\EE}{\mathcal E}
\newcommand{\FF}{\mathcal F}

\newcommand{\BBB}{\mathscr B}

\newcommand{\EEE}{\mathscr E}

\newcommand{\UUU}{\mathscr U}

\newcommand{\XXX}{\mathscr X}
\newcommand{\YYY}{\mathscr Y}
\newcommand{\ZZZ}{\mathscr Z}

\theoremstyle{plain}
\newtheorem{theo}{Theorem}[section]
\newtheorem{prop}[theo]{Proposition}
\newtheorem{lem}[theo]{Lemma}
\newtheorem{cor}[theo]{Corollary}
\theoremstyle{remark}
\newtheorem{rem}[theo]{Remark}
\theoremstyle{definition}

\numberwithin{equation}{section}

\def\le{\leqslant}
\def\ge{\geqslant}
\def\leq{\leqslant}
\def\geq{\geqslant}

\DeclareMathOperator{\Div}{div}

\DeclareMathOperator{\Id}{Id}

\def\eps{{\varepsilon}}

\newcommand{\Nt}{|\hskip-0.04cm|\hskip-0.04cm|}

\newcommand{\la}{\left\langle}
\newcommand{\ra}{\right\rangle}
\renewcommand{\d}{\mathrm{d}}
\newcommand{\dt}{\mathrm{d}t}
\newcommand{\dx}{\mathrm{d}x}
\newcommand{\dv}{\mathrm{d}v}

\newcommand{\black}{\color{black}}

\allowdisplaybreaks

\begin{document}

\title[Regularization and hydrodynamical limit for the Landau equation]{Regularization estimates and hydrodynamical limit for the Landau equation}
\date{\today}

\author[K. Carrapatoso]{Kleber Carrapatoso}
\author[M. Rachid]{Mohamad Rachid}   
\author[I. Tristani]{Isabelle Tristani}

\address[K. Carrapatoso]{Centre de Math\'ematiques Laurent Schwartz, \'Ecole
  Polytechnique, Institut Polytechnique de Paris, 91128 Palaiseau cedex, France}
\email{kleber.carrapatoso@polytechnique.edu}

\address[M. Rachid]{Laboratoire de Math\'ematiques Jean Leray, Universit\'e de Nantes, 2 rue de la Houssini\`ere
BP 92208 F-44322 Nantes Cedex 3, France}
\email{Mohamad.Rachid@univ-nantes.fr}

\address[I. Tristani]{D\'epartement de Math\'ematiques et Applications, \'Ecole Normale Sup\'erieure, CNRS, PSL University, 75005 Paris, France}
\email{Isabelle.Tristani@ens.fr}



\begin{abstract}
In this paper, we study the Landau equation under the Navier-Stokes scaling in the torus for hard and moderately soft potentials. 
More precisely, we investigate the Cauchy theory in a perturbative framework and establish some new short time regularization estimates for our rescaled nonlinear Landau equation. These estimates are quantified in time and optimal, indeed, we obtain the instantaneous expected anisotropic gain of regularity (see~\cite{Rachid1} for the corresponding hypoelliptic estimates on the linearized Landau collision operator). Moreover, the estimates giving the gain of regularity in the velocity variable are uniform in the Knudsen number. Intertwining these new estimates on the Landau equation with estimates on the Navier-Stokes-Fourier system, we are then able to obtain a result of strong convergence towards this fluid system. 
\end{abstract}

%

\maketitle

\vspace{-0.2cm}

\tableofcontents

\section{Introduction}

In this paper, we are interested in the link between the Landau and Navier-Stokes equations. It has been a major challenge to establish rigorous links between microscopic and macroscopic equations for many years, this problem goes back to Hilbert~\cite{Hilbert} and the main goal is to obtain a unified description of gas dynamics. The equations of kinetic theory (including the Boltzmann and the Landau equations) can be seen as an intermediate step between the microscopic and macroscopic scales of description.  In order to link the Landau and Navier-Stokes equations, we study a suitable rescaling of the Landau equation, as described in Subsection~\ref{subsec:kineticmodel}. 

\smallskip
The first part of our paper is dedicated to the study of this rescaled Landau equation thanks to hypocoercivity methods, (linear and nonlinear) regularization estimates and sharp nonlinear estimates on the Landau collision operator. We study the Cauchy theory in a close-to-equilibrium framework for this equation and establish new and sharp regularization estimates in short time. The second part of our paper focuses on the aforementioned hydrodynamical limit problem. More precisely, we give a result of strong convergence of the solutions to the Landau equation constructed in the first part of the paper towards strong and global solutions to the incompressible Navier-Stokes-Fourier system. Our approach is reminiscent of the one used in~\cite{Bardos-Ukai,Gallagher-Tristani} for the hard spheres Boltzmann equation and improves the result obtained in~\cite{GuoBNS,Rachid2} in terms of type of convergence or functional framework in the case of not too soft potentials. Our analysis heavily relies on the estimates on the Landau equation established in the first part of the paper as well as on results of spectral analysis for the linearized Landau equation performed in~\cite{Yang-Yu} and some refined estimates on the fluid problem (as in~\cite{Gallagher-Tristani}).

\subsection{The kinetic model} \label{subsec:kineticmodel}

We start by introducing the Landau equation which models the evolution of charged particles in a plasma through the evolution of the density of particles $f=f(t,x,v)$ which depends on time $t \in \R^+$, position $x \in \T^3$ the~$3$-dimensional unit periodic box and velocity $v \in \R^3$, when only binary collisions are taken into account. The Landau equation reads:
	$$
	\partial_t f + v \cdot \nabla_x f = \frac{1}{\eps} Q(f,f),
	$$
where $\eps >0$ is the Knudsen number which is the inverse of the average number of collisions for each particle per unit time and $Q$ is the Landau collision operator. It is defined as
	\begin{equation}\label{eq:oplandau0}
		Q(g,f)(v) = \partial_{v_i}  \int_{\R^3} a_{ij}(v-v_*) \left[ g(v_*) \partial_{v_j} f(v) - f(v) \partial_{v_j} g(v_*) \right]  \d v_*,
	\end{equation}
where we use the convention of summation of repeated indices. The matrix $a_{ij}$ is symmetric, semi-positive and is given by
	\begin{equation}\label{eq:aij}
		a_{ij}(v) = |v|^{\gamma+2}\left( \delta_{ij} - \frac{v_i v_j}{|v|^2}\right), \quad -3 \le \gamma \le 1.
	\end{equation}
We have the following classification: We call hard potentials if $\gamma\in(0,1]$, Maxwellian molecules if $\gamma=0$, moderately soft potentials if $\gamma \in [-2,0)$, very soft potentials if $\gamma \in (-3,-2)$ and Coulomb potential if $\gamma=-3$.
Hereafter, we shall consider the cases of hard potentials, Maxwellian molecules and moderately soft potentials, i.e. 
$$
- 2 \le \gamma \le 1.
$$ 
The Landau equation preserves mass, momentum and energy. Indeed, at least formally, for any test function $\varphi$, we have
\begin{multline}  \label{prop:conserv}
\int_{\R^3} Q(f,f)(v) \varphi(v) \, \dv \\
= - \frac12 \int_{\R^3 \times \R^3} a_{ij}(v-v_*) f(v) f(v_*)  \left(  \frac{\partial_{v_i} f(v)}{f(v)} -  \frac{\partial_{v_i} f(v_*)}{f(v_*)} \right)   \left( \partial_{v_j} \varphi(v) - \partial_{v_j} \varphi(v_*) \right) \, \d v_* \, \dv,
\end{multline}
from which we deduce that
\begin{equation} \label{eq:conserv}
\begin{aligned}
\frac{\d}{\d t} \int_{\T^3 \times \R^3} f(x,v) \varphi(v)  \, \dv \, \dx
&= \int_{\T^3 \times \R^3} \left[\frac{1}{\eps} Q(f,f)(x,v) - v \cdot \nabla_x f(x,v) \right] \varphi(v)  \, \dv \, \dx \\
&= 0 \quad \text{for} \quad \varphi(v) = 1,v,|v|^2.
\end{aligned}
\end{equation}
Moreover, the Landau version of the Boltzmann $H$-theorem asserts that the entropy
$$
H(f) := \int_{\T^3 \times \R^3} f \, \log f  \, \dv \, \dx
$$
is non increasing. Indeed, at least formally, since $a_{ij}$ is nonnegative, we have the following inequality for the entropy dissipation $D(f)$:
\begin{multline*}
D(f) := -\frac{\d}{\dt} H(f) 
=\frac12 \int_{\T^3 \times \R^3 \times \R^3} a_{ij}(v-v_*) f(v) f(v_*) \\
\left( \frac{\partial_{v_i} f(v)}{f(v)} -  \frac{\partial_{v_i} f(v_*)}{f(v_*)} \right) 
\left( \frac{\partial_{v_j} f(v)}{f(v)} -  \frac{\partial_{v_j} f(v_*)}{f(v_*)} \right) \, \d v_* \, \d v \, \dx \geq 0.
\end{multline*}
The second part of the $H$-theorem asserts that local equilibria of the Landau equation are local Maxwellian distributions in velocity. In what follows, we shall consider the following centered normalized Maxwellian independent of time $t$ and space $x$ which is a global equilibrium of our equation defined by
	$$
	M(v):=\frac1{(2\pi)^\frac 32} e^{-\frac{|v|^2}2} .
	$$

Taking $\eps$ small has the effect of enhancing the role of collisions and thus when~$\eps$ goes to $0$, in view of the above mentioned Landau version of the Boltzmann~$H$-theorem, the solution looks more and more like a local thermodynamical equilibrium. 
As suggested in previous works (see for example~\cite{BGL1}), we consider the following rescaled Landau equation in which an additional dilatation of the macroscopic  time scale has been done in order to be able to reach the Navier-Stokes equation in the limit:
	\begin{equation}\label{eq:scaledLandau}
		\partial_t f^\eps + \frac1\eps v \cdot \nabla_x f^\eps 
		= \frac1{\eps^2} Q(f^\eps,f^\eps) \quad \mbox{in} \quad \R^+ \times  \T^3 \times  \R^3 \, .
	\end{equation}
To relate the Landau equation to the incompressible Navier-Stokes equation, we look at equation~\eqref{eq:scaledLandau} under the following linearization of order $\eps$:
	\begin{equation} \label{eq:linearization}
		f^\eps(t,x,v)= M(v) + \eps \sqrt{M}(v) g^\eps(t,x,v).
	\end{equation}
Let us recall that taking $\eps$ small in this linearization corresponds to taking a small Mach number, which enables one to get in the limit the incompressible Navier-Stokes equation. If~$f^\eps$ solves~\eqref{eq:scaledLandau}, then equivalently~$g^\eps$ solves 
	\begin{equation}\label{eq:geps}
		\partial_t g^\eps + \frac1\eps v \cdot\nabla_x g^\eps 
		= \frac1{\eps^2}  Lg^\eps +  \frac1\eps\Gamma(g^\eps,g^\eps) \quad \mbox{in} \quad \R^+ \times  \T^3\times \R^3\, 
	\end{equation}
where the nonlinear collision operator $\Gamma$ is defined by
	\begin{equation} \label{def:operatorGamma}
	\begin{aligned}
		\Gamma (f_1,f_2)  :=\frac{1}{\sqrt{M}} Q\left(\sqrt{M} f_1,\sqrt{M} f_2\right) 
	\end{aligned}
	\end{equation}
and the linearized collision operator $L$ by
\begin{equation} \label{def:operatorL}
	\begin{aligned} 
		Lf  :=\Gamma\left(\sqrt{M},f\right) + \Gamma\left(f,\sqrt{M}\right).
	\end{aligned}
	\end{equation}
Notice that the property~\eqref{prop:conserv} implies that for any suitable functions $f_1$ and $f_2$,
\begin{equation} \label{prop:conservbis}
\int_{\R^3} \Gamma(f_1,f_2)(v) \, \varphi(v) \, \dv = 0 
\quad \text{for} \quad
\varphi(v) = \sqrt{M}, v \sqrt{M}, |v|^2 \sqrt{M}.
\end{equation}
We also define the full linearized operator $\Lambda_\eps$ as 
	\begin{equation} \label{def:Lambdaeps}
		\Lambda_\eps  := \frac{1}{\eps^2}L - \frac{1}{\eps} v \cdot \nabla_x.
	\end{equation}

It is well known that the kernel of $L$ is given by 
	$$
	\operatorname{Ker} L = \hbox{Span} \left\{\sqrt{M}, v_1\sqrt{M}, v_2 \sqrt{M}, v_3\sqrt{M}, |v|^2\sqrt{M}\right\} ,
	$$ 
and we shall denote by $\pi$ the orthogonal projector onto $\operatorname{Ker} L$ which is defined by: 
	\begin{multline} \label{def:pi}
		\pi f(v) = \Bigg(\int_{\R^3} f(w) \sqrt{M}(w) \, \d w + \int_{\R^3} w f(w)  \sqrt{M}(w) \, \d w \cdot v \\
		+ \int_{\R^3} \frac{|w|^2-3}{3} f(w) \sqrt{M}(w) \, \d w \, \frac{|v|^2-3}{2} \Bigg) \sqrt{M}(v). 
	\end{multline}

Throughout the paper, we shall also use the following notation: For a given kinetic distribution $f=f(x,v)$, we denote by $f^\perp$ its microscopic part, namely 
	\begin{equation} \label{def:micro}
		f^\perp := (\Id - \pi) f
	\end{equation}
and by $(\rho_f,u_f,\theta_f)$ its first macroscopic quantities defined through
	\begin{equation}\label{rhof}
		\rho_{f}(x):= \int_{\R^3} f (x,v)\sqrt{M}(v)  \, \d v  ,
	\end{equation}
	\begin{equation}\label{uf}
		u_{f} (x):= \int_{\R^3} v\,  f (x,v)\sqrt{M}(v)  \, \d v , 
	\end{equation}
and
	\begin{equation}\label{thetaf}
		\theta_{f} (x):= \frac{1}{3} \int_{\R^3} (|v|^2-3) f (x,v)\sqrt{M}(v)  \, \d v  ,
	\end{equation}
so that $\pi f = \left( \rho_f  + u_f \cdot v + \theta_f \frac{|v|^2-3}{2} \right)\sqrt M$.
Using for example Proposition~3.1 from~\cite{Briant}, we also have:
	$$
	\operatorname{Ker} \Lambda_\eps 
	= \hbox{Span} \left\{\sqrt{M}, v_1\sqrt{M}, v_2\sqrt{M}, v_3\sqrt{M}, |v|^2\sqrt{M}\right\} 
	$$ 
and the projector $\Pi$ onto $\operatorname{Ker} \Lambda_\eps$ is given by
	\begin{multline}\label{def:Pi}
		\Pi f(v) = \Bigg(\int_{\T^3 \times\R^3} f(x,w) \sqrt{M}(w) \,  \d w \, \d x 
		+ \int_{\T^3 \times\R^3} f(x,w) w \sqrt{M}(w) \, \d w \, \d x \cdot v \\
		+ \int_{\T^3 \times \R^3} \frac{|w|^2-3}{3} f(x,w) \sqrt{M}(w) \, \d w \, \d x \, \frac{|v|^2-3}{2} \Bigg) \sqrt{M}(v). 
	\end{multline}
Notice that $\displaystyle{\Pi f(v) = \int_{\T^3} \pi f(x,v) \, \d x}$. 

\subsection{Cauchy theory, decay and regularization for the Landau equation}

We introduce the following $H^1$-norm in velocity which naturally arises in the study of the Landau equation:
	\begin{equation} \label{def:H1v*}
		\|f\|^2_{H^1_{v,*}} :=
		\Vert \langle v \rangle^{\frac{\gamma}{2}+1}  f \Vert_{L^{2}_{v}}^{2}
		+\Vert {\langle v \rangle}^{\frac{\gamma}{2}}P_{v}\nabla_{v}f \Vert_{L^{2}_{v}}^{2}
		+\Vert {\langle v \rangle}^{\frac{\gamma}{2}+1}(\Id-P_{v})\nabla_{v}f \Vert_{L^{2}_{v}}^{2},
	\end{equation}
where $P_v$ stands for the projection on $v$, namely, $P_v w = \left(w\cdot\frac{v}{\vert v \vert}\right)\frac{v}{\vert v \vert}$. 
We define the weighted Sobolev-type spaces $\XXX$ and $\YYY_1$ as the spaces associated to the following norms: 
	\begin{equation}\label{def:normXXX}
	\| f \|_{\XXX}^2 := \| \langle v \rangle^{3(\frac{\gamma}{2}+1)} f \|_{L^2_{x,v}}^2 
	+ \| \langle v \rangle^{2(\frac{\gamma}{2}+1)} \nabla_x f \|_{L^2_{x,v}}^2 
	+ \| \langle v \rangle^{\frac{\gamma}{2}+1} \nabla^2_x f \|_{L^2_{x,v}}^2 
	+ \| \nabla^3_x f \|_{L^2_{x,v}}^2 ,
	\end{equation}
and
	\begin{equation}\label{def:normYYY1}
	\begin{aligned}
	\| f \|_{\YYY_1}^2 
	&:= \|  \langle v \rangle^{3(\frac{\gamma}{2}+1)}  f \|_{L^2_x (H^1_{v,*})}^2 
	+ \| \langle v \rangle^{2(\frac{\gamma}{2}+1)} \nabla_x f \|_{L^2_x (H^1_{v,*})}^2 \\
	&\quad	+ \| \langle v \rangle^{\frac{\gamma}{2}+1} \nabla^2_x f \|_{L^2_x (H^1_{v,*})}^2 
	+ \| \nabla^3_x f \|_{L^2_x (H^1_{v,*})}^2.
	\end{aligned}
	\end{equation}
Remark that as in~\cite{GuoLandau,CTW}, we work with ``twisted'' Sobolev spaces in which the weights depend on the order of the derivative in $x$, it allows us to close our nonlinear estimates.

\medskip

Let us now state our main result on the well-posedness, decay and regularization of the rescaled Landau equation~\eqref{eq:geps}.
\begin{theo}\label{theo:main1}
There is $\eta_0>0$ small enough such that for any $\eps \in (0,1)$, 
if $g_{\rm in}^\eps \in \XXX$ satisfies
\begin{equation} \label{eq:DI}
\int_{\T^3 \times \R^3} g_{\rm in}^\eps(x,v) \, \varphi(v) \, \dv \, \dx = 0 
\quad \text{for} \quad
\varphi(v) = \sqrt{M}, v \sqrt{M}, |v|^2 \sqrt{M},
\end{equation} 
and $\| g_{\mathrm{in}}^\eps \|_{\XXX} \le \eta_0$, then the following holds:

\medskip\noindent
(i) There is a unique global solution $ g^\eps \in L^\infty (\R_+ ; \XXX) \cap  L^2(\R_+ ; \YYY_1)$ to \eqref{eq:geps} associated to the initial data $g_{\mathrm{in}}^\eps$, which verifies moreover 
	\begin{equation}\label{eq:theo:main1:Linftybound}
	\sup_{t \ge 0} e^{2\sigma t} \| g^\eps(t) \|_{\XXX}^2 
	+ \frac{1}{\eps^2}\int_0^\infty e^{2\sigma t}\| (g^\eps(t))^\perp \|_{\YYY_1}^2 \, \dt 
	+ \int_0^\infty e^{2\sigma t}\|  g^\eps(t) \|_{\YYY_1}^2 \, \dt
	\lesssim \| g_{\mathrm{in}}^\eps \|_{\XXX}^2 ,
	\end{equation}
for any $0 < \sigma < \sigma_0$, where $\sigma_0$ is the decay rate of linearized operator $\Lambda_\eps$ given in Proposition~\ref{prop:hypoL2}, and where we recall that $(g^\eps)^\perp$ is defined in~\eqref{def:micro}.

\medskip\noindent
(ii) In addition, the solution satisfies the following regularization estimates, for all $t >0$,
	\begin{equation}\label{eq:theo:main1:regbound}
	\|g^\eps(t)\|_{\YYY_1} \lesssim \frac{e^{-\sigma t}}{\min(1,\sqrt{t})} \|g_{\rm in}^\eps\|_\XXX,
\quad \text{and} \quad 
\eps \|\widetilde \nabla_x g^\eps(t)\|_{\XXX} \lesssim \frac{e^{-\sigma t}}{\min(1,t^{3/2})} \|g_{\rm in}^\eps\|_\XXX,
	\end{equation}
where $\widetilde \nabla_x$ is a weighted anisotropic gradient defined in \eqref{def:DvDx}.
\end{theo}

\begin{rem}
It is worth noticing that the condition~\eqref{eq:DI} is equivalent to $g_{\rm in}^\eps \in (\operatorname{Ker} \Lambda_\eps)^\perp$. 
\end{rem}

\begin{rem}\label{rem1}
Let us point out that the results obtained in Theorems~\ref{theo:main1} could be obtained in larger spaces of the type $\EEE := H^3_xL^2_v(\langle v \rangle^k \sqrt{M})$ for $k$ large enough. More precisely, due to the linearization~\eqref{eq:linearization}, working in spaces like $\EEE$ means that the original data $f^\eps$ lie in polynomially weighted Sobolev spaces, which is more relevant from a physical point of view. We chose to only present the proof in the functional space $\XXX$ because this functional framework is compatible with the second part of the paper which is about hydrodynamical limit of the Landau equation~\eqref{eq:geps}. Let us though explain the strategy to perform such an extension of our results from the functional space $\XXX$ to $\EEE$. The strategy is the same as the one used in~\cite{BMM} by Briant, Merino-Aceituno and Mouhot where uniform in $\eps$ estimates on solutions to the hard-spheres Boltzmann equation have been obtained. The trick is to rewrite the equation~\eqref{eq:geps} as an equivalent system of two equations thanks to the splitting of the linearized operator~$\Lambda_\eps=\AA_\eps + \BB_\eps$ introduced in Section~\ref{sec:linear2}: We write $g^\eps=g^\eps_1+g^\eps_2$ with 
	$$
	\partial_t g^\eps_1 = \BB_\eps g^\eps_1 
	+ \frac{1}{\eps} \Gamma(g^\eps_1,g^\eps_1) 
	+ \frac{1}{\eps} \Gamma(g^\eps_1,g^\eps_2) 
	+ \frac{1}{\eps} \Gamma(g^\eps_2,g^\eps_1) 
	\quad \text{and} \quad g^\eps_1(t=0) = g_{\rm in}^\eps \in \EEE
	$$
and
	$$
	\partial_t g^\eps_2 = \Lambda_\eps g^\eps_2 + \frac{1}{\eps} \Gamma(g^\eps_2,g^\eps_2) + \AA_\eps g^\eps_1
	\quad \text{and} \quad g^\eps_2(t=0) = 0 \in \XXX. 
	$$
The first equation can be studied in the large space $\EEE$ thanks to the nice properties of~$\BB_\eps$ in all type of spaces and the second one can be studied in the smaller space $\XXX$ since it starts from $0$. Moreover, we have some nice estimates on this equation because the operator $\AA_\eps$ enjoys some regularizing properties, it is bounded from $\EEE$ into $\XXX$, we can thus use the estimates obtained for the first equation satisfied by~$g^\eps_1$. Following those ideas, one can obtain some nice nested a priori estimates on the system, which allow to conclude. 
\end{rem}

\begin{rem}\label{rem2}
Our method should be robust enough to also treat the case of very soft and Coulomb potentials $-3 \le \gamma < - 2$ in which the linearized operator does not have a spectral gap (in this case, the inequality~\eqref{eq:spectralgap} does not provide coercivity anymore). More precisely, we should be able to obtain a similar result of Theorem~\ref{theo:main1} with the exponential time-decay in \eqref{eq:theo:main1:Linftybound} being replaced by a sub-exponential one, by combining the arguments developed in this paper together with the study made in the case $\eps=1$ in~\cite{CM} by Carrapatoso and Mischler. We do not treat this case in the present paper. 
\end{rem}

\medskip

The Cauchy theory and the large-time behavior of the Landau equation for $\eps=1$ have been extensively studied. We here give a small sample of the existing literature: Let us mention~\cite{Alexandre-Villani} for renormalized solutions with defect measure,~\cite{Desvillettes-Villani} for the convergence to equilibrium for a priori smooth solutions with general initial data,~\cite{GuoLandau,Mouhot-Neumann,CTW,CM} for strong solutions in a perturbative framework. 
Concerning the well-posedness of our rescaled equation, it has already been obtained in~\cite{GuoBNS,Rachid1} respectively by Guo and Rachid in Sobolev spaces (respectively in $H^N_{x,v}$ with~$N \geq 8$ and in $H^3_xL^2_v$) thanks to nonlinear energy methods. In~\cite{Briant}, Briant has obtained a similar result in $H^N_{x,v}$ with~$N \geq 4$ thanks to the so-called $H^1$-hypocoercivity method at the linear level. 

\smallskip

Our global strategy to prove Theorem~\ref{theo:main1}-(i) is based on the study of the linearized equation. And then, we go back to the fully nonlinear problem. This is a standard strategy to develop a Cauchy theory in a close-to-equilibrium regime. However, we have to emphasize here that our study is quite involved as explained below.

\smallskip

At the linear level, our strategy is based on a $L^2$-hypocoercivity method which heavily relies on the micro-macro decomposition and is thus particularly adapted to the study of hydrodynamical problems. Recall that the challenge of hypocoercivity is to understand the interplay between the collision operator that provides dissipativity in the velocity variable and the transport one which is conservative, in order to obtain global dissipativity for the whole linear problem (see~\cite{Villani-hypo,HerauCours} for a presentation of this topic). 
The $L^2$-hypocoercivity method has been introduced by Hérau~\cite{Herau-L2} (see also~\cite{Dolbeault-Mouhot-Schmeiser}) for one dimensional space of collisional invariants and introduced by Guo in~\cite{Guo-bounded} for a space of collisional invariants of dimension larger than one (including the Boltzmann and Landau cases). Let us explain into more details the strategy, we first define a norm $\Nt \cdot \Nt_{L^2_{x,v}}$ (associated to the scalar product $\la\!\la\cdot,\cdot\ra\!\ra_{L^2_{x,v}}$) which is equivalent to the usual one $\|\cdot\|_{L^2_{x,v}}$ uniformly in $\eps$ and is such that 
	$$
	\la\!\la \Lambda_\eps f, f \ra\!\ra_{L^2_{x,v}} \lesssim -\frac{1}{\eps^2} \|f^\perp\|^2_{L^2_x (H^1_{v,*})} - \|f\|^2_{L^2_{x,v}}.
	$$
Such a norm is defined in Subsection~\ref{subsec:hypo} and is inspired by~\cite{Guo-bounded} (see also~\cite{BCMT}) in which the more complex case of bounded domains with various boundary conditions is treated. Due to the fact that derivatives in $x$ commute with $\Lambda_\eps$, it is easy to deduce a similar result on the space $H^3_xL^2_v$. However, it is not an easy matter to recover such an energy estimate in larger or smaller spaces than $H^3_xL^2_v$. Actually, the methods presented in~\cite{GMM,MM} by Gualdani, Mischler and Mouhot to develop shrinkage or enlargement arguments at the level of energy estimates is not easily adaptable to rescaled equations if one wants to get uniform estimates in the parameter of rescaling. To develop an enlargement argument, one can use the trick introduced in~\cite{GMM,MM} of splitting the original equation into several ones. This trick was already used in~\cite{BMM,ALT} to obtain uniform in $\eps$ estimates on the rescaled Boltzmann equation for respectively elastic and inelastic hard spheres in a large class of Sobolev spaces (see also Remark~\ref{rem1}). However, we do not have such a method of splitting to perform a ``shrinkage'' argument.
In the present paper, we exhibit a norm equivalent to the usual one that provides dissipativity for $\Lambda_\eps$ in a smaller space than $H^3_xL^2_v$ (namely in the space $\XXX$ defined in~\eqref{def:normXXX}) and that also preserves the gain of~$1/\eps$ on the microscopic part of the solution. This is done in Subsection~\ref{subsec:hypo}. 
Notice that it is also possible to obtain decay estimates directly on the semigroup associated to $\Lambda_\eps$ thanks to Duhamel formula once one has exhibited a nice splitting of $\Lambda_\eps$ (see Section~\ref{sec:linear2}). 

\smallskip
We then prove some new and sharp nonlinear estimates on the Landau collision operator (see Subsection~\ref{subsec:nonlinearestim}) to be able to develop our Cauchy theory for the whole nonlinear problem in a close-to-equilibrium framework. It is worth mentioning that to prove good a priori estimates on the nonlinear problem, we use the hypocoercive norm defined in Subsection~\ref{subsec:hypo} and we only perform energy estimates. It is actually important that our analysis does not rely on the use of Duhamel formula because of the rescaling parameter (see the beginning of Section~\ref{sec:Cauchyreg} for more details). 

\smallskip
The strategy that we use to prove the regularization estimate in Theorem~\ref{theo:main1}-(ii) is quite classical for linear hypoelliptic equations and has been introduced by H\'erau and Nier~\cite{Herau-Nier} for the kinetic Fokker-Planck equation. Such a method has been used for many hypoelliptic equations: In~\cite{HTT2} for the fractional kinetic Fokker-Planck equation, in~\cite{CTW,CM} for the linearized Landau equation, in~\cite{HTT} for the linearized Boltzmann equation without cutoff etc... To our knowledge, it is the first time that such a strategy is used for a nonlinear equation (even in the simpler case $\eps=1$) and for a rescaled equation (with uniform estimates in the rescaling parameter). Roughly speaking, the idea is to introduce a functional with weights in time which is a Lyapunov functional for our equation for small times. From this property, we are then able to recover some regularization estimates quantified in time as stated in Theorem~\ref{theo:main1}-(ii). Here, the difficulties are threefold: 
\begin{enumerate}[itemsep=5pt,leftmargin=*]
\item[-] First, we study a nonlinear equation, our computations are thus much more intricate, the idea behind our computations being that we work with small data which allows us to absorbe the nonlinearity. Our proof also requires some new and sharp nonlinear estimates on the collision operator (see Subsection~\ref{subsec:nonlinearestim}). 
\item[-] Then, the functional has to be suitably defined to handle the dependencies in $\eps$. The differences of behaviors between microscopic and macroscopic parts of the solution have to be taken into account and in the spirit of the definition of the $H^1$-hypocoercive norm of Briant~\cite{Briant}, some terms of the functional only involve the microscopic part of the solution (see~\eqref{eq:Lyapunov} for the definition of the functional). 
\item[-] Finally, since we want to obtain the optimal gain of regularity (the corresponding hypoelliptic estimates are provided in~\cite{Rachid1} by Rachid), our functional has to be defined accordingly. For example, in~\cite{CTW} in which the authors were not interested in getting the optimal gain of regularity (and in which only the case $\eps=1$ was treated), only classical derivation operators were involved in the definition of the functional. Here, the definition of the functional is much more intricate: We have to work with the anisotropic operators~$\widetilde \nabla_v$ and $\widetilde \nabla_x$ defined in~\eqref{def:DvDx} and our functional also involves additional terms which are necessary to close our estimates. 
\end{enumerate}
To end this part, we mention that our proof also provides a regularization estimate in the space variable which is not uniform in $\eps$ (see~\eqref{eq:theo:main1:regbound}). The non-uniformity in $\eps$ of such a gain is expected since the transport operator and the linearized collision operator (which gives the gain of regularity in velocity) do not act at the same scale.

\subsection{The fluid model}

In the second part of the paper, we shall prove that the hydrodynamical limit of~\eqref{eq:scaledLandau} as~$\eps$ goes to zero is the incompressible Navier-Stokes-Fourier system associated with the Boussinesq equation which writes
	\begin{equation}\label{eq:NSF}
	\left\{
		\begin{aligned}
			\partial_t u + u \cdot \nabla_x u- \nu_1 \Delta_x u &= - \nabla_x p \\
			\partial_t \theta  + u \cdot \nabla_x \theta- \nu_2 \Delta_x \theta &= 0 \\
			\Div_x u & = 0 \\
			\nabla_x(\rho+\theta) &= 0  .
		\end{aligned}
	\right.
	\end{equation}
In this system, the temperature~$\theta$, the density~$\rho$ and the pressure~$p$ are scalar unknowns, whereas the velocity~$u$ is an unknown vector field. The pressure can actually be eliminated from the equations by applying to the momentum equation the projector~${\mathbb P}$ onto the space of divergence-free vector fields. This projector is bounded over~$H^\ell_x$ for all~$\ell$, and in~$L^p_x$ for all~$1<p<\infty$.  To define the viscosity coefficients $\nu_i$, let us introduce the two unique functions $\Phi$ (which is a matrix-valued function) and $\Psi$ (which is a vector-valued function) orthogonal to~$\operatorname{Ker} L$ such that 
	$$
	\frac{1}{\sqrt{M}} L\left(\sqrt{M} \Phi\right) =  \frac{|v|^2}{3} {\rm{Id}} -v\otimes v 
	\quad \text{and} \quad 
	\frac{1}{\sqrt{M}}L\left(\sqrt{M} \Psi\right) =\frac{5-|v|^2}{2} \, v.
	$$
The viscosity coefficients are then defined by
	$$ 
	\nu_1 :=\frac{1}{10}\int_{\R^3} \Phi : L\left(\sqrt{M}\Phi\right) \sqrt{M} \, \d v \quad \text{and} \quad
	\nu_2  :=\frac{2}{15} \int_{\R^3} \Psi \cdot L\left(\sqrt{M}\Psi\right) \sqrt{M} \, \d v.
	$$

In what follows, we call {\em well-prepared data} the class of functions $f \in \operatorname{Ker} L$ that write 
	\begin{multline} \label{def:wp}
		f(x,v) = \sqrt{M}(v) \left(\rho_f(x) + u_f(x) \cdot v + \frac{|v|^2-3}{2} \theta_f(x)\right) \\
		\text{with} \quad \nabla_x \cdot u_f = 0 \quad \text{and} \quad \rho_f+\theta_f=0
	\end{multline}
where we recall that $\rho_f$, $u_f$ and $\theta_f$ are defined in~\eqref{rhof},~\eqref{uf},~\eqref{thetaf}. 

It is known that for mean free $(\rho_0,u_0,\theta_0) \in H^3_x$ small enough and satisfying
	\begin{equation} \label{eq:divBoussinesq}
	\nabla_x \cdot u_0 = 0 \quad \text{and} \quad \rho_0+\theta_0=0,
	\end{equation}
there exists a unique solution $(\rho,u,\theta) \in H^3_x$ to~\eqref{eq:NSF} defined on $\R^+$ with associated initial data~$(\rho_0,u_0,\theta_0)$ (see~\cite{Fujita-Kato,LR2,LR1,Gallagher-Tristani}). For such an initial data, we also define $g_0$ a well-prepared data with $(\rho_0,u_0,\theta_0)$ as associated first macroscopic quantities, namely
	\begin{multline} \label{def:g0}
		g_0(x,v) := \sqrt{M}(v) \left(\rho_0(x) + u_0(x) \cdot v + \frac{|v|^2-3}{2} \theta_0(x)\right) \\
		\text{with} \quad \nabla_x \cdot u_0 = 0 \quad \text{and} \quad \rho_0+\theta_0=0. 
	\end{multline}
Notice that from the definition of the space $\XXX$, we in particular have that $g_0 \in \XXX$ and the mean-free assumption made on $(\rho_0,u_0,\theta_0)$ implies that $g_0 \in (\operatorname{Ker} \Lambda_\eps)^\perp$. 
Notice also that due to the definition of $\XXX$, the smallness assumption made on $(\rho_0,u_0,\theta_0)$ can be translated into a smallness assumption on $g_0$. Indeed, given the form of $g_0$, by triangular inequality, it is clear that 
	$$
	\|g_0\|_\XXX \lesssim \|(\rho_0,u_0,\theta_0)\|_{H^3_x}.
	$$
Moreover, since $\left\{\sqrt{M}, \sqrt{M} v_i, \sqrt{M} (|v|^2-3)/2\right\}$ is an orthogonal system in $L^2_v$, we also have that 
	$$
	\|g_0\|_\XXX \gtrsim \|g_0\|_{H^3_xL^2_v} \gtrsim \|(\rho_0,u_0,\theta_0)\|_{H^3_x}.
	$$
As a consequence, there exists $\eta_1>0$ such that if $g_0$ is of the form~\eqref{def:g0} and satisfies~$\|g_0\|_\XXX \leq \eta_1$, then there exists $(\rho,u,\theta) \in H^3_x$ defined on $\R^+$ solution to~\eqref{eq:NSF}.
We define the kinetic distribution lying in~$\operatorname{Ker} L$ with associated macroscopic quantities~$(\rho,u,\theta)$
	\begin{equation} \label{def:g}
		g(t,x,v) := \sqrt{M}(v)\left(\rho (t,x) + u (t,x) \cdot v + \frac{|v|^2- 3}{2} \theta(t,x) \right) .
	\end{equation}
We also have the following estimate
	\begin{equation} \label{eq:expNSF}
	\|g\|_{L^\infty_t(\XXX)} \lesssim C(\|g_{0}\|_\XXX)
	\end{equation}
where $C(\|g_{0}\|_\XXX)$ is a constant only depending on the $\XXX$-norm of the data $g_0$. 
The aforementioned results on the system~\eqref{eq:NSF} can be found in~\cite[Appendix~B.3]{Gallagher-Tristani} in which more details and references on the subject are given.

\subsection{Hydrodynamical limit result} 

For the statement of the main hydrodynamical result, we first introduce the following notation for functional spaces: If~$X_1$ and~$X_2$ are two function spaces, we say that a function~$f$ belongs to~$X_1+X_2$ if there are~$f_1 \in X_1$ and~$f_2\in X_2$ such that~$f =f_1+f_2$ and we define
	$$
	\|f\|_{X_1+X_2} := \min_{\tiny\begin{array}{c}f = f_1+f_2\\f_i \in X_i\end{array}}\left(\|f_1\|_{X_1} + \|f_2\|_{X_2} \right) .
	$$

\begin{theo} \label{theo:main2}
Let $g_{\rm in}^\eps \in \XXX \cap (\operatorname{Ker} \Lambda_\eps)^\perp$ for $\eps \in (0,1)$ such that $\|g_{\rm in}^\eps\|_\XXX \leq \eta_0$ (where~$\eta_0$ is defined in Theorem~\ref{theo:main1}) and~$g^\eps \in L^\infty_t(\XXX)$ being the associated solutions of~\eqref{eq:geps} with initial data $g_{\rm in}^\eps$ constructed in Theorem~\ref{theo:main1}.
Consider also $g_0 \in \XXX \cap (\operatorname{Ker} \Lambda_\eps)^\perp$ such that~$\|g_0\|_\XXX \leq \eta_1$ and $g$ defined respectively as in~\eqref{def:g0} and~\eqref{def:g}. 

There exists $\eta_2 \in (0,\min(\eta_0,\eta_1))$ such that if $\max\left(\|g_{\rm in}^\eps\|_\XXX,\|g_0\|_\XXX\right) \leq \eta_2$ and
	\begin{equation} \label{eq:DIcv1}
	\|g_{\rm in}^\eps - g_0\|_\XXX \xrightarrow[\eps \to 0]{} 0,
	\end{equation}
then we have  
	\begin{equation} \label{eq:cv1}
		\|g^\eps-g\|_{L^\infty_t(\XXX)} \xrightarrow[\eps \to 0]{} 0.
	\end{equation}
If $\max\left(\|g_{\rm in}^\eps\|_\XXX,\|g_0\|_\XXX\right) \leq \eta_2$ and
	\begin{equation} \label{eq:DIcv2}
	\|\pi g_{\rm in}^\eps - g_0\|_\XXX \xrightarrow[\eps \to 0]{} 0 ,
	\end{equation}
then we have 
	\begin{equation} \label{eq:cv2}
		\|g^\eps-g\|_{L^1_t(\YYY_1) + L^\infty_t(\XXX)} \xrightarrow[\eps \to 0]{} 0.
	\end{equation}
\end{theo}
	
\begin{rem}
One can get a rate of convergence in~\eqref{eq:cv1} and
~\eqref{eq:cv2} if we suppose that $g_0$ has some additional regularity in $x$, namely a rate of $\eps^\delta$ if the regularity is supposed to be~$H^{3+\delta}_x$ for $\delta\in(0,1/2]$. We refer to Theorem~\ref{theo:main2precise} for a quantitative version of this result. 
\end{rem}

\begin{rem}
As explained in Remark~\ref{rem1}, the results of Theorem~\ref{theo:main1}~could be obtained in larger spaces $\EEE = H^3_x L^2_v (\langle v \rangle^k \sqrt M)$. A similar approach as the one used by Gervais in~\cite{Gervais1,Gervais2} in which the hard spheres Boltzmann equation is treated in ``large'' Sobolev spaces, might yield the associated hydrodynamical result. 
\end{rem}

\begin{rem}
As explained in Remark~\ref{rem2}, the strategy of the proof of Theorem~\ref{theo:main1}~ should also work in order to treat the case of very soft and Coulomb potentials. However, in order to obtain the associated hydrodynamical result, our method employs some fine spectral estimates that are known to hold only for the case $-2 \le \gamma \le 1$ by \cite{Yang-Yu}. Therefore, if we are able to extend the results of \cite{Yang-Yu} to the case of very soft and Coulomb potentials, we might then be able to obtain the analogous result of Theorem~\ref{theo:main2}.
\end{rem}

\medskip

We first give a short overview of the existing literature on the problem of deriving fluid equations from kinetic ones (we refer to the book by Saint-Raymond~\cite{SRbook} for a thorough presentation of the topic). 
The first justifications of the link between kinetic and fluid equations were formal and based on asymptotic expansions by Hilbert, Chapman, Enskog and Grad (see~\cite{Hilbert,ChapEns,Gradhydro}). The first rigorous convergence proofs based also on asymptotic expansions were given by Caflisch~\cite{Caflisch} (see also~\cite{Lachowicz} and~\cite{DeMasi-Esposito-Lebowitz}). In those papers, the limit is justified up to the first singular time for the fluid equation. By using the nonlinear energy method introduced by himself in~\cite{GuoLandau}, Guo~\cite{GuoBNS} has justified the limit towards the Navier-Stokes equation and beyond in Hilbert's expansion from Boltzmann and Landau equations (see below for more details on this result).

\smallskip
There has also been some convergence proofs based on spectral analysis in the framework of strong solutions close to   equilibrium introduced by Grad~\cite{Grad} and Ukai~\cite{Ukai} for the Boltzmann equation. In this respect, we refer to the works by Nishida, Bardos and Ukai, Gallagher and Tristani~\cite{Nishida,Bardos-Ukai,Gallagher-Tristani}. These results use the description of the spectrum of the linearized Boltzmann equation in Fourier space in the space variable performed in~\cite{Nicolaenko,CIP,Ellis-Pinsky}. Our approach is reminiscent of these ones and relies on the generalization of the paper~\cite{Ellis-Pinsky} to several kinetic equations (including the Landau one) made in~\cite{Yang-Yu} by Yang and Yu. Notice also that such a spectral result has recently been obtained in~\cite{Gervais1} by Gervais for the hard-spheres Boltzmann equation in a larger class of Sobolev spaces. 

\smallskip
More recently, some uniform in $\eps$ estimates on kinetic equations have allowed to prove (at least) weak convergence towards the Navier-Stokes equation. Let us mention~\cite{Jiang-Xu-Zhao,Rachid2} in which the cases of Boltzmann equation without cutoff and the Landau equations are treated.
In~\cite{Briant,BMM}, the authors have obtained  convergence to equilibrium results for the rescaled Boltzmann equation (and also the Landau equation in~\cite{Briant}) uniformly in the rescaling parameter using respectively hypocoercivity and enlargement methods. In~\cite{BMM}, the authors are able to weaken the assumptions on the data down to Sobolev spaces with polynomial weights. We also refer to~\cite{ALT} in which a similar approach combined with perturbative arguments has been used to derive a fluid system from the inelastic Boltzmann equation. Notice that Briant~\cite{Briant} has combined this with Ellis and Pinsky result~\cite{Ellis-Pinsky} to recover strong convergence in the case of the Boltzmann equation.

\smallskip
Finally, let us mention that this problem has been extensively studied in the framework of weak solutions, the goal being to obtain solutions for the fluid models from renormalized solutions introduced by DiPerna and Lions in~\cite{DiPerna-Lions} for the Boltzmann equation. We shall not make an extensive presentation of this program as it is out of the realm of this study, but let us mention that it was started by Bardos, Golse and Levermore at the beginning of the nineties in~\cite{BGL1,BGL2} and was continued by those  authors, Saint-Raymond, Masmoudi, Lions among others. We mention here  a (non exhaustive) list of papers which are part of this program (see~\cite{GSR1,GSR2,Levermore-Masmoudi,Lions-Masmoudi,SRbook}).

\smallskip
Let us focus on the Landau equation for which the literature is scarcer. As mentioned above, in~\cite{GuoBNS}, Guo justifies the limit the Navier-Stokes limit and beyond in the Hilbert's expansion from (the Boltzmann and) the Landau equations (even for the case of very soft potentials) in the torus by using his nonlinear energy method. Our result on the hydrodynamical limit is reminiscent of the one obtained in~\cite{GuoBNS} for the hard and moderately soft potentials in the sense that we work with strong solutions and we prove a strong convergence result. It is however worth noticing that our functional framework is less restrictive (we only work with $3$ derivatives in $x$ and no derivative in $v$ whereas in~\cite{GuoBNS}, regularity on $8$ derivatives in both variables $x$ and $v$ is required). Moreover, there is an important loss of regularity in the estimates of convergence proven in~\cite{GuoBNS} whereas we only lose $\delta$ derivatives in $x$ and there is no loss in $v$ to get a rate of convergence of $\eps^\delta$ from Landau to Navier-Stokes equation (see Theorem~\ref{theo:main2precise}). In the present paper and in~\cite{GuoBNS}, the fluid initial data are supposed to be well-prepared, namely the divergence free condition and the Boussinesq relation~\eqref{eq:divBoussinesq} are supposed to hold. We refer to~\cite{Jiang-Xiong} by Jiang and Xiong for an extension to the case where the fluid part is not supposed to be well-prepared and the creation and propagation of initial layers is studied. In~\cite{GuoBNS,Jiang-Xiong}, the kinetic initial data is supposed to have a specific form so that there is no creation of kinetic initial layers. Our presentation is slightly different since we do not use Hilbert expansion to study the limit towards the Navier-Stokes equation, our assumption being the following: The projection of the kinetic initial data onto the kernel of the linearized operator $L$ is supposed to converge towards the well-prepared fluid initial data as $\eps \to 0$.
Finally, in~\cite{Rachid2}, Rachid obtained a result of weak-$\star$ convergence in $L^\infty_t(H^3_xL^2_v)$ towards the incompressible Navier-Stokes-Fourier system, we have thus strengthen this result for the type of initial data that we consider here. One can notice that the assumptions made on the fluid initial data in~\cite{Rachid2} are a bit less restrictive since the divergence free condition for $u_0$ and the Boussinesq relation for $\rho_0$ and $\theta_0$ are not supposed to hold. The initial layer that such an initial condition creates is absorbed there in the weak convergence. In our framework, we can not hope to absorbe it in a strong convergence framework because the initial layer is  propagated over time in the periodic domain (see ~\cite{Jiang-Xiong}). Note also that in~\cite{Gallagher-Tristani}, the authors were able to treat this type of ``completely ill-prepared'' data only in the case of the whole space since those terms has some dispersive properties in the whole space.

\smallskip

Let us describe into more details our strategy to obtain strong convergence. It is inspired by the ones used in~\cite{Bardos-Ukai,Briant,Gallagher-Tristani}. Indeed, as in~\cite{Gallagher-Tristani}, using the spectral analysis performed in~\cite{Yang-Yu} by Yang and Yu, in order to prove our main convergence result, we reformulate the fluid equation in a kinetic fashion and we then study the equation satisfied by the difference satisfied between the kinetic and the fluid solutions. However, let us point out that we are not able to perform a fixed point argument as in the aforementioned paper. This is due to the fact that the structure of the Landau bilinear operator is more complicated than the hard-spheres Boltzmann one. Indeed, there is an anisotropic loss of derivatives and weights in the nonlinear estimates which prevent us from closing a fixed point estimate. To circumvent this difficulty, we use some new a priori estimates on the solution of the linearized rescaled Landau equation and on the nonlinear rescaled Landau equation~\eqref{eq:geps} that are uniform in the Knudsen number and that have been presented in Theorem~\ref{theo:main1}. By intertwining these refined and sharp kinetic estimates and fluid mechanics ones, we are able to prove a result of strong convergence from the solutions of the Landau equation to the Navier-Stokes one as stated in Theorem~\ref{theo:main2}.

\subsection{Outline of the paper} In Section~\ref{sec:preliminaries}, we give some technical results on the Landau collision operator that will be used all along the paper. In Section~\ref{sec:linear}, we develop hypocoercivity and regularization estimates for the linearized problem. In Section~\ref{sec:Cauchyreg}, we develop our perturbative Cauchy theory for the whole nonlinear problem as well as some regularization estimates on it. In Section~\ref{sec:linear2}, we develop some new estimates on the linearized problem that are useful to prove our hydrodynamical result, which is proven in Section~\ref{sec:hydro}. 

\medskip
\noindent{\bf Acknowledgements.} The authors thank Frédéric Hérau for his valuable advice and fruitful discussions. This work has been partially supported by the Projects EFI: ANR-17-CE40-0030 (K.C.\ and I.T.) and SALVE: ANR-19-CE40-0004 (I.T.) of the French National Research Agency (ANR). 

\section{Preliminary results} \label{sec:preliminaries}

In this section, we present some technical results and tools that will be useful throughout the paper. 

\subsection{Collision operator} 

Recalling the definition of the matrix $a=(a_{ij})_{ij}$ in \eqref{eq:aij} and that we use the convention of summation of repeated indices through the paper, we define the following quantities
	\begin{align*}
		& b_i(v) = \partial_j a_{ij}(v) = - 2 \, |v|^\gamma \, v_i, \\
		& c(v) =  \partial_{ij} a_{ij}(v)  = - 2 (\gamma+3) \, |v|^\gamma ,
	\end{align*}		
in such a way that one can rewrite the Landau operator \eqref{eq:oplandau0} as
	\begin{equation}\label{eq:oplandau}
	\begin{aligned}
		Q(g,f) 
		&= ( a_{ij}* g) \partial_{v_i, v_j} f - (c* g) f\\
		&= \partial_{v_i}  \left\{ (a_{ij} * g) \partial_{v_j} f - (b_i * g)f \right\},
	\end{aligned}
	\end{equation}
where $*$ denotes the convolution in the velocity variable $v$.

We now state a technical lemma in which we provide a rewriting of the nonlinear operator~$\Gamma$ defined in \eqref{def:operatorGamma} and the linearized collision operator $L$ defined in \eqref{def:operatorL}.
	\begin{lem}
		There holds
		\begin{equation}\label{eq:Gamma}
		\begin{aligned}
			\Gamma(f_1,f_2)
			& = \partial_{v_i} \left\{ \left(a_{ij}* [\sqrt M f_1] \right) \partial_{v_j} f_2   \right\}
			-\partial_{v_i} \left\{ \left(b_{i}* [\sqrt M f_1] \right) f_2   \right\} \\
			&\quad
			-  \left( a_{ij}* [\sqrt M f_1]\right) v_i \partial_{v_j} f_2 
			+  \frac14 \left( a_{ij}* [\sqrt M f_1]\right) v_i v_j f_2 \\
			&\quad
			-  \frac12 \left( a_{ii}* [\sqrt M f_1]\right) f_2 ,
		\end{aligned}
		\end{equation}
		and 
		\begin{equation}\label{eq:L}
		\begin{aligned}
			L f
			& = \partial_{v_i} \left\{ \left(a_{ij}* M \right) \partial_{v_j} f  \right\}
			+ \left\{ -\frac14 \left( a_{ij}*M  \right)v_i v_j + \frac12  \partial_{v_i} [ \left( a_{ij}*M \right) v_j]   \right\} f \\
			&\quad 
			+ \left(a_{ij}* [\sqrt M f] \right) v_i v_j \sqrt M  
			- \left(a_{ii}* [\sqrt M f] \right) \sqrt M
			- \left( c* [\sqrt M f] \right) \sqrt M .
		\end{aligned}
		\end{equation}
	\end{lem}

	\begin{proof}
		From the definition of $\Gamma$ in \eqref{def:operatorGamma} and using the formulation \eqref{eq:oplandau} for $Q$, we first obtain 
		\begin{align*}
			\Gamma(g_1, g_2) 
                        &= \frac{1}{\sqrt M} \partial_{v_i} \left\{ \left(a_{ij} * [\sqrt M g_1] \right) \partial_{v_j} (\sqrt M g_2) 
                        - \left(b_i* [\sqrt M g_1] \right) \sqrt M g_2
                        \right\} .
 		\end{align*}
By writing $\partial_{v_j} (\sqrt M g_2)  = \sqrt M \partial_{v_j} g_2 - \frac12 v_j \sqrt M g_2$ we thus obtain
\begin{align*}
			\Gamma(g_1, g_2) 
                        &= \frac{1}{\sqrt M} \partial_{v_i} \left\{ \left(a_{ij} * [\sqrt M g_1] \right) \sqrt M \partial_{v_j} g_2 \right\}
                        -\frac12 \frac{1}{\sqrt M} \partial_{v_i} \left\{ \left(a_{ij} * [\sqrt M g_1] \right) v_j \sqrt M g_2 
                        \right\} \\
                        &\quad 
                        - \frac{1}{\sqrt M} \partial_{v_i} \left\{  \left(b_i* [\sqrt M g_1] \right) \sqrt M g_2
                        \right\} .
 		\end{align*}
Applying the derivative $\partial_{v_i}$ inside the brackets to the term $\sqrt M$ we then get
\begin{align*}
			\Gamma(g_1, g_2) 
                        &= \partial_{v_i} \left\{ \left(a_{ij} * [\sqrt M g_1] \right) \partial_{v_j} g_2 \right\} 
                        -\frac12  \left(a_{ij} * [\sqrt M g_1] \right) v_i \partial_{v_j} g_2 \\
                        &\quad
                        -\frac12  \partial_{v_i} \left\{ \left(a_{ij} * [\sqrt M g_1] \right) v_j  g_2 
                        \right\}
                        +\frac14  \left(a_{ij} * [\sqrt M g_1] \right) v_i v_j  g_2  \\
                        &\quad 
                        -  \partial_{v_i} \left\{  \left(b_i* [\sqrt M g_1] \right) g_2
                        \right\} 
                        +\frac12  \left(b_i* [\sqrt M g_1] \right) v_i g_2 .
 		\end{align*}
Finally, we apply the derivative to the third term in the right-hand side of the above equation and thus we get
		\begin{align*}
			\Gamma(g_1, g_2) 
                        &= \partial_{v_i} \left\{ \left(a_{ij} * [\sqrt M g_1] \right) \partial_{v_j} g_2 \right\} 
                        -\frac12  \left(a_{ij} * [\sqrt M g_1] \right) v_i \partial_{v_j} g_2 \\
                        &\quad
                        -\frac12 \left(b_{j} * [\sqrt M g_1] \right) v_j  g_2 
                        -\frac12 \left(a_{ii} * [\sqrt M g_1] \right)  g_2 
                        -\frac12   \left(a_{ij} * [\sqrt M g_1] \right) v_j  \partial_{v_i} g_2  \\
                        &\quad
                        +\frac14  \left(a_{ij} * [\sqrt M g_1] \right) v_i v_j  g_2  
                        -  \partial_{v_i} \left\{  \left(b_i* [\sqrt M g_1] \right) g_2
                        \right\}  
                        +\frac12  \left(b_i* [\sqrt M g_1] \right) v_i g_2 \\
                        &= \partial_{v_i} \left\{ \left(a_{ij} * [\sqrt M g_1] \right) \partial_{v_j} g_2 \right\} 
                        - \left(a_{ij} * [\sqrt M g_1] \right) v_i \partial_{v_j} g_2 \\
                        &\quad
                        -\frac12 \left(a_{ii} * [\sqrt M g_1] \right)  g_2   
                        +\frac14  \left(a_{ij} * [\sqrt M g_1] \right) v_i v_j  g_2 
                        -  \partial_{v_i} \left\{  \left(b_i* [\sqrt M g_1] \right) g_2
                        \right\}  
		\end{align*}
	which completes the proof of \eqref{eq:Gamma}.

\smallskip

        We now prove \eqref{eq:L}. From \eqref{eq:Gamma} we get 
                \begin{align*}
                        \Gamma(\sqrt M , f)
                        &= \partial_{v_i} \left\{ \left(a_{ij}* M \right) \partial_{v_j} f \right\}
                        -\partial_{v_i} \left\{ \left(b_{i}* M \right) f   \right\} \\
                        &\quad
                        -  \left( a_{ij}* M\right) v_i \partial_{v_j} f 
                        +  \frac14 \left( a_{ij}* M \right) v_i v_j f
                        -  \frac12 \left( a_{ii}* M\right) f.
                \end{align*}
For the second term in the right-hand side of above equation, we observe that 
$$
-\partial_{v_i} \left\{ \left(b_{i}* M \right) f   \right\} 
= - \left(c* M \right) f  - \left(b_{i}* M \right) \partial_{v_i}f ,
$$ 
as well as
$$
\begin{aligned}
- \left(b_{i}* M \right) \partial_{v_i}f 
&= - \left(\partial_{v_j} a_{ij}* M \right) \partial_{v_i}f 
=  \left(a_{ij}* v_j M \right) \partial_{v_i}f 
=  \left(a_{ij}*  M \right) v_j \partial_{v_i}f 
\end{aligned}
$$
by using that $a_{ij}(v-w)(v_i-w_i)=0$. We also remark, using that $a_{ij}(v-w)(v_i-w_i)(v_j-w_j)=0$,
$$
\begin{aligned}
-(c*M) 
&= -(a_{ij} * \partial_{v_i} \partial_{v_j} M ) 
= (a_{ij} * \delta_{ij} M )- (a_{ij} *  v_i v_j M) \\
&= (a_{ii} * M) - (a_{ij} * M) v_i v_j,
\end{aligned}
$$
as well as
$$
\partial_{v_i} [ (a_{ij}*M) v_j]
= -  (a_{ij}*  M) v_i v_j +   (a_{ii}*M).
$$
Putting together previous equalities, we finally obtain
$$
\Gamma(\sqrt M , f)
= \partial_{v_i} \left\{ \left(a_{ij}* M \right) \partial_{v_j} f  \right\}
+ \left\{ -\frac14 \left( a_{ij}*M  \right)v_i v_j + \frac12  \partial_{v_i} [ (a_{ij}*M) v_j]  \right\} f .
$$

Moreover, from \eqref{eq:Gamma} we get 
		\begin{align*}
                        \Gamma(f , \sqrt M)
                        &= \partial_{v_i} \left\{ \left(a_{ij}* [\sqrt M f] \right) \partial_{v_j} \sqrt M   \right\}
                        -\partial_{v_i} \left\{ \left(b_{i}* [\sqrt M f] \right) \sqrt M  \right\} \\
                        &\quad
                        -  \left( a_{ij}* [\sqrt M f_1]\right) v_i \partial_{v_j} \sqrt M 
                        +  \frac14 \left( a_{ij}* [\sqrt M f]\right) v_i v_j \sqrt M \\
                        &\quad
                        -  \frac12 \left( a_{ii}* [\sqrt M f]\right) \sqrt M 
		\end{align*}
so that, developing the derivatives $\partial_{v_i}$, we get
	\begin{align*}
                        \Gamma(f , \sqrt M)
                        &=  \left(b_{j}* [\sqrt M f] \right) \partial_{v_j} \sqrt M   
                        + \left(a_{ij}* [\sqrt M f] \right) \partial_{v_i}\partial_{v_j} \sqrt M    \\
                        &\quad 
                        - \left( c* [\sqrt M f] \right) \sqrt M 
                        - \left(b_{i}* [\sqrt M f] \right) \partial_{v_i}\sqrt M   \\
                        &\quad
                        -  \left( a_{ij}* [\sqrt M f]\right) v_i \partial_{v_j} \sqrt M 
                        +  \frac14 \left( a_{ij}* [\sqrt M f]\right) v_i v_j \sqrt M \\
                        &\quad 
                        -  \frac12 \left( a_{ii}* [\sqrt M f]\right) \sqrt M.
		\end{align*}
Observing that 
                $$
                \partial_{v_j} \sqrt M = - \frac12 v_j \sqrt M 
				\quad \text{and} \quad
                \partial_{v_i}\partial_{v_j} \sqrt M 
                = - \frac12 \delta_{ij} \sqrt M
                + \frac14 v_i v_j \sqrt M ,
                $$
we then get
		\begin{align*}
                        \Gamma(f , \sqrt M)
                        &= 
                        \left(a_{ij}* [\sqrt M f] \right) v_i v_j \sqrt M  
                        - \left(a_{ii}* [\sqrt M f] \right) \sqrt M
                        - \left( c* [\sqrt M f] \right) \sqrt M  ,
                \end{align*}
which concludes the proof.
\end{proof}

For $v \in \R^3$, we define the symmetric matrix $\mathbf{A}(v)=({A}_{ij}(v))_{1\leq  i,j \leq 3}$ whose coefficients are given by
$$ 
{A}_{ij}(v)= \left( {a}_{ij}*_{v} M \right) (v).
$$
We can decompose ${\bf A}(v)$ as $\mathbf{A}(v)= \mathbf{B}^{\top}(v)\mathbf{B}(v)$, where $\mathbf{B}(v)=(B_{ij}(v))_{1\leq  i,j \leq 3}$ is a matrix with real-valued smooth entries (see~\cite{Rachid1}).
Recall also that from \cite{DegondLemou}, for $v \in \R^3 \setminus \{0\}$, the matrix~${\bf A}(v)$ has a simple eigenvalue~$\ell_1(v)>0$ associated with the eigenvector $v$ and a double eigenvalue~$\ell_2(v)>0$ associated with the eigenspace $v^{\perp}$. Moreover,
when $|v|\to +\infty$, we have
	\begin{equation} \label{eq:vp}
	\ell_1(v) \sim  2 \langle v \rangle^\gamma \quad\text{and}\quad
	\ell_2(v) \sim  \langle v \rangle^{\gamma+2} .
	\end{equation}
As such, one can write that for any $v \in \R^3 \setminus \{0\}$, 
	$$
	\mathbf{A}(v) 
	= \ell_1(v) \frac{v}{|v|} \otimes \frac{v}{|v|}  + \ell_2(v) \left( \Id - \frac{v}{|v|} \otimes \frac{v}{|v|} \right),
	$$
where $\Id$ denotes the identity matrix and
	\begin{equation} \label{def:B}
	{\bf B}(v) = \sqrt{\ell_1}(v) \frac{v}{|v|} \otimes \frac{v}{|v|} + \sqrt{\ell_2}(v) \left( \Id - \frac{v}{|v|} \otimes \frac{v}{|v|} \right).
	\end{equation}
In what follows, we will use the following differential operators
	\begin{equation} \label{def:DvDx}
		\widetilde{\nabla}_v := \mathbf{B}(v) \nabla_v \quad \text{and} \quad \widetilde{\nabla}_x := \mathbf{B}(v) \nabla_x
	\end{equation} 
as well as their adjoint operators in $L^2_v$ given by, for $F: \R^d \to \R^d$, 
\begin{equation} \label{def:Dv*Dx*}
		(\widetilde\nabla_v)^* F = - \nabla_v \cdot ({\bf B}(v) F ) \quad \text{and} \quad  
		(\widetilde{\nabla}_x)^* F = - \widetilde \nabla_x \cdot F.
\end{equation}

Using the formulation \eqref{eq:L} of the linearized collision operator $L$, we can therefore rewrite it as 
\begin{equation}\label{eq:Lbis}
\begin{aligned}
L f
&=  - \widetilde \nabla_v^* \widetilde \nabla_v f 
- \left\{ \frac14 |\mathbf{B}(v) v |^2 
-\frac12 \nabla_{v}\cdot\left[ \mathbf{B}^{\top}(v)\mathbf{B}(v) v \right]  \right\} f \\
&\quad 
- \left\{ \left( a_{ij} * \sqrt M f \right)  v_i v_j  
- \left(a_{ii} * \sqrt M f \right)
+ \left( c * \sqrt M f \right) \right\} \sqrt M.
\end{aligned}
\end{equation}

The functions ${A}_{ij}, {B}_{ij}$ verify the following properties (see for example~\cite{GuoBNS,Rachid1}):
For any multi-index $ \alpha \in \N^{3}$ we have, for all $v\in \R^{3}$,
\begin{equation}\label{eq:nablaBij}
	\vert\partial_{v}^{\alpha} {A}_{ij}(v)\vert \lesssim {\langle v \rangle}^{{\gamma+2-\vert \alpha \vert }},\quad 
	\vert \partial_{v}^{\alpha} B_{ij}(v)\vert \lesssim {\langle v \rangle}^{\frac{\gamma}{2}+1-\vert \alpha \vert }. 
\end{equation}

From \cite{DegondLemou,GuoLandau,Mouhot,MouhotStrain}, we also know that $L$ has a spectral gap, more precisely, there is~$\sigma_L >0$ such that 
\begin{equation}\label{eq:spectralgap}
\la L f , f \ra_{L^2_v} \le - \sigma_L \| f - \pi f \|_{H^1_{v,*}}^2
\end{equation}
for any $f \in \mathrm{Dom}(L)$, where we recall that $\pi$ is the projector onto the kernel of $L$ defined in~\eqref{def:pi} and the $H^1_{v,*}$-norm is defined in~\eqref{def:H1v*}. Notice that in the case $-2 \le \gamma \le 1$ which we study in the present paper, the previous inequality is indeed a coercivity estimate because $\|\cdot\|_{H^1_{v,*}} \ge \|\cdot\|_{L^2_v}$.

\smallskip

\subsection{Functional spaces} \label{subsec:funcspaces}

We first notice that the $H^1_{v,*}$-norm in velocity defined in~\eqref{def:H1v*} is also equal to
$$
\| f \|_{H^1_{v,*}}^2 
= \|  \langle v \rangle^{\frac{\gamma}{2} + 1} f \|_{L^{2}_{v}}^{2}
+\| \widetilde \nabla_{v} f \|_{L^{2}_{v}}^{2}
$$
where $\widetilde \nabla_v$ is defined in~\eqref{def:DvDx}.  
We also introduce some $H^2$~norm in velocity defined through the following norm:
\begin{equation*}\label{def:H2v*}
	\| f \|_{H^2_{v,*}}^2 :=
	\| \langle v \rangle^{\gamma + 2} f \|_{L^2_v}^2
	+ \| \widetilde \nabla_v ( \langle v \rangle^{\frac{\gamma}{2} + 1} f) \|_{L^2_v}^2
	+ \| \langle v \rangle^{\frac{\gamma}{2} + 1} \widetilde \nabla_v f \|_{L^2_v}^2 
	+ \| \widetilde \nabla_v  \widetilde \nabla_v f \|_{L^2_v}^2.
\end{equation*}
Similarly to the definitions of~$\XXX$ and~$\YYY_1$, we define the weighted Sobolev-type space $\YYY_2$ as the space associated to the norm
	\begin{equation} \label{def:normYYY2}
	\begin{aligned}
	\| f \|_{\YYY_2}^2 
	&:= \| \langle v \rangle^{3(\frac{\gamma}{2}+1)} f \|_{L^2_x (H^2_{v,*})}^2 
	+ \| \langle v \rangle^{2(\frac{\gamma}{2}+1)} \nabla_x f \|_{L^2_x (H^2_{v,*})}^2 \\
	&\quad 
	+ \| \langle v \rangle^{\frac{\gamma}{2}+1} \nabla^2_x f \|_{L^2_x (H^2_{v,*})}^2 
	+ \|\nabla^3_x f \|_{L^2_x (H^2_{v,*})}^2.
	\end{aligned}	
	\end{equation}

We then introduce the spaces $\ZZZ_i^\eps$ for $i=1,2$ that involve derivatives in $x$: $\ZZZ_1^\eps$ is associated with the norm
	\begin{equation} \label{def:normZZZ1eps}
	\|f\|^2_{\ZZZ_1^\eps} := \|f\|^2_\XXX + \|f\|^2_{\YYY_1} + \eps^2 \|\widetilde{\nabla}_x f\|^2_\XXX,
	\end{equation}
and $\ZZZ_2^\eps$ is associated with 
	\begin{equation} \label{def:normZZZ2eps}
	\|f\|^2_{\ZZZ_2^\eps} := \|f\|^2_{\ZZZ_1^\eps} + \|f\|^2_{\YYY_2} + \eps^4 \|\widetilde{\nabla}_x^2 f\|^2_\XXX.
	\end{equation}
For the sequel, it is worth noticing that if $f \in \ZZZ_2^\eps$, then $\eps \widetilde \nabla_v \widetilde \nabla_x f \in \XXX$. Indeed, a simple computation based on integrations by parts shows that 
	$$
	\eps \| \widetilde \nabla_v \widetilde \nabla_x f \|_\XXX 
	\lesssim \eps^2 \| \widetilde \nabla_x^2 f \|_\XXX + \|f\|_{\YYY_2}.
	$$

For $i=1,2$, we also define the associated dual spaces $(\YYY_i)'$ and $(\ZZZ_i^\eps)'$ with $\XXX$ as a pivot space, more precisely, they are associated with the following norms:
	\begin{equation} \label{def:YYYi'}
	\|f\|_{\YYY_i'} := \sup_{\|\varphi\|_{\YYY_i} \leq 1} \langle f, \varphi\rangle_\XXX
	\end{equation}
and
	\begin{equation} \label{def:ZZZieps'}
	\|f\|_{(\ZZZ_i^\eps)'} := \sup_{\|\varphi\|_{\ZZZ_i^\eps} \leq 1} \langle f, \varphi\rangle_\XXX
	\end{equation}
where $\langle \cdot,\cdot \rangle_\XXX$ is the scalar product associated to $\|\cdot\|_\XXX$ defined in~\eqref{def:normXXX}.
Notice that we have the following interpolation result: 
	\begin{equation} \label{eq:interp}
		\left[\XXX,(\ZZZ_2^\eps)'\right]_{1/2,2} = (\ZZZ_1^\eps)'.
	\end{equation}
The notation used above is the classical one of real interpolation (see~\cite{BookBL}). For sake of completeness, we briefly recall the meaning of this notation. For $C$ and $D$ two Banach spaces which are both embedded in the same Hausdorff topological vector space, for any $z \in C+D$, we define the $K$-function by
$$
K(t,z) := \inf_{z=c+d} \left(\|c\|_C+t\|d\|_D\right), \quad \forall \, t>0.
$$
The space $[C,D]_{\theta,p}$ for $\theta \in (0,1)$ and $p \in [1,+\infty]$ is then defined by:
$$
[C,D]_{\theta,p} := \left\{z \in C+D, \, \, t \mapsto K(t,z)/t^\theta \in L^p\left(\d t/t\right)\right\}.
$$

\subsection{Basic estimates} 

We gather in this subsection some basic estimates concerning the collision operator $L$ that will be useful in the sequel.

In order to simplify, we recall the formulation of the operator $L$ in \eqref{eq:Lbis} and we introduce the function
\begin{equation}\label{eq:def:psi}
\psi(v) := \frac14 |\mathbf{B}(v) v |^2 
-\frac12 \nabla_{v}\cdot\left[ \mathbf{B}^{\top}(v)\mathbf{B}(v) v \right]
\end{equation}
as well as the operators 
\begin{equation}\label{eq:def:L1}
L_1 f := - \widetilde \nabla_v^* \widetilde \nabla_v f - \psi f
\end{equation}
and
\begin{equation}\label{eq:def:L2}
L_2 f:= - \left\{ \left( a_{ij} * \sqrt M f \right) v_i v_j 
- \left(a_{ii} * \sqrt M f \right)
+ \left( c * \sqrt M f \right) \right\} \sqrt M
\end{equation}
so that we have $L = L_1 + L_2$.

\medskip

We start with some basic commutator estimates.
\begin{lem}\label{lem:commutator}
For any suitable function $f=f(x,v)$ and any $(x,v) \in \T^3 \times \R^3$, there holds:

\medskip\noindent
(i)
$[\widetilde \nabla_v , v \cdot \nabla_x] f(x,v) = \widetilde \nabla_x f(x,v)$.

\medskip\noindent
(ii) For any $\alpha \in \R$ and $1 \le j \le 3$, one has
			$$ 
			|[ \langle v \rangle^\alpha, \widetilde \nabla_{v_j}] f |(x,v) 
			= |[ \langle v \rangle^\alpha, \widetilde \nabla_{v_j}^*] f |(x,v) \lesssim \langle v \rangle^{\frac{\gamma}{2} + \alpha-1} |f|(x,v).
			$$
		
\medskip\noindent
(iii) For any $1 \le i,j \le 3$, one has
			$$
			|[\widetilde \nabla_{v_i} , \widetilde \nabla_{v_j}] f |(x,v)  \lesssim \langle v \rangle^{\gamma+1}  |\nabla_v f|(x,v) .
			$$

\medskip\noindent
(iv) For any $1 \le i,j \le 3$, one has
			$$ 
			|[\widetilde \nabla_{v_i} , \widetilde \nabla_{v_j}^*] f |(x,v) 
			\lesssim \langle v \rangle^{\gamma+1}  |\nabla_v f|(x,v)  +  \langle v \rangle^{\gamma} |f|(x,v) .
			$$

\medskip\noindent
(v) For any $1 \le i,j \le 3$, one has
			$$ 
			|[\widetilde \nabla_{v_i} , \widetilde \nabla_{x_j}] f |(x,v) 
			= |[\widetilde \nabla_{v_i}^* , \widetilde \nabla_{x_j}] f | (x,v) 
			\lesssim \langle v \rangle^{\gamma+1}  |\nabla_x f|(x,v) .
			$$
		
\medskip\noindent
(vi) For any $\alpha \in \R$, one has
			$$ 
			\left| \left[ [\langle v \rangle^{\alpha}, \widetilde \nabla_{v_j}^*] , \widetilde \nabla_{v_j} \right] f \right|(x,v) 
			\lesssim   \langle v \rangle^{\gamma + \alpha - 1}  |f|(x,v) .
			$$

\medskip\noindent
(vii) For any $1 \le i,j,k \le 3$, one has
			$$ 
			\left| \left[ [\widetilde \nabla_{v_i} , \widetilde \nabla_{v_j}] , \widetilde \nabla_{v_k} \right] f \right|(x,v) 
			\lesssim \langle v \rangle^{\frac{3\gamma}{2}+1}  |\nabla_v f|(x,v)   .
			$$
			and 
			$$ 
			\left| \left[ [\widetilde \nabla_{v_i} , \widetilde \nabla_{v_j}^*] , \widetilde \nabla_{v_k} \right] f \right|(x,v) 
			\lesssim \langle v \rangle^{\frac{3\gamma}{2}+1}  |\nabla_v f|(x,v)  +  \langle v \rangle^{\frac{3\gamma}{2}}  |f|(x,v) .
			$$

\medskip\noindent
(viii) For any $1 \le i,j,k \le 3$, one has
			$$ 
			\left| \left[ [\widetilde \nabla_{x_i} , \widetilde \nabla_{v_j}] , \widetilde \nabla_{v_k} \right] f \right|(x,v) = \left| \left[ [\widetilde \nabla_{x_i} , \widetilde \nabla_{v_j}^*] , \widetilde \nabla_{v_k} \right] f \right|(x,v) 
			\lesssim \langle v \rangle^{\frac{3\gamma}{2}+1}  |\nabla_x f|(x,v) .
			$$
\end{lem}

\begin{proof} 
Recall that we denote $\mathbf B(v) = (B_{ij})_{1 \le i,j \le 3}$ and that we have 
$$
\widetilde \nabla_{v_i} f = B_{im} \partial_{v_m} f, \,\,\,\,
\widetilde \nabla_{x_i} f = B_{im} \partial_{x_m} f, \,\,\,\,
\widetilde \nabla_{v_i}^* f =  -\partial_{v_m} ( B_{im} f), \,\,\,\,
\widetilde \nabla_{x_i}^* f = - B_{im} \partial_{x_m} f = - \widetilde \nabla_{x_i} f.
$$

\medskip\noindent
(i) We have 
$$
\begin{aligned}{}
[\widetilde \nabla_{v_i} , v_\ell \partial_{x_\ell} ] f 
&= B_{im} \partial_{v_m} ( v_\ell \partial_{x_\ell}f) 
- v_\ell \partial_{x_\ell} ( B_{im} \partial_{v_m} f) \\
&= B_{i\ell} \partial_{x_\ell}f
+ B_{im}  v_\ell \partial_{v_m} \partial_{x_\ell}f 
- v_\ell B_{im} \partial_{x_\ell}  \partial_{v_m} f \\
&= \widetilde \nabla_{x_i} f.
\end{aligned}
$$

\medskip\noindent
(ii) We easily compute 
$$
\begin{aligned}{}
[\langle v \rangle^{\alpha} , \widetilde \nabla_{v_j}] f 
&= - ( \widetilde \nabla_{v_j} \langle v \rangle^{\alpha}) f
\end{aligned}
$$
as well as 
$$
\begin{aligned}{}
[\langle v \rangle^{\alpha} , \widetilde \nabla_{v_j}^*] f 
&= ( \widetilde \nabla_{v_j} \langle v \rangle^{\alpha}) f.
\end{aligned}
$$
We conclude the proof by remarking that $\widetilde \nabla_{v_j} \langle v \rangle^{\alpha} = \alpha (\mathbf B(v) v)_{j} \langle v \rangle^{\alpha-2}$ and using that $|\mathbf B(v) v| \lesssim \langle v \rangle^{\frac{\gamma}{2} +1}$ thanks to the definition \eqref{def:B} of $\mathbf B(v)$.

\medskip\noindent
(iii) We easily compute 
$$
\begin{aligned}{}
[\widetilde \nabla_{v_i} , \widetilde \nabla_{v_j}] f 
&= B_{im} \partial_{v_m} ( B_{j \ell} \partial_{v_\ell} f) - B_{j\ell} \partial_{v_\ell} ( B_{im} \partial_{v_m} f) \\
&= B_{im}   B_{j \ell} \partial_{v_m} \partial_{v_\ell} f
+ B_{im} (\partial_{v_m} B_{j \ell}) \partial_{v_\ell} f
- B_{j\ell}  B_{im} \partial_{v_\ell}  \partial_{v_m} f
- B_{j\ell} (\partial_{v_\ell} B_{im}) \partial_{v_m} f \\
&= 
B_{im} (\partial_{v_m} B_{j \ell}) \partial_{v_\ell} f
- B_{j\ell} (\partial_{v_\ell} B_{im}) \partial_{v_m} f ,
\end{aligned}
$$
and we conclude the proof using \eqref{eq:nablaBij}.

\medskip\noindent
(iv) We have 
$$
\begin{aligned}{}
[\widetilde \nabla_{v_i} , \widetilde \nabla_{v_j}^*] f 
&= -B_{im} \partial_{v_m} \partial_{v_\ell} ( B_{j \ell} f) + \partial_{v_\ell}(B_{j\ell} B_{im} \partial_{v_m} f) \\
&= -B_{im}B_{j \ell}  \partial_{v_m} \partial_{v_\ell} f 
-B_{im} (\partial_{v_m}B_{j \ell}) \partial_{v_\ell} f \\
&\quad
-B_{im} (\partial_{v_m} \partial_{v_\ell} B_{j \ell}) f
-B_{im} (\partial_{v_\ell} B_{j \ell}) \partial_{v_m}  f  \\
&\quad
+ B_{im} (\partial_{v_\ell} B_{j\ell}) \partial_{v_m} f
+ B_{j\ell} (\partial_{v_\ell}B_{im}) \partial_{v_m} f
+ B_{j\ell} B_{im} \partial_{v_\ell}\partial_{v_m} f \\
&= 
-B_{im} (\partial_{v_m}B_{j \ell}) \partial_{v_\ell} f 
+ B_{j\ell} (\partial_{v_\ell}B_{im}) \partial_{v_m} f
-B_{im} (\partial_{v_m} \partial_{v_\ell} B_{j \ell}) f.
\end{aligned}
$$
We can simplify last expression by relabelling the indices $m$ and $\ell$ of the second term, which gives
$$
\begin{aligned}{}
[\widetilde \nabla_{v_i} , \widetilde \nabla_{v_j}^*] f 
&= 
\left[ B_{jm} (\partial_{v_m}B_{i\ell}) -B_{im} (\partial_{v_m}B_{j \ell}) \right]\partial_{v_\ell} f
-B_{im} (\partial_{v_m} \partial_{v_\ell} B_{j \ell}) f.
\end{aligned}
$$
We then conclude the proof by using using \eqref{eq:nablaBij}.

\medskip\noindent
(v) We have 
$$
\begin{aligned}{}
[\widetilde \nabla_{v_i} , \widetilde \nabla_{x_j}] f 
&= B_{im} \partial_{v_m} ( B_{j \ell} \partial_{x_\ell} f) - B_{j\ell} \partial_{x_\ell} ( B_{im} \partial_{v_m} f) \\
&= B_{im}   B_{j \ell} \partial_{v_m} \partial_{x_\ell} f
+ B_{im} (\partial_{v_m} B_{j \ell}) \partial_{x_\ell} f
- B_{j\ell}  B_{im} \partial_{x_\ell}  \partial_{v_m} f \\
&= 
B_{im} (\partial_{v_m} B_{j \ell}) \partial_{x_\ell} f
\end{aligned}
$$
as well as
$$
\begin{aligned}{}
-[\widetilde \nabla_{v_i}^* , \widetilde \nabla_{x_j}] f 
&=\partial_{v_m} (  B_{im} B_{j \ell} \partial_{x_\ell} f) - B_{j\ell} \partial_{x_\ell} \partial_{v_m}( B_{im}  f) \\
&= B_{j \ell} \partial_{x_\ell} \partial_{v_k} (B_{im} f) 
+ B_{im}(\partial_{v_m} B_{j \ell}) \partial_{x_\ell} f
- B_{j\ell} \partial_{x_\ell} \partial_{v_m}( B_{im}  f) \\
&= B_{im}(\partial_{v_m} B_{j \ell}) \partial_{x_\ell} f .
\end{aligned}
$$
We then obtain the estimate by using \eqref{eq:nablaBij}.


\medskip\noindent
(vi) Thanks to the proof of item (ii), we write
$$
\begin{aligned}{}
[\langle v \rangle^\alpha , \widetilde \nabla_{v_j}^*]  \widetilde \nabla_{v_j} f 
&= (\widetilde \nabla_{v_j} \langle v \rangle^{\alpha}) \widetilde \nabla_{v_j} f \\
&= \widetilde \nabla_{v_j} ((\widetilde \nabla_{v_j} \langle v \rangle^{\alpha})  f )
- ( \widetilde \nabla_{v_j} (\widetilde \nabla_{v_j} \langle v \rangle^{\alpha}) ) f 
\\
&= \widetilde \nabla_{v_j} [\langle v \rangle^\alpha , \widetilde \nabla_{v_j}^*]   f  
- (\widetilde \nabla_{v_j} (\widetilde \nabla_{v_j} \langle v \rangle^{\alpha}) ) f .
\end{aligned}
$$
We conclude the proof by using $\widetilde \nabla_{v_j} \langle v \rangle^{\alpha} = \alpha ({\bf B}(v)v)_j \la v \ra^{\alpha-2} $ and the upper bound \eqref{eq:nablaBij}.

\medskip\noindent
(vii) Thanks to the proof of item (iii), we first write
$$
\begin{aligned}{}
[\widetilde \nabla_{v_i} , \widetilde \nabla_{v_j}]  \widetilde \nabla_{v_k} f 
& = B_{im} (\partial_{v_m} B_{j \ell}) \partial_{v_\ell} (B_{kp} \partial_{v_p} f) 
- B_{j \ell} (\partial_{v_\ell} B_{im}) \partial_{v_m} (B_{kp} \partial_{v_p} f) \\
&\quad 
= B_{im} (\partial_{v_m} B_{j \ell})  \widetilde \nabla_{v_k} \partial_{v_\ell}f  
+ B_{im} (\partial_{v_m} B_{j \ell}) (\partial_{v_\ell} B_{kp}) \partial_{v_p} f \\
&\quad\quad
- B_{j \ell} (\partial_{v_\ell} B_{im}) \widetilde \nabla_{v_k} \partial_{v_m}  f
- B_{j \ell} (\partial_{v_\ell} B_{im}) (\partial_{v_m} B_{kp}) \partial_{v_p} f.
\end{aligned}
$$
We then obtain 
$$
\begin{aligned}{}
[\widetilde \nabla_{v_i} , \widetilde \nabla_{v_j}]  \widetilde \nabla_{v_k} f 
& = \widetilde \nabla_{v_k}[B_{im} (\partial_{v_m} B_{j \ell})  \partial_{v_\ell}f ]
-(\widetilde \nabla_{v_k}[B_{im} (\partial_{v_m} B_{j \ell}) ])  \partial_{v_\ell} f  \\
&\quad\quad
- \widetilde \nabla_{v_k}[B_{j \ell} (\partial_{v_\ell} B_{im}) \partial_{v_m}  f ]
+ (\widetilde \nabla_{v_k} [B_{j \ell} (\partial_{v_\ell} B_{im}) ]) \partial_{v_m}  f \\
&\quad\quad 
+ [B_{im} (\partial_{v_m} B_{j \ell}) (\partial_{v_\ell} B_{kp})  
- B_{j \ell} (\partial_{v_\ell} B_{im}) (\partial_{v_m} B_{kp})] \partial_{v_p} f \\
& = \widetilde \nabla_{v_k} [\widetilde \nabla_{v_i} , \widetilde \nabla_{v_j}] f 
-(\widetilde \nabla_{v_k}[B_{im} (\partial_{v_m} B_{j \ell}) ])  \partial_{v_\ell} f 
+ (\widetilde \nabla_{v_k} [B_{j \ell} (\partial_{v_\ell} B_{im}) ]) \partial_{v_m}  f \\
&\quad
+ [B_{im} (\partial_{v_m} B_{j \ell}) (\partial_{v_\ell} B_{kp})  
- B_{j \ell} (\partial_{v_\ell} B_{im}) (\partial_{v_m} B_{kp})] \partial_{v_p} f
\end{aligned}
$$
and we conclude the proof of the first estimate using \eqref{eq:nablaBij}. The second estimate is obtained in a similar way by using the computation of item (iv), thus we omit it.

\medskip\noindent
(viii) Thanks to the proof of item (v), we write
$$
\begin{aligned}{}
[\widetilde \nabla_{x_i} , \widetilde \nabla_{v_j}]  \widetilde \nabla_{v_k} f 
& = - B_{jm} (\partial_{v_m} B_{i \ell}) \partial_{x_\ell} \widetilde \nabla_{v_k} f 
\\
& 
= - \widetilde \nabla_{v_k} \left(B_{jm} (\partial_{v_m} B_{i \ell}) \partial_{x_\ell}  f \right)
- \left( \widetilde \nabla_{v_k} [B_{jm} (\partial_{v_m} B_{i \ell})] \right) \partial_{x_\ell} f \\
& 
= \widetilde \nabla_{v_k}[\widetilde \nabla_{x_i} , \widetilde \nabla_{v_j}^*]  f 
- \left( \widetilde \nabla_{v_k} [B_{jm} (\partial_{v_m} B_{i \ell})] \right) \partial_{x_\ell} f. 
\end{aligned}
$$
The estimate then follows from \eqref{eq:nablaBij}.

\end{proof}

Using the above result, we shall now compute some commutators related to the $L_1$ term defined in \eqref{eq:def:L1} of the collision operator $L$.

\begin{lem}\label{lem:commutatorL}
There holds 

\medskip\noindent
(i) For any $\alpha \in \R$, one has
$$
\begin{aligned}{}
[ \langle v \rangle^\alpha , L_1] f
&= -\widetilde \nabla_{v_\ell}^* [\langle v \rangle^{\alpha}, \widetilde \nabla_{v_\ell}] f  
- \widetilde \nabla_{v_\ell} [\langle v \rangle^{\alpha}  ,
\widetilde \nabla_{v_\ell}^* ] f
- \left[ [\langle v \rangle^{\alpha}  ,
\widetilde \nabla_{v_\ell}^* ], \widetilde \nabla_{v_\ell} \right] f.
\end{aligned}
$$

\medskip\noindent
(ii) For any $1 \le k \le 3$, one has
$$
\begin{aligned}{}
[\widetilde \nabla_{v_k} , L_1] f
&= -  \widetilde \nabla_{v_\ell}^* [\widetilde \nabla_{v_k},  \widetilde \nabla_{v_\ell}] f
-  \widetilde \nabla_{v_\ell} [\widetilde \nabla_{v_k},  \widetilde \nabla_{v_\ell}^*]  f  
- \left[ [\widetilde \nabla_{v_k},  \widetilde \nabla_{v_\ell}^*] , \widetilde \nabla_{v_\ell} \right] f
- (\widetilde \nabla_{v_k} \psi) f . 
\end{aligned}
$$

\medskip\noindent
(iii) For any $1 \le k \le 3$, one has
$$
\begin{aligned}{}
[\widetilde \nabla_{x_k} , L_1] f
&= -  \widetilde \nabla_{v_\ell}^* [\widetilde \nabla_{x_k},  \widetilde \nabla_{v_\ell}] f
-  \widetilde \nabla_{v_\ell} [\widetilde \nabla_{x_k},  \widetilde \nabla_{v_\ell}^*]  f
-  \left[ [\widetilde \nabla_{x_k},  \widetilde \nabla_{v_\ell}^*], \widetilde \nabla_{v_\ell} \right] f 
\end{aligned}
$$
and
$$
\begin{aligned}{}
[\langle v \rangle^{\frac{\gamma}{2}} \partial_{x_k} , L_1] f  
= -  \widetilde \nabla_{v_\ell}^* [\langle v \rangle^{\frac{\gamma}{2}} \partial_{x_k},  \widetilde \nabla_{v_\ell}] f
-  \widetilde \nabla_{v_\ell} [\langle v \rangle^{\frac{\gamma}{2}} \partial_{x_k},  \widetilde \nabla_{v_\ell}^*]  f 
-  \left[ [\langle v \rangle^{\frac{\gamma}{2}} \partial_{x_k},  \widetilde \nabla_{v_\ell}^*] , \widetilde \nabla_{v_\ell} \right] f. 
\end{aligned}
$$
\end{lem}

\begin{proof} 
(i) We first write 
$$
\begin{aligned}
\langle v \rangle^{\alpha} (L_1 f)
&= -  \langle v \rangle^{\alpha}  
\widetilde \nabla_{v_\ell}^* \widetilde \nabla_{v_\ell} f 
-   \psi  \langle v \rangle^{\alpha}  f 
\end{aligned}
$$
and 
$$
\begin{aligned}
L_1(\langle v \rangle^{\alpha} f)
&= -   
\widetilde \nabla_{v_\ell}^* \widetilde \nabla_{v_\ell} (\langle v \rangle^{\alpha}  f) 
-   \psi \langle v \rangle^{\alpha}  f .
\end{aligned}
$$
We now observe that 
$$
\begin{aligned}
\langle v \rangle^{\alpha}  
\widetilde \nabla_{v_\ell}^* \widetilde \nabla_{v_\ell} f 
&=   
\widetilde \nabla_{v_\ell}^* \langle v \rangle^{\alpha} \widetilde \nabla_{v_\ell} f  
+ [\langle v \rangle^{\alpha}  ,
\widetilde \nabla_{v_\ell}^* ] \widetilde \nabla_{v_\ell} f \\
&=   
\widetilde \nabla_{v_\ell}^*  \widetilde \nabla_{v_\ell} ( \langle v \rangle^{\alpha} f )
+\widetilde \nabla_{v_\ell}^* [\langle v \rangle^{\alpha}, \widetilde \nabla_{v_\ell}] f  
+ \widetilde \nabla_{v_\ell} [\langle v \rangle^{\alpha}  ,
\widetilde \nabla_{v_\ell}^* ] f
+ \left[ [\langle v \rangle^{\alpha}  ,
\widetilde \nabla_{v_\ell}^* ], \widetilde \nabla_{v_\ell} \right] f ,
\end{aligned}
$$
which completes the proof.

\medskip\noindent
(ii) We first compute 
$$
\begin{aligned}
\widetilde \nabla_{v_k} (L_1 f)
&= -  \widetilde \nabla_{v_k}  \widetilde \nabla_{v_\ell}^* \widetilde \nabla_{v_\ell} f 
-  \widetilde \nabla_{v_k} ( \psi f) 
\end{aligned}
$$
and 
$$
\begin{aligned}
L_1(\widetilde \nabla_{v_k}  f)
&= -  \widetilde \nabla_{v_\ell}^* \widetilde \nabla_{v_\ell} (\widetilde \nabla_{v_k} f) 
-    \psi \widetilde \nabla_{v_k} f. 
\end{aligned}
$$
We conclude the proof by observing that 
$$
\begin{aligned}
\widetilde \nabla_{v_k}  \widetilde \nabla_{v_\ell}^* \widetilde \nabla_{v_\ell} f 
&=  \widetilde \nabla_{v_\ell}^* \widetilde \nabla_{v_k}  \widetilde \nabla_{v_\ell} f 
+ [\widetilde \nabla_{v_k},  \widetilde \nabla_{v_\ell}^*] \widetilde \nabla_{v_\ell} f \\
&=  \widetilde \nabla_{v_\ell}^*   \widetilde \nabla_{v_\ell} \widetilde \nabla_{v_k} f 
+\widetilde \nabla_{v_\ell}^* [\widetilde \nabla_{v_k},  \widetilde \nabla_{v_\ell}] f 
+ \widetilde \nabla_{v_\ell}[\widetilde \nabla_{v_k},  \widetilde \nabla_{v_\ell}^*]  f
+ \left[ [\widetilde \nabla_{v_k},  \widetilde \nabla_{v_\ell}^*], \widetilde \nabla_{v_\ell} \right]f
\end{aligned}
$$
and writing $\widetilde \nabla_{v_k} ( \psi f) = \psi \widetilde \nabla_{v_k} f + (\widetilde \nabla_{v_k}  \psi ) f$.

\medskip\noindent
(iii) We compute 
$$
\begin{aligned}
\widetilde \nabla_{x_k} (L_1 f)
&= -  \widetilde \nabla_{x_k}  \widetilde \nabla_{v_\ell}^* \widetilde \nabla_{v_\ell} f 
-  \psi \widetilde \nabla_{x_k} f
\end{aligned}
$$
and
$$
\begin{aligned}
L_1(\widetilde \nabla_{x_k} f)
&= -  \widetilde \nabla_{v_\ell}^* \widetilde \nabla_{v_\ell} \widetilde \nabla_{x_k}  f 
-  \psi \widetilde \nabla_{x_k} f. 
\end{aligned}
$$
We now observe that 
$$
\begin{aligned}
\widetilde \nabla_{x_k}  \widetilde \nabla_{v_\ell}^* \widetilde \nabla_{v_\ell} f 
&=  \widetilde \nabla_{v_\ell}^* \widetilde \nabla_{x_k}  \widetilde \nabla_{v_\ell} f 
+ [\widetilde \nabla_{x_k},  \widetilde \nabla_{v_\ell}^*] \widetilde \nabla_{v_\ell} f \\
&=  \widetilde \nabla_{v_\ell}^*   \widetilde \nabla_{v_\ell} \widetilde \nabla_{x_k} f 
+\widetilde \nabla_{v_\ell}^* [\widetilde \nabla_{x_k},  \widetilde \nabla_{v_\ell}] f 
+ \widetilde \nabla_{v_\ell}[\widetilde \nabla_{x_k},  \widetilde \nabla_{v_\ell}^*]  f
+ \left[ [\widetilde \nabla_{x_k},  \widetilde \nabla_{v_\ell}^*], \widetilde \nabla_{v_\ell} \right]f
\end{aligned}
$$
which gives the first estimate.

The computation for $[\langle v \rangle^{\frac{\gamma}{2}} \partial_{x_k} , L_1]$ can be obtained in a similar fashion, thus we omit it.
\end{proof}

\begin{lem}\label{lem:conv}
For any suitable function $f=f(v)$ and any $v \in \R^3$ there holds:

\medskip\noindent
(i) For any $i,j \in \{1,2,3\}$ one has 
$$
|(a_{ij}* f ) (v) | + |(a_{ij}* f ) (v) v_i| + |(a_{ij}* f ) (v) v_i v_j| 
\lesssim \langle v \rangle^{\gamma+2} \| \langle v \rangle^{7} f \|_{L^2_v}
$$

\medskip\noindent
(ii) For any $i,j,\ell \in \{1,2,3\}$ one has 
$$
|(\partial_{v_\ell} a_{ij}* f ) (v) |
+ |(b_{i}* f ) (v) |
\lesssim \langle v \rangle^{\gamma+1} \| \langle v \rangle^{4} f \|_{L^2_v}
$$
and
$$
|\left(\partial_{v_\ell} a_{ij}*  f \right)(v) v_i |
+ |\left(\partial_{v_\ell} a_{ij}*  f \right)(v) v_i v_j | \lesssim \langle v \rangle^{\gamma+2} \| \langle v \rangle^{8} f  \|_{L^2_v} 
$$

\medskip\noindent
(iii) If $\gamma \ge 0$, for any $\ell \in \{1,2,3 \}$ one has 
$$
|\left(c* f \right) (v) |
+ |\left(\partial_{v_\ell} b_i * f \right) (v) |
\lesssim \langle v \rangle^{\gamma} \| \langle v \rangle^{3} f \|_{L^2_v}.
$$

\medskip\noindent
(iv) If $\gamma \in [-2,0)$, for any $\ell \in \{1,2,3 \}$ one has 
$$
|\left(c* f \right) (v) |
+ |\left(\partial_{v_\ell} b_i * f \right) (v) |
\lesssim \langle v \rangle^{\gamma} \| \langle v \rangle^{3} f \|_{L^4_v}.
$$
\end{lem}

\begin{proof}
All estimates in point (i) come from~\cite[Lemma~3.4]{CTW}. The estimates on the term $|(b_{i}* f ) (v) |$ in point (ii) and on the term $|(c*f)(v)|$ in (iii)--(iv) also come from~\cite[Lemma~3.4]{CTW}, and the proof of the estimates on $|(\partial_{v_\ell} a_{ij}* f ) (v) |$ in (ii) and on $|\left(\partial_{v_\ell} b_i* f \right) (v) |$ in (iii)--(iv) follow the same lines since $|\partial_{v_\ell} a_{ij}| \lesssim |v|^{\gamma+1}$ and $|\partial_{v_\ell}b_i| \lesssim |v|^\gamma$.

We thus only prove the second estimate in (ii). 
Using that $a_{ij}(v-v_*)v_iv_j=a_{ij}(v-v_*)(v_*)_i(v_*)_j$, we observe that  
$$
\begin{aligned}
\left(\partial_{v_\ell} a_{ij}*  f \right)(v) v_i v_j  
&= (a_{ij}* \partial_{v_\ell} f )(v) v_i v_j \\
&= (a_{ij}* v_iv_j\partial_{v_\ell} f ) (v) \\
&= (\partial_{v_\ell} a_{ij}* [v_iv_j  f] ) (v) 
- (a_{ij}* [\partial_{v_\ell}(v_iv_j) f] ) (v).
\end{aligned}
$$
Using the the estimate in (i) and the first estimate in (ii), we thus deduce 
$$
\begin{aligned}
|(\partial_{v_\ell} a_{ij}*  f )(v) v_i v_j 
| 
&\lesssim \langle v \rangle^{\gamma+1} \| \langle v \rangle^{6} f  \|_{L^2_v} + \langle v \rangle^{\gamma+2} \| \langle v \rangle^{8} f  \|_{L^2_v}  .
\end{aligned}
$$
\end{proof}

\begin{lem}\label{lem:boundL20}
For any function $f=f(v)$ smooth enough and $\alpha \in \R$, there holds:
$$
\|(c*(\sqrt{M}f)) \la v \ra^\alpha \sqrt{M}\|_{L^2_{v}} \lesssim \|M^{\frac14} f\|_{L^2_{v}}.
$$
\end{lem}
\begin{proof}
See \cite[Proof of Lemma 2.12]{CTW}.
\end{proof}

We are now able to obtain some upper bounds on the term $L_2$ defined in \eqref{eq:def:L2} of the collision operator $L$ as follows.

\begin{lem}\label{lem:boundL2}
For any function $f=f(x,v)$ smooth enough, there holds:

\medskip\noindent
(i) For any $\alpha \in \R$, one has 
\begin{align*}
\| \langle v \rangle^\alpha L_2 f \|_{L^2_{x,v}} 
&\lesssim  \| M^{\frac14} f \|_{L^2_{x,v}}.
\end{align*}

\medskip\noindent
(ii) For any $1 \le k \le 3$, one has
\begin{align*}
\|  \widetilde \nabla_{v_k} (L_2 f )\|_{L^2_{x,v}} 
&\lesssim \| M^{\frac14} f \|_{L^2_{x,v}} +\| M^{\frac14} \widetilde \nabla_v f \|_{L^2_{x,v}}.
\end{align*}

\medskip\noindent
(iii) For any $1 \le k \le 3$, one has
\begin{align*}
\|  \widetilde \nabla_{x_k}( L_2 f )\|_{L^2_{x,v}} 
&\lesssim  \| M^{\frac14} \widetilde \nabla_xf \|_{L^2_{x,v}}.
\end{align*}
\end{lem}

\begin{proof}
(i) We write
$$
\begin{aligned}
\langle v \rangle^\alpha L_2 f
&= 
- \left\{ \left( a_{ij} *_v \sqrt M f \right) v_i v_j 
- \left(a_{ii} *_v \sqrt M f \right)
+ \left( c *_v \sqrt M f \right) \right\} \langle v \rangle^\alpha \sqrt M
\end{aligned}
$$
from which, thanks to Lemma~\ref{lem:conv}, we deduce
$$
\begin{aligned}
|\langle v \rangle^\alpha L_2 f|
&\lesssim \| M^{\frac14} f \|_{L^2_v}  \langle v \rangle^{\gamma + \alpha + 2} \sqrt M
+ |( c *_v \sqrt M f )|  \langle v \rangle^\alpha \sqrt M .
\end{aligned}
$$
We conclude the proof by taking the $L^2$ norm of the last estimate and using Lemma~\ref{lem:boundL20}.

\medskip\noindent
(ii) Writing $\widetilde \nabla_{v_k} = B_{k\ell} \partial_{v_\ell}$ we compute
$$
\begin{aligned}
\widetilde \nabla_{v_k} (L_2 f) 
&= - B_{k\ell} \bigg\{ \left( a_{ij} *_v \partial_{v_\ell}(\sqrt M f) \right) v_i v_j
+ \left( a_{ij} *_v \sqrt M f \right) \partial_{v_\ell}(v_i v_j) \\
&\quad
- \left(a_{ii} *_v \partial_{v_\ell}(\sqrt M f) \right)
+ \left( c *_v \partial_{v_\ell}(\sqrt M f) \right) \bigg\} \sqrt M \\
&\quad 
- \left\{ \left( a_{ij} *_v \sqrt M f \right) v_i v_j 
- \left(a_{ii} *_v \sqrt M f \right)
+ \left( c *_v \sqrt M f \right) \right\} \widetilde \nabla_{v_k} \sqrt M.
\end{aligned}
$$
Thanks to Lemmas~\ref{lem:conv} and~\ref{lem:boundL20}, we obtain that
$$
\|\widetilde \nabla_{v_k} (L_2 f) (x,\cdot)\|_{L^2_v} 
\lesssim \| M^{\frac14}  f (x,\cdot)\|_{L^2_v} + \| M^{\frac14} \widetilde \nabla_{v} f(x,\cdot) \|_{L^2_v} ,
$$
and we conclude by integrating in $x$ this last estimate.

\medskip\noindent
(iii) Writing $\widetilde \nabla_{x_k} = B_{k\ell} \partial_{x_\ell}$ we compute
$$
\begin{aligned}
\widetilde \nabla_{x_k} (L_2 f) 
&= - B_{k\ell} \bigg\{ \left( a_{ij} *_v \sqrt M \partial_{x_\ell} f \right) v_i v_j
- \left(a_{ii} *_v \sqrt M\partial_{x_\ell} f \right)
+ \left( c *_v \sqrt M \partial_{x_\ell} f \right) \bigg\} \sqrt M ,
\end{aligned}
$$
and we obtain the wanted result thanks to Lemmas~\ref{lem:conv} and~\ref{lem:boundL20}.
\end{proof}

\section{Estimates on the linearized problem} \label{sec:linear}

We recall that the functional spaces $\XXX$, $\YYY_1$ and $\ZZZ_1^\eps$ are respectively defined in~\eqref{def:normXXX},~\eqref{def:normYYY1} and~\eqref{def:normZZZ1eps}, that the operator $\Lambda_\eps$ is given in~\eqref{def:Lambdaeps} and that $\Pi$ is the projector onto the kernel of $\Lambda_\eps$ (see~\eqref{def:Pi}).
We consider~$U^\eps(t)$ the semigroup associated with~$\Lambda_\eps$ and study its decay and regularization properties.

\begin{theo} \label{theo:mainlinear}
For any $\sigma \in (0,\sigma_0)$, we have:
	$$
	\|U^\eps(t) (\Id- \Pi)\|_{\XXX \to \XXX} \lesssim e^{-\sigma t},
	$$
	$$
	\|U^\eps(t) (\Id- \Pi)\|_{\XXX \to \YYY_1} \lesssim \frac{e^{-\sigma t}}{\min(1,\sqrt{t})}
	\quad \text{and} \quad 
	\|U^\eps(t) (\Id- \Pi)\|_{\XXX \to \ZZZ_1^\eps} \lesssim \frac{e^{-\sigma t}}{\min(1,t^{3/2})},
	$$
where $\sigma_0$ is defined in Proposition~\ref{prop:hypoL2}. 
\end{theo} 

Our proof is based on hypocoercivity tricks and thus on direct energy estimates on the whole problem $\partial_t f = \Lambda_\eps f$. Our method to prove this theorem is of particular interest when one wants to extend the analysis to the whole nonlinear problem (see the paragraph at the beginning of Section~\ref{sec:Cauchyreg}). Indeed, it is based on a micro-macro decomposition of the solution and thus allows to identify the different behaviors of the microscopic and macroscopic parts of the solution. We will use this approach in Section~\ref{sec:Cauchyreg} in which we develop a Cauchy theory for the nonlinear Landau equation as well as some regularization estimates on it.

Let us finally notice that we are actually able to prove Theorem~\ref{theo:mainlinear} by using a splitting of $\Lambda_\eps$ as presented in Section~\ref{sec:linear2}. By establishing nice estimates on each part of the splitting, we are then able to recover the wanted estimates on the whole semigroup $U^\eps(t)$ thanks to Duhamel formula. Such an analysis does not allow us to develop our Cauchy theory and regularization estimates for the nonlinear problem but will be useful to study our problem of hydrodynamical limit, we thus postpone it to Section~\ref{sec:linear2}.

\subsection{Hypocoercivity estimates} \label{subsec:hypo}

In this part, we state some hypocoercivity results for our linearized operator $\Lambda_\eps$ defined in~\eqref{def:Lambdaeps}. The first one provides a result of hypocoercivity in~$L^2_{x,v}$ and the proof is a mere adaptation of the one provided in~\cite[Theorem~5.1]{BCMT} in the more complicated case of bounded domains with various boundary conditions. For sake of completeness, we give the proof of Proposition~\ref{prop:hypoL2} in Appendix~\ref{app:hypo}.

\begin{prop} \label{prop:hypoL2}
There exists a norm $\Nt \cdot \Nt_{L^2_{x,v}}$ on $L^2_{x,v}$ (with associated scalar product~$\langle \! \langle \cdot, \cdot \rangle \! \rangle_{L^2_{x,v}}$) equivalent to the standard norm $\| \cdot \|_{L^2_{x,v}}$ which satisfies the following property: For any $f \in \operatorname{Dom} \Lambda_\eps \cap (\operatorname{Ker} \Lambda_\eps)^\perp$,
$$
\la\!\la \Lambda_\eps f, f \ra\!\ra_{L^2_{x,v}} \leq
 - \sigma_0 \Nt f \Nt_{L^2_{x,v}}^2 
- \kappa_0 \| f \|_{L^2_x(H^1_{v,*})}^2
- \frac{\kappa_0}{\eps^2} \| f^\perp \|_{L^2_x (H^1_{v,*})}^2,
$$
for some constructive constants $0 < \sigma_0 < \sigma_L$ (where $\sigma_L$ is defined in~\eqref{eq:spectralgap}), $ \kappa_0 >0$ and where $f^\perp$ is defined in~\eqref{def:micro}. 
\end{prop}

Roughly speaking, the norm $\Nt \cdot \Nt_{L^2_{x,v}}$ is of the following form 
	$$
	\Nt f \Nt^2_{L^2_{x,v}} 
	= \|f\|^2_{L^2_{x,v}} + \eps \sum_{i=1}^3 \eta_i \left\langle \partial_{x_i} \Delta_x^{-1} \pi f, \widetilde \pi_i f\right\rangle_{L^2_x}
	$$
where $\widetilde \pi_i : L^2_{x,v} \to L^2_x$ is some suitable moment operator, the inverse laplacian $\Delta_x^{-1}$ is suitably defined and the constants $\eta_i$ are chosen to be small enough (see~\eqref{def:Nt} in the proof). 
The norm $\Nt \cdot \Nt_{L^2_{x,v}}$ thus depends on $\eps$ but is equivalent to the usual norm~$\|\cdot\|_{L^2_{x,v}}$ uniformly in $\eps$, this explains the fact that we do not mention the dependency in $\eps$ in our notation.

\smallskip
In the following proposition, we provide a result of hypocoercivity in $\XXX$. Notice that obtaining hypocoercivity in $H^3_xL^2_v$ is a straightforward consequence of the previous proposition since derivatives in $x$ commute with $\Lambda_\eps$. Due to the presence of additional weights in the definition of the space $\XXX$, we have to exhibit a new norm equivalent to the usual one for which we can recover a suitable energy estimate.

\begin{prop} \label{prop:hypoX}
There exists a norm $\Nt \cdot \Nt_\XXX$ on $\XXX$ (with associated scalar product denoted by~$\langle \! \langle \cdot, \cdot \rangle \! \rangle_{\XXX}$) equivalent to the standard norm $\| \cdot \|_{\XXX}$ which satisfies the following property: For any $f \in \operatorname{Dom} \Lambda_\eps \cap (\operatorname{Ker} \Lambda_\eps)^\perp$,
\begin{equation}\label{eq:hypo-Lambda}
	\la\!\la \Lambda_\eps f, f \ra\!\ra_\XXX
	\leq - \sigma \Nt f \Nt_{\XXX}^2 - \kappa \| f \|_{\YYY_1}^2 - \frac{\kappa}{\eps^2} \| f^\perp \|_{\YYY_1}^2, 
\end{equation}
for any $0 < \sigma < \sigma_0$ (where $\sigma_0$ is defined in Proposition~\ref{prop:hypoL2}) and for some $\kappa >0$. As a consequence there holds, for all $t \ge 0$, 
\begin{equation}\label{eq:Ueps-decay}
\| U^\eps(t)(\Id - \Pi) \|_{\XXX \to \XXX} \le C e^{-\sigma t}.
\end{equation}
\end{prop}

\begin{proof}
We define the inner product $\la\!\la \cdot , \cdot \ra\!\ra_{\XXX}$ on $\XXX$ by 
\begin{equation}\label{eq:def:ps-triple-XXX}
\begin{aligned}
\la\!\la f , g \ra\!\ra_{\XXX} 
&:= \sum_{i=0}^2 \left( \delta \la \langle v \rangle^{(3-i)(\frac{\gamma}{2}+1)} \, \nabla_x^i f^\perp , \langle v \rangle^{(3-i)(\frac{\gamma}{2}+1)} \, \nabla_x^i g^\perp  \ra_{L^2_{x,v}} 
+ \la\!\la \nabla_x^i f, \nabla_x^i g \ra\!\ra_{L^2_{x,v}} \right) \\
&\quad 
+  \la\!\la \nabla_x^3 f , \nabla_x^3 g \ra\!\ra_{L^2_{x,v}}
\end{aligned}
\end{equation}
so that its associated norm is given by
\begin{equation}\label{eq:def:norme-triple-XXX}
\Nt f \Nt^2_\XXX := \sum_{i=0}^2 \left( \delta\| \langle v \rangle^{(3-i)(\frac{\gamma}{2}+1)} \, \nabla_x^i f^\perp\|_{L^2_{x,v}}^2 
+ \Nt \nabla_x^i f \Nt_{L^2_{x,v}}^2 \right)
+  \Nt \nabla_x^3 f \Nt_{L^2_{x,v}}^2
\end{equation}
for some constant $\delta\in (0,1)$ to be chosen small enough, and where $\la\!\la \cdot , \cdot \ra\!\ra_{L^2_{x,v}}$ and $\Nt \cdot \Nt_{L^2_{x,v}}$ are defined in Proposition~\ref{prop:hypoL2} (see~\eqref{def:Nt}). We first observe that this norm is equivalent to the norm $\| \cdot \|_\XXX$.

Let $\sigma' \in (\sigma, \sigma_0)$ be fixed and $f \in \XXX \cap (\operatorname{Ker} \Lambda_\eps)^\perp$. 
We shall prove 
\begin{equation}\label{eq:hypo-Lambda-bis}
\la\!\la \Lambda_\eps f, f \ra\!\ra_\XXX
\le
- \sigma' \Nt f \Nt_\XXX^2 - \frac{\kappa'}{\eps^2} \| f^\perp \|_{\YYY_1}^2 
\end{equation}
for some constant $\kappa'>0$, which readily implies \eqref{eq:hypo-Lambda} with some constant $\kappa \in (0 , \min (\sigma' - \sigma , \kappa'))$ by decomposing the first term with $\sigma' = \sigma + (\sigma'-\sigma)$ and using the fact that $\| \pi f \|_{\YYY_1} \lesssim \Nt \pi f \Nt_\XXX $. Estimate \eqref{eq:Ueps-decay} is then a direct consequence of \eqref{eq:hypo-Lambda}.

\medskip\noindent
\textit{Step 1.}
We first compute 
\begin{align*}
    \la\!\la \Lambda_\eps f, f \ra\!\ra_\XXX
	&=  \sum_{i=0}^2  \delta \la  \langle v \rangle^{(3-i)(\frac{\gamma}{2}+1)} \, \nabla_x^i (\Lambda_\eps f)^\perp , \langle v \rangle^{(3-i)(\frac{\gamma}{2}+1)} \, \nabla_x^i f^\perp \ra_{L^2_{x,v}} \\ 
	&\quad +  \sum_{i=0}^3 \la\!\la \nabla_x^i \Lambda_\eps f , \nabla_x^i f \ra\!\ra_{L^2_{x,v}},
\end{align*}
and we observe that, thanks to Proposition~\ref{prop:hypoL2} and the fact that $\nabla_x$ commutes with $\Lambda_\eps$, we already have
\begin{align*}
&\sum_{i=0}^3 \la\!\la \nabla_x^i \Lambda_\eps f , \nabla_x^i f \ra\!\ra_{L^2_{x,v}} \\
&\qquad\le 
- \sum_{i=0}^3 \left\{ \sigma_0  \Nt \nabla_x^if \Nt_{L^2_{x,v}}^2
+ \kappa_0 \| \nabla_x^if \|_{L^2_x(H^1_{v,*})}^2 
+\frac{\kappa_0}{\eps^2} \| \nabla_x^if^\perp \|_{L^2_x (H^1_{v,*})}^2 \right\}.
\end{align*}

\medskip\noindent
\textit{Step 2.} We first observe that 
$$
(\Lambda_\eps f)^\perp 
= (\mathrm{Id}-\pi) \Lambda_\eps f
= \frac{1}{\eps^2}(\mathrm{Id}-\pi) L f - \frac{1}{\eps} (\mathrm{Id}-\pi) (v \cdot \nabla_x f),
$$
from which we obtain, since $\pi L = 0$ and $L f = L f^\perp$, that 
\begin{equation}\label{eq:dtfperp}
(\Lambda_\eps f)^\perp  
= \frac{1}{\eps^2} L f^\perp 
- \frac{1}{\eps} \left\{  v \cdot \nabla_x f^\perp + v \cdot \nabla_x (\pi f) - \pi (v \cdot \nabla_x f)   \right\}.
\end{equation}
We therefore get, for any $i \in \{0,1,2\}$ and using that the transport operator $v \cdot \nabla_x$ is skew-adjoint,
$$
\begin{aligned}
\la \langle v \rangle^{(3-i)(\frac{\gamma}{2}+1)} \, \nabla_x^i (\Lambda_\eps f)^\perp , \langle v \rangle^{(3-i)(\frac{\gamma}{2}+1)} \, \nabla_x^if^\perp  \ra_{L^2_{x,v}} 
&=: \frac{1}{\eps^2} R_1^i 
- \frac{1}{\eps} R_2^i
+ \frac{1}{\eps} R_3^i
\end{aligned}
$$
with
$$
\begin{aligned}
R_1^i &:= \la \langle v \rangle^{(3-i)(\frac{\gamma}{2}+1)} \, \nabla_x^i L f^\perp , \langle v \rangle^{(3-i)(\frac{\gamma}{2}+1)} \, \nabla_x^if^\perp  \ra_{L^2_{x,v}}, \\
R_2^i &:= \la \langle v \rangle^{(3-i)(\frac{\gamma}{2}+1)} \, \nabla_x^i (v \cdot \nabla_x (\pi f) ) , \langle v \rangle^{(3-i)(\frac{\gamma}{2}+1)} \, \nabla_x^if^\perp  \ra_{L^2_{x,v}}, \\
R_3^i &:= \la \langle v \rangle^{(3-i)(\frac{\gamma}{2}+1)} \, \nabla_x^i ( \pi (v \cdot \nabla_x f) ) , \langle v \rangle^{(3-i)(\frac{\gamma}{2}+1)} \, \nabla_x^if^\perp  \ra_{L^2_{x,v}},
\end{aligned}
$$
and we treat each term separately. For simplicity we denote $\omega_i =  \langle v \rangle^{(3-i)(\frac{\gamma}{2}+1)} $ in the sequel.

\medskip\noindent
\textit{Step 3.} We deal with the term $R_1^i$. Since $\nabla_x$ commutes with $L$ we have 
$$
\begin{aligned}
R_1^i 
&= \la L (\omega_i \nabla_x^i f^\perp) , \omega_i \nabla_x^i f^\perp \ra_{L^2_{x,v}}
+ \la [\omega_i, L_1] \nabla_x^i f^\perp , \omega_i \nabla_x^i f^\perp \ra_{L^2_{x,v}} \\
&\quad + \la [\omega_i, L_2] \nabla_x^i f^\perp , \omega_i \nabla_x^i f^\perp \ra_{L^2_{x,v}} \\
&=: R_{11}^i + R_{12}^i + R_{13}^i,
\end{aligned}
$$
where we recall that $L_1$ and $L_2$ are defined in \eqref{eq:def:L1} and \eqref{eq:def:L2}, respectively.
Thanks to the spectral gap estimate~\eqref{eq:spectralgap}, we have 
$$
\begin{aligned}
R_{11}^i
&\le - \sigma_{L} \| \omega_i \nabla_x^i f^\perp - \pi(\omega_i \nabla_x^i f^\perp) \|_{L^2_x (H^1_{v,*})}^2  \\
&\le - \sigma_{L} \| \omega_i \nabla_x^i f^\perp \|_{L^2_x (H^1_{v,*})}^2 
+ C \|  \nabla_x^i f^\perp \|_{L^2_{x,v}}^2,
\end{aligned}
$$
for some constant $C>0$. From Lemma~\ref{lem:commutatorL} we get
$$
\begin{aligned}
R_{12}^i
&= -\la [\omega_i , \widetilde \nabla_{v_\ell}] \nabla_x^i f^\perp, \widetilde \nabla_{v_\ell} (\omega_i \nabla_x^i f^\perp) \ra_{L^2_{x,v}}
- \la  [\omega_i , \widetilde \nabla_{v_\ell}^*] \nabla_x^i f^\perp, \widetilde \nabla_{v_\ell}^* (\omega_i \nabla_x^i f^\perp) \ra_{L^2_{x,v}} \\
&\quad 
-\la \left[  [\omega_i , \widetilde \nabla_{v_\ell}^*], \widetilde \nabla_{v_\ell} \right] \nabla_x^i f^\perp , \omega_i \nabla_x^i f^\perp \ra_{L^2_{x,v}} .
\end{aligned}
$$
Using Lemma~\ref{lem:commutator} and observing that $\| \widetilde \nabla_v g \|_{L^2_{x,v}} + \| \widetilde \nabla_v^* g \|_{L^2_{x,v}} \lesssim \| g \|_{L^2_x (H^1_{v,*})}$, we then obtain
$$
\begin{aligned}
|R_{12}^i|
&\le C \| \omega_i \nabla_x^i f^\perp \|_{L^2_x (H^1_{v,*})} 
\| \langle v \rangle^{\frac{\gamma}{2}-1} \omega_i \nabla_x^i f^\perp    \|_{L^2_{x,v}}
+ C \| \langle v \rangle^{\frac{\gamma}{2}-\frac12} \omega_i \nabla_x^i f^\perp \|_{L^2_{x,v}}^2 \\
&\le (\sigma_L - \sigma'') \| \omega_i \nabla_x^i f^\perp \|_{L^2_x (H^1_{v,*})}^2
+ C \| \nabla_x^i f^\perp \|_{L^2_{x,v}}^2
\end{aligned}
$$
for any $\sigma'' \in (\sigma' , \sigma_0)$, where we have used Young's inequality in last line.
For the term $R_{13}^i$, we use Lemma~\ref{lem:boundL2} to obtain
$$
|R_{13}^i|
\le C \| \nabla_x^i f^\perp \|_{L^2_{x,v}}^2.
$$
Gathering previous estimates we finally get 
$$
R_1^i \le - \sigma'' \| \omega_i \nabla_x^i f^\perp \|_{L^2_x (H^1_{v,*})}^2 
+ C \| \nabla_x^i f^\perp \|_{L^2_{x,v}}^2.
$$

\medskip\noindent
\textit{Step 4.} We deal with the terms $R_2^i$ and $R_3^i$. Observing that 
$$
\| \omega_i \nabla_x^i (v \cdot \nabla_x (\pi f)) \|_{L^2_{x,v}} 
+
\| \omega_i \nabla_x^i ( \pi (v \cdot \nabla_x f)) \|_{L^2_{x,v}}
\lesssim \| \nabla_x^{i+1} f \|_{L^2_{x,v}}
$$ 
we obtain
$$
\begin{aligned}
|R_2^i| + |R_3^i| 
&\le C \| \nabla_x^{i+1} f \|_{L^2_{x,v}} \| \omega_i \nabla_x^{i} f^\perp \|_{L^2_{x,v}} \\
&\le \frac{(\sigma'' - \sigma''')}{\eps}\| \omega_i \nabla_x^i f^\perp \|_{L^2_x (H^1_{v,*})}^2 + C \eps \Nt \nabla_x^{i+1} f \Nt_{L^2_{x,v}}^2,
\end{aligned}
$$ 
for any $\sigma''' \in (\sigma' , \sigma'')$, where we have used Young's inequality and that $\| \cdot \|_{L^2_{x,v}}$ and $\Nt \cdot \Nt_{L^2_{x,v}}$ are equivalent in last line.

\medskip\noindent
\textit{Step 5.} Gathering previous estimates we then obtain
$$ 
\begin{aligned}   
\la\!\la \Lambda_\eps f, f \ra\!\ra_\XXX
&\le  - \sum_{i=0}^3 \left\{ \sigma_0  \Nt \nabla_x^if \Nt_{L^2_{x,v}}^2
+ \kappa_0 \| \nabla_x^if \|_{L^2_x(H^1_{v,*})}^2 
+\frac{\kappa_0}{\eps^2} \| \nabla_x^if^\perp \|_{L^2_x (H^1_{v,*})}^2 \right\} \\
&\quad 
+ \sum_{i=0}^2 \left\{ - \frac{\sigma''' \delta}{\eps^2} \| \omega_i \nabla_x^i f^\perp \|_{L^2_x(H^1_{v,*})}^2
+ \frac{C \delta}{\eps^2} \| \nabla_x^i f^\perp \|_{L^2_{x,v}}^2 
+ C \delta  \Nt \nabla_x^{i+1} f \Nt_{L^2_{x,v}}^2 \right\}.
\end{aligned}
$$
Recalling that $\| \cdot \|_{L^2_{x,v}} \le \| \cdot \|_{L^2_x(H^1_{v,*})}$, it follows 
$$ 
\begin{aligned}   
\la\!\la \Lambda_\eps f, f \ra\!\ra_\XXX
&\le  
-\sigma_0  \Nt f \Nt_{L^2_{x,v}}^2
-(\sigma_0 - C \delta) \sum_{i=1}^3 \Nt \nabla_x^if \Nt_{L^2_{x,v}}^2 
- \frac{\sigma'}{\eps^2}  \sum_{i=0}^2 \delta \| \omega_i \nabla_x^i f^\perp \|_{L^2_{x,v}}^2 \\
&\quad 
- \frac{(\kappa_0 - C \delta)}{\eps^2} \sum_{i=0}^3  \| \nabla_x^i f^\perp \|_{L^2_x (H^1_{v,*})}^2
- \frac{(\sigma''' - \sigma')}{\eps^2}  \sum_{i=0}^2 \delta \| \omega_i \nabla_x^i f^\perp \|_{L^2_x(H^1_{v,*})}^2.
\end{aligned}
$$
We then choose $\delta \in (0,1)$ small enough such that $\sigma_0 - C \delta \ge \sigma'$ and $\kappa_0 - C \delta > 0$, therefore we obtain \eqref{eq:hypo-Lambda-bis} with $\kappa' = \min ((\sigma''' - \sigma')\delta, \kappa_0-C\delta)>0$, which completes the proof.
\end{proof}

\subsection{Regularization estimates} \label{subsec:reg}

\begin{prop}\label{prop:regLambda}
The solution $f(t) = U^\eps(t) f_{\mathrm{in}}$ to the equation 
$$
\left\{
\begin{aligned}
& \partial_t f = \Lambda_\eps f \\
& f(0) = f_{\mathrm{in}} \in \XXX \cap (\operatorname{Ker} \Lambda_\eps)^\perp
\end{aligned}
\right.
$$ 
satisfies, for all $t > 0$,
\begin{equation}\label{eq:Lambda-reg-1}
\| f(t) \|_{\YYY_1} \le C \frac{e^{-\sigma t}}{\min (1 , \sqrt t)} \, \| f_\mathrm{in} \|_\XXX
\end{equation}
and
\begin{equation}\label{eq:Lambda-reg-2}
\| f(t) \|_{\ZZZ_1^\eps} \le C \frac{e^{-\sigma t}}{\min (1 ,  t^{3/2})} \, \| f_\mathrm{in} \|_\XXX,
\end{equation}
for any $0 < \sigma < \sigma_0$ (where $\sigma_0$ is defined in Proposition~\ref{prop:hypoL2}).
\end{prop}

\begin{rem}
Notice that thanks to the second inequality, one can in particular recover a gain of one derivative in the spatial variable (with the associated anisotropic gain of weight in velocity), at the price of loosing a $1/\eps$. As already mentioned, this is explained by the fact that the gain comes from the transport operator which does not act as the same scale as the collision operator in velocity. Notice also that in~\cite{BMM}, the authors were facing a similar singularity in~$\eps$ when wanting to obtain a gain of regularity in the spatial variable for the hard-spheres Boltzmann equation. The latter equation is not hypoelliptic but thanks to a suitable use of averaging lemmas, the authors were also able to obtain regularization properties in the spatial variable with the same singularity in $\eps$. 
\end{rem}

\begin{proof}[Proof of Proposition~\ref{prop:regLambda}]
We shall prove that for any $t \in (0,1]$ there holds 
\begin{equation}\label{eq:Lambda-reg-1-bis}
\| f(t) \|_{\YYY_1} \le \frac{C}{\sqrt t} \, \| f_\mathrm{in} \|_\XXX
\end{equation}
and
\begin{equation}\label{eq:Lambda-reg-2-bis}
\| f(t) \|_{\ZZZ_1^\eps} \le \frac{C}{t^{3/2}} \, \| f_\mathrm{in} \|_\XXX,
\end{equation}
which readily imply \eqref{eq:Lambda-reg-1} and \eqref{eq:Lambda-reg-2} thanks to the exponential decay of $U^\eps$ in $\XXX$ from Proposition~\ref{prop:hypoX}.

\medskip\noindent
\textit{Step 1.}
Define the functional
	\begin{multline*}
	\UUU_\eps(t,f) 
	= \Nt f \Nt_\XXX^2 
	+ \alpha_1 {t} \left(\| \widetilde \nabla_v f^{\perp} \|_{\XXX}^2 +  K \| \langle v \rangle^{\frac{\gamma}{2}+1} f^\perp\|^2_\XXX\right) \\
	+\eps \alpha_2 {t}^2 \la \widetilde \nabla_v f , \widetilde \nabla_x f \ra_{\XXX} 
	+ \eps^2 \alpha_3 {t}^3 \left(\| \widetilde \nabla_x f \|_{\XXX}^2 + K \| \langle v \rangle^{\frac{\gamma}{2}} \nabla_x f\|^2_\XXX \right),
	\end{multline*}
with constants $0 < \alpha_3 \ll \alpha_2 \ll \alpha_1 \ll 1$  so that $\alpha_2 \le \sqrt{\alpha_1 \alpha_3}$ and $K>0$. The constants~$\alpha_i$ will be chosen small enough and $K$ large enough during the proof. 

We easily observe that 
\begin{align*}
\begin{split}
\Nt f \Nt_\XXX^2  + t \left(\| \widetilde \nabla_v f^{\perp} \|_{\XXX}^2 +  \| \langle v \rangle^{\frac{\gamma}{2}+1} f^\perp\|^2_\XXX\right)
+ \eps^2 t^3 \| \widetilde \nabla_x f \|_\XXX^2
\lesssim \UUU_\eps(t,f).
\end{split}
\end{align*}
Remarking that $\| f \|_{\YYY_1} \lesssim \| f^{\perp} \|_{\YYY_1} + \| \pi f \|_\XXX$, we thus obtain 
$$
\| f \|_{\YYY_1} \lesssim \| \widetilde \nabla_v f^{\perp} \|_{\XXX}^2 +  \| \langle v \rangle^{\frac{\gamma}{2}+1} f^\perp\|_\XXX + \| \pi f \|_\XXX,
$$ 
from which we deduce the following lower bounds: For all $t \in [0,1]$ there holds 
\begin{equation}\label{eq:lowerbdd}
t \| f \|_{\YYY_1}^2 \lesssim \UUU_\eps(t,f)
\quad\text{and}\quad
t^3 \| f \|_{\ZZZ_1^\eps}^2 \lesssim \UUU_\eps(t,f).
\end{equation}
Therefore, in order to prove \eqref{eq:Lambda-reg-1-bis} and \eqref{eq:Lambda-reg-2-bis}, it is sufficient to show that 
$$
\frac{\d}{\dt} \UUU_\eps (t , f) \le 0, \quad \forall \, t \in [0,1],
$$
which we shall do next.
We thus compute 
\begin{equation}\label{eq:dtUeps}
\begin{aligned}
\frac{\mathrm{d}}{\mathrm{d}t}\UUU_\eps(t,f)
&=\frac{\mathrm{d}}{\mathrm{d}t}\Nt f \Nt_\XXX^2
+\alpha_{1}\Big( \| \langle v \rangle^{\frac{\gamma}{2}+1} f^{\perp}  \|_\XXX^2 
+ \|  \widetilde{\nabla}_v f^{\perp} \|_\XXX^2 \Big) \\
&\quad
+\alpha_{1}t\frac{\mathrm{d}}{\mathrm{d}t}\Big( \| \langle v \rangle^{\frac{\gamma}{2}+1} f^{\perp}  \|_\XXX^2 
+ \|  \widetilde{\nabla}_v f^{\perp} \|_\XXX^2 \Big)\\
&\quad
+2\alpha_{2}{\varepsilon} t\la \widetilde{\nabla}_v f ,  \widetilde{\nabla}_x f \ra_\XXX +\alpha_{2}\varepsilon t^2  \frac{\mathrm{d}}{\mathrm{d}t}\la \widetilde{\nabla}_v f ,  \widetilde{\nabla}_x f \ra_\XXX \\
&\quad
+3\alpha_{3}\varepsilon^2t^{2} \Big(\|  \widetilde{\nabla}_x f \|_\XXX^2
+ K\|  \langle v \rangle^{\frac{\gamma}{2}} \nabla_x f \|_\XXX^2 \Big) \\
&\quad
+\alpha_{3}\varepsilon^2t^{3}\frac{\mathrm{d}}{\mathrm{d}t} \Big(\|  \widetilde{\nabla}_x f \|_\XXX^2
+ K\|  \langle v \rangle^{\frac{\gamma}{2}} \nabla_x f \|_\XXX^2 \Big)
\end{aligned}
\end{equation}
and we estimate each term separately in the sequel. In order to simplify, we introduce the notations $g_i = \nabla_x^i f$, $g_i^\perp = (\nabla_x^i f)^\perp = \nabla_x^i f^\perp$  and $\omega_i = \langle v \rangle^{(3-i) (\frac{\gamma}{2} +1)}$ so that 
$$
\| f \|_\XXX^2 = \sum_{i=0}^3 \| \omega_i g_i \|_{L^2_{x,v}}^2 , 
\quad 
\| f \|_{\YYY_1}^2 = \sum_{i=0}^3 \| \omega_i g_i \|_{L^2_x (H^1_{v,*})}^2 , 
\quad 
\| f \|_{\YYY_2}^2 = \sum_{i=0}^3 \| \omega_i g_i \|_{L^2_x(H^2_{v,*})}^2.
$$

\medskip\noindent 
\textit{Step 2.} 
From Proposition~\ref{prop:hypoX}, we already have
\begin{align}\label{eq:dtfXXX}
\frac{\d}{\dt} \Nt f \Nt_{\XXX}^2
\le - \kappa_0 \| f \|_{\YYY_1}^2 - \frac{\kappa_0}{\eps^2} \| f^\perp \|_{\YYY_1}^2,
\end{align}
for some constant $\kappa_0>0$.

\medskip\noindent
\textit{Step 3.} 
We deal in this step with the term $( K\| \langle v \rangle^{\frac{\gamma}{2}+1} f^{\perp}  \|_\XXX^2 
+ \|  \widetilde{\nabla}_v f^{\perp} \|_\XXX^2 )$.
We split the computations into two parts.

\medskip\noindent
\textit{Step 3.1.} We first compute 
$$
\begin{aligned}
\frac12 \frac{\d}{\dt} \| \langle v \rangle^{\frac{\gamma}{2}+1} f^\perp \|_\XXX^2
& = \frac12 \sum_{i=0}^3 \frac{\d}{\dt} \| \omega_i  \langle v \rangle^{\frac{\gamma}{2}+1} g_i^\perp \|_{L^2_{x,v}}^2.
\end{aligned}
$$
Observing that $f^\perp$ satisfies the equation $\partial_t f^\perp = (\Lambda_\eps f)^\perp$, with $(\Lambda_\eps f)^\perp$ given by \eqref{eq:dtfperp}, using that the transport operator is skew-adjoint and that derivatives in $x$ commute with $\pi$ and $\Lambda_\eps$, for any $i \in \{0,1,2,3\}$ we obtain
$$
\begin{aligned}
\frac12 \frac{\d}{\dt} \| \omega_i \langle v \rangle^{\frac{\gamma}{2}+1} g_i^\perp \|_{L^2_{x,v}}^2
&= \la \omega_i \langle v \rangle^{\frac{\gamma}{2}+1} \partial_t g_i^\perp , \omega_i \langle v \rangle^{\frac{\gamma}{2}+1} g_i^\perp \ra_{L^2_{x,v}} 
=: \frac{1}{\eps^2}J_1^i - \frac{1}{\eps}J_2^i + \frac{1}{\eps}J_3^i
\end{aligned}
$$
with
$$
\begin{aligned}
J_1^i &:= \la \omega_i \langle v \rangle^{\frac{\gamma}{2}+1} L g_i^\perp , \omega_i \langle v \rangle^{\frac{\gamma}{2}+1} g_i^\perp  \ra_{L^2_{x,v}} \\
J_2^i &:= \la \omega_i \langle v \rangle^{\frac{\gamma}{2}+1}(v \cdot \nabla_x (\pi g_i) ) , \omega_i \langle v \rangle^{\frac{\gamma}{2}+1} g_i^\perp \ra_{L^2_{x,v}} \\
J_3^i &:= \la \omega_i \langle v \rangle^{\frac{\gamma}{2}+1}( \pi (v \cdot \nabla_x g_i) ) , \omega_i \langle v \rangle^{\frac{\gamma}{2}+1} g_i^\perp \ra_{L^2_{x,v}}.
\end{aligned}
$$
For the first term, we write
$$
\begin{aligned}
J_1^i 
&= \la L (\omega_i \langle v \rangle^{\frac{\gamma}{2}+1} g_i^\perp )  + [\omega_i\langle v \rangle^{\frac{\gamma}{2}+1}, L_1] g_i^\perp 
+  [\omega_i \langle v \rangle^{\frac{\gamma}{2}+1}, L_2] g_i^\perp, \omega_i \langle v \rangle^{\frac{\gamma}{2}+1} g_i^\perp \ra_{L^2_{x,v}} \\
&=: J_{11}^i + J_{12}^i + J_{13}^i,
\end{aligned}
$$
where we recall that $L_1$ and $L_2$ are defined in \eqref{eq:def:L1} and \eqref{eq:def:L2}, respectively. Thanks to the spectral gap estimate~\eqref{eq:spectralgap}, one has
$$
\begin{aligned}
J_{11}^i
&\le - 2\kappa \| \omega_i \langle v \rangle^{\frac{\gamma}{2}+1} g_i^\perp - \pi(\omega_i \langle v \rangle^{\frac{\gamma}{2}+1} g_i^\perp) \|_{L^2_x (H^1_{v,*})}^2  \\
&\le - 2\kappa \| \omega_i \langle v \rangle^{\frac{\gamma}{2}+1} g_i^\perp \|_{L^2_x (H^1_{v,*})}^2 
+ C \|  g_i^\perp \|_{L^2_{x,v}}^2,
\end{aligned}
$$
for some constants $\kappa,C>0$. From Lemma~\ref{lem:commutatorL}, we get
$$
\begin{aligned}
J_{12}^i
&=-\la [\omega_i\langle v \rangle^{\frac{\gamma}{2}+1}, \widetilde \nabla_{v_\ell}] g_i^\perp , \widetilde \nabla_{v_\ell}(\omega_i \langle v \rangle^{\frac{\gamma}{2}+1} g_i^\perp )\ra_{L^2_{x,v}} \\
&\quad - \la [\omega_i\langle v \rangle^{\frac{\gamma}{2}+1}, \widetilde \nabla_{v_\ell}^*] g_i^\perp ,  \widetilde \nabla_{v_\ell}^*(\omega_i \langle v \rangle^{\frac{\gamma}{2}+1} g_i^\perp) \ra_{L^2_{x,v}} \\
&\quad
- \la \left[ [\omega_i\langle v \rangle^{\frac{\gamma}{2}+1}, \widetilde \nabla_{v_\ell}^*] , \widetilde \nabla_{v_\ell}  \right] g_i^\perp , \omega_i \langle v \rangle^{\frac{\gamma}{2}+1} g_i^\perp\ra_{L^2_{x,v}} .
\end{aligned}
$$
Using Lemma~\ref{lem:commutator} and observing that $\| \widetilde \nabla_v h \|_{L^2_{x,v}} + \| \widetilde \nabla_v^* h \|_{L^2_{x,v}} \lesssim \| h \|_{L^2_x (H^1_{v,*})}$, we then obtain
$$
\begin{aligned}
|J_{12}^i|
&\le C \|  \omega_i \langle v \rangle^{\frac{\gamma}{2}+1} g_i^\perp \|_{L^2_x (H^1_{v,*})} 
\| \langle v \rangle^{\frac{\gamma}{2}-1}  \omega_i \langle v \rangle^{\frac{\gamma}{2}+1} g_i^\perp    \|_{L^2_{x,v}}
+ C \| \langle v \rangle^{\frac{\gamma}{2}-\frac12}  \omega_i \langle v \rangle^{\frac{\gamma}{2}+1} g_i^\perp \|_{L^2_{x,v}}^2 \\
&\le {\kappa} \|  \omega_i \langle v \rangle^{\frac{\gamma}{2}+1} g_i^\perp \|_{L^2_x (H^1_{v,*})}^2
+ C \| \omega_i  \langle v \rangle^{\frac{\gamma}{2}+1} g_i^\perp \|_{L^2_{x,v}}^2
\end{aligned}
$$
where we have used that $\gamma \le 1$ and Young's inequality in last line.
For the term $J_{13}^i$, we use Lemma~\ref{lem:boundL2} to obtain
$$
|J_{13}^i|
\le C \| \omega_i g_i^\perp \|_{L^2_{x,v}}^2.
$$
We deal with the terms $J_2^i$ and $J_3^i$. Observing that 
$$
\| \omega_i \langle v \rangle^{\frac{\gamma}{2}+1}   (v \cdot \nabla_x (\pi g_i)) \|_{L^2_{x,v}} 
+
\| \omega_i \langle v \rangle^{\frac{\gamma}{2}+1}  ( \pi (v \cdot \nabla_x g_i)) \|_{L^2_{x,v}}
\lesssim \| \omega_i \widetilde \nabla_x g_i \|_{L^2_{x,v}},
$$
we obtain
$$
\begin{aligned}
|J_2^i| + |J_3^i| 
&\le C \|  \omega_i \widetilde \nabla_x g_i \|_{L^2_{x,v}} \| \omega_i \langle v \rangle^{\frac{\gamma}{2}+1} g_i^\perp \|_{L^2_{x,v}} .
\end{aligned}
$$ 
Gathering previous estimates and using that $\| \la v \ra^{\gamma+\frac12} f^\perp \|_{\XXX} \lesssim \| f^\perp \|_{\YYY_1}$, we obtain
\begin{equation}\label{dtmfperp}
\begin{aligned}
\frac12 \frac{\d}{\dt} \| \langle v \rangle^{\frac{\gamma}{2}+1} f^\perp \|_\XXX^2
& \le - \frac{\kappa}{\eps^2} \| \langle v \rangle^{\frac{\gamma}{2}+1} f^\perp \|_{\YYY_1}^2 + \frac{C}{\eps^2} \| f^\perp \|_{\YYY_1}^2 
+ \frac{C}{\eps} \|  f^\perp \|_{\YYY_1} \| \widetilde \nabla_x f \|_{\XXX}
\end{aligned}
\end{equation}
for some constants $\kappa , C >0$.

\medskip\noindent
\textit{Step 3.2.} We now compute 
$$
\begin{aligned}
\frac12 \frac{\d}{\dt} \| \widetilde \nabla_v f^\perp \|_\XXX^2
& = \frac12 \sum_{i=0}^3 \frac{\d}{\dt} \| \omega_i  \widetilde \nabla_v g_i^\perp \|_{L^2_{x,v}}^2 \\
&=: \sum_{i=0}^3  \left(\frac{1}{\eps^2} I_1^i - \frac{1}{\eps} I_2^i + \frac{1}{\eps} I_3^i
- \frac{1}{\eps} I_4^i \right)
\end{aligned}
$$
with
$$
\begin{aligned}
I_1^i &:= \la \omega_i \widetilde \nabla_{v_k} L g_i^\perp , \omega_i \widetilde \nabla_{v_k} g_i^\perp  \ra_{L^2_{x,v}} \\
I_2^i &:= \la \omega_i \widetilde \nabla_{v_k} (v \cdot \nabla_x (\pi g_i) ) , \omega_i \widetilde \nabla_{v_k} g_i^\perp \ra_{L^2_{x,v}} \\
I_3^i &:= \la \omega_i \widetilde \nabla_{v_k} ( \pi (v \cdot \nabla_x g_i) ) , \omega_i \widetilde \nabla_{v_k} g_i^\perp\ra_{L^2_{x,v}}
\\
I_4^i &:= \la \omega_i \widetilde \nabla_{v_k} ( v \cdot \nabla_x g_i ^\perp) , \omega_i \widetilde \nabla_{v_k} g_i^\perp\ra_{L^2_{x,v}}.
\end{aligned}
$$
For the first term, we write
$$
\begin{aligned}
I_1^i 
&= \la L (\omega_i \widetilde \nabla_{v_k} g_i^\perp ) , \omega_i \widetilde \nabla_{v_k} g_i^\perp \ra_{L^2_{x,v}} 
+\la   [\omega_i , L_1] \widetilde \nabla_{v_k} g_i^\perp , \omega_i \widetilde \nabla_{v_k} g_i^\perp \ra_{L^2_{x,v}} \\
&\quad 
+\la  \omega_i [\widetilde \nabla_{v_k}, L_1] g_i^\perp, \omega_i \widetilde \nabla_{v_k} g_i^\perp \ra_{L^2_{x,v}} 
+\la  [\omega_i , L_2] \widetilde \nabla_{v_k} g_i^\perp, \omega_i \widetilde \nabla_{v_k} g_i^\perp \ra_{L^2_{x,v}} \\
&\quad 
+\la  \omega_i [\widetilde \nabla_{v_k}, L_2] g_i^\perp, \omega_i \widetilde \nabla_{v_k} g_i^\perp \ra_{L^2_{x,v}} \\
&=: I_{11}^i + I_{12}^i + I_{13}^i + I_{14}^i + I_{15}^i .
\end{aligned}
$$
Thanks to the spectral gap estimate~\eqref{eq:spectralgap}, one has
$$
\begin{aligned}
I_{11}^i
&\le - 2\kappa \| \omega_i \widetilde \nabla_{v_k} g_i^\perp - \pi(\omega_i \widetilde \nabla_{v_k} g_i^\perp) \|_{L^2_x (H^1_{v,*})}^2  \\
&\le - 2\kappa \| \omega_i \widetilde \nabla_{v_k} g_i^\perp \|_{L^2_x (H^1_{v,*})}^2 
+ C \|  g_i^\perp \|_{L^2_{x,v}}^2,
\end{aligned}
$$
for some constants $\kappa,C>0$. From Lemma~\ref{lem:commutatorL}, we get
$$
\begin{aligned}
I_{12}^i
&= 
- \la  [\omega_i,\widetilde \nabla_{v_\ell}] \widetilde \nabla_{v_k} g_i^\perp , \widetilde \nabla_{v_\ell} (\omega_i \widetilde \nabla_{v_k} g_i^\perp) \ra_{L^2_{x,v}}
- \la    [\omega_i,\widetilde \nabla_{v_\ell}^*] \widetilde \nabla_{v_k} g_i^\perp , \widetilde \nabla_{v_\ell}^*(\omega_i \widetilde \nabla_{v_k} g_i^\perp) \ra_{L^2_{x,v}}\\
&\quad
- \la  \left[ [\omega_i , \widetilde \nabla_{v_\ell}^*] , \widetilde \nabla_{v_\ell} \right] \widetilde \nabla_{v_k} g_i^\perp , \omega_i \widetilde \nabla_{v_k} g_i^\perp \ra_{L^2_{x,v}},
\end{aligned}
$$
and
$$
\begin{aligned}
I_{13}^i
&= 
-\la  [\widetilde \nabla_{v_k}, \widetilde \nabla_{v_\ell}] g_i^\perp, \widetilde \nabla_{v_\ell} (\omega_i^2 \widetilde \nabla_{v_k} g_i^\perp) \ra_{L^2_{x,v}}
-\la   [\widetilde \nabla_{v_k}, \widetilde \nabla_{v_\ell}^*] g_i^\perp, \widetilde \nabla_{v_\ell}^*(\omega_i^2 \widetilde \nabla_{v_k} g_i^\perp) \ra_{L^2_{x,v}}\\
&\quad
-\la  \omega_i \left[ [\widetilde \nabla_{v_k}, \widetilde \nabla_{v_\ell}^*] , \widetilde \nabla_{v_\ell} \right] g_i^\perp, \omega_i \widetilde \nabla_{v_k} g_i^\perp \ra_{L^2_{x,v}}
-\la  \omega_i (\widetilde \nabla_{v_k} \psi) g_i^\perp ,  \omega_i \widetilde \nabla_{v_k} g_i^\perp \ra_{L^2_{x,v}}
.
\end{aligned}
$$
For $I_{12}^i$, using Lemma~\ref{lem:commutator} and observing that $\| \widetilde \nabla_v h \|_{L^2_{x,v}} + \| \widetilde \nabla_v^* h \|_{L^2_{x,v}} \lesssim \| h \|_{L^2_x (H^1_{v,*})}$, we obtain
$$
\begin{aligned}
|I_{12}^i|
&\le C \|  \omega_i \widetilde \nabla_{v_k} g_i^\perp \|_{L^2_x (H^1_{v,*})} 
\| \langle v \rangle^{\frac{\gamma}{2}-1}  \omega_i \widetilde \nabla_{v_k} g_i^\perp    \|_{L^2_{x,v}}
+ C \| \langle v \rangle^{\frac{\gamma}{2}-\frac12}  \omega_i \widetilde \nabla_{v_k} g_i^\perp \|_{L^2_{x,v}}^2. 
\end{aligned}
$$
For $I_{13}^i$, we first observe that 
writing 
$$
\begin{aligned}
\widetilde \nabla_{v_\ell} (\omega_i^2 h)
&= \omega_i \widetilde \nabla_{v_\ell} (\omega_i h) + (\widetilde \nabla_{v_\ell} \omega_i) \omega_i h \\
\widetilde \nabla_{v_\ell}^* (\omega_i^2 h)
&= -\omega_i \widetilde \nabla_{v_\ell} (\omega_i h)
- \left[ (\partial_{v_p} B_{\ell p}) \omega_i + (\widetilde \nabla_{v_\ell} \omega_i) \right] \omega_i h
\end{aligned}
$$ 
and using the bounds \eqref{eq:nablaBij}, we have 
\begin{equation}\label{nablaomega2h}
\| \omega_i^{-1} \widetilde \nabla_{v_\ell} (\omega_i^2 h) \|_{L^2_{x,v}} + \| \omega_i^{-1} \widetilde \nabla_{v_\ell}^* (\omega_i^2 h) \|_{L^2_{x,v}} \lesssim \| \omega_i h \|_{L^2_x(H^1_{v,*})}. 
\end{equation}
Therefore, using Lemma~\ref{lem:commutator} and noticing that 
$$
|\widetilde \nabla_v \psi| 
\lesssim \la v \ra^{\frac{\gamma}{2}+1} |\nabla_v \psi| 
\lesssim \la v \ra^{\frac{\gamma}{2}+1} \la v \ra^{\gamma+1},
$$
we obtain
$$
\begin{aligned}
|I_{13}^i|
&\le C \|  \omega_i \widetilde \nabla_{v_k} g_i^\perp \|_{L^2_x (H^1_{v,*})} \| \omega_i  \langle v \rangle^{\gamma+1} \nabla_v g_i^\perp \|_{L^2_{x,v}}
+ C  \|  \omega_i \widetilde \nabla_{v_k} g_i^\perp \|_{L^2_x (H^1_{v,*})} \| \omega_i \langle v \rangle^{\gamma} g_i^\perp \|_{L^2_{x,v}} \\
&\quad 
+ C \|  \omega_i \langle v \rangle^{\gamma}  \nabla_v g_i^\perp \|_{L^2_{x,v}} \| \langle v \rangle^{\frac{\gamma}{2}+1} \omega_i \widetilde \nabla_{v_k} g_i^\perp \|_{L^2_{x,v}} \\
&\quad 
+ C \|  \omega_i \langle v \rangle^{\gamma-1} g_i^\perp \|_{L^2_{x,v}} \| \langle v \rangle^{\frac{\gamma}{2}+1} \omega_i \widetilde \nabla_{v_k} g_i^\perp \|_{L^2_{x,v}} \\
&\quad
+ C \|  \omega_i \langle v \rangle^{\gamma+1} g_i^\perp \|_{L^2_{x,v}} \| \langle v \rangle^{\frac{\gamma}{2}+1}  \omega_i \widetilde \nabla_{v_k} g_i^\perp \|_{L^2_{x,v}}.
\end{aligned}
$$
Putting together the estimates for $I_{12}^i$ and $I_{13}^i $ and using Young's inequality, it follows
$$
|I_{12}^i|+|I_{13}^i|
\le \kappa \|  \omega_i \widetilde \nabla_{v_k} g_i^\perp \|^2_{L^2_x (H^1_{v,*})}
+ C \| \omega_i \langle v \rangle^{\gamma+1} g_i^\perp \|_{L^2_{x,v}}^2
+ C \| \omega_i \langle v \rangle^{\gamma+1} \nabla_v g_i^\perp \|_{L^2_{x,v}}^2.
$$
For the terms $I_{14}^i$ and $I_{15}^i$, Lemma~\ref{lem:boundL2} yields
$$
|I_{14}^i|
+ |I_{15}^i|
\le C \| \omega_i g_i^\perp \|_{L^2_{x,v}}^2 
+ C \| \omega_i \widetilde \nabla_v g_i^\perp \|_{L^2_{x,v}}^2.
$$
The terms $I_2^i$ and $I_3^i$ can be dealt as before in Step 3.1, and we obtain 
$$
\begin{aligned}
|I_2^i| + |I_3^i| 
&\le C \|  \omega_i \widetilde \nabla_x g_i \|_{L^2_{x,v}} \| \omega_i \widetilde \nabla_{v_k} g_i^\perp \|_{L^2_{x,v}}.
\end{aligned}
$$ 
For the remainder term $I_4^i$, we use Lemma~\ref{lem:commutator} and the fact that the transport operator is skew-symmetric to obtain
$$
I_4^i = 
\la \omega_i \widetilde \nabla_{x_k} g_i^\perp , \omega_i \widetilde \nabla_{v_k} g_i^\perp \ra_{L^2_{x,v}}
\le \| \omega_i \widetilde \nabla_{x_k} g_i^\perp \|_{L^2_{x,v}} \| \omega_i \widetilde \nabla_{v_k} g_i^\perp \|_{L^2_{x,v}}.
$$ 

Gathering previous estimates, we obtain
\begin{equation}\label{dtnablavfperp}
\begin{aligned}
\frac12 \frac{\d}{\dt} \| \widetilde \nabla_{v} f^\perp \|_\XXX^2
& \le - \frac{\kappa}{\eps^2} \| \widetilde \nabla_{v}f^\perp \|_{\YYY_1}^2 
+ \frac{C}{\eps^2} \| \langle v \rangle^{\gamma+1} f^\perp \|_{\XXX}^2 
+ \frac{C}{\eps^2} \| \langle v \rangle^{\gamma+1} \nabla_v f^\perp \|_{\XXX}^2\\
&\quad
+ \frac{C}{\eps} \| \widetilde \nabla_v f^\perp \|_{\XXX} \| \widetilde \nabla_x f \|_{\XXX}.
\end{aligned}
\end{equation}

Finally, we observe that $\| \langle v \rangle^{\gamma+1} f^\perp \|_{\XXX}^2 + \| \langle v \rangle^{\gamma+1} \nabla_v f^\perp \|_{\XXX}^2 \lesssim \| \langle v \rangle^{\frac{\gamma}{2}+1} f^\perp \|_{\YYY_1}$ and thus, gathering \eqref{dtmfperp} and \eqref{dtnablavfperp} and taking $K >0$ large enough, we obtain
\begin{equation}\label{eq:dtm+nablavfperpXXX}
\begin{aligned}
&\frac12 \frac{\d}{\dt} \left( K \| \langle v \rangle^{\frac{\gamma}{2}+1} f^\perp \|_\XXX^2 + \| \widetilde \nabla_{v} f^\perp \|_\XXX^2 \right)\\
&\qquad 
\le - \frac{\kappa_1}{\eps^2} \| f^\perp \|_{\YYY_2}^2 
+ \frac{C}{\eps^2} \|  f^\perp \|_{\YYY_1}^2 
+\frac{C}{\eps} \| f^\perp \|_{\YYY_1} \| \widetilde \nabla_x f \|_{\XXX}
\end{aligned}
\end{equation}
for some constants $\kappa_1 , C >0$.

\medskip\noindent
\textit{Step 4.} 
We deal in this step with the derivative in time of term $ \langle \widetilde{\nabla}_v f , \widetilde \nabla_x f \rangle_\XXX$.
We compute
$$
\begin{aligned}
\frac{\d}{\dt} \la \widetilde{\nabla}_v f , \widetilde \nabla_x f \ra_\XXX
& =  \sum_{i=0}^3 \frac{\d}{\dt} \la \omega_i  \widetilde \nabla_v g_i , \omega_i  \widetilde \nabla_x g_i   \ra_{L^2_{x,v}} \\
&=: \sum_{i=0}^3  \left( \frac{1}{\eps^2} R_1^i - \frac{1}{\eps} R_2^i 
+  \frac{1}{\eps^2} S_1^i - \frac{1}{\eps} S_2^i \right)
\end{aligned}
$$
with
$$
\begin{aligned}
R_1^i &:= \la \omega_i \widetilde \nabla_{v_k} (L g_i^\perp) , \omega_i \widetilde \nabla_{x_k} g_i  \ra_{L^2_{x,v}} 
\\
S_1^i &:= \la \omega_i \widetilde \nabla_{v_k} g_i , \omega_i \widetilde \nabla_{x_k} (L g_i^\perp)  \ra_{L^2_{x,v}}
\\
R_2^i &:= \la \omega_i \widetilde \nabla_{v_k} ( v \cdot \nabla_x g_i) , \omega_i \widetilde \nabla_{x_k} g_i \ra_{L^2_{x,v}}
\\
S_2^i &:= \la \omega_i \widetilde \nabla_{v_k} g_i , \omega_i \widetilde \nabla_{x_k} (v \cdot \nabla_x g_i) \ra_{L^2_{x,v}}.
\end{aligned}
$$

Recalling that $\widetilde \nabla_{x_k}^* = - \widetilde \nabla_{x_k}$, we then write
$$
\begin{aligned}
R_1^i + S_1^i
&= \la L g_i^\perp, \widetilde \nabla_{v_k}^* (\omega_i^2 \widetilde \nabla_{x_k} g_i) - \widetilde \nabla_{x_k} (\omega_i^2 \widetilde \nabla_{v_k} g_i) \ra_{L^2_{x,v}} \\
&\le \| \omega_i L g_i^\perp \|_{L^2_{x,v}} \left( \| \omega_i^{-1} \widetilde \nabla_{v_k}^* (\omega_i^2 \widetilde \nabla_{x_k} g_i) \|_{L^2_{x,v}} 
+ \| \omega_i^{-1} \widetilde \nabla_{x_k} (\omega_i^2 \widetilde \nabla_{v_k} g_i) \|_{L^2_{x,v}} \right).
\end{aligned}
$$
Observe that 
$$
\omega_i^{-1} \widetilde \nabla_{x_k} (\omega_i^2 \widetilde \nabla_{v_k} g_i)
=    \widetilde \nabla_{v_k} (\omega_i \widetilde \nabla_{x_k} g_i)
+[\omega_i , \widetilde \nabla_{v_k}] \widetilde \nabla_{x_k} g_i
+\omega_i [\widetilde \nabla_{x_k} ,\widetilde \nabla_{v_k}] g_i
$$
so that using Lemma~\ref{lem:commutator}, we get 
\begin{equation}\label{nablaxomega2}
\| \omega_i^{-1} \widetilde \nabla_{x_k} (\omega_i^2 \widetilde \nabla_{v_k} g_i) \|_{L^2_{x,v}} \lesssim \| \omega_i \widetilde \nabla_x g_i \|_{L^2_x(H^1_{v,*})} .
\end{equation}
Writing $\omega_i L g_i^\perp = \omega_i L_1 g_i^\perp + \omega_i L_2 g_i^\perp$ and using the explicit formula \eqref{eq:def:L1} of $L_1$ together with the bound of Lemma~\ref{lem:boundL2} for $L_2$, we obtain
\begin{equation} \label{eq:boundL}
\| \omega_i L g_i^\perp \|_{L^2_{x,v}} 
\lesssim \| \omega_i g_i^\perp \|_{L^2_x (H^2_{v,*})}.
\end{equation}
Together with \eqref{nablaomega2h}, we hence get
$$
|R_1^i + S_1^i|
\le C \| \omega_i g_i^\perp \|_{L^2_x (H^2_{v,*})} \| \omega_i \widetilde \nabla_x g_i \|_{L^2_x(H^1_{v,*})} .
$$

For the remainder terms, we observe that using Lemma~\ref{lem:commutator} and the fact that the transport operator is skew-adjoint, 
$$
\begin{aligned}
R_2^i + S_2^i
&=  \| \omega_i \widetilde \nabla_{x_k}  g_i \|_{L^2_{x,v}}^2
+ \la \omega_i  v \cdot \nabla_x (\widetilde \nabla_{v_k} g_i) , \omega_i \widetilde \nabla_{x_k} g_i \ra_{L^2_{x,v}} \\
&\qquad \qquad \qquad
 + \la \omega_i \widetilde \nabla_{v_k} g_i , \omega_i  v \cdot \nabla_x (\widetilde \nabla_{x_k} g_i) \ra_{L^2_{x,v}} 
=\| \omega_i \widetilde \nabla_{x_k}  g_i \|_{L^2_{x,v}}^2.
\end{aligned}
$$

Gathering previous estimates we obtain
\begin{equation}\label{eq:dtnablavxfXXX}
\begin{aligned}
\frac{\d}{\dt} \la \widetilde{\nabla}_v f , \widetilde \nabla_x f \ra_\XXX
&\le \frac{C}{\eps^2} \| f^\perp \|_{\YYY_2} \|  \widetilde \nabla_x f \|_{\YYY_1} - \frac{1}{\eps} \| \widetilde \nabla_x f \|_{\XXX}^2
\end{aligned}
\end{equation}
for some constant $C >0$.

\medskip\noindent
\textit{Step 5.} 
We deal in this step with the derivative in time of the term $( K\| \langle v \rangle^{\frac{\gamma}{2}} \nabla_x f   \|_\XXX^2 
+ \|  \widetilde \nabla_x f \|_\XXX^2 )$.
We first compute, using that the transport operator is skew-adjoint, 
$$
\begin{aligned}
\frac12 \frac{\d}{\dt} \|\langle v \rangle^{\frac{\gamma}{2}} \nabla_x f \|_\XXX^2
& = \frac12 \sum_{i=0}^3 \frac{\d}{\dt} \| \omega_i  \langle v \rangle^{\frac{\gamma}{2}} \nabla_x g_i \|_{L^2_{x,v}}^2 \\
&=:  \sum_{i=0}^3 \frac{1}{\eps^2}  \la \omega_i \langle v \rangle^{\frac{\gamma}{2}} \partial_{x_k} (L g_i) , \omega_i \langle v \rangle^{\frac{\gamma}{2}} \partial_{x_k} g_i \ra_{L^2_{x,v}}.
\end{aligned}
$$
Since $\partial_{x_k}$ commutes with $L$, we can argue as for the term $J_1^i$ in Step 3.1 above to obtain
$$
\la \omega_i \langle v \rangle^{\frac{\gamma}{2}} \partial_{x_k} (L g_i) , \omega_i \langle v \rangle^{\frac{\gamma}{2}} \partial_{x_k} g_i \ra_{L^2_{x,v}} \le - \kappa \| \omega_i \langle v \rangle^{\frac{\gamma}{2}} \nabla_x g_i \|_{L^2_x(H^1_{v,*})}^2 + C \| \omega_i \widetilde \nabla_x g_i \|_{L^2_{x,v}}^2,
$$
for some constants $\kappa , C >0$, therefore 
\begin{equation}\label{dtnablaxfXXX-0}
\begin{aligned}
\frac12 \frac{\d}{\dt} \|\langle v \rangle^{\frac{\gamma}{2}} \nabla_x f \|_\XXX^2
& \le - \frac{\kappa}{\eps^2} \| \langle v \rangle^{\frac{\gamma}{2}} \nabla_x f \|_{\YYY_1}^2
+ \frac{C}{\eps^2} \| \widetilde \nabla_x f \|_{\XXX}^2.
\end{aligned}
\end{equation}

\medskip

Using that the transport operator is skew-adjoint and commutes with $\widetilde \nabla_{x_k}$, we now compute
$$
\begin{aligned}
\frac12 \frac{\d}{\dt} \| \widetilde \nabla_x f \|_\XXX^2
& = \frac12 \sum_{i=0}^3 \frac{\d}{\dt} \| \omega_i  \widetilde \nabla_x g_i \|_{L^2_{x,v}}^2 \\
&=  \sum_{i=0}^3 \frac{1}{\eps^2}  \la \omega_i \widetilde \nabla_{x_k} (L g_i) , \omega_i \widetilde \nabla_{x_k} g_i \ra_{L^2_{x,v}} 
=:  \sum_{i=0}^3 \frac{1}{\eps^2} N^i.
\end{aligned}
$$
We then write
$$
\begin{aligned}
N^i 
&= \la L (\omega_i \widetilde \nabla_{x_k} g_i) , \omega_i \widetilde \nabla_{x_k} g_i \ra_{L^2_{x,v}} 
+\la   [\omega_i , L_1] \widetilde \nabla_{x_k} g_i , \omega_i \widetilde \nabla_{v_k} g_i \ra_{L^2_{x,v}} \\
&\quad 
+\la  \omega_i [\widetilde \nabla_{x_k}, L_1] g_i, \omega_i \widetilde \nabla_{x_k} g_i \ra_{L^2_{x,v}} 
+\la  [\omega_i , L_2] \widetilde \nabla_{x_k} g_i, \omega_i \widetilde \nabla_{x_k} g_i \ra_{L^2_{x,v}} \\
&\quad 
+\la  \omega_i [\widetilde \nabla_{x_k}, L_2] g_i, \omega_i \widetilde \nabla_{x_k} g_i \ra_{L^2_{x,v}} \\
&=: N_{1}^i + N_{2}^i + N_{3}^i + N_{4}^i + N_{5}^i.
\end{aligned}
$$
Thanks to the spectral gap estimate~\eqref{eq:spectralgap}, one has
$$
\begin{aligned}
N_{1}^i
&\le - 2\kappa \| \omega_i \widetilde \nabla_{x_k} g_i - \pi(\omega_i \widetilde \nabla_{x_k} g_i) \|_{L^2_x (H^1_{v,*})}^2  \\
&\le - 2\kappa \| \omega_i \widetilde \nabla_{x_k} g_i^\perp \|_{L^2_x (H^1_{v,*})}^2 
+ C \|  \omega_i \widetilde \nabla_{x_k} g_i \|_{L^2_{x,v}}^2,
\end{aligned}
$$
for some constants $\kappa,C>0$. From Lemma~\ref{lem:commutatorL}, we get
$$
\begin{aligned}
N_{2}^i
&= 
- \la  [\omega_i,\widetilde \nabla_{v_\ell}] \widetilde \nabla_{x_k} g_i , \widetilde \nabla_{v_\ell} (\omega_i \widetilde \nabla_{x_k} g_i) \ra_{L^2_{x,v}}
- \la    [\omega_i,\widetilde \nabla_{v_\ell}^*] \widetilde \nabla_{x_k} g_i , \widetilde \nabla_{v_\ell}^*(\omega_i \widetilde \nabla_{x_k} g_i) \ra_{L^2_{x,v}}\\
&\quad
- \la  \left[ [\omega_i , \widetilde \nabla_{v_\ell}^*] , \widetilde \nabla_{v_\ell} \right] \widetilde \nabla_{x_k} g_i , \omega_i \widetilde \nabla_{x_k} g_i \ra_{L^2_{x,v}},
\end{aligned}
$$
and
$$
\begin{aligned}
N_{3}^i
&= 
-\la  [\widetilde \nabla_{x_k}, \widetilde \nabla_{v_\ell}] g_i, \widetilde \nabla_{v_\ell} (\omega_i^2 \widetilde \nabla_{x_k} g_i) \ra_{L^2_{x,v}}
-\la   [\widetilde \nabla_{x_k}, \widetilde \nabla_{v_\ell}^*] g_i, \widetilde \nabla_{v_\ell}^*(\omega_i^2 \widetilde \nabla_{x_k} g_i) \ra_{L^2_{x,v}}\\
&\quad
-\la  \omega_i \left[ [\widetilde \nabla_{x_k}, \widetilde \nabla_{v_\ell}^*] , \widetilde \nabla_{v_\ell} \right] g_i, \omega_i \widetilde \nabla_{x_k} g_i \ra_{L^2_{x,v}}
.
\end{aligned}
$$
Arguing in a similar way as in Step 3.2 above (for the term $I_{12}^i$ and $I_{13}^i$), it follows
$$
|N_{2}^i|+|N_{3}^i|
\le \kappa \|  \omega_i \widetilde \nabla_{x_k} g_i \|^2_{L^2_x (H^1_{v,*})}
+ C \| \omega_i \langle v \rangle^{\gamma+1} \nabla_x g_i \|_{L^2_{x,v}}^2.
$$
For the terms $N_{4}^i$ and $N_{5}^i$, Lemma~\ref{lem:boundL2} yields
$$
|N_{4}^i|
+ |N_{5}^i|
\le C \| \omega_i \widetilde \nabla_x g_i \|_{L^2_{x,v}}^2.
$$
Gathering previous estimates, we obtain
\begin{equation}\label{dtnablaxfXXX}
\begin{aligned}
\frac12 \frac{\d}{\dt} \| \widetilde \nabla_{x} f \|_\XXX^2
& \le - \frac{\kappa}{\eps^2} \| \widetilde \nabla_{x}f \|_{\YYY_1}^2 
+ \frac{C}{\eps^2} \| \langle v \rangle^{\gamma+1} \nabla_x f \|_{\XXX}^2.
\end{aligned}
\end{equation}

Finally, we observe that $\| \langle v \rangle^{\gamma+1} \nabla_x f \|_{\XXX}^2 \lesssim \| \langle v \rangle^{\frac{\gamma}{2}} \nabla_x f \|_{\YYY_1}$ and thus, gathering \eqref{dtnablaxfXXX-0} and~\eqref{dtnablaxfXXX} and taking $K >0$ large enough, we obtain
\begin{equation}\label{eq:dtnablaxfXXX}
\begin{aligned}
&\frac12 \frac{\d}{\dt} \left( K \| \langle v \rangle^{\frac{\gamma}{2}} \nabla_x f \|_\XXX^2 + \| \widetilde \nabla_{x} f^\perp \|_\XXX^2 \right)
\le - \frac{\kappa_2}{\eps^2} \| \widetilde \nabla_x f \|_{\YYY_1}^2 
+ \frac{C}{\eps^2} \| \widetilde \nabla_x f \|_{\XXX}^2
\end{aligned}
\end{equation}
for some constants $\kappa_2 , C >0$.

\medskip\noindent
\textit{Step 6. Conclusion.} 
Gathering \eqref{eq:dtfXXX}--\eqref{eq:dtm+nablavfperpXXX}--\eqref{eq:dtnablavxfXXX}--\eqref{eq:dtnablaxfXXX}, we thus obtain 
\begin{align*}
\frac{\d}{\dt}\UUU_\eps(t,f)
&\le - \kappa_0 \| f \|_{\YYY_1}^2 - \frac{\kappa_0}{\eps^2} \| f^\perp \|_{\YYY_1}^2
+\alpha_{1}\left( K\| \langle v \rangle^{\frac{\gamma}{2}+1} f^{\perp}  \|_\XXX^2 
+ \|  \widetilde{\nabla}_v f^{\perp} \|_\XXX^2 \right) \\
&\quad
+\alpha_{1}t \left( - \frac{\kappa_1}{\eps^2} \| f^\perp \|_{\YYY_2}^2 
+ \frac{C}{\eps^2} \|  f^\perp \|_{\YYY_1}^2 
+\frac{C}{\eps} \| f^\perp \|_{\YYY_1} \| \widetilde \nabla_x f \|_{\XXX} \right)\\
&\quad
+2\alpha_{2} \eps t \la \widetilde{\nabla}_v f ,  \widetilde{\nabla}_x f \ra_\XXX 
+ \alpha_{2}\eps t^2  \left( \frac{C}{\eps^2} \| f^\perp \|_{\YYY_2} \| \widetilde \nabla_x f \|_{\YYY_1} - \frac{1}{\eps} \| \widetilde \nabla_x f \|_{\XXX}^2 \right) \\
&\quad
+3\alpha_{3}\eps^2 t^{2} \Big(\|  \widetilde{\nabla}_x f \|_\XXX^2
+ K\|  \langle v \rangle^{\frac{\gamma}{2}} \nabla_x f \|_\XXX^2 \Big) \\
&\quad
+\alpha_{3}\eps^2t^{3}  \left( - \frac{\kappa_2}{\eps^2} \| \widetilde \nabla_x f \|_{\YYY_1}^2 + \frac{C}{\eps^2} \| \widetilde \nabla_x f \|_{\XXX}^2 \right).
\end{align*}
Observe that 
$$
K\| \langle v \rangle^{\frac{\gamma}{2}+1} f^{\perp}  \|_\XXX^2 
+ \|  \widetilde{\nabla}_v f^{\perp} \|_\XXX^2 \lesssim \| f^\perp \|_{\YYY_1}^2
$$
and 
$$
\|  \widetilde{\nabla}_x f \|_\XXX^2
+ K\|  \langle v \rangle^{\frac{\gamma}{2}} \nabla_x f \|_\XXX^2 \lesssim \| \widetilde \nabla_x f \|_{\XXX}^2
$$
and also that, thanks to Young's inequality, there holds
$$
\begin{aligned}
\frac{\alpha_1 t C}{\eps}  \| f^{\perp} \|_{\YYY_1} \| \widetilde{\nabla}_x f \|_{\XXX}
& \le \frac{\alpha_2}{4}{t}^2 \| \widetilde{\nabla}_x f \|_{\XXX}^2 + C \frac{\alpha_1^2}{\alpha_2} \frac{1}{\eps^2} \| f^{\perp} \|_{\YYY_1}^2
\\
\frac{\alpha_2 t^2 C}{\eps}  \| f^{\perp} \|_{\YYY_2} \| \widetilde \nabla_x f \|_{\YYY_1}   
&\le \frac{\alpha_3 \kappa_2 }{2}t^3 \| \widetilde \nabla_x f \|_{\YYY_1}^2 +C\frac{\alpha_2^2}{\alpha_3} \frac{t}{\eps^2}\| f^{\perp} \|_{\YYY_2}^2
\\
2 \alpha_{2} \eps t \la  \widetilde{\nabla}_x f ,  \widetilde{\nabla}_v f\ra_\XXX 
&\le \frac{\alpha_2}{4}{t^2} \|  \widetilde \nabla_x f\|_\XXX^2 + C\alpha_2\eps^2 \| f\|_{\YYY_1}^2.
\end{aligned}
$$
We therefore deduce that, for any $t \in [0,1]$, there holds  
\begin{equation}\label{eq:dtUeps-1}
\begin{aligned} 
\frac{\d}{\dt}\UUU_\eps(t,f)
&\le
-\left( \kappa_0 - C \alpha_2  \right) \| f \|_{\YYY_1}^2 
- \frac{1}{\eps^2} \left( \kappa_0  - C\alpha_1 - C \frac{\alpha_1^2}{\alpha_2}  \right) \| f^\perp \|_{\YYY_1}^2 \\
&\quad 
- \frac{t}{\eps^2} \left( \alpha_1 \kappa_1   
-  C \frac{\alpha_2^2}{\alpha_3} 
\right)  \| f^{\perp} \|_{\YYY_2}^2 
- {t}^2 \left( \frac{\alpha_2}{2} - C \alpha_3 \right)  \|  \widetilde \nabla_x f\|_\XXX^2 \\
&\quad 
- \frac{\alpha_3 \kappa_2 t^3}{2}  \| \widetilde \nabla_x f \|_{\YYY_1}^2.
\end{aligned}
\end{equation}
We now choose $\alpha_{1}=\eta$, $\alpha_{2}=\eta^{\frac{3}{2}}$, and $\alpha_{3}=\eta^{\frac{5}{3}}$, with $\eta \in (0,1)$ small enough such that each quantity appearing inside the parentheses in above inequality is  positive.
We hence obtain that $\frac{\d}{\dt}\UUU_\eps(t,f) \le 0$ for any $t \in [0,1]$, which concludes the proof as explained in Step 1.
\end{proof}

\section{Cauchy theory and regularization estimates for the nonlinear problem} \label{sec:Cauchyreg}

In this section, we provide a Cauchy theory for~\eqref{eq:geps} for small initial data as well as some new regularization estimates for this equation. Notice that our proofs are based on the results developed in Subsection~\ref{subsec:hypo}. 
It is actually crucial to be able to avoid the use of Duhamel formula to obtain nice estimates on the nonlinear problem because of the singularity in $\eps$ that is in front of the nonlinear term in~\eqref{eq:geps}. Our strategy is to perform direct energy estimates with the norm $\Nt \cdot \Nt_\XXX$ introduced in Subsection~\ref{subsec:hypo} (see~\eqref{def:Nt} and~\eqref{eq:def:norme-triple-XXX} for the precise definition) and exploit the facts that $\Gamma(f,g) = (\Gamma(f,g))^\perp$ and
$\langle \Gamma(f,g) , h \rangle_{L^2_{x,v}} = \langle \Gamma(f,g) , h^\perp \rangle_{L^2_{x,v}}$ so that 
	$$
	\left\langle \! \left\langle \Gamma(f,g),h \right\rangle \! \right\rangle_{L^2_{x,v}} = \left\langle \Gamma(f,g) , h^\perp \right\rangle_{L^2_{x,v}}
	$$
and thus
	\begin{equation*} 
	\begin{aligned}
		\la \! \la \Gamma(f,g), h \ra \! \ra_\XXX 
		& = \sum_{i=0}^2  
		\delta \left\langle \langle v \rangle^{(3-i)(\frac{\gamma}{2}+1)} \, \nabla_x^i \Gamma(f,g), \langle v \rangle^{(3-i)(\frac{\gamma}{2}+1)} \, \nabla_x^i h^\perp \right\rangle_{L^2_{x,v}} \\
&\quad  
		+ \sum_{i=0}^3 
		\left\langle  \nabla_x^i \Gamma(f,g),  \nabla_x^i h^\perp \right\rangle_{L^2_{x,v}},
	\end{aligned}
	\end{equation*}
where we recall that $\delta \in (0,1)$ is a small enough constant chosen in Proposition~\ref{prop:hypoX}. Notice also that we used the particular form of $\Nt\cdot\Nt_{L^2_{x,v}}$ defined in~\eqref{def:Nt} and the fact that $\pi \Gamma(f,g)=0$. Rearranging terms, we then deduce 
	\begin{equation} \label{eq:Gammaperp}
	\begin{aligned}
		\la \! \la \Gamma(f,g), h \ra \! \ra_\XXX 
		& = \delta \la \Gamma(f,g), h^\perp \ra_\XXX
		+ \sum_{i=0}^2 \la \nabla_x^i \Gamma(f,g) , \nabla_x^i h^\perp \ra_{L^2_{x,v}} \\
		&\quad
		+ (1-\delta) \la \nabla_x^3 \Gamma(f,g) , \nabla_x^3 h^\perp \ra_{L^2_{x,v}}.
	\end{aligned}
	\end{equation}

\subsection{Bilinear estimates for the Landau operator} \label{subsec:nonlinearestim}
In this part, we start by establishing some new and sharp nonlinear estimates on the Landau collision operator, we recall that the matrix ${\bf B}(v)$ is defined in~\eqref{def:B} and that the spaces $\XXX$, $\YYY_1$, $\YYY_2$ and $\YYY_1'$ are respectively defined in~\eqref{def:normXXX},~\eqref{def:normYYY1},~\eqref{def:normYYY2} and~\eqref{def:YYYi'}.

We start by establishing some convolution estimates for the coefficients $a_{ij}$ and $b_i$:
\begin{lem}\label{lem:afGH}
For any suitable function $f=f(v)$, vector fields $G=G(v),H=H(v)$ and $\ell \in \{1,2,3\}$ there holds, for any $v\in \R^3$:
\begin{equation}\label{lem:afGH-1}
|( a_{ij} * f ) G_i H_j| (v)  
\lesssim \| \langle v \rangle^{7} f \|_{L^2_v} 
| {\bf B}(v) G(v) | | {\bf B}(v) H(v) |
\end{equation}
\begin{equation}\label{lem:afGH-2}
|( \partial_{v_\ell} a_{ij} * f ) G_i H_j |(v)   
\lesssim \| \langle v \rangle^{8} f \|_{L^2_v} 
| {\bf B}(v) G(v) | | {\bf B}(v) H(v) | 
\end{equation}
\begin{equation}\label{lem:afGH-3}
|(\partial_{v_\ell} b_i * f) G_i| (v)  
\lesssim \| \langle v \rangle^{3} f \|_{L^2_v} \langle v \rangle^{\frac{\gamma}{2}} |{\bf B}(v) G(v)| \quad \text{if} \quad 0 \le \gamma \le 1 
\end{equation}
\begin{equation}\label{lem:afGH-4}
\begin{aligned}
&|(\partial_{v_\ell} b_i * f) G_i| (v)   
\lesssim \left( \|\langle v \rangle^{4} f \|_{H^1_v} \langle v \rangle^{-1} 
+ \|\langle v \rangle^{4} f \|_{L^2_v} \right) \langle v \rangle^{\frac{\gamma}{2}} |{\bf B}(v) G(v)|
\quad \text{if} \quad -2 \le \gamma < 0.
\end{aligned}
\end{equation}

\end{lem}

\begin{proof}
We split the proof into three steps.

\medskip\noindent
\textit{Step 1. Proof of \eqref{lem:afGH-1}.} We only prove the estimate for $|v| \ge 1$, the case $|v|<1$ being trivial. Recalling that $P_v$ denotes the projection onto $v$, we decompose 
$$
G_i = (P_v)_i(G)+ (\mathrm{Id}-P_v)_i (G)= v_i \left( G \cdot \frac{v}{|v|^2} \right) + (\mathrm{Id}-P_v)_i (G)
$$
and
$$
H_j =  (P_v)_i(H)+ (\mathrm{Id}-P_v)_i (H) = v_j \left( H \cdot \frac{v}{|v|^2} \right) + (\mathrm{Id}-P_v)_j (H).
$$
We thus obtain
$$
\begin{aligned}
( a_{ij} * f )(v) G_i(v) H_j(v) 
&= ( a_{ij} * f )(v) v_i v_j  \left( G (v)\cdot \frac{v}{|v|^2} \right) \left( H(v) \cdot \frac{v}{|v|^2} \right)  \\
&\quad + ( a_{ij} * f )(v) v_i \left( G(v) \cdot \frac{v}{|v|^2} \right) (\mathrm{Id}-P_v)_j (H(v))  \\
&\quad + ( a_{ij} * f )(v) v_j  (\mathrm{Id}-P_v)_i (G(v)) \left( H(v) \cdot \frac{v}{|v|^2} \right)   \\
&\quad + ( a_{ij} * f )(v) (\mathrm{Id}-P_v)_i (G(v))  (\mathrm{Id}-P_v)_j (H(v)) .
\end{aligned}
$$
Using Lemma \ref{lem:conv}, we estimate each term of the previous splitting. First, 
$$
\begin{aligned}
&\left| ( a_{ij} * f )(v) v_i v_j  \left( G(v) \cdot \frac{v}{|v|^2} \right) \left( H(v) \cdot \frac{v}{|v|^2} \right) \right| \\
&\qquad \lesssim \| \langle v \rangle^{7} f \|_{L^2_v}  \langle v \rangle^{\frac{\gamma}{2}}  |G(v) |
\langle v \rangle^{\frac{\gamma}{2}}  |H(v) |.
\end{aligned}
$$
Then,
$$
\begin{aligned}
&\left| ( a_{ij} * f )(v) v_i \left( G(v) \cdot \frac{v}{|v|^2} \right) (\mathrm{Id}-P_v)_j (H(v)) \right| \\
&\qquad \lesssim \| \langle v \rangle^{7} f \|_{L^2_v} \langle v \rangle^{\frac{\gamma}{2}}  |G (v)|
\langle v \rangle^{\frac{\gamma}{2}+1} |(\mathrm{Id}-P_v)H(v)|
\end{aligned}
$$
and
$$
\begin{aligned}
&\left| ( a_{ij} * f )(v) v_j  (\mathrm{Id}-P_v)_i (G(v)) \left( H(v) \cdot \frac{v}{|v|^2} \right) \right| \\
&\qquad \lesssim \| \langle v \rangle^{7} f \|_{L^2_v}   \langle v \rangle^{\frac{\gamma}{2}+1}  
|(\mathrm{Id}-P_v)G(v) |
\langle v \rangle^{\frac{\gamma}{2}} | H(v) |.
\end{aligned}
$$
Finally,
$$
\begin{aligned}
&\left| ( a_{ij} * f )(v) (\mathrm{Id}-P_v)_i (G(v))  (\mathrm{Id}-P_v)_j (H(v)) \right| \\
&\qquad \lesssim \| \langle v \rangle^{7} f \|_{L^2_v}  \langle v \rangle^{\frac{\gamma}{2}+1}  
|(\mathrm{Id}-P_v)G (v)| 
\langle v \rangle^{\frac{\gamma}{2}+1} 
|(\mathrm{Id}-P_v)H (v) |.
\end{aligned}
$$
We conclude the proof of \eqref{lem:afGH-1} by gathering previous estimates and recalling that 
$$
|\mathbf B(v) G(v)| \lesssim \langle v \rangle^{\frac{\gamma}{2}} |P_v G(v)| + \langle v \rangle^{\frac{\gamma}{2}+1} |(\mathrm{Id}-P_v)G(v)| \lesssim |\mathbf B(v) G(v)|.
$$

\medskip\noindent
\textit{Step 2. Proof of \eqref{lem:afGH-2}.} The proof of \eqref{lem:afGH-2} is similar to the one of \eqref{lem:afGH-1} by using the bounds on $\left( \partial_{v_\ell} a_{ij} * f \right)$, $\left( \partial_{v_\ell} a_{ij} * f \right) v_i$, $\left( \partial_{v_\ell} a_{ij} * f \right)v_j$ and $\left( \partial_{v_\ell} a_{ij} * f \right) v_i v_j$ given by Lemma~\ref{lem:conv}. We thus skip it.

\medskip\noindent
\textit{Step 3. Proof of \eqref{lem:afGH-3} and \eqref{lem:afGH-4}.} Again, we only prove the estimate for $|v| \ge 1$. Recall that $|\partial_{v_\ell} b_i| \lesssim | v |^\gamma$. 
If $0 \le \gamma \le 1$, then using Lemma~\ref{lem:conv}, we have 
$$
|(\partial_{v_\ell} b_i * f)(v) G_i(v) | \lesssim \| \langle v \rangle^{3} f \|_{L^2_v} \langle v \rangle^{\gamma} |G(v)|. 
$$
If $-2 \le \gamma < 0$, we use the above decomposition of $G$ to write 
$$
(\partial_{v_\ell} b_i * f)(v) G_i(v) = 
(\partial_{v_\ell} b_i * f)(v) v_i \left( G(v) \cdot \frac{v}{|v|^2} \right) + (\partial_{v_\ell} b_i * f)(v) (\mathrm{Id}-P_v)_i(G(v)).
$$
Remarking that $ \left(\partial_{v_\ell} b_{i}* f \right) = \left(b_{i}* \partial_{v_\ell}f \right)$, we also observe that
\begin{align*}
(b_{i}* \partial_{v_\ell}f )(v) v_i
&= ( \partial_{v_j} a_{ij}* \partial_{v_\ell}f )(v) v_i \\
&= ( a_{ij}* v_i\partial_{v_j}\partial_{v_\ell}f )(v) \\
&=( b_{i}*  (v_i\partial_{v_\ell}f ) )(v)
- ( a_{ii}*  \partial_{v_\ell}f ) (v) ,
\end{align*}
from which we obtain
\begin{align*}
(b_{i}* \partial_{v_\ell}f )(v) v_i
&= ( b_{i}*  \partial_{v_\ell}[v_i f] ) (v)
-( b_{i}*  (\partial_{v_\ell}v_i)f ) )(v)
- ( a_{ii}*  \partial_{v_\ell}f ) (v) \\
&= ( \partial_{v_\ell} b_{i}* [v_i f] )(v) 
- ( b_{\ell}*  f ) (v)
- ( \partial_{v_\ell} a_{ii}*  f ) (v).
\end{align*}
From Lemma~\ref{lem:conv} and using classical Sobolev embeddings, we have 
$$
|( \partial_{v_\ell} b_{i}* f )(v)
| \lesssim \langle v \rangle^{\gamma} \|\langle v \rangle^{3} f \|_{L^4_v} 
\lesssim \langle v \rangle^{\gamma} \|\langle v \rangle^{3} f \|_{H^1_v} .
$$
Therefore, using once more Lemma~\ref{lem:conv}, we obtain
$$
|( \partial_{v_\ell} b_{i}* f )(v) v_i
| \lesssim \langle v \rangle^{\gamma} \|\langle v \rangle^{4} f \|_{H^1_v}  +
\langle v \rangle^{\gamma+1} \|\langle v \rangle^{4} f \|_{L^2_v} .
$$
Hence, 
\begin{align*}
&|(\partial_{v_\ell} b_i * f)(v) G_i(v) | \\
&\qquad \lesssim  \|\langle v \rangle^{4} f \|_{H^1_v} \bigg( \langle v \rangle^{\gamma} |(\mathrm{Id}-P_v) G(v)|
+ \langle v \rangle^{\gamma-1}  |G(v)| \bigg)
+ \|\langle v \rangle^{4} f \|_{L^2_v} \langle v \rangle^{\gamma}  |G(v)|,
\end{align*}
which concludes the proof. 
\end{proof}

We shall now establish bilinear estimates for the nonlinear operator $\Gamma$ in Propositions~\ref{prop:Gamma-NL},~\ref{prop:DvBGamma-NL},~\ref{prop:DxBGamma-NL} and~\ref{prop:Gamma-NL-XXX} below.
Recall from \eqref{eq:Gamma} that 
$$
\Gamma(g_1,g_2)
= \Gamma_1(g_1,g_2) + \Gamma_2(g_1,g_2) + \Gamma_3(g_1,g_2) + \Gamma_4(g_1,g_2) + \Gamma_5(g_1,g_2) 
$$
with
\begin{equation}\label{eq:def-Gamma1}
\Gamma_1(g_1,g_2)
= \partial_{v_i} \left\{ \left(a_{ij}* [\sqrt M g_1] \right) \partial_{v_j} g_2   \right\},
\end{equation}
\begin{equation}\label{eq:def-Gamma2}
\Gamma_2(g_1,g_2)
= -\partial_{v_i} \left\{ \left(b_{i}* [\sqrt M g_1] \right) g_2   \right\} ,
\end{equation}
\begin{equation}\label{eq:def-Gamma3}
\Gamma_3(g_1,g_2)
= -  \left( a_{ij}* [\sqrt M g_1]\right) v_i \partial_{v_j} g_2 ,
\end{equation}
\begin{equation}\label{eq:def-Gamma4}
\Gamma_4(g_1,g_2)
= \frac14 \left( a_{ij}* [\sqrt M g_1]\right) v_i v_j g_2,
\end{equation}
\begin{equation}\label{eq:def-Gamma5}
\Gamma_5(g_1,g_2)
= -  \frac12 \left( a_{ii}* [\sqrt M g_1]\right) g_2 .
\end{equation}

\begin{prop}\label{prop:Gamma-NL}
Let $g_1$, $g_2$ and $g_3$ be smooth enough functions.
For any $\alpha \in \R$, there holds 
\begin{equation}\label{eq:Gamma-NL}
\la \langle v \rangle^{\alpha}  \Gamma(g_1,g_2) , g_3 \ra_{\XXX} \lesssim  \| g_1 \|_{\XXX} \| \langle v \rangle^{\alpha}  g_2 \|_{\YYY_1} \| g_3 \|_{\YYY_1}.
\end{equation}
As a consequence, one has by duality
\begin{equation}\label{eq:Gamma-NL-YYY0'}
\| \Gamma(g_1,g_2) \|_{\YYY_1'} \lesssim \| g_1 \|_{\XXX} \| g_2 \|_{\YYY_1} .
\end{equation}
\end{prop}

\begin{proof}[Proof of Proposition~\ref{prop:Gamma-NL}]
We shall first prove that for any $\alpha \in \R$, there holds 
\begin{equation}\label{eq:Gamma-NL1}
\la \langle v \rangle^{\alpha}  \Gamma(g_1,g_2) , g_3 \ra_{L^2_v} \lesssim \| M^{\frac14} g_1 \|_{L^2_v}
\| \langle v \rangle^{\alpha}  g_2 \|_{H^1_{v,*}} 
\| g_3 \|_{H^1_{v,*}} .
\end{equation}
Once this estimate is established, we shall prove \eqref{eq:Gamma-NL} in the final step of the proof by integrating it in $x$ and using Sobolev embeddings. Estimate \eqref{eq:Gamma-NL-YYY0'} is then a direct consequence of \eqref{eq:Gamma-NL}.
We thus write $\Gamma(g_1,g_2) = \Gamma_1(g_1,g_2) + \cdots +\Gamma_5(g_1,g_2)$ as in \eqref{eq:def-Gamma1}--\eqref{eq:def-Gamma5}, and we estimate each term separately in the sequel.

\medskip\noindent
\textit{Step 1.}
We write, from \eqref{eq:def-Gamma1} and making an integration by parts,
$$
\begin{aligned}
&\la \langle v \rangle^{\alpha}  \Gamma_1(g_1,g_2) , g_3 \ra_{L^2_v} 
= - \la \left(a_{ij}* [\sqrt M g_1] \right) \partial_{v_j} g_2 , \partial_{v_i} ( \langle v \rangle^{\alpha}  g_3)\ra_{L^2_v} \\
&\qquad = - \la \left(a_{ij}* [\sqrt M g_1] \right) \partial_{v_j} g_2 , \langle v \rangle^{\alpha}  \partial_{v_i} g_3)\ra_{L^2_v}  
- \la \left(a_{ij}* [\sqrt M g_1] \right) \partial_{v_j} g_2 , (\partial_{v_i}  \langle v \rangle^{\alpha}  ) g_3\ra_{L^2_v} \\
&\qquad =: \mathrm{I}_1 + \mathrm{I}_2.
\end{aligned}
$$
For the term $\mathrm{I}_1$, we use Lemma~\ref{lem:afGH}, which yields
$$
\begin{aligned}
\mathrm{I}_1 
&\lesssim \| M^{\frac14} g_1 \|_{L^2_v}  
\la \langle v \rangle^{\alpha}  |\widetilde \nabla_v g_2| , |\widetilde \nabla_v g_3| \ra_{L^2_v}\\
&\lesssim \| M^{\frac14} g_1 \|_{L^2_v}  
\| \langle v \rangle^{\alpha}  \widetilde \nabla_v g_2 \|_{L^2_v} \| \widetilde \nabla_v g_3 \|_{L^2_v} .
\end{aligned}
$$
In a similar way, thanks to Lemma~\ref{lem:afGH} and using that $|\widetilde \nabla_v \langle v \rangle^{\alpha} | \lesssim \langle v \rangle^{\frac{\gamma}{2}-1 + \alpha} $, we obtain
$$
\begin{aligned}
\mathrm{I}_2
&\lesssim \| M^{\frac14} g_1 \|_{L^2_v}  
\la \langle v \rangle^{\frac{\gamma}{2}-1} \langle v \rangle^{\alpha}  |\widetilde \nabla_v g_2|, |g_3| \ra_{L^2_v}\\
&\lesssim \| M^{\frac14} g_1 \|_{L^2_v}  
\| \langle v \rangle^{\alpha}  \widetilde \nabla_v g_2 \|_{L^2_v} \| \langle v \rangle^{\frac{\gamma}{2} - 1} g_3 \|_{L^2_v} .
\end{aligned}
$$
We therefore obtain
\begin{equation}\label{eq:Gamma1}
\la \langle v \rangle^{\alpha}   \Gamma_1(g_1,g_2) , g_3 \ra_{L^2_v}
\lesssim \| M^{\frac14} g_1 \|_{L^2_v}
\| \langle v \rangle^{\alpha}  \widetilde \nabla_v g_2 \|_{L^2_v}  
\| g_3 \|_{H^1_{v,*}}.
\end{equation}

\medskip\noindent
\textit{Step 2.}
Starting from \eqref{eq:def-Gamma2} and making an integration by parts, we get
$$
\begin{aligned}
&\la \langle v \rangle^{\alpha}  \Gamma_2(g_1,g_2) , g_3 \ra_{L^2_v}
= \la \left(b_{i}* [\sqrt M g_1] \right)  g_2  , \partial_{v_i} (\langle v \rangle^{\alpha}  g_3) \ra_{L^2_v} \\
&\qquad 
= \la \left(b_{i}* [\sqrt M g_1] \right)  g_2  , \langle v \rangle^{\alpha}  \partial_{v_i} g_3 \ra_{L^2_v} 
+ \la \left(b_{i}* [\sqrt M g_1] \right)  g_2  , (\partial_{v_i} \langle v \rangle^{\alpha}  )g_3 \ra_{L^2_v} \\
&\qquad =: \mathrm{II}_1 + \mathrm{II}_2.
\end{aligned}
$$
For the term $\mathrm{II}_1$, we use Lemma~\ref{lem:conv} to obtain
$$
\begin{aligned}
\mathrm{II}_1
& \lesssim \| M^{\frac14} g_1 \|_{L^2_v} \la \langle v \rangle^{\gamma+1} \langle v \rangle^{\alpha}  |g_2| , |\nabla_v g_3| \ra_{L^2_v} \\
&\lesssim \| M^{\frac14} g_1 \|_{L^2_v}  \| \langle v \rangle^{\alpha}  \langle v \rangle^{\frac{\gamma}{2} +1} g_2 \|_{L^2_v} \| \langle v \rangle^{\frac{\gamma}{2}} \nabla_v g_3 \|_{L^2_v}.
\end{aligned}
$$
In a similar fashion, Lemma~\ref{lem:conv} yields
$$
\begin{aligned}
\mathrm{II}_2
& \lesssim \| M^{\frac14} g_1 \|_{L^2_v} \la \langle v \rangle^{\gamma+1} \la v \ra^{\alpha-1} |g_2|, |g_3| \ra_{L^2_v} \\
&\lesssim \| M^{\frac14} g_1 \|_{L^2_v}  \| \langle v \rangle^{\alpha}  \langle v \rangle^{\frac{\gamma}{2} +1} g_2 \|_{L^2_v} \| \langle v \rangle^{\frac{\gamma}{2}-1} g_3 \|_{L^2_v}.
\end{aligned}
$$
We thus get
\begin{equation}\label{eq:Gamma2}
\la \langle v \rangle^{\alpha}  \Gamma_2(g_1,g_2) , g_3 \ra_{L^2_v}
\lesssim 
\| M^{\frac14} g_1 \|_{L^2_v}  \| \langle v \rangle^{\alpha}  \langle v \rangle^{\frac{\gamma}{2} + 1} g_2 \|_{L^2_v} \| g_3 \|_{H^1_{v,*}} .
\end{equation}

\medskip\noindent
\textit{Step 3.} All the remainder terms associated to  $\Gamma_3$, $\Gamma_4$ and $\Gamma_5$ can be estimated directly thanks to Lemma~\ref{lem:conv} or Lemma~\ref{lem:afGH} and Cauchy-Schwarz inequality. Indeed we first have, using Lemma~\ref{lem:afGH} and $|\mathbf B(v) v| \lesssim \langle v \rangle^{\frac{\gamma}{2}+1}$,
\begin{equation}\label{eq:Gamma3}
\begin{aligned}
\la \langle v \rangle^{\alpha}  \Gamma_3(g_1,g_2) , g_3 \ra_{L^2_v}
&= - \la \left(a_{ij}* [\sqrt M g_1] \right) v_i \partial_{v_j} g_2  , \langle v \rangle^{\alpha}  g_3 \ra_{L^2_v} \\
&\lesssim 
\| M^{\frac14} g_1 \|_{L^2_v}  \la \langle v \rangle^{\alpha}  \langle v \rangle^{\frac{\gamma}{2}+1} |\widetilde \nabla_v g_2| , |g_3| \ra_{L^2_v} \\
&\lesssim 
\| M^{\frac14} g_1 \|_{L^2_v}  \| \langle v \rangle^{\alpha}  \widetilde \nabla_v g_2 \|_{L^2_v} \| \langle v \rangle^{\frac{\gamma}{2} + 1} g_3 \|_{L^2_v}.
\end{aligned}
\end{equation}
In a similar way, we also get
\begin{equation}\label{eq:Gamma4}
\begin{aligned}
\la \langle v \rangle^{\alpha}  \Gamma_4(g_1,g_2) , g_3 \ra_{L^2_v}
&= \frac14 \la \left(a_{ij}* [\sqrt M g_1] \right) v_i v_j g_2  , \langle v \rangle^{\alpha}  g_3 \ra_{L^2_v} \\
&\lesssim 
\| M^{\frac14} g_1 \|_{L^2_v} \la \langle v \rangle^{\gamma+2} \langle v \rangle^{\alpha}  |g_2|, |g_3| \ra_{L^2_v}\\
&\lesssim 
\| M^{\frac14} g_1 \|_{L^2_v}  \| \langle v \rangle^{\alpha}  \langle v \rangle^{\frac{\gamma}{2}+1} g_2 \|_{L^2_v}  \| \langle v \rangle^{\frac{\gamma}{2}+1} g_3 \|_{L^2_v}.
\end{aligned}
\end{equation}
Finally, now using Lemma~\ref{lem:conv}, we obtain
\begin{equation}\label{eq:Gamma5}
\begin{aligned}
\la \langle v \rangle^{\alpha}  \Gamma_5(g_1,g_2) , g_3 \ra_{L^2_v}
&= -\frac12 \la \langle v \rangle^{\alpha}  \left(a_{ii}* [\sqrt M g_1] \right) g_2  ,  g_3 \ra_{L^2_v} \\
&\lesssim 
\| M^{\frac14} g_1 \|_{L^2_v} \la \langle v \rangle^{\alpha}  \langle v \rangle^{\gamma+2}  |g_2| ,  |g_3| \ra_{L^2_v}\\
&\lesssim 
\| M^{\frac14} g_1 \|_{L^2_v}  \| \langle v \rangle^{\alpha}  \langle v \rangle^{\frac{\gamma}{2}+1} g_2 \|_{L^2_v}  \| \langle v \rangle^{\frac{\gamma}{2}+1} g_3 \|_{L^2_v}  .
\end{aligned}
\end{equation}

We thus conclude the proof of \eqref{eq:Gamma-NL1} by gathering estimates \eqref{eq:Gamma1}--\eqref{eq:Gamma2}--\eqref{eq:Gamma3}--\eqref{eq:Gamma4}--\eqref{eq:Gamma5} and observing that 
$$
\| \langle v \rangle^{\alpha}  g_2 \|_{H^1_{v,*}}^2
\lesssim
\| \langle v \rangle^{\alpha}  \widetilde \nabla_v g_2 \|_{L^2_v}^2
 + \| \langle v \rangle^{\alpha}  \langle v \rangle^{\frac{\gamma}{2} + 1}  g_2 \|_{L^2_v}^2
\lesssim
\| \langle v \rangle^{\alpha}  g_2 \|_{H^1_{v,*}}^2. 
$$

\medskip\noindent
\textit{Step 4. Proof of \eqref{eq:Gamma-NL}.} 
Recalling the definition of $\la \cdot , \cdot \ra_\XXX$ in \eqref{def:normXXX}, we have 
$$
\begin{aligned}
\la \langle v \rangle^{\alpha}  \Gamma(g_1,g_2) , g_3 \ra_{\XXX}
&= \sum_{i=0}^3 \la \langle v \rangle^{\alpha}  \langle v \rangle^{(3-i)(\frac{\gamma}{2}+1)} \nabla_x^i \Gamma(g_1,g_2) , \langle v \rangle^{(3-i)(\frac{\gamma}{2}+1)} \nabla_x^i g_3 \ra_{L^2_{x,v}}. \\
&=: T_0 + T_1 + T_2 + T_3.
\end{aligned}
$$
Thanks to \eqref{eq:Gamma-NL1} and the fact that $\|\cdot\|_{L^\infty_x} \lesssim \|\cdot\|_{H^2_x}$, we get
$$
\begin{aligned}
T_0
&\lesssim \int_{\T^3} \| M^{\frac14} g_1 \|_{L^2_v} \| \langle v \rangle^{\alpha}\langle v \rangle^{3(\frac{\gamma}{2}+1)} g_2 \|_{H^1_{v,*} } \| \langle v \rangle^{3(\frac{\gamma}{2}+1)} g_3 \|_{H^1_{v,*}} \, \dx \\
&\lesssim \| M^{\frac14} g_1 \|_{H^2_x L^2_v} \| \langle v \rangle^{\alpha}  \langle v \rangle^{3(\frac{\gamma}{2}+1)} g_2 \|_{L^2_x (H^1_{v,*} )} \| \langle v \rangle^{3(\frac{\gamma}{2}+1)} g_3 \|_{L^2_x(H^1_{v,*})} \\
&\lesssim \| g_1 \|_{\XXX} \| \langle v \rangle^{\alpha}  g_2 \|_{\YYY_1} \| g_3 \|_{\YYY_1}.
\end{aligned}
$$
Using that $\partial_{x_k} \Gamma(g_1,g_2) = \Gamma(\partial_{x_k} g_1 , g_2) + \Gamma(g_1 , \partial_{x_k} g_2)$, H\"older inequality and the fact that from classical Sobolev embeddings, $\|\cdot\|_{L^6_x} + \|\cdot\|_{L^3_x} \lesssim \|\cdot\|_{H^1_x}$, estimate \eqref{eq:Gamma-NL1} yields
$$
\begin{aligned}
T_1
& \lesssim \int_{\T^3} \| M^{\frac14} \nabla_x g_1 \|_{L^2_v} \| \langle v \rangle^{2(\frac{\gamma}{2}+1)} \langle v \rangle^{\alpha}  g_2 \|_{H^1_{v,*}} \| \langle v \rangle^{2(\frac{\gamma}{2}+1)}\nabla_x g_3 \|_{H^1_{v,*}} \, \dx \\
&\quad
+ \int_{\T^3} \| M^{\frac14}  g_1 \|_{L^2_v} \| \langle v \rangle^{2(\frac{\gamma}{2}+1)} \langle v \rangle^{\alpha}  \nabla_x g_2 \|_{H^1_{v,*} } \| \langle v \rangle^{2(\frac{\gamma}{2}+1)} \nabla_x g_3 \|_{H^1_{v,*}} \, \dx \\ 
&\lesssim \| M^{\frac14} \nabla_x g_1 \|_{H^1_x L^2_v} \| \langle v \rangle^{2(\frac{\gamma}{2}+1)} \langle v \rangle^{\alpha}  g_2 \|_{H^1_x (H^1_{v,*})} \| \langle v \rangle^{2(\frac{\gamma}{2}+1)} \nabla_x g_3 \|_{L^2_x(H^1_{v,*})} \\
&\quad 
+ \| M^{\frac14}  g_1 \|_{H^2_x L^2_v} \| \langle v \rangle^{2(\frac{\gamma}{2}+1)} \langle v \rangle^{\alpha}  \nabla_x g_2 \|_{L^2_x (H^1_{v,*})} \| \langle v \rangle^{2(\frac{\gamma}{2}+1)} \nabla_x g_3 \|_{L^2_x(H^1_{v,*})} \\
&\lesssim \| g_1 \|_{\XXX} \| \langle v \rangle^{\alpha}g_2 \|_{\YYY_1} \| g_3 \|_{\YYY_1}.
\end{aligned}
$$
Moreover, for the term $T_2$, similarly, we have 
$$
\begin{aligned}
T_2
&\lesssim 
\int_{\T^3} \bigg( \| M^{\frac14} \nabla^2_x g_1 \|_{L^2_v} \| \langle v \rangle^{\frac{\gamma}{2}+1} \langle v \rangle^{\alpha}  g_2 \|_{H^1_{v,*} }
+ \| M^{\frac14} \nabla_x g_1 \|_{L^2_v} \| \langle v \rangle^{\frac{\gamma}{2}+1} \langle v \rangle^{\alpha}  \nabla_x g_2 \|_{H^1_{v,*} } \\
&\qquad\qquad \qquad\qquad
+ \| M^{\frac14}  g_1 \|_{L^2_v} \| \langle v \rangle^{\frac{\gamma}{2}+1} \langle v \rangle^{\alpha}  \nabla_x^2 g_2 \|_{H^1_{v,*} }  \bigg)\| \langle v \rangle^{\frac{\gamma}{2}+1} \nabla^2_x g_3 \|_{H^1_{v,*}} \, \dx 
\\ 
& \lesssim \bigg(\| M^{\frac14} \nabla^2_x g_1 \|_{L^2_x L^2_v} \| \langle v \rangle^{\frac{\gamma}{2}+1} \langle v \rangle^{\alpha}  g_2 \|_{H^2_x (H^1_{v,*})} 
\\
&\qquad\qquad \qquad\qquad+ \| M^{\frac14}  \nabla_x g_1 \|_{H^1_x L^2_v} \|  \langle v \rangle^{\frac{\gamma}{2}+1} \langle v \rangle^{\alpha}  \nabla_x g_2 \|_{H^1_x (H^1_{v,*})}  \\
&\qquad\qquad \qquad\qquad
+ \| M^{\frac14} g_1 \|_{H^2_x L^2_v} \| \langle v \rangle^{\frac{\gamma}{2}+1} \langle v \rangle^{\alpha}  \nabla^2_x g_2 \|_{L^2_x (H^1_{v,*})} 
\bigg) \| \langle v \rangle^{\frac{\gamma}{2}+1} \nabla^2_x g_3 \|_{L^2_x(H^1_{v,*})} \\
&\lesssim \| g_1 \|_{\XXX} \| \langle v \rangle^{\alpha}  g_2 \|_{\YYY_1} \| g_3 \|_{\YYY_1}.
\end{aligned}
$$
Finally, for the term $T_3$, we have 
$$
\begin{aligned}
&T_3
\lesssim 
\int_{\T^3} \bigg( \| M^{\frac14} \nabla^3_x g_1 \|_{L^2_v} \| \langle v \rangle^{\alpha}  g_2 \|_{H^1_{v,*} }
+ \| M^{\frac14} \nabla^2_x g_1 \|_{L^2_v} \| \langle v \rangle^{\alpha}  \nabla_x g_2 \|_{H^1_{v,*}} \\
&\qquad\qquad
+ \| M^{\frac14} \nabla_x g_1 \|_{L^2_v} \|  \langle v \rangle^{\alpha}  \nabla_x^2 g_2 \|_{H^1_{v,*} } 
+ \| M^{\frac14}  g_1 \|_{L^2_v} \|  \langle v \rangle^{\alpha}  \nabla_x^3 g_2 \|_{H^1_{v,*} }  \bigg)\|  \nabla^3_x g_3 \|_{H^1_{v,*}} \, \dx 
\\ 
& \lesssim \bigg(\| M^{\frac14} \nabla^3_x g_1 \|_{L^2_x L^2_v} \|  \langle v \rangle^{\alpha}  g_2 \|_{H^2_x (H^1_{v,*})} 
+ \| M^{\frac14}  \nabla^2_x g_1 \|_{H^1_x L^2_v} \|   \langle v \rangle^{\alpha}  \nabla_x g_2 \|_{H^1_x (H^1_{v,*})}  \\
&
+ \| M^{\frac14} \nabla_x g_1 \|_{H^2_x L^2_v} \|  \langle v \rangle^{\alpha}  \nabla^2_x g_2 \|_{L^2_x (H^1_{v,*})} 
+ \| M^{\frac14} g_1 \|_{H^2_x L^2_v} \| \langle v \rangle^{\alpha}  \nabla^3_x g_2 \|_{L^2_x (H^1_{v,*})} 
\bigg) \|  \nabla^2_x g_3 \|_{L^2_x(H^1_{v,*})} \\
&\lesssim \| g_1 \|_{\XXX} \| \langle v \rangle^{\alpha}  g_2 \|_{\YYY_1} \| g_3 \|_{\YYY_1},
\end{aligned}
$$
which concludes the proof of~\eqref{eq:Gamma-NL}.

\medskip\noindent
\textit{Step 5. Proof of \eqref{eq:Gamma-NL-YYY0'}.} The result is immediate using the definition of the norm of $\YYY_1'$ given in~\eqref{def:YYYi'} and~\eqref{eq:Gamma-NL}. 
\end{proof}

\begin{prop}\label{prop:DvBGamma-NL}
Let $g_1$, $g_2$ be smooth enough functions and $G_3$ a smooth enough vector field, then
\begin{equation}\label{eq:DvBGamma-NL}
\la \widetilde \nabla_v \Gamma(g_1,g_2) , G_3 \ra_{\XXX} 
\lesssim \Big( \| g_1 \|_{\XXX} \| g_2 \|_{\YYY_2} + \| g_1 \|_{\YYY_1} \| g_2 \|_{\XXX} \Big) \| G_3 \|_{\YYY_1}.
\end{equation}

\end{prop}

\begin{proof}[Proof of Proposition~\ref{prop:DvBGamma-NL}]  
We shall only prove that for any $\alpha \in \R$, there holds 
\begin{equation}\label{eq:DvBGamma-NL-1}
\begin{aligned}
&\la \langle v \rangle^{\alpha} \widetilde \nabla_v \Gamma(g_1,g_2) , G_3 \ra_{L^2_v} \\
&\qquad 
\lesssim \left( \| M^{\frac14}  g_1 \|_{L^2_v} 
\| \langle v \rangle^{\alpha} g_2 \|_{H^2_{v,*}}  
+\| M^{\frac14} g_1 \|_{H^1_v}  \| \langle v \rangle^{\alpha} g_2 \|_{L^2_v} \right)
\| G_3 \|_{H^1_{v,*}},
\end{aligned}
\end{equation}
from which we obtain the desired result by integrating in $x$ and arguing as in Step~4 of the proof of Proposition~\ref{prop:Gamma-NL}. 
We thus write $\Gamma(g_1,g_2) = \Gamma_1(g_1,g_2) + \cdots +\Gamma_5(g_1,g_2)$ as in~\eqref{eq:def-Gamma1}--\eqref{eq:def-Gamma5}, and we estimate each term separately in the sequel. We shall use during the proof the following equivalence:
\begin{align*}
\| \langle v \rangle^{\alpha} g_2 \|_{H^2_{v,*}} 
\lesssim 
\| \langle v \rangle^{\alpha} \langle v \rangle^{\gamma+2} g_2 \|_{L^2_{v}} 
&+\| \langle v \rangle^{\alpha}  \langle v \rangle^{\frac{\gamma}{2} + 1} \widetilde \nabla_v g_2 \|_{L^2_v} \\
&\qquad \qquad + \| \langle v \rangle^{\alpha} \widetilde \nabla_v (\widetilde \nabla_v g_2) \|_{L^2_v}  
\lesssim 
\| \langle v \rangle^{\alpha} g_2 \|_{H^2_{v,*}} .
\end{align*}

\medskip\noindent
\textit{Step 1. Term associated to $\Gamma_1$.}
Writing $\widetilde \nabla_{v_k} = B_{k \ell} \partial_{v_\ell}$ and observing that $[\widetilde \nabla_{v_k} , \partial_{v_i}] = - (\partial_{v_i} B_{k\ell}) \partial_{v_\ell}$, we first get that for any $k \in \{1,2,3\}$,
$$
\begin{aligned}
&\widetilde \nabla_{v_k} \Gamma_1(g_1,g_2)
= \partial_{v_i} \left\{ \left(a_{ij}* [\sqrt M g_1] \right) \partial_{v_j} (\widetilde \nabla_{v_k} g_2)   \right\} 
+\partial_{v_i} \left\{ \widetilde \nabla_{v_k}\left(a_{ij}* [\sqrt M g_1] \right) \partial_{v_j} g_2   \right\} \\
&\qquad\quad 
-\partial_{v_i} \left\{ \left(a_{ij}* [\sqrt M g_1] \right) (\partial_{v_j} B_{k\ell}) \partial_{v_\ell} g_2   \right\} 
- (\partial_{v_i} B_{k\ell}) \partial_{v_\ell} \left\{ \left(a_{ij}* [\sqrt M g_1] \right) \partial_{v_j} g_2   \right\}  
\end{aligned}
$$
whence 
$$
\la \langle v \rangle^{\alpha} \widetilde \nabla_{v_k} \Gamma_1(g_1,g_2) , G_{3,k} \ra_{L^2_v} 
= \mathrm{I}_1 + \mathrm{I}_2 + \mathrm{I}_3 + \mathrm{I}_4,
$$
where $G_{3,k}$ denotes the $k$-th component of $G_3$.

For the term $I_1$, we first make an integration by parts, then we use Lemma~\ref{lem:afGH} and the fact that $|\widetilde \nabla_v \langle v \rangle^{\alpha}|  \lesssim \langle v \rangle^{\frac{\gamma}{2}-1} \langle v \rangle^{\alpha}$ together with Cauchy-Schwarz inequality to obtain
$$
\begin{aligned}
\mathrm{I}_1 
&= - \la \langle v \rangle^{\alpha} \left(a_{ij}* [\sqrt M g_1] \right) \partial_{v_j} (\widetilde \nabla_{v_k} g_2) , \partial_{v_i} G_{3,k} \ra_{L^2_v} \\
&\qquad \qquad - \la \left(a_{ij}* [\sqrt M g_1] \right) \partial_{v_j} (\widetilde \nabla_{v_k} g_2) \partial_{v_i} \langle v \rangle^{\alpha}, G_{3,k} \ra_{L^2_v}  \\
&\lesssim 
\| M^{\frac14} g_1 \|_{L^2_v} \bigg(
\la \langle v \rangle^{\alpha} |\widetilde \nabla_v (\widetilde \nabla_{v_k} g_2)|, |\widetilde \nabla_v G_{3,k}| \ra_{L^2_v}
+  \la \langle v \rangle^{\frac{\gamma}{2}-1} \langle v \rangle^{\alpha} |\widetilde \nabla_v (\widetilde \nabla_{v_k} g_2)|, |G_{3,k}| \ra_{L^2_v} \bigg) \\
&\lesssim 
\| M^{\frac14} g_1 \|_{L^2_v} 
\| \langle v \rangle^{\alpha} \widetilde \nabla_v (\widetilde \nabla_v g_2) \|_{L^2_v} 
\bigg( \| \widetilde \nabla_v G_3 \|_{L^2_v}
+  \| \langle v \rangle^{\frac{\gamma}{2} - 1} G_3 \|_{L^2_v} \bigg).
\end{aligned}
$$
We argue in a similar fashion for the term $\mathrm{I}_2$. We first make an integration by parts and write that $\widetilde \nabla_{v_k}\left(a_{ij}* [\sqrt M g_1] \right) = B_{k\ell} \left( \partial_{v_\ell} a_{ij}* [\sqrt M g_1] \right)$, then we use Lemma~\ref{lem:afGH} together with $|B_{k\ell}| \lesssim \langle v \rangle^{\frac{\gamma}{2}+1}$, thus we obtain
$$
\begin{aligned}
\mathrm{I}_2
&= - \la \langle v \rangle^{\alpha} B_{k \ell}  \left(\partial_{v_\ell} a_{ij}* [\sqrt M g_1] \right) \partial_{v_j} g_2    , \partial_{v_i} G_{3,k} \ra_{L^2_v}  \\
&\qquad \qquad - \la  B_{k \ell}  \left( \partial_{v_\ell} a_{ij}* [\sqrt M g_1] \right) \partial_{v_j} g_2   \partial_{v_i} \langle v \rangle^{\alpha}, G_{3,k} \ra_{L^2_v}   \\
&\lesssim 
\| M^{\frac14} g_1 \|_{L^2_v} \bigg(
\la \langle v \rangle^{\alpha} \langle v \rangle^{\frac{\gamma}{2}+1} |\widetilde \nabla_v g_2|, |\widetilde \nabla_v G_{3,k}| \ra_{L^2_v}
+  \la \langle v \rangle^{\frac{\gamma}{2}+1}\langle v \rangle^{\frac{\gamma}{2}-1} \langle v \rangle^{\alpha} |\widetilde \nabla_v g_2|, |G_{3,k}| \ra_{L^2_v} \bigg) \\
&\lesssim 
\| M^{\frac14} g_1 \|_{L^2_v} 
\| \langle v \rangle^{\alpha}  \langle v \rangle^{\frac{\gamma}{2} + 1} \widetilde \nabla_v g_2 \|_{L^2_v} 
\bigg( \| \widetilde \nabla_v G_3 \|_{L^2_v}
+  \| \langle v \rangle^{\frac{\gamma}{2} - 1} G_3 \|_{L^2_v} \bigg).
\end{aligned}
$$
For the term $\mathrm{I}_3$, arguing similarly as above using also that $|\widetilde \nabla_v  B_{k \ell}| \lesssim \langle v \rangle^{\frac{\gamma}{2}+1} \langle v \rangle^{\frac{\gamma}{2}}$, we get 
$$
\begin{aligned}
\mathrm{I}_3
&= \la \langle v \rangle^{\alpha} \left(a_{ij}* [\sqrt M g_1] \right) (\partial_{v_j} B_{k \ell}) \partial_{v_\ell} g_2  , \partial_{v_i} G_{3,k} \ra_{L^2_v}  \\
&\qquad \qquad + \la  \left(a_{ij}* [\sqrt M g_1] \right) (\partial_{v_j} B_{k \ell}) \partial_{v_\ell} g_2  \partial_{v_i}\langle v \rangle^{\alpha} , G_{3,k} \ra_{L^2_v}    \\
&\lesssim 
\| M^{\frac14} g_1 \|_{L^2_v} \bigg(
\la \langle v \rangle^{\alpha} \langle v \rangle^{\frac{\gamma}{2}+1} \langle v \rangle^{\frac{\gamma}{2}} |\nabla_v g_2|, |\widetilde \nabla_v G_{3}| \ra_{L^2_v} \\
&\qquad \qquad +  \la \langle v \rangle^{\frac{\gamma}{2}+1} \langle v \rangle^{\frac{\gamma}{2}} \langle v \rangle^{\frac{\gamma}{2}-1} \langle v \rangle^{\alpha} |\nabla_v g_2|, |G_{3}| \ra_{L^2_v} \bigg) \\
&\lesssim 
\| M^{\frac14} g_1 \|_{L^2_v} 
\| \langle v \rangle^{\alpha}  \langle v \rangle^{\frac{\gamma}{2} + 1} \widetilde \nabla_v g_2 \|_{L^2_v} 
\bigg( \| \widetilde \nabla_v G_3 \|_{L^2_v}
+  \| \langle v \rangle^{\frac{\gamma}{2} - 1} G_3 \|_{L^2_v} \bigg).
\end{aligned}
$$
We treat the term $\mathrm{I}_4$ in the same way, first performing an integration by parts and using also that $| \partial_{v_i} \partial_{v_\ell}  B_{k \ell}| \lesssim \langle v \rangle^{\frac{\gamma}{2}-1}$ and $\langle v \rangle^{\frac{\gamma}{2}} |\nabla_v G_3| \lesssim |\widetilde \nabla_v G_3|$, it gives us
$$
\begin{aligned}
&\mathrm{I}_4
= \la \langle v \rangle^{\alpha} \left( a_{ij} * [\sqrt M g_1]\right)  \partial_{v_j} g_2 (\partial_{v_i} B_{k\ell} ) ,\partial_{v_\ell} G_{3,k} \ra_{L^2_v} \\
&\quad
+\la \left( a_{ij} * [\sqrt M g_1]\right)  \partial_{v_j} g_2   (\partial_{v_i} B_{k\ell} ) \partial_{v_\ell} \langle v \rangle^{\alpha} , G_{3,k} \ra_{L^2_v} \\
&\quad
+\la \langle v \rangle^{\alpha} \left( a_{ij} * [\sqrt M g_1]\right)  \partial_{v_j} g_2,  ( \partial_{v_i} \partial_{v_\ell} B_{k\ell} ) G_{3,k} \ra_{L^2_v}  \\
&\lesssim 
\| M^{\frac14} g_1 \|_{L^2_v} \bigg(
\la \langle v \rangle^{\alpha} \langle v \rangle^{\frac{\gamma}{2}+1} \langle v \rangle^{\frac{\gamma}{2}} |\widetilde \nabla_v g_2|, |\nabla_v G_{3}| \ra_{L^2_v}
+ \la  \langle v \rangle^{\alpha} \langle v \rangle^{\frac{\gamma}{2}+1} \langle v \rangle^{\frac{\gamma}{2}-1} |\widetilde \nabla_v g_2|,  |G_{3}| \ra_{L^2_v} \bigg) \\
&\lesssim 
\| M^{\frac14} g_1 \|_{L^2_v} 
\| \langle v \rangle^{\alpha}  \langle v \rangle^{\frac{\gamma}{2} + 1} \widetilde \nabla_v g_2 \|_{L^2_v} 
\bigg( \| \widetilde \nabla_v G_3 \|_{L^2_v}
+  \| \langle v \rangle^{\frac{\gamma}{2} - 1} G_3 \|_{L^2_v} \bigg).
\end{aligned}
$$
Finally, gathering previous estimates, we get
\begin{equation}\label{eq:DvBGamma1}
\la \langle v \rangle^{\alpha} \widetilde \nabla_v \Gamma_1(g_1,g_2) , G_3 \ra_{L^2_v}
\lesssim \| M^{\frac14} g_1 \|_{L^2_v} \| \langle v \rangle^{\alpha} g_2 \|_{H^2_{v,*}} \| G_3 \|_{H^1_{v,*}}.
\end{equation}

\medskip\noindent
\textit{Step 2. Term associated to $\Gamma_2$.}
We write $\widetilde \nabla_{v_k} = B_{k \ell} \partial_{v_\ell}$ so that $[\widetilde \nabla_{v_k} , \partial_{v_i}] = - (\partial_{v_i} B_{k\ell}) \partial_{v_\ell}$, thus we get, for any $k \in \{1,2,3\}$,
$$
\begin{aligned}
\widetilde \nabla_{v_k} \Gamma_2
&= -  \widetilde \nabla_{v_k} \partial_{v_i} \left\{ \left(b_{i}* [\sqrt M g_1] \right)  g_2   \right\} \\
&= - \partial_{v_i} \left\{ \left(b_{i}* [\sqrt M g_1] \right)  \widetilde \nabla_{v_k}  g_2   \right\} 
-  \partial_{v_i} \left\{ \widetilde \nabla_{v_k} \left(b_{i}* [\sqrt M g_1] \right)  g_2   \right\} \\
&\quad
+ (\partial_{v_i} B_{k\ell} ) \partial_{v_\ell} \left\{  \left(b_{i}* [\sqrt M g_1] \right) g_2   \right\} ,
\end{aligned}
$$
whence 
$$
\la \langle v \rangle^{\alpha} \widetilde \nabla_{v_k} \Gamma_2 (g_1,g_2) , G_{3,k} \ra 
= \mathrm{II}_1 + \mathrm{II}_2 + \mathrm{II}_3.
$$
For the term $\mathrm{II}_1$, we make an integrating by parts and then use Lemma~\ref{lem:conv} together with $\langle v \rangle^{\frac{\gamma}{2}} |\nabla_v G_3| \lesssim |\widetilde \nabla_v G_3|$, which yields
$$
\begin{aligned}
\mathrm{II}_1
&= \la \langle v \rangle^{\alpha}\left(b_{i}* [\sqrt M g_1] \right)  \widetilde \nabla_{v_k} g_2  , \partial_{v_i} G_{3,k} \ra_{L^2_v}
+ \la \left(b_{i}* [\sqrt M g_1] \right)  \widetilde \nabla_{v_k} g_2  \partial_{v_i}\langle v \rangle^{\alpha}, G_{3,k} \ra_{L^2_v} \\
&\lesssim \| M^{\frac14} g_1 \|_{L^2_v} \bigg(
\la \langle v \rangle^{\alpha}\langle v \rangle^{\gamma+1} |\widetilde \nabla_{v_k} g_2|,  |\nabla_v G_{3,k}| \ra_{L^2_v}
+ \la \langle v \rangle^{\gamma+1} |\widetilde \nabla_{v_k} g_2| \langle v \rangle^{\alpha-1},|G_{3,k}| \ra_{L^2_v}
\bigg)
\\
&\lesssim \| M^{\frac14} g_1 \|_{L^2_v} \| \langle v \rangle^{\alpha}\langle v \rangle^{\frac{\gamma}{2} + 1} \widetilde \nabla_v g_2 \|_{L^2_v}  \bigg( \| \widetilde \nabla_v G_3 \|_{L^2_v}
+  \| \langle v \rangle^{\frac{\gamma}{2} - 1} G_3 \|_{L^2_v} \bigg).
\end{aligned}
$$
For the term $\mathrm{II}_2$, we first make an integration by parts using that $\widetilde \nabla_{v_k}\left(b_{i}* [\sqrt M g_1] \right) = B_{k\ell} \left( \partial_{v_\ell} b_{i}* [\sqrt M g_1] \right)$, which yields
$$
\begin{aligned}
\mathrm{II}_2
&= \la \langle v \rangle^{\alpha}B_{k\ell} \left( \partial_{v_\ell} b_{i}* [\sqrt M g_1] \right) g_2  , \partial_{v_i} G_{3,k} \ra_{L^2_v}\\
&\quad + \la B_{k\ell} \left(\partial_{v_\ell} b_{i}* [\sqrt M g_1] \right) g_2  \partial_{v_i}\langle v \rangle^{\alpha} , G_{3,k}  \ra_{L^2_v} .
\end{aligned}
$$
We now split into two cases according to the estimates of Lemma~\ref{lem:afGH}: If $0 \le \gamma \le 1$, using that $|\widetilde \nabla_v \langle v \rangle^{\alpha} | \lesssim \langle v \rangle^{\frac{\gamma}{2}-1 + \alpha}$ we get
$$
\begin{aligned}
\mathrm{II}_2
&\lesssim \| M^{\frac14} g_1 \|_{L^2_v} \bigg(\la \langle v \rangle^{\alpha} \langle v \rangle^{\frac{\gamma}{2}+1} \langle v \rangle^{\frac{\gamma}{2}}  |g_2|, |\widetilde \nabla_v G_3| \ra_{L^2_v} 
+\la  \langle v \rangle^{\frac{\gamma}{2}+1} \langle v \rangle^{\frac{\gamma}{2}}  |g_2| \langle v \rangle^{\frac{\gamma}{2}-1 + \alpha},  |G_3| \ra_{L^2_v} \bigg)
\\
&\lesssim \| M^{\frac14} g_1 \|_{L^2_v} \| \langle v \rangle^{\alpha} \langle v \rangle^{\gamma + 1} g_2 \|_{L^2_v} \bigg( \| \widetilde \nabla_v G_3 \|_{L^2_v}
+  \| \langle v \rangle^{\frac{\gamma}{2} - 1} G_3 \|_{L^2_v} \bigg) ,
\end{aligned}
$$
and if $-2 \le \gamma <0$ we get, using also that $\langle v \rangle^{\gamma} \lesssim 1$,
$$
\begin{aligned}
\mathrm{II}_2
&\lesssim 
\| M^{\frac14} g_1 \|_{L^2_v} \la \langle v \rangle^{\alpha} \langle v \rangle^{\frac{\gamma}{2}+1} \langle v \rangle^{\frac{\gamma}{2}} |g_2|, |\widetilde \nabla_v G_3| \ra_{L^2_v}  \\
&\quad +\| M^{\frac14} g_1 \|_{H^1_v} \la \langle v \rangle^{\alpha} \langle v \rangle^{\frac{\gamma}{2}+1} \langle v \rangle^{\frac{\gamma}{2}-1} |g_2|, |\widetilde \nabla_v G_3| \ra_{L^2_v} \\
&\quad
+\| M^{\frac14} g_1 \|_{L^2_v} \la \langle v \rangle^{\frac{\gamma}{2}+1} \langle v \rangle^{\frac{\gamma}{2}}  \langle v \rangle^{\frac{\gamma}{2}-1 + \alpha}  |g_2|, |G_3| \ra_{L^2_v}  \\
&\quad +\| M^{\frac14} g_1 \|_{H^1_v} \la \langle v \rangle^{\frac{\gamma}{2}+1} \langle v \rangle^{\frac{\gamma}{2}-1} \langle v \rangle^{\frac{\gamma}{2}-1 + \alpha}  |g_2|, |G_3| \ra_{L^2_v} \\
&\lesssim
\| M^{\frac14} g_1 \|_{L^2_v} \| \langle v \rangle^{\alpha} \langle v \rangle^{\gamma+1} g_2 \|_{L^2_v} \left( \| \widetilde \nabla_v G_3 \|_{L^2_v} + \| \langle v \rangle^{\frac{\gamma}{2}-1} G_3 \|_{L^2_v} \right) \\
&\quad
+ \| M^{\frac14} g_1 \|_{H^1_v} \| \langle v \rangle^{\alpha}  g_2 \|_{L^2_v} \left( \| \widetilde \nabla_v G_3 \|_{L^2_v} + \| \langle v \rangle^{\frac{\gamma}{2}-1} G_3 \|_{L^2_v} \right).
\end{aligned}
$$
We deal with the term $\mathrm{II}_3$ by first making an integration by parts and then using Lemma~\ref{lem:conv} together with $|\partial_{v_i} B_{k\ell} | \lesssim \langle v \rangle^{\frac{\gamma}{2}}$ and $|\partial_{v_i} \partial_{v_\ell} B_{k\ell} | \lesssim \langle v \rangle^{\frac{\gamma}{2}-1}$, which yields 
$$
\begin{aligned}
\mathrm{II}_3
&= - \la \langle v \rangle^{\alpha} \left( b_{i} * [\sqrt M g_1]\right)  g_2 (\partial_{v_i} B_{k\ell} ) ,\partial_{v_\ell} G_{3,k} \ra_{L^2_v} \\
&\quad
- \la  \left( b_{i} * [\sqrt M g_1]\right)  g_2 (\partial_{v_i} B_{k\ell} ) \partial_{v_\ell} \langle v \rangle^{\alpha} ,  G_{3,k} \ra_{L^2_v} \\
&\quad
- \la \langle v \rangle^{\alpha} \left( b_{i} * [\sqrt M g_1]\right)  g_2 (\partial_{v_i} \partial_{v_\ell}B_{k\ell} )  ,  G_{3,k} \ra_{L^2_v}  \\
&\lesssim 
\| M^{\frac14} g_1 \|_{L^2_v} \bigg(
\la \langle v \rangle^{\alpha}  \langle v \rangle^{\gamma+1}  |g_2|  \langle v \rangle^{\frac{\gamma}{2}}, |\nabla_v G_{3}| \ra_{L^2_v}
+ \la \langle v \rangle^{\alpha} \langle v \rangle^{\gamma+1} |g_2| \langle v \rangle^{\frac{\gamma}{2}-1}, |G_{3}| \ra_{L^2_v}\bigg) \\
&\lesssim 
\| M^{\frac14} g_1 \|_{L^2_v} 
\| \langle v \rangle^{\alpha}  \langle v \rangle^{\gamma + 1} g_2 \|_{L^2_v} 
\bigg( \| \widetilde \nabla_v G_3 \|_{L^2_v}
+  \| \langle v \rangle^{\frac{\gamma}{2} - 1} G_3 \|_{L^2_v} \bigg).
\end{aligned}
$$

Finally, gathering previous estimates, we get
\begin{equation}\label{eq:DvBGamma2}
\begin{aligned}
\la \langle v \rangle^{\alpha} \widetilde \nabla_v \Gamma_2(g_1,g_2) , G_3 \ra_{L^2_v}
&\lesssim \| M^{\frac14} g_1 \|_{L^2_v}
\| \langle v \rangle^{\alpha} g_2 \|_{H^2_{v,*}}
\|  G_3 \|_{H^1_{v,*}} \\
&\quad 
+\| M^{\frac14} g_1 \|_{H^1_v}
\| \langle v \rangle^{\alpha} g_2 \|_{L^2_v}
\| G_3 \|_{H^1_{v,*}}.
\end{aligned}
\end{equation}

\medskip\noindent
\textit{Step 3. Term associated to $\Gamma_3$.} We write $\widetilde \nabla_{v_k} = B_{k \ell} \partial_{v_\ell}$ so that $[\widetilde \nabla_{v_k} , \partial_{v_i}] = - (\partial_{v_i} B_{k\ell}) \partial_{v_\ell}$, thus we get, for any $k \in \{1,2,3\}$,
$$
\begin{aligned}
\widetilde \nabla_{v_k} \Gamma_3(g_1,g_2)
&= -\frac12 \widetilde \nabla_{v_k} \left\{ \left(a_{ij}* [\sqrt M g_1] \right)  v_i \partial_{v_j} g_2   \right\} \\
&= -\frac12  \left(a_{ij}* [\sqrt M g_1] \right)  v_i \partial_{v_j} (\widetilde \nabla_{v_k} g_2)  \
-\frac12 \widetilde \nabla_{v_k} \left(a_{ij}* [\sqrt M g_1] \right)  v_i \partial_{v_j} g_2   \\
&\quad
-\frac12  \left(a_{ij}* [\sqrt M g_1] \right)  (\widetilde \nabla_{v_k} v_i) \partial_{v_j} g_2     
+\frac12  \left(a_{ij}* [\sqrt M g_1] \right)  v_i (\partial_{v_j} B_{k\ell}) \partial_{v_\ell} g_2   
\end{aligned}
$$
whence 
$$
\la \langle v \rangle^{\alpha} \widetilde \nabla_{v_k} \Gamma_3 (g_1,g_2) , G_{3,k} \ra 
= \mathrm{III}_1 + \mathrm{III}_2 + \mathrm{III}_3 + \mathrm{III}_4.
$$
We now bound each term separately using Lemma~\ref{lem:afGH}. For the term $\mathrm{III}_1$, using that $|\mathbf B(v) v| \lesssim \langle v \rangle^{\frac{\gamma}{2}+1}$, we first obtain, 
$$
\begin{aligned}
\mathrm{III}_1
&= -\frac12 \la \langle v \rangle^{\alpha}  \left(a_{ij}* [\sqrt M g_1] \right)  v_i \partial_{v_j} (\widetilde \nabla_{v_k} g_2) , G_{3,k} \ra_{L^2_v} \\
&\lesssim \| M^{\frac14} g_1 \|_{L^2_v}  \la \langle v \rangle^{\alpha} \langle v \rangle^{\frac{\gamma}{2}+1} |\widetilde \nabla_v(\widetilde \nabla_{v_k} g_2)|,  |G_{3,k}| \ra_{L^2_v} 
\\
&\lesssim \| M^{\frac14} g_1 \|_{L^2_v} \| \langle v \rangle^{\alpha} \widetilde \nabla_v (\widetilde \nabla_v g_2) \|_{L^2_v}
\| \langle v \rangle^{\frac{\gamma}{2} + 1} G_3 \|_{L^2_v} .
\end{aligned}
$$
Using also that $|B_{k \ell}| \lesssim \langle v \rangle^{\frac{\gamma}{2}+1}$, we then get
$$
\begin{aligned}
\mathrm{III}_2
&= -\frac12 \la \langle v \rangle^{\alpha} B_{k\ell} \left(\partial_{v_\ell} a_{ij}* [\sqrt M g_1] \right)  v_i \partial_{v_j} g_2 , G_{3,k} \ra_{L^2_v} \\
&\lesssim \| M^{\frac14} g_1 \|_{L^2_v} 
\la \langle v \rangle^{\alpha} \langle v \rangle^{\frac{\gamma}{2}+1} \langle v \rangle^{\frac{\gamma}{2}+1} |\widetilde \nabla_v g_2|, |G_3| \ra_{L^2_v} \\
&\lesssim \| M^{\frac14} g_1 \|_{L^2_v} \| \langle v \rangle^{\alpha} \langle v \rangle^{\frac{\gamma}{2} + 1} \widetilde \nabla_v  g_2 \|_{L^2_v}
\| \langle v \rangle^{\frac{\gamma}{2} + 1} G_3 \|_{L^2_v}.
\end{aligned}
$$
Since $\widetilde \nabla_{v_k} v_i  = B_{ki}$, we obtain  
$$
\begin{aligned}
\mathrm{III}_3
&= -\frac12 \la \langle v \rangle^{\alpha} \left( a_{ij}* [\sqrt M g_1] \right) B_{ki} \partial_{v_j} g_2 , G_{3,k} \ra_{L^2_v} \\
&\lesssim \| M^{\frac14} g_1 \|_{L^2_v} \la \langle v \rangle^{\alpha}  \langle v \rangle^{\gamma+2} |\widetilde \nabla_v g_2|, |G_3| \ra_{L^2_v} \\
&\lesssim \| M^{\frac14} g_1 \|_{L^2_v} \| \langle v \rangle^{\alpha} \langle v \rangle^{\frac{\gamma}{2}+1} \widetilde \nabla_v g_2 \|_{L^2_v}
\| \langle v \rangle^{\frac{\gamma}{2}+1} G_3 \|_{L^2_v}
.
\end{aligned}
$$
Using now that $|\widetilde \nabla_v B_{k\ell}| \lesssim \langle v \rangle^{\frac{\gamma}{2}+1} \langle v \rangle^{\frac{\gamma}{2}}$ and $\langle v \rangle^{\frac{\gamma}{2}} |\nabla_v g_2|  \lesssim |\widetilde \nabla_v g_2|$, we thus get 
$$
\begin{aligned}
\mathrm{III}_4
&= \frac12 \la \langle v \rangle^{\alpha} \left( a_{ij}* [\sqrt M g_1] \right) v_i \partial_{v_j} B_{k\ell} \partial_{v_\ell} g_2 , G_{3,k} \ra_{L^2_v} \\
&\lesssim \| M^{\frac14} g_1 \|_{L^2_v} \la \langle v \rangle^{\alpha} \langle v \rangle^{\frac{\gamma}{2}+1} \langle v \rangle^{\frac{\gamma}{2}+1} \langle v \rangle^{\frac{\gamma}{2}} |\nabla_v g_2|, |G_3| \ra_{L^2_v} \\
&\lesssim \| M^{\frac14} g_1 \|_{L^2_v} \| \langle v \rangle^{\alpha} \langle v \rangle^{\frac{\gamma}{2} + 1} \widetilde \nabla_v  g_2 \|_{L^2_v}
\| \langle v \rangle^{\frac{\gamma}{2} + 1} G_3 \|_{L^2_v}.
\end{aligned}
$$

Finally, gathering previous estimates, we get
\begin{equation}\label{eq:DvBGamma3}
\begin{aligned}
\la \langle v \rangle^{\alpha} \widetilde \nabla_v \Gamma_3(g_1,g_2) , G_3 \ra_{L^2_v}
&\lesssim \| M^{\frac14} g_1 \|_{L^2_v}
\| \langle v \rangle^{\alpha} g_2 \|_{H^2_{v,*}}
\| G_3 \|_{H^1_{v,*}}.
\end{aligned}
\end{equation}

\medskip\noindent
\textit{Step 4. Term associated to $\Gamma_4$.} Writing $\widetilde \nabla_{v_k} = B_{k \ell} \partial_{v_\ell}$ and observing that $[\widetilde \nabla_{v_k} , \partial_{v_i}] = - (\partial_{v_i} B_{k\ell}) \partial_{v_\ell}$ we first get, for any $k \in \{1,2,3\}$, 
$$
\begin{aligned}
\widetilde \nabla_{v_k} \Gamma_4
&= \frac14 \widetilde \nabla_{v_k} \left\{ \left(a_{ij}* [\sqrt M g_1] \right)  v_i v_j g_2   \right\} \\
&= \frac14  \left(a_{ij}* [\sqrt M g_1] \right)  v_i v_j \widetilde \nabla_{v_k} g_2  
+\frac14 \widetilde \nabla_{v_k}  \left(a_{ij}* [\sqrt M g_1] \right)  v_i v_j g_2 \\
&\quad
+\frac14 \left(a_{ij}* [\sqrt M g_1] \right)  \widetilde \nabla_{v_k}(v_i v_j) g_2  
\end{aligned}
$$
whence 
$$
\la \langle v \rangle^{\alpha} \widetilde \nabla_{v_k} \Gamma_4 (g_1,g_2) , G_{3,k} \ra 
= \mathrm{IV}_1 + \mathrm{IV}_2 + \mathrm{IV}_3.
$$
We now bound each term separately using Lemma~\ref{lem:conv}. For the term $\mathrm{IV}_1$, we first obtain
$$
\begin{aligned}
\mathrm{IV}_1
&= \frac14 \la  \langle v \rangle^{\alpha} \left(a_{ij}* [\sqrt M g_1] \right)  v_i v_j \widetilde \nabla_{v_k} g_2 , G_{3,k} \ra_{L^2_v} \\
&\lesssim \| M^{\frac14} g_1 \|_{L^2_v} \la \langle v \rangle^{\alpha} \langle v \rangle^{\gamma+2} |\widetilde \nabla_{v_k} g_2|, |G_{3,k}| \ra_{L^2_v} \\
&\lesssim \| M^{\frac14} g_1 \|_{L^2_v} \| \langle v \rangle^{\alpha} \langle v \rangle^{\frac{\gamma}{2}+1} \widetilde \nabla_v g_2 \|_{L^2_v}
\| \langle v \rangle^{\frac{\gamma}{2}+1} G_3 \|_{L^2_v}
.
\end{aligned}
$$
For the term $\mathrm{IV}_2$, we also use that $|B_{k \ell}| \lesssim \langle v \rangle^{\frac{\gamma}{2}+1}$ to obtain
$$
\begin{aligned}
\mathrm{IV}_2
&= \frac14 \la \langle v \rangle^{\alpha} B_{k\ell} \left(\partial_{v_\ell} a_{ij}* [\sqrt M g_1] \right)  v_i v_j g_2 , G_{3,k} \ra_{L^2_v} \\
&\lesssim \| M^{\frac14} g_1 \|_{L^2_v} \la \langle v \rangle^{\alpha} \langle v \rangle^{\frac{\gamma}{2}+1} \langle v \rangle^{\gamma+2} |g_2|, |G_{3}| \ra_{L^2_v} \\
&\lesssim \| M^{\frac14} g_1 \|_{L^2_v} \| \langle v \rangle^{\alpha} \langle v \rangle^{\gamma+2} g_2 \|_{L^2_v}
\| \langle v \rangle^{\frac{\gamma}{2}+1} G_3 \|_{L^2_v}
.
\end{aligned}
$$
For the last term $\mathrm{IV}_3$, we write $\widetilde \nabla_{v_k} (v_i v_j)  = v_i B_{kj} + v_j B_{ki}$, and thus we get 
$$
\begin{aligned}
\mathrm{IV}_3
&= \frac14 \la \langle v \rangle^{\alpha} \left( a_{ij}* [\sqrt M g_1] \right) [v_i B_{kj} + v_j B_{ki}] g_2 , G_{3,k} \ra_{L^2_v} \\
&\lesssim \| M^{\frac14} g_1 \|_{L^2_v} \la \langle v \rangle^{\alpha} \langle v \rangle^{\gamma+2} \langle v \rangle^{\frac{\gamma}{2}+1} |g_2| , |G_3| \ra_{L^2_v} \\
&\lesssim \| M^{\frac14} g_1 \|_{L^2_v} \| \langle v \rangle^{\alpha} \langle v \rangle^{\gamma+2} g_2 \|_{L^2_v}
\| \langle v \rangle^{\frac{\gamma}{2}+1} G_3 \|_{L^2_v}
.
\end{aligned}
$$

Finally, gathering previous estimates we get
\begin{equation}\label{eq:DvBGamma4}
\begin{aligned}
\la \langle v \rangle^{\alpha} \widetilde \nabla_v \Gamma_4(g_1,g_2) , G_3 \ra_{L^2_v}
&\lesssim \| M^{\frac14} g_1 \|_{L^2_v}
\| \langle v \rangle^{\alpha} g_2 \|_{H^2_{v,*}}
\| G_3 \|_{H^1_{v,*}}.
\end{aligned}
\end{equation}

\medskip\noindent
\textit{Step 5. Term associated to $\Gamma_5$.} We write $\widetilde \nabla_{v_k} = B_{k \ell} \partial_{v_\ell}$ so that $[\widetilde \nabla_{v_k} , \partial_{v_i}] = - (\partial_{v_i} B_{k\ell}) \partial_{v_\ell}$, thus we get, for any $k \in \{1,2,3\}$, 
$$
\begin{aligned}
\widetilde \nabla_{v_k} \Gamma_5(g_1,g_2)
&= -\frac12 \widetilde \nabla_{v_k} \left\{ \left(a_{ii}* [\sqrt M g_1] \right)  g_2   \right\} \\
&= -\frac12  \left(a_{ii}* [\sqrt M g_1] \right) \widetilde \nabla_{v_k} g_2 
-\frac12 \widetilde \nabla_{v_k}  \left(a_{ii}* [\sqrt M g_1] \right)   g_2   
\end{aligned}
$$
whence 
$$
\la \langle v \rangle^{\alpha} \widetilde \nabla_{v_k} \Gamma_5 (g_1,g_2) , G_{3,k} \ra_{L^2_v}
= \mathrm{V}_1 + \mathrm{V}_2.
$$
We now bound each term separately using Lemma~\ref{lem:conv}. For the term $\mathrm{V}_1$, we get
$$
\begin{aligned}
\mathrm{V}_1
&= -\frac12 \la  \langle v \rangle^{\alpha} \left(a_{ii}* [\sqrt M g_1] \right)  \widetilde \nabla_{v_k} g_2 , G_{3,k} \ra_{L^2_v} \\
&\lesssim \| M^{\frac14} g_1 \|_{L^2_v} \la \langle v \rangle^{\alpha} \langle v \rangle^{\gamma+2} |\widetilde \nabla_{v_k} g_2|, |G_{3,k}| \ra_{L^2_v} \\
&\lesssim \| M^{\frac14} g_1 \|_{L^2_v} \| \langle v \rangle^{\alpha} \langle v \rangle^{\frac{\gamma}{2}+1} \widetilde \nabla_v g_2 \|_{L^2_v}
\| \langle v \rangle^{\frac{\gamma}{2}+1} G_3 \|_{L^2_v}
,
\end{aligned}
$$
and, for the term $\mathrm{V}_2$, we use $|B_{k \ell}| \lesssim \langle v \rangle^{\frac{\gamma}{2}+1}$ to deduce
$$
\begin{aligned}
\mathrm{V}_2
&= -\frac12 \la \langle v \rangle^{\alpha} B_{k\ell} \left(\partial_{v_\ell} a_{ii}* [\sqrt M g_1] \right)  g_2 , G_{3,k} \ra_{L^2_v} \\
&\lesssim \| M^{\frac14} g_1 \|_{L^2_v} \la \langle v \rangle^{\alpha} \langle v \rangle^{\frac{\gamma}{2}+1} \langle v \rangle^{\gamma+1} |g_2|, |G_{3}| \ra_{L^2_v} \\
&\lesssim \| M^{\frac14} g_1 \|_{L^2_v} \| \langle v \rangle^{\alpha} \langle v \rangle^{\gamma+1} g_2 \|_{L^2_v}
\| \langle v \rangle^{\frac{\gamma}{2}+1} G_3 \|_{L^2_v}
.
\end{aligned}
$$

Finally, gathering previous estimates we get
\begin{equation}\label{eq:DvBGamma5}
\begin{aligned}
\la \langle v \rangle^{\alpha} \widetilde \nabla_v \Gamma_5(g_1,g_2) , G_3 \ra_{L^2_v}
&\lesssim \| M^{\frac14} g_1 \|_{L^2_v}
\| \langle v \rangle^{\alpha} g_2 \|_{H^2_{v,*}}
\| G_3 \|_{H^1_{v,*}}.
\end{aligned}
\end{equation}

\medskip\noindent
\textit{Step 6. Proof of \eqref{eq:DvBGamma-NL-1}.} We gather estimates \eqref{eq:DvBGamma1}--\eqref{eq:DvBGamma2}--\eqref{eq:DvBGamma3}--\eqref{eq:DvBGamma4}--\eqref{eq:DvBGamma5}.
\end{proof}

\begin{prop}\label{prop:DxBGamma-NL}
Let $g_1$, $g_2$ be smooth enough functions and $G_3$ a smooth enough vector field, then
\begin{equation}\label{eq:DxBGamma-NL}
\begin{aligned}
&\la \widetilde \nabla_x \Gamma(g_1,g_2) , G_3 \ra_{\XXX} \\
&\quad 
\lesssim \left( \| g_1 \|_{\XXX} \| \widetilde \nabla_x g_2 \|_{\YYY_1}  
+ \| \widetilde \nabla_x g_1 \|_{\XXX}  \| g_2 \|_{\YYY_1}
+\| g_1 \|_{\XXX} \| g_2 \|_{\YYY_2} \right) \| G_3 \|_{\YYY_1}.
\end{aligned}
\end{equation}

\end{prop}

\begin{proof}[Proof of Proposition~\ref{prop:DxBGamma-NL}]  
We shall first prove that for any $\alpha \in \R$, there holds 
\begin{equation}\label{eq:DxBGamma-NL-1}
\begin{aligned}
&\la \langle v \rangle^{\alpha} \widetilde \nabla_x \Gamma(g_1,g_2) , G_3 \ra_{L^2_v} \\
&\quad \lesssim \left(\| M^{\frac14} g_1 \|_{L^2_v}
\| \langle v \rangle^{\alpha} \widetilde \nabla_x g_2 \|_{H^1_{v,*}} 
+ \| M^{\frac14} \nabla_x g_1 \|_{L^2_v} 
\| \langle v \rangle^{\alpha} \langle v \rangle^{\frac{\gamma}{2}+1} g_2 \|_{H^1_{v,*} }   \right)
\| G_3 \|_{H^1_{v,*}}.
\end{aligned}
\end{equation}
We thus write $\Gamma(g_1,g_2) = \Gamma_1(g_1,g_2) + \cdots +\Gamma_5(g_1,g_2)$ as in \eqref{eq:def-Gamma1}--\eqref{eq:def-Gamma5}, and we estimate each term separately in the sequel. 
The desired result~\eqref{eq:DxBGamma-NL} is then obtained by integrating in $x$, the proof is given in Step~7 because it differs from Step~4 of the proof of Proposition~\ref{prop:Gamma-NL}.

\medskip\noindent
\textit{Step 1. Term associated to $\Gamma_1$.}
Writing $\widetilde \nabla_{x_k} = B_{k \ell} \partial_{x_\ell}$, observing that $[\widetilde \nabla_{x_k} , \partial_{v_i}] = - (\partial_{v_i} B_{k\ell}) \partial_{x_\ell}$ and using that $\widetilde \nabla_{x_k}\left(a_{ij}* [\sqrt M g_1] \right) = B_{k\ell} \left( a_{ij}* [\sqrt M \partial_{x_\ell} g_1] \right)$, we first get that for any $k \in \{1,2,3\}$,
$$
\begin{aligned}
\widetilde \nabla_{x_k} \Gamma_1(g_1,g_2)
&= \partial_{v_i} \left\{ \left(a_{ij}* [\sqrt M g_1] \right) \partial_{v_j} (\widetilde \nabla_{x_k} g_2)   \right\} 
+\partial_{v_i} \left\{ B_{k\ell}\left(a_{ij}* [\sqrt M \partial_{x_\ell} g_1] \right) \partial_{v_j} g_2   \right\} \\
&\quad 
-\partial_{v_i} \left\{ \left(a_{ij}* [\sqrt M g_1] \right) (\partial_{v_j} B_{k\ell}) \partial_{x_\ell} g_2   \right\} 
- (\partial_{v_i} B_{k\ell}) \left(a_{ij}*  [\sqrt M \partial_{x_\ell} g_1] \right) \partial_{v_j} g_2   \\
&\quad 
- (\partial_{v_i} B_{k\ell}) \left(a_{ij}* [\sqrt M g_1] \right) \partial_{v_j} (\partial_{x_\ell} g_2).  
\end{aligned}
$$
whence 
$$
\la \langle v \rangle^{\alpha} \widetilde \nabla_{x_k} \Gamma_1(g_1,g_2) , G_{3,k} \ra_{L^2_v} 
= \mathrm{I}_1 + \mathrm{I}_2 + \mathrm{I}_3 + \mathrm{I}_4 + \mathrm{I}_5.
$$
We can then estimate each term separately using Lemma~\ref{lem:afGH} and arguing similarly as in Step~1 of the proof of Proposition~\ref{prop:DvBGamma-NL}, which we brief explain below.

For the term $\mathrm{I}_1$, we make an integration by parts and use the fact that $|\widetilde \nabla_v \langle v \rangle^{\alpha}|  \lesssim \langle v \rangle^{\frac{\gamma}{2}-1} \langle v \rangle^{\alpha}$ to obtain
$$
\begin{aligned}
\mathrm{I}_1 
&= - \la \langle v \rangle^{\alpha} \left(a_{ij}* [\sqrt M g_1] \right) \partial_{v_j} ( \widetilde \nabla_{x_k} g_2) , \partial_{v_i} G_{3,k} \ra_{L^2_v} \\
&\quad 
- \la \left(a_{ij}* [\sqrt M g_1] \right) \partial_{v_j} ( \widetilde \nabla_{x_k} g_2) \partial_{v_i}\langle v \rangle^{\alpha} , G_{3,k} \ra_{L^2_v} \\
&\lesssim \| M^{\frac14} g_1 \|_{L^2_v} \| \langle v \rangle^{\alpha} \widetilde \nabla_v (\widetilde \nabla_x g_2) \|_{L^2_v} \bigg( \| \widetilde \nabla_v G_3 \|_{L^2_v} + \| \langle v \rangle^{\frac{\gamma}{2} + 1} G_3 \|_{L^2_v} \bigg).
\end{aligned}
$$
In a similar way, for the term $\mathrm{I}_2$, we make an integration by parts and use that $|B_{k\ell}| \lesssim \langle v \rangle^{\frac{\gamma}{2}+1}$, which yields
$$
\begin{aligned}
\mathrm{I}_2 
&= - \la  \langle v \rangle^{\alpha} B_{k \ell} \left(a_{ij}* [\sqrt M \partial_{x_\ell} g_1] \right) \partial_{v_j} g_2    , \partial_{v_i} G_{3,k} \ra_{L^2_v} \\
&\quad
- \la  \langle v \rangle^{\alpha} B_{k \ell} \left(a_{ij}* [\sqrt M \partial_{x_\ell} g_1] \right) \partial_{v_j} g_2  \partial_{v_i} \langle v \rangle^{\alpha} ,  G_{3,k} \ra_{L^2_v}   \\
&\lesssim \| M^{\frac14} \nabla_x g_1 \|_{L^2_v} \| \langle v \rangle^{\alpha} \langle v \rangle^{\frac{\gamma}{2} + 1} \widetilde \nabla_v g_2 \|_{L^2_v} \bigg( \| \widetilde \nabla_v G_3 \|_{L^2_v} + \| \langle v \rangle^{\frac{\gamma}{2} + 1} G_3 \|_{L^2_v} \bigg).
\end{aligned}
$$
For the term $\mathrm{I}_3$, arguing as above and using also that $|\widetilde \nabla_v  B_{k \ell}| \lesssim \langle v \rangle^{\frac{\gamma}{2}+1} \langle v \rangle^{\frac{\gamma}{2}}$, we get 
$$
\begin{aligned}
\mathrm{I}_3
&= \la \langle v \rangle^{\alpha} \left( a_{ij} * [\sqrt M g_1]\right) (\partial_{v_j} B_{k\ell} ) \partial_{x_\ell} g_2 , \partial_{v_i} G_{3,k} \ra_{L^2_v} \\
&\quad
+ \la \left( a_{ij} * [\sqrt M g_1]\right) (\partial_{v_j} B_{k\ell} ) \partial_{x_\ell} g_2  \partial_{v_i}\langle v \rangle^{\alpha} , G_{3,k} \ra_{L^2_v} \\
&\lesssim 
\| M^{\frac14} g_1 \|_{L^2_v} \| \langle v \rangle^{\alpha} \langle v \rangle^{\frac{\gamma}{2} + 1} \widetilde \nabla_x g_2 \|_{L^2_v} \bigg( \| \widetilde \nabla_v G_3 \|_{L^2_v} + \| \langle v \rangle^{\frac{\gamma}{2} + 1} G_3 \|_{L^2_v} \bigg).
\end{aligned}
$$
For the term $\mathrm{I}_4$, we use that $|\widetilde \nabla_v B_{k\ell}| \lesssim \langle v \rangle^{\gamma+1}$ to get
$$
\begin{aligned}
\mathrm{I}_4
&= - \la  \langle v \rangle^{\alpha} \left(a_{ij}* [\sqrt M \partial_{x_\ell} g_1] \right) (\partial_{v_i} B_{k\ell}) \partial_{v_j} g_2 , G_{3,k} \ra_{L^2_v} \\
&\lesssim \| M^{\frac14} \nabla_x g_1 \|_{L^2_v} \| \langle v \rangle^{\alpha} \langle v \rangle^{\frac{\gamma}{2}} \widetilde \nabla_v g_2 \|_{L^2_v} \| \langle v \rangle^{\frac{\gamma}{2}+1} G_3 \|_{L^2_v}  .
\end{aligned}
$$
Performing an integration by parts in the term $\mathrm{I}_5$, we first obtain
$$
\begin{aligned}
\mathrm{I}_5
&= - \la \langle v \rangle^{\alpha} \left( a_{ij} * [\sqrt M g_1]\right) (\partial_{v_i} B_{k \ell})  \partial_{v_j} (\partial_{x_\ell} g_2) , G_{3,k} \ra_{L^2_v}\\
&=    
\la \langle v \rangle^{\alpha} \left( a_{ij} * [\sqrt M g_1]\right) (\partial_{v_i} B_{k \ell})   \partial_{x_\ell} g_2 , \partial_{v_j} G_{3,k} \ra_{L^2_v} \\
&\quad
+\la \langle v \rangle^{\alpha} \left( a_{ij} * [\sqrt M g_1]\right) (\partial_{v_j} \partial_{v_i} B_{k \ell})   \partial_{x_\ell} g_2 , G_{3,k} \ra_{L^2_v} \\
&\quad 
+\la \left( a_{ij} * [\sqrt M g_1]\right) (\partial_{v_i} B_{k \ell})   (\partial_{x_\ell} g_2) \partial_{v_j} \langle v \rangle^{\alpha} , G_{3,k} \ra_{L^2_v}\\
&\quad
+\la \langle v \rangle^{\alpha} \left( b_{i} * [\sqrt M g_1]\right) (\partial_{v_i} B_{k \ell})   \partial_{x_\ell} g_2 , G_{3,k} \ra_{L^2_v}.
\end{aligned}
$$
Using Lemmas~\ref{lem:conv} and \ref{lem:afGH} together with $| \partial_{v_j} B_{k\ell}| \lesssim \langle v \rangle^{\frac{\gamma}{2}} $ and $|\partial_{v_\ell}\partial_{v_j} B_{k\ell} | \lesssim \langle v \rangle^{\frac{\gamma}{2} - 1}$, we hence obtain
$$
\begin{aligned}
\mathrm{I}_5
&\lesssim \| M^{\frac14} g_1 \|_{L^2_v} \la \langle v \rangle^{\alpha} \langle v \rangle^{\gamma+1}  |\nabla_x g_2| |\widetilde \nabla_v G_3| \ra_{L^2_v} \\
&\quad
+ \| M^{\frac14} g_1 \|_{L^2_v} \la \langle v \rangle^{\alpha} \langle v \rangle^{\gamma+2} \langle v \rangle^{\frac{\gamma}{2}-1} |\nabla_x g_2| |G_3| \ra_{L^2_v}  \\
&\quad 
+\| M^{\frac14} g_1 \|_{L^2_v} \la \langle v \rangle^{\gamma+1} \langle v \rangle^{\frac{\gamma}{2}-1} \langle v \rangle^{\alpha} |\nabla_x g_2| |G_3| \ra_{L^2_v} \\
&\quad
+ \| M^{\frac14} g_1 \|_{L^2_v} \la \langle v \rangle^{\alpha} \langle v \rangle^{\gamma+1} \langle v \rangle^{\frac{\gamma}{2}} |\nabla_x g_2| |G_3| \ra_{L^2_v} \\
&\lesssim \| M^{\frac14} g_1 \|_{L^2_v} \| \langle v \rangle^{\alpha} \langle v \rangle^{\frac{\gamma}{2} + 1} \widetilde \nabla_x g_2 \|_{L^2_v} \bigg( \| \widetilde \nabla_v G_3 \|_{L^2_v} + \| \langle v \rangle^{\frac{\gamma}{2} + 1} G_3 \|_{L^2_v} \bigg).
\end{aligned}
$$

Finally, gathering previous estimates, we get
\begin{equation}\label{eq:DxBGamma1}
\begin{aligned}
&\la \langle v \rangle^{\alpha} \widetilde \nabla_x \Gamma_1(g_1,g_2) , G_3 \ra_{L^2_v} \\
&\quad 
\lesssim \left(\| M^{\frac14} g_1 \|_{L^2_v}
\| \langle v \rangle^{\alpha} \widetilde \nabla_x g_2 \|_{H^1_{v,*}} 
+ \| M^{\frac14} \nabla_x g_1 \|_{L^2_v} 
\| \langle v \rangle^{\alpha} \langle v \rangle^{\frac{\gamma}{2} + 1} \widetilde \nabla_v g_2 \|_{L^2_{v}} \right)
\| G_3 \|_{H^1_{v,*}}.
\end{aligned}
\end{equation}

\medskip\noindent
\textit{Step 2. Term associated to $\Gamma_2$.}
Similarly as for $\Gamma_1$, we have that for any $k \in \{1,2,3\}$,
$$
\begin{aligned}
\widetilde \nabla_{x_k} \Gamma_2
&= -  \widetilde \nabla_{x_k} \partial_{v_i} \left\{ \left(b_{i}* [\sqrt M g_1] \right)  g_2   \right\} \\
&= - \partial_{v_i} \left\{ \left(b_{i}* [\sqrt M g_1] \right)  \widetilde \nabla_{x_k}  g_2   \right\} 
-  \partial_{v_i}  \left\{ B_{k\ell} \left(b_{i}* [\sqrt M \partial_{x_\ell} g_1] \right)  g_2 \right\}   \\
&\quad
+ (\partial_{v_i} B_{k\ell} )   \left(b_{i}* [\sqrt M g_1] \right) \partial_{x_\ell} g_2  
+ (\partial_{v_i} B_{k\ell} )  \left(b_{i}* [\sqrt M \partial_{x_\ell}g_1] \right) g_2  .
\end{aligned}
$$
We can then estimate each term separately using Lemma~\ref{lem:afGH} and arguing as in Step~2 of the proof of Proposition~\ref{prop:DvBGamma-NL}, which yields 
\begin{equation}\label{eq:DxBGamma2}
\begin{aligned}
&\la \langle v \rangle^{\alpha} \widetilde \nabla_x \Gamma_2(g_1,g_2) , G_3 \ra_{L^2_v} \\
&\qquad 
\lesssim \left(\| M^{\frac14} g_1 \|_{L^2_v}
\| \langle v \rangle^{\alpha} \widetilde \nabla_x g_2 \|_{H^1_{v,*}} 
+ \| M^{\frac14} \nabla_x g_1 \|_{L^2_v} 
\| \langle v \rangle^{\alpha} \langle v \rangle^{\gamma + 2} g_2 \|_{L^2_v} \right)
\| G_3 \|_{H^1_{v,*}}.
\end{aligned}
\end{equation}

\medskip\noindent
\textit{Step 3. Term associated to $\Gamma_3$.} 
As previously, we first get that for any $k \in \{1,2,3\}$,
$$
\begin{aligned}
\widetilde \nabla_{x_k} \Gamma_3
&= - \widetilde \nabla_{x_k} \left\{ \left(a_{ij}* [\sqrt M g_1] \right)  v_i \partial_{v_j} g_2   \right\} \\
&= -  \left(a_{ij}* [\sqrt M g_1] \right)  v_i \partial_{v_j} (\widetilde \nabla_{x_k} g_2)  
- B_{k \ell } \left(a_{ij}* [\sqrt M  \partial_{x_\ell} g_1] \right)  v_i \partial_{v_j} g_2   \\
&\quad
+  \left(a_{ij}* [\sqrt M g_1] \right)  v_i (\partial_{v_j} B_{k\ell}) \partial_{x_\ell} g_2   .
\end{aligned}
$$
We can then estimate each term separately using Lemma~\ref{lem:afGH} and arguing as in Step~3 of the proof of Proposition~\ref{prop:DvBGamma-NL}, which yields
\begin{equation}\label{eq:DxBGamma3}
\begin{aligned}
&\la \langle v \rangle^{\alpha} \widetilde \nabla_x \Gamma_3(g_1,g_2) , G_3 \ra_{L^2_v} \\
&\quad
\lesssim \left(\| M^{\frac14} g_1 \|_{L^2_v}
\| \langle v \rangle^{\alpha} \widetilde \nabla_x g_2 \|_{H^1_{v,*}} 
+ \| M^{\frac14} \nabla_x g_1 \|_{L^2_v} 
\| \langle v \rangle^{\alpha} \langle v \rangle^{\frac{\gamma}{2} + 1} \widetilde \nabla_v g_2 \|_{L^2_v} \right)
\| G_3 \|_{H^1_{v,*}}.
\end{aligned}
\end{equation}

\medskip\noindent
\textit{Step 4. Term associated to $\Gamma_4$.} 
As previously, we first get that for any $k \in \{1,2,3\}$,
$$
\begin{aligned}
\widetilde \nabla_{x_k} \Gamma_4
&= \frac14 \widetilde \nabla_{x_k} \left\{ \left(a_{ij}* [\sqrt M g_1] \right)  v_i v_j g_2   \right\} \\
&= \frac14  \left(a_{ij}* [\sqrt M g_1] \right)  v_i v_j \widetilde \nabla_{x_k} g_2  
+\frac14 B_{k\ell}  \left(a_{ij}* [\sqrt M \partial_{x_\ell} g_1] \right)  v_i v_j g_2   .
\end{aligned}
$$
We can then estimate each term separately using Lemma~\ref{lem:conv} and arguing as in Step~4 of the proof of Proposition~\ref{prop:DvBGamma-NL}, which yields
\begin{equation}\label{eq:DxBGamma4}
\begin{aligned}
&\la \langle v \rangle^{\alpha} \widetilde \nabla_x \Gamma_4(g_1,g_2) , G_3 \ra_{L^2_v} \\
&\qquad 
\lesssim \left(\| M^{\frac14} g_1 \|_{L^2_v}
\| \langle v \rangle^{\alpha} \widetilde \nabla_x g_2 \|_{H^1_{v,*}} 
+ \| M^{\frac14} \nabla_x g_1 \|_{L^2_v} 
\| \langle v \rangle^{\alpha} \langle v \rangle^{\gamma + 2} g_2 \|_{L^2_v} \right)
\| G_3 \|_{H^1_{v,*}}.
\end{aligned}
\end{equation}

\medskip\noindent
\textit{Step 5. Term associated to $\Gamma_5$.} 
As previously, we first get that for any $k \in \{1,2,3\}$,
$$
\begin{aligned}
\widetilde \nabla_{x_k} \Gamma_5
&=-\frac12 \widetilde \nabla_{x_k} \left\{ \left(a_{ii}* [\sqrt M g_1] \right)   g_2   \right\} \\
&= -\frac12  \left(a_{ii}* [\sqrt M g_1] \right)  \widetilde \nabla_{x_k} g_2  
-\frac12 B_{k\ell}  \left(a_{ii}* [\sqrt M \partial_{x_\ell} g_1] \right)  g_2  . 
\end{aligned}
$$
We can then estimate each term separately using Lemma~\ref{lem:conv} and arguing as in Step~5 of the proof of Proposition~\ref{prop:DvBGamma-NL}, which yields
\begin{equation}\label{eq:DxBGamma5}
\begin{aligned}
&\la \langle v \rangle^{\alpha} \widetilde \nabla_x \Gamma_5(g_1,g_2) , G_3 \ra_{L^2_v} \\
&\qquad
\lesssim \left(\| M^{\frac14} g_1 \|_{L^2_v}
\| \langle v \rangle^{\alpha} \widetilde \nabla_x g_2 \|_{H^1_{v,*}} 
+ \| M^{\frac14} \nabla_x g_1 \|_{L^2_v} 
\| \langle v \rangle^{\alpha} \langle v \rangle^{\gamma + 2} g_2 \|_{L^2_v} \right)
\| G_3 \|_{H^1_{v,*}}.
\end{aligned}
\end{equation}

\medskip\noindent
\textit{Step 6. Proof of \eqref{eq:DxBGamma-NL-1}.} We gather estimates \eqref{eq:DxBGamma1}--\eqref{eq:DxBGamma2}--\eqref{eq:DxBGamma3}--\eqref{eq:DxBGamma4}--\eqref{eq:DxBGamma5}.

\medskip\noindent
\textit{Step 7. Proof of \eqref{eq:DxBGamma-NL}.} 
We first write
$$
\begin{aligned}
\la \widetilde \nabla_x \Gamma (g_1, g_2) , G_3 \ra_{\XXX} 
&= \sum_{i=0}^3 \la \langle v \rangle^{(3-i)(\frac{\gamma}{2}+1)} \nabla_x^i \widetilde \nabla_x \Gamma (g_1, g_2) , \langle v \rangle^{(3-i)(\frac{\gamma}{2}+1)} \nabla_x^i G_3 \ra_{L^2_{x,v}} \\
&=: T_0 + T_1 + T_2 + T_3 
\end{aligned}
$$
and recall that 
$
\partial_{x_\ell} \Gamma(g_1, g_2) = \Gamma (\partial_{x_\ell} g_1 , g_2) + \Gamma ( g_1 , \partial_{x_\ell} g_2)
$
and that $\nabla_x$ and $\widetilde \nabla_x$ commute. In the remainder of the proof, we shall use~\eqref{eq:DxBGamma-NL-1} as well as some classical Sobolev embeddings as in Step~4 of the proof of Proposition~\ref{prop:Gamma-NL} without no further mention. We here point out that in estimate~\eqref{eq:DxBGamma-NL-1}, it is important to keep~$\| \langle v \rangle^{\alpha} \langle v \rangle^{\frac{\gamma}{2}+1} g_2 \|_{H^1_{v,*} }$ in our estimate instead of bounding this term by $\| \langle v \rangle^{\alpha}  g_2 \|_{H^2_{v,*} }$ and also that to close our estimate, we widely use that the weights in our functional spaces depend on the order of the derivatives in $x$. 

For $T_0$, we have:
 \begin{multline*}
 T_0
 \lesssim \bigg( \| M^{\frac14} g_1 \|_{H^2_x L^2_v} \| \langle v \rangle^{3(\frac{\gamma}{2}+1)} \widetilde \nabla_x g_2 \|_{L^2_x(H^1_{v,*})} \\
+ \| M^{\frac14} \nabla_x g_1 \|_{H^2_x L^2_v} \| \langle v \rangle^{3(\frac{\gamma}{2}+1)} \langle v \rangle^{\frac{\gamma}{2}+1}  g_2 \|_{L^2_x (H^1_{v,*})} \bigg) \| \langle v \rangle^{3(\frac{\gamma}{2}+1)} G_3 \|_{L^2_x(H^1_{v,*})} \\ 
 \lesssim \left( \|g_1 \|_{\XXX} \| \widetilde \nabla_x g_2 \|_{\YYY_1} 
 + \| g_1 \|_{\XXX} \| g_2 \|_{\YYY_2} \right) \| G_3 \|_{\YYY_1}.
 \end{multline*}
 For $T_1$, we have:
 $$
 \begin{aligned}
 T_1
 &\lesssim 
 \bigg( \| M^{\frac14} g_1 \|_{H^2_x L^2_v} \| \langle v \rangle^{2(\frac{\gamma}{2}+1)} \nabla_x (\widetilde \nabla_x g_2) \|_{L^2_x(H^1_{v,*})} \\
 &\qquad\qquad
 + \| M^{\frac14} \nabla_x g_1 \|_{H^1_x L^2_v} \| \langle v \rangle^{2(\frac{\gamma}{2}+1)} \widetilde \nabla_x g_2 \|_{H^1_x(H^1_{v,*})} \bigg) \| \langle v \rangle^{2(\frac{\gamma}{2}+1)} \nabla_x G_3 \|_{L^2_x(H^1_{v,*})} \\
 &\quad 
 + \bigg( \| M^{\frac14} \nabla_x g_1 \|_{H^2_x L^2_v} \| \langle v \rangle^{2(\frac{\gamma}{2}+1)} \nabla_x (\langle v \rangle^{\frac{\gamma}{2}+1}  g_2) \|_{L^2_x(H^1_{v,*})} \\
 &\qquad\qquad
 + \| M^{\frac14} \nabla^2_x g_1 \|_{H^1_x L^2_v} \| \langle v \rangle^{2(\frac{\gamma}{2}+1)} \langle v \rangle^{\frac{\gamma}{2}+1}  g_2 \|_{H^1_x(H^1_{v,*})} \bigg) \| \langle v \rangle^{2(\frac{\gamma}{2}+1)} \nabla_x G_3 \|_{L^2_x(H^1_{v,*})}  \\
 &\lesssim \left(
 \| g_1 \|_{\XXX} \| \widetilde \nabla_x g_2 \|_{\YYY_1}   
 + \| g_1 \|_{\XXX} \| g_2 \|_{\YYY_2} \right) \|G_3 \|_{\YYY_1}.
 \end{aligned}
 $$
 For $T_2$, we have:
 $$
 \begin{aligned}
 T_2
 &\lesssim 
 \bigg( 
 \| M^{\frac14} g_1 \|_{H^2_x L^2_v} \| \langle v \rangle^{\frac{\gamma}{2}+1} \nabla^2_x (\widetilde \nabla_x g_2) \|_{L^2_x(H^1_{v,*})} \\
&\qquad \qquad + \| M^{\frac14} \nabla_x g_1 \|_{H^1_xL^2_v} \| \langle v \rangle^{\frac{\gamma}{2}+1} \nabla_x(\widetilde \nabla_x g_2) \|_{H^1_x(H^1_{v,*})} \\
 &\qquad\qquad
 + \| M^{\frac14} \nabla^2_x g_1 \|_{L^2_x L^2_v} \| \langle v \rangle^{\frac{\gamma}{2}+1} \widetilde \nabla_x g_2 \|_{H^2_x(H^1_{v,*})} \bigg) 
 \| \langle v \rangle^{\frac{\gamma}{2}+1} \nabla^2_x G_3 \|_{L^2_x(H^1_{v,*})} \\
 &\quad
 + \bigg( 
 \| M^{\frac14} \nabla_x g_1 \|_{H^2_x L^2_v} \| \langle v \rangle^{\frac{\gamma}{2}+1} \nabla^2_x (\langle v \rangle^{\frac{\gamma}{2}+1}  g_2) \|_{L^2_x(H^1_{v,*})} \\
&\qquad \qquad + \| M^{\frac14} \nabla^2_x g_1 \|_{H^1_x L^2_v} \| \langle v \rangle^{\frac{\gamma}{2}+1} \nabla_x(\langle v \rangle^{\frac{\gamma}{2}+1}  g_2) \|_{H^1_x(H^1_{v,*})} \\
 &\qquad\qquad
 + \| M^{\frac14} \nabla^3_x g_1 \|_{L^2_x L^2_v} \| \langle v \rangle^{\frac{\gamma}{2}+1} \langle v \rangle^{\frac{\gamma}{2}+1}  g_2 \|_{H^2_x(H^1_{v,*})}
 \bigg) \| \langle v \rangle^{\frac{\gamma}{2}+1} \nabla^2_x G_3 \|_{L^2_x(H^1_{v,*})}  \\
 &\lesssim \left(
 \| g_1 \|_{\XXX} \| \widetilde \nabla_x g_2 \|_{\YYY_1}   
 + \| g_1 \|_{\XXX} \| g_2 \|_{\YYY_2}  \right) \| G_3 \|_{\YYY_1}.
 \end{aligned}
 $$
 For $T_3$, we have:
 $$
 \begin{aligned}
 &T_3
 \lesssim 
 \bigg( 
 \| M^{\frac14} g_1 \|_{H^2_x L^2_v} \|  \nabla^3_x (\widetilde \nabla_x g_2) \|_{L^2_x(H^1_{v,*})} 
 + \| M^{\frac14} \nabla_x g_1 \|_{H^1_xL^2_v} \|  \nabla^2_x (\widetilde \nabla_x g_2) \|_{H^1_x(H^1_{v,*})} \\
 &\qquad \qquad
 + \| M^{\frac14} \nabla^2_x g_1 \|_{L^2_x L^2_v} \|  \nabla_x(\widetilde \nabla_x g_2) \|_{H^2_x(H^1_{v,*})} \\
&\qquad \qquad
+ \| M^{\frac14} \nabla^3_x g_1 \|_{L^2_x L^2_v} \|  \widetilde \nabla_x g_2 \|_{H^2_x(H^1_{v,*})} \bigg) 
 \|  \nabla^3_x G_3 \|_{L^2_x(H^1_{v,*})} \\
 &\quad
 + \bigg( 
 \| M^{\frac14} \nabla_x g_1 \|_{H^2_x L^2_v} \|  \nabla^3_x (\langle v \rangle^{\frac{\gamma}{2}+1}  g_2) \|_{L^2_x(H^1_{v,*})} 
 + \| M^{\frac14} \nabla^2_x g_1 \|_{H^1_x L^2_v} \|  \nabla^2_x (\langle v \rangle^{\frac{\gamma}{2}+1}  g_2) \|_{H^1_x(H^1_{v,*})} \\
 &\qquad \qquad
 + \| M^{\frac14} \nabla^3_x g_1 \|_{L^2_x L^2_v} \|  \nabla_x(\langle v \rangle^{\frac{\gamma}{2}+1}  g_2) \|_{H^2_x(H^1_{v,*})} \\
 &\qquad \qquad
 + \| M^{\frac14} \nabla^4_x g_1 \|_{L^2_x L^2_v} \|  \langle v \rangle^{\frac{\gamma}{2}+1}  g_2 \|_{H^2_x(H^1_{v,*})}
 \bigg) \|  \nabla^3_x G_3 \|_{L^2_x(H^1_{v,*})}  \\
 &\qquad \lesssim \left(
 \| g_1 \|_{\XXX} \| \widetilde \nabla_x g_2 \|_{\YYY_1}   
 + \| g_1 \|_{\XXX} \| g_2 \|_{\YYY_2}   
 + \| M^{\frac14} \nabla^4_x g_1 \|_{L^2_x L^2_v} \|  \langle v \rangle^{\frac{\gamma}{2}+1}  g_2 \|_{H^2_x(H^1_{v,*})}  \right) \| G_3 \|_{\YYY_1}.
 \end{aligned}
 $$
 We conclude the proof by gathering previous estimates and observing that 
 $$
 \| M^{\frac14} \nabla^4_x g_1 \|_{L^2_x L^2_v} \|  \la v \ra^{\frac{\gamma}{2}+1} g_2 \|_{H^2_x(H^1_{v,*})} 
 \lesssim \| \widetilde \nabla_x g_1 \|_{\XXX} \| g_2 \|_{\YYY_1}.
 $$
\end{proof}

\begin{prop}\label{prop:Gamma-NL-XXX}
For any smooth enough functions $g_1$ and $g_2$ there holds 
$$
\| \Gamma(g_1,g_2) \|_{\XXX} 
\lesssim \| g_1 \|_{\XXX} \| g_2 \|_{\YYY_2}  
+ \| g_1 \|_{\YYY_1} \| g_2 \|_{\XXX} .
$$
\end{prop}

\begin{proof}
We shall only prove that for any $\alpha \in \R$, there holds 
\begin{equation}\label{eq:Gamma-NL-XXX-1}
\| \langle v \rangle^{\alpha} \Gamma(g_1, g_2) \|_{L^2_v} 
\lesssim \| M^{\frac14} g_1 \|_{L^2_v} \| \langle v \rangle^{\alpha} g_2 \|_{H^2_{v,*}}
+ \| M^{\frac18} g_1 \|_{H^1_{v,*}} \| \langle v \rangle^{\alpha} g_2 \|_{L^2_v},
\end{equation}
from which we obtain the desired result by integrating in $x$ and arguing as in Step~4 of the proof of Proposition~\ref{prop:Gamma-NL}. 

Starting from the formulation \eqref{eq:Gamma} of $\Gamma(g_1,g_2)$, we perform the $\partial_{v_i}$ derivative in the first two terms, which gives
$$
\begin{aligned}
\Gamma(g_1,g_2)
& =   \left(a_{ij}* [\sqrt M g_1] \right) \left\{ \partial_{v_i} \partial_{v_j} g_2 - v_i \partial_{v_j} g_2 + \frac14 v_i v_j g_2   \right\}  
-  \left( \left[ \frac12 a_{ii} + c \right] * [\sqrt M g_1]\right) g_2 \\
&=: \widetilde \Gamma_1(g_1,g_2) + \widetilde \Gamma_2(g_1,g_2) + \widetilde \Gamma_3(g_1,g_2) + \widetilde \Gamma_4(g_1,g_2) ,
\end{aligned}
$$
and we estimate each term separately in the sequel.

The terms $ \widetilde \Gamma_2(g_1,g_2)$ and $\widetilde \Gamma_3(g_1,g_2)$ can be easily estimated thanks to Lemma~\ref{lem:afGH}. Indeed, we have 
$$
\begin{aligned}
&\| \langle v \rangle^{\alpha} \widetilde \Gamma_2(g_1 , g_2) \|_{L^2_v} 
= \| \langle v \rangle^{\alpha}\left(a_{ij}* [\sqrt M g_1] \right) v_i \partial_{v_j} g_2 \|_{L^2_v} \\
&\qquad \lesssim \| M^{\frac14} g_1 \|_{L^2_v} \| \langle v \rangle^{\alpha}  \langle v \rangle^{\frac{\gamma}{2}+1} \widetilde \nabla_v g_2 \|_{L^2_v}  
\lesssim \| M^{\frac14} g_1 \|_{L^2_v}  \| \langle v \rangle^{\alpha} g_2 \|_{H^2_{v,*}},
\end{aligned}
$$
and also
$$
\begin{aligned}
&\| \langle v \rangle^{\alpha} \widetilde \Gamma_3(g_1 , g_2) \|_{L^2_v} 
= \frac14 \| \langle v \rangle^{\alpha}\left(a_{ij}* [\sqrt M g_1] \right) v_i v_j g_2 \|_{L^2_v} \\
&\qquad
\lesssim \| M^{\frac14} g_1 \|_{L^2_v} \| \langle v \rangle^{\alpha}  \langle v \rangle^{\gamma+2} g_2 \|_{L^2_v}  
\lesssim \| M^{\frac14} g_1 \|_{L^2_v}  \| \langle v \rangle^{\alpha} g_2 \|_{H^2_{v,*}}.
\end{aligned}
$$
Moreover, we can easily estimate $ \widetilde \Gamma_4 (g_1,g_2)$ thanks to Lemma~\ref{lem:conv}. We have
$$
\begin{aligned}
\| \langle v \rangle^{\alpha} \left(a_{ii}* [\sqrt M g_1] \right)  g_2 \|_{L^2_v} 
&\lesssim \| M^{\frac14} g_1 \|_{L^2_v} \| \langle v \rangle^{\alpha}  \langle v \rangle^{\gamma + 2}  g_2 \|_{L^2_v}  \\
&\lesssim \| M^{\frac14} g_1 \|_{L^2_v}  \| \langle v \rangle^{\alpha} g_2 \|_{H^2_{v,*}},
\end{aligned}
$$
as well as, if $\gamma \ge 0$,
$$
\begin{aligned}
\| \langle v \rangle^{\alpha} \left( c * [\sqrt M g_1] \right)  g_2 \|_{L^2_v} 
&\lesssim \| M^{\frac14} g_1 \|_{L^2_v} \| \langle v \rangle^{\alpha}  \langle v \rangle^{\gamma}  g_2 \|_{L^2_v}  \\
&\lesssim \| M^{\frac14} g_1 \|_{L^2_v}  \| \langle v \rangle^{\alpha} g_2 \|_{H^2_{v,*}},
\end{aligned}
$$
and, if $-2 \le \gamma < 0$, since $\|\cdot\|_{L^4_v} \lesssim \|\cdot\|_{H^1_v}$, we have
$$
\begin{aligned}
\| \langle v \rangle^{\alpha} \left( c * [\sqrt M g_1] \right)  g_2 \|_{L^2_v} 
&\lesssim \| M^{\frac14} g_1 \|_{H^1_{v}} \| \langle v \rangle^{\alpha}  \langle v \rangle^{\gamma}  g_2 \|_{L^2_v}  \\
&\lesssim \| M^{\frac18} g_1 \|_{H^1_{v,*}}  \| \langle v \rangle^{\alpha} g_2 \|_{L^2_v}.
\end{aligned}
$$

\medskip

We now prove \eqref{eq:Gamma-NL-XXX-1} for $\widetilde \Gamma_1(g_1,g_2)$. When $|v| \leq 1$, the result is straightforward using Lemma~\ref{lem:conv}. Consider now $|v| \geq 1$. 
We first write 
$$
\partial_{v_j} g_2 =  \frac{v_j}{|v|} \frac{v_\ell}{|v|} \partial_{v_\ell} g_2  + (\mathrm{Id}-P_v)_j \nabla_v g_2
$$
and
$$
\partial_{v_i} \partial_{v_j} g_2 
=  \frac{v_i}{|v|}  \frac{v_k}{|v|}  \partial_{v_k} \partial_{v_j} g_2 + (\mathrm{Id}-P_v)_i \nabla_v (\partial_{v_j} g_2),
$$
thus we obtain
$$
\begin{aligned}
\partial_{v_i} \partial_{v_j} g_2 
&= \frac{v_i}{|v|}  \frac{v_k}{|v|} \partial_{v_k} \left( \frac{v_j}{|v|} \frac{v_\ell}{|v|} \partial_{v_\ell} g_2  \right)    
+\frac{v_i}{|v|}  \frac{v_k}{|v|} \partial_{v_k}  \left( (\mathrm{Id}-P_v)_j \nabla_v g_2 \right)   \\
&\quad
+ (\mathrm{Id}-P_v)_i \nabla_v \left( \frac{v_j}{|v|} \frac{v_\ell}{|v|} \partial_{v_\ell} g_2 \right)
+ (\mathrm{Id}-P_v)_i \nabla_v \left(  (\mathrm{Id}-P_v)_j \nabla_v g_2  \right).
\end{aligned}
$$
Thanks to a straightforward computation, we remark that 
$$
\begin{aligned}
&(\mathrm{Id}-P_v)_i \nabla_v \left( \frac{v_j}{|v|} \frac{v_\ell}{|v|} \partial_{v_\ell} g_2 \right)  \\
&\quad
= \frac{v_j}{|v|} \frac{v_\ell}{|v|} (\mathrm{Id}-P_v)_i \nabla_v (\partial_{v_\ell} g_2)
+ \partial_{v_k} \left(\frac{v_j}{|v|}  \frac{v_\ell}{|v|} \right) \left( 1 - \frac{v_i}{|v|} \frac{v_k}{|v|}\right) \partial_{v_\ell} g_2 .
\end{aligned}
$$
and 
$$
\begin{aligned}
\frac{v_i}{|v|}  \frac{v_k}{|v|} \partial_{v_k} \left( \frac{v_j}{|v|} \frac{v_\ell}{|v|} \partial_{v_\ell} g_2  \right) 
&= \frac{v_i}{|v|} \frac{v_j}{|v|} \left( \frac{v_k}{|v|} \frac{v_\ell}{|v|} \partial_{v_k} \partial_{v_\ell} g_2    \right).
\end{aligned}
$$
Therefore, we get 
$$
\begin{aligned}
&\langle v \rangle^{\alpha} \widetilde \Gamma_1(g_1 , g_2) (v) 
=  \langle v \rangle^{\alpha}\left(a_{ij}* [\sqrt M g_1] \right) \frac{v_i}{|v|} \frac{v_j}{|v|} \left( \frac{v_k}{|v|} \frac{v_\ell}{|v|} \partial_{v_k} \partial_{v_\ell} g_2    \right) \\
&\qquad \qquad
+ \langle v \rangle^{\alpha}\left(a_{ij}* [\sqrt M g_1] \right) \frac{v_i}{|v|} \frac{v_k}{|v|} \partial_{v_k}[ (\mathrm{Id}-P_v)_j \nabla_v g_2 ] \\
&\qquad \qquad
+ \langle v \rangle^{\alpha}\left(a_{ij}* [\sqrt M g_1] \right) \frac{v_j}{|v|} \frac{v_\ell}{|v|} (\mathrm{Id}-P_v)_i \nabla_v (\partial_{v_\ell} g_2) \\
&\qquad\qquad
+ \langle v \rangle^{\alpha}\left(a_{ij}* [\sqrt M g_1] \right) \partial_{v_k} \left(\frac{v_j}{|v|}  \frac{v_\ell}{|v|} \right) \left( 1 - \frac{v_i}{|v|} \frac{v_k}{|v|}\right) \partial_{v_\ell} g_2  \\
&\qquad \qquad
+ \langle v \rangle^{\alpha}\left(a_{ij}* [\sqrt M g_1] \right) (\mathrm{Id}-P_v)_i \nabla_v \left(  (\mathrm{Id}-P_v)_j \nabla_v g_2  \right) \\
&\qquad\qquad\qquad\quad\,\,=: I_1 + I_2 + I_3 + I_4 + I_5.
\end{aligned}
$$
We can now estimate each of these terms using Lemma~\ref{lem:conv} and Lemma~\ref{lem:afGH}. We obtain
$$
\| I_1 \|_{L^2_v} 
\lesssim \| M^{\frac14} g_1 \|_{L^2_v}  \| \langle v \rangle^{\alpha} \langle v \rangle^{\gamma} \nabla_v \nabla_v g_2 \|_{L^2_v} 
\lesssim \| M^{\frac14} g_1 \|_{L^2_v}  \| \langle v \rangle^{\alpha} g_2 \|_{H^2_{v,*}},
$$
$$
\| I_2 \|_{L^2_v} 
\lesssim \| M^{\frac14} g_1 \|_{L^2_v} \| \langle v \rangle^{\alpha} \langle v \rangle^{\gamma+1} \nabla_v [ (\mathrm{Id}-P_v)\nabla_v g_2] \|_{L^2_v} 
\lesssim \| M^{\frac14} g_1 \|_{L^2_v}  \| \langle v \rangle^{\alpha} g_2 \|_{H^2_{v,*}} ,
$$
$$
\| I_3 \|_{L^2_v} \lesssim \| M^{\frac14} g_1 \|_{L^2_v} \| \langle v \rangle^{\alpha} \langle v \rangle^{\gamma+1} (\mathrm{Id}-P_v)\nabla_v (\nabla_v g_2) \|_{L^2_v}
\lesssim \| M^{\frac14} g_1 \|_{L^2_v}  \| \langle v \rangle^{\alpha} g_2 \|_{H^2_{v,*}},
$$
$$
\| I_4 \|_{L^2_v} 
\lesssim \| M^{\frac14} g_1 \|_{L^2_v} \| \langle v \rangle^{\alpha} \langle v \rangle^{\gamma+1} \nabla_v g_2 \|_{L^2_v}
\lesssim \| M^{\frac14} g_1 \|_{L^2_v}  \| \langle v \rangle^{\alpha} g_2 \|_{H^2_{v,*}},
$$
and finally
$$
\| I_5 \|_{L^2_v} \lesssim \| M^{\frac14} g_1 \|_{L^2_v} \| \langle v \rangle^{\alpha} \langle v \rangle^{\gamma+2} (\mathrm{Id}-P_v)\nabla_v [ (\mathrm{Id}-P_v)\nabla_v g_2] \|_{L^2_v}
\lesssim \| M^{\frac14} g_1 \|_{L^2_v}  \| \langle v \rangle^{\alpha} g_2 \|_{H^2_{v,*}}.
$$

We conclude the proof of \eqref{eq:Gamma-NL-XXX-1} by gathering previous estimates.
\end{proof}

\color{black}

\subsection{Well-posedness for the Landau equation}

In this section, we shall prove the well-posedness part of Theorem~\ref{theo:main1}.

\begin{proof}[Proof of Theorem~\ref{theo:main1}-(i)]
Let $g^\eps$ be a solution to \eqref{eq:geps} associated to $g^\eps_{\mathrm{in}} \in \XXX$ satisfying~\eqref{eq:DI}. Notice that it implies that $g_{\rm in}^\eps \in (\operatorname{Ker} \Lambda_\eps)^\perp$ and thus $g^\eps(t) \in (\operatorname{Ker} \Lambda_\eps)^\perp$ for all~$t>0$ from the conservation laws~\eqref{eq:conserv}. 
We shall use the norm $\Nt \cdot \Nt_\XXX$ (and the associated inner product $\la\!\la \cdot , \cdot \ra\!\ra_\XXX$) defined in \eqref{eq:def:norme-triple-XXX} during the proof of Proposition~\ref{prop:hypoX}, in order to establish below an a priori estimate for $g^\eps$. 

Let $\sigma \in (0, \sigma_0)$ be fixed and compute
$$
\begin{aligned}
\frac12 \frac{\d}{\dt} \Nt g^\eps \Nt_\XXX^2
& = \la\!\la \Lambda_\eps g^\eps , g^\eps \ra\!\ra_\XXX + \frac{1}{\eps} \la\!\la \Gamma (g^\eps , g^\eps) , g^\eps  \ra\!\ra_\XXX .
\end{aligned}
$$
For the linear part, estimate \eqref{eq:hypo-Lambda} in Proposition~\ref{prop:hypoX} already gives us
$$
\la\!\la \Lambda_\eps g^\eps , g^\eps \ra\!\ra_\XXX
\le - \sigma \Nt g^\eps \Nt_\XXX^2
- \kappa \| g^\eps \|_{\YYY_1}^2
- \frac{\kappa}{\eps^2} \| (g^\eps)^\perp \|_{\YYY_1}^2.
$$
For the nonlinear part, from \eqref{eq:Gammaperp}, we get
$$
\begin{aligned}
\la\!\la \Gamma (g^\eps , g^\eps) , g^\eps \ra\!\ra_\XXX 
& 
= \delta \la \Gamma (g^\eps , g^\eps) , (g^\eps)^\perp \ra_\XXX 
+ \sum_{i=0}^2 \la \nabla_x^i \Gamma (g^\eps , g^\eps) , \nabla_x^i (g^\eps)^\perp \ra_{L^2_{x,v}} \\
&\quad + (1-\delta) \la \nabla_x^3 \Gamma (g^\eps , g^\eps) , \nabla_x^3 (g^\eps)^\perp \ra_{L^2_{x,v}}.
\end{aligned}
$$
Thanks to Proposition~\ref{prop:Gamma-NL}, we have
$$
\la \Gamma (g^\eps , g^\eps) , (g^\eps)^\perp \ra_\XXX \lesssim \| g^\eps \|_\XXX \| g^\eps \|_{\YYY_1} \| (g^\eps)^\perp \|_{\YYY_1}
$$
and also 
$$
\la \nabla_x^3 \Gamma (g^\eps , g^\eps) , \nabla_x^3 (g^\eps)^\perp \ra_{L^2_{x,v}} \lesssim \| g^\eps \|_\XXX \| g^\eps \|_{\YYY_1} \| (g^\eps)^\perp \|_{\YYY_1}.
$$
Arguing as in Step 4 of the proof of Proposition~\ref{prop:Gamma-NL}, we also obtain 
$$
\sum_{i=0}^2 \la \nabla_x^i \Gamma (g^\eps , g^\eps) , \nabla_x^i (g^\eps)^\perp \ra_{L^2_{x,v}} 
\lesssim \| g^\eps \|_\XXX \| g^\eps \|_{\YYY_1} \| (g^\eps)^\perp \|_{\YYY_1}.
$$
In summary, and recalling that $\| \cdot \|_\XXX$ and $\Nt \cdot \Nt_\XXX$ are equivalent, we then have 
\begin{equation}\label{eq:GammaXXX}
\frac{1}{\eps} 	\la\!\la \Gamma (g^\eps , g^\eps) , g^\eps  \ra\!\ra_\XXX 
\le \frac{C}{\eps} \Nt g^\eps \Nt_\XXX \| g^\eps \|_{\YYY_1} \| (g^\eps)^\perp \|_{\YYY_1}
\end{equation}
for some constant $C>0$.

Denoting $\mathrm{e}_\sigma : t \mapsto e^{\sigma t}$ and $g_\sigma^\eps = \mathrm{e}_\sigma g^\eps$, we therefore obtain 
$$
\frac12 \frac{\d}{\dt} \Nt g_\sigma^\eps \Nt_\XXX^2
\le 
- \kappa \| g_\sigma^\eps \|_{\YYY_1}^2
- \frac{\kappa}{\eps^2} \| (g_\sigma^\eps)^\perp \|_{\YYY_1}^2
+\frac{C}{\eps} \| g_\sigma^\eps \|_\XXX \| g_\sigma^\eps \|_{\YYY_1} \| (g_\sigma^\eps)^\perp \|_{\YYY_1}.
$$
Thanks to Young's inequality we write
$$
\frac{C}{\eps} \| g_\sigma^\eps \|_\XXX \| g_\sigma^\eps \|_{\YYY_1} \| (g_\sigma^\eps)^\perp \|_{\YYY_1} 
\le C \Nt g_\sigma^\eps \Nt_\XXX^2 \| g_\sigma^\eps \|_{\YYY_1}^2 + \frac{\kappa}{2\eps^2} \| (g_\sigma^\eps)^\perp \|_{\YYY_1}^2 ,
$$
which then gives the following a priori estimate
\begin{equation}\label{eq:apriori1}
\frac12 \frac{\d}{\dt} \Nt g_\sigma^\eps \Nt_\XXX^2
\le 
- \left( \kappa - C \Nt g_\sigma^\eps \Nt_\XXX^2\right) \| g_\sigma^\eps \|_{\YYY_1}^2
- \frac{\kappa}{2\eps^2} \| (g_\sigma^\eps)^\perp \|_{\YYY_1}^2.
\end{equation}

At least formally, from this differential inequality we easily obtain that if $\Nt g_{\mathrm{in}}^\eps \Nt_\XXX$ is small enough then $g^\eps$ satisfies the uniform in time estimate \eqref{eq:theo:main1:Linftybound}. The proof of existence and uniqueness of a solution $g^\eps$ to \eqref{eq:geps} satisfying \eqref{eq:theo:main1:Linftybound} for small data $\| g^\eps_{\mathrm{in}} \|_{\XXX} \le \eta_0$ follows a standard iterative scheme that uses estimate \eqref{eq:apriori1} (see for example~\cite{CTW}).
\end{proof}

\subsection{Regularity for the Landau equation} \label{subsec:decayregnonlinear}

In this part, we provide a result of regularization for the solutions to the nonlinear Landau equation which is quantified in time, namely the regularization estimate of Theorem~\ref{theo:main1}. Notice here that if we only wanted to handle the case $\eps=1$, we could have used the triple norm introduced in~\cite{GMM} (see~\cite{CTW} for the Landau equation) which is dissipative for the whole linearized operator and equivalent to the usual one. Here, to handle the $\eps$-dependencies, we have to use our hypocoercive norm defined in Proposition~\ref{prop:hypoX} and separate carefully the behaviors of microscopic and macroscopic parts of the solution. Some additional remainder terms coming from the fact that the transport operator and the projector $\pi$ onto the kernel of $L$ (see~\eqref{def:pi}) do not commute have to be treated. The computations are thus much more intricate. 

\begin{proof}[Proof of Theorem~\ref{theo:main1}-(ii)]
Let $g^\eps$ be a global solution to \eqref{eq:geps} associated to the initial data~$g^\eps_{\mathrm{in}} \in \XXX$ satisfying~\eqref{eq:DI}, with $\| g^\eps_{\mathrm{in}}\|_\XXX \le \eta_0$, provided by Theorem~\ref{theo:main1}-{(i)}. As in the proof of Theorem~\ref{theo:main1}-(i), we shall only obtain an a priori estimate implying the desired regularization estimates.

We recall that the spaces $\YYY_1$ and $\ZZZ_1^\eps$ are defined in~\eqref{def:normYYY1}-\eqref{def:normZZZ1eps} and we shall prove that for any $t \in (0, 1]$, one has
\begin{equation}\label{eq:regNL}
	\| g^\eps (t) \|_{\YYY_1} \lesssim \frac{1}{\sqrt t} \, \| g^\eps_{\mathrm{in}} \|_{\XXX}
	\quad \text{and} \quad 
	\| g^\eps (t) \|_{\ZZZ_1^\eps} \lesssim \frac{1}{t^{3/2}} \, \| g^\eps_{\mathrm{in}} \|_{\XXX},
\end{equation}
which readily implies that, for all $t >0$, there holds 
$$
	\| g^\eps (t) \|_{\YYY_1} \lesssim \frac{e^{-\sigma t}}{\min(1,\sqrt t)} \, \| g^\eps_{\mathrm{in}} \|_{\XXX}
	\quad \text{and} \quad 
	\| g^\eps (t) \|_{\ZZZ_1^\eps} \lesssim \frac{e^{-\sigma t}}{\min(1,t^{3/2})} \, \| g^\eps_{\mathrm{in}} \|_{\XXX}
	$$
by using the exponential decay in $\XXX$ given by Theorem~\ref{theo:main1}-(i) and hence concludes the proof of Theorem~\ref{theo:main1}-(ii). We split the proof of \eqref{eq:regNL} into several steps. We shall use without no further mention that since $\pi \in \BBB(\XXX,\YYY_i)$, we have $\| g^\eps \|_{\YYY_i}\lesssim \| g^\eps \|_{\XXX}+\| (g^\eps)^\perp \|_{\YYY_i}$ for $i=1,2$.

\medskip\noindent
\textit{Step 1.}
We consider the same functional $\UUU_\eps$ as defined in Step~1 of Proposition~\ref{prop:regLambda}, which we recall is given by
	\begin{multline} \label{eq:Lyapunov}
	\UUU_\eps(t,g^\eps) 
	= \Nt g^\eps \Nt_\XXX^2 
	+ \alpha_1 {t} \left(\| \widetilde \nabla_v (g^\eps)^{\perp} \|_{\XXX}^2 +  K\| \langle v \rangle^{\frac{\gamma}{2}+1} (g^\eps)^\perp\|^2_\XXX\right) \\
	+\eps \alpha_2 {t}^2 \la \widetilde \nabla_v g^\eps , \widetilde \nabla_x g^\eps \ra_{\XXX} 
	+ \eps^2 \alpha_3 {t}^3 \left(\| \widetilde \nabla_x g^\eps \|_{\XXX}^2 + K \| \langle v \rangle^{\frac{\gamma}{2}} \nabla_x g^\eps\|^2_\XXX \right),
	\end{multline}
with constants $K>0$ and $0 < \alpha_3 \ll \alpha_2 \ll \alpha_1 \ll 1$  so that $\alpha_2 \le \sqrt{\alpha_1 \alpha_3}$. We recall that the constant $K>0$ is chosen large enough in the proof of Proposition~\ref{prop:regLambda}. The constants~$\alpha_i$ will be chosen small enough during the proof here. We also recall that for any $t \in (0,1]$, one has the lower bounds 
\begin{equation}\label{eq:lowerbdd-bis}
t \| g^\eps \|_{\YYY_1}^2 \lesssim \UUU_\eps(t,g^\eps)
\quad\text{and}\quad
t^3 \| g^\eps \|_{\ZZZ_1^\eps}^2 \lesssim \UUU_\eps(t,g^\eps).
\end{equation}

\medskip\noindent
\textit{Step 2.} Thanks to the proof of Theorem~\ref{theo:main1}-(i) we already have
\begin{equation}\label{eq:dtgperpNL}
\begin{aligned}
\frac12\frac{\d}{\dt} \Nt g^\eps \Nt_\XXX^2
&= \la\!\la \Lambda_\eps g^\eps ,  g^\eps \ra \! \ra_\XXX+\frac{1}{\eps}\la\!\la \Gamma(g^\eps,g^\eps),g^\eps \ra \! \ra_\XXX\\
&\le - \left(\kappa - C \Nt g^\eps \Nt_{\XXX}^2 \right)  \| g^\eps \|_{\YYY_1}^2 - \frac{\kappa}{2\eps^2} \| (g^\eps)^\perp \|_{\YYY_1}^2 
\end{aligned}
\end{equation}
for some constants $\kappa , C >0$.

\medskip\noindent
\textit{Step 3.} We first observe that since $\pi \Gamma(g^\eps,g^\eps)=0$, $ (g^\eps)^\perp$ satisfies the equation
$$
\partial_t  (g^\eps)^\perp
= (\mathrm{Id} - \pi)\Lambda_\eps g^\eps 
+ \frac{1}{\eps} \Gamma( g^\eps , g^\eps).
$$
We then compute
\begin{align*}
&\frac12 \frac{\d}{\dt} \left\{ K \| \langle v \rangle^{\frac{\gamma}{2}+1} (g^\eps)^\perp \|_{\XXX}^2 + \| \widetilde \nabla_v (g^\eps)^\perp \|_{\XXX}^2  \right\}\\
&\quad 
= K\la \langle v \rangle^{\frac{\gamma}{2}+1} (\mathrm{Id}-\pi) \Lambda_\eps g^\eps, \langle v \rangle^{\frac{\gamma}{2}+1} (g^\eps)^\perp \ra_{\XXX}
+\la  \widetilde \nabla_v(\mathrm{Id}-\pi) \Lambda_\eps g^\eps, \widetilde\nabla_v (g^\eps)^\perp \ra_{\XXX}
\\
&\quad\quad 
+\frac{K}{\eps}\la \langle v \rangle^{\frac{\gamma}{2}+1} \Gamma(g^\eps , g^\eps), \langle v \rangle^{\frac{\gamma}{2}+1} (g^\eps)^\perp \ra_{\XXX}
+\frac{1}{\eps}\la \widetilde\nabla_v \Gamma(g^\eps , g^\eps), \widetilde\nabla_v (g^\eps)^\perp \ra_{\XXX}\\
&\quad =: I_1 + I_2+I_3+I_4.
\end{align*}
From \eqref{eq:dtm+nablavfperpXXX} in the proof of  Proposition~\ref{prop:regLambda}, we already know that
$$
I_1 + I_2 \le - \frac{\kappa_1}{\eps^2} \|  (g^\eps)^\perp \|_{\YYY_2}^2 + \frac{C}{\eps^2} \|  (g^\eps)^\perp \|_{\YYY_1}^2 + \frac{C}{\eps}\|  (g^\eps)^\perp \|_{\YYY_1} \| \widetilde \nabla_x g^\eps \|_{\XXX}
$$
for some constants $\kappa_1,C >0$.
Thanks to \eqref{eq:Gamma-NL} in Proposition~\ref{prop:Gamma-NL}, we have
$$
\begin{aligned}
I_3 
&\lesssim \frac{1}{\eps} \| g^\eps \|_{\XXX} \| \langle v \rangle^{\frac{\gamma}{2}+1} g^\eps \|_{\YYY_1} \| \langle v \rangle^{\frac{\gamma}{2}+1} (g^\eps)^\perp \|_{\YYY_1} 
\lesssim \frac{1}{\eps} \| g^\eps \|_{\XXX} \| g^\eps \|_{\YYY_2} \|  (g^\eps)^\perp \|_{\YYY_2}. 
\end{aligned}
$$
Moreover, thanks to \eqref{eq:DvBGamma-NL} in Proposition~\ref{prop:DvBGamma-NL}, we have
$$
\begin{aligned}
I_4
&\lesssim \frac{1}{\eps} \| g^\eps \|_{\XXX} \|g^\eps \|_{\YYY_2}  \|  \widetilde\nabla_v (g^\eps)^\perp \|_{\YYY_1} 
\lesssim \frac{1}{\eps} \| g^\eps \|_{\XXX} \|g^\eps \|_{\YYY_2}  \|  (g^\eps)^\perp \|_{\YYY_2} \\
&\qquad \qquad \qquad \quad\lesssim \frac{1}{\eps} \| g^\eps \|_{\XXX}^2  \|  (g^\eps)^\perp \|_{\YYY_2}
+ \frac{1}{\eps} \| g^\eps \|_{\XXX}  \|  (g^\eps)^\perp \|_{\YYY_2}^2.
\end{aligned}
$$

We therefore obtain, using that $\| h \|_{\XXX} \lesssim \| h \|_{\YYY_1}$,
\begin{equation}\label{eq:dtm+nablavgperpNL}
\begin{aligned}
&\frac12 \frac{\d}{\dt} \left\{ K \| \langle v \rangle^{\frac{\gamma}{2}+1} (g^\eps)^\perp \|_{\XXX}^2 + \| \widetilde \nabla_v (g^\eps)^\perp \|_{\XXX}^2  \right\}\\
&\quad 
\le  - \frac{\kappa_1}{\eps^2} \|  (g^\eps)^\perp \|_{\YYY_2}^2 + \frac{C}{\eps^2} \|  (g^\eps)^\perp \|_{\YYY_1}^2 + \frac{C}{\eps}\|  (g^\eps)^\perp \|_{\YYY_1} \| \widetilde \nabla_x  g^\eps \|_{\XXX} \\
&\quad\quad 
+\frac{C}{\eps} \| g^\eps \|_{\XXX} \| g^\eps \|_{\YYY_1}  \|  (g^\eps)^\perp \|_{\YYY_2}
+\frac{C}{\eps} \| g^\eps \|_{\XXX}  \|  (g^\eps)^\perp \|_{\YYY_2}^2.
\end{aligned}
\end{equation}

\medskip\noindent
\textit{Step 4.} We compute 
\begin{align*}
\frac{\d}{\dt} \la \widetilde \nabla_v g^\eps , \widetilde \nabla_x g^\eps \ra_{\XXX}
&= \la  \widetilde \nabla_v (\Lambda_\eps g^\eps) , \widetilde \nabla_x g^\eps  \ra_{\XXX}
+\la \widetilde \nabla_x (\Lambda_\eps g^\eps) , \widetilde \nabla_v g^\eps   \ra_{\XXX} 
\\
&\quad 
+ \frac{1}{\eps} \la   \widetilde \nabla_v \Gamma(g^\eps , g^\eps)  , \widetilde \nabla_x g^\eps \ra_{\XXX}
+\frac{1}{\eps} \la   \widetilde \nabla_x \Gamma(g^\eps , g^\eps) , \widetilde \nabla_v g^\eps \ra_{\XXX}\\
&=: J_1 + J_2 + J_3 + J_4.
\end{align*}
Thanks to \eqref{eq:dtnablavxfXXX} in Proposition~\ref{prop:regLambda}, we already have 
$$
J_1 + J_2
\le \frac{C}{\eps^2} \| (g^\eps)^\perp \|_{\YYY_2} \| \widetilde \nabla_x g^\eps \|_{\YYY_1} 
-\frac{1}{\eps} \| \widetilde \nabla_x g^\eps \|_{\XXX}^2
$$
fo some constant $C>0$. For the term $J_3$, estimate
\eqref{eq:DvBGamma-NL} in Proposition~\ref{prop:DvBGamma-NL} yields
$$
\begin{aligned}
J_3
&\lesssim \frac{1}{\eps} \| g^\eps \|_{\XXX} \| g^\eps \|_{\YYY_2}  \|  \widetilde\nabla_x {g^\eps} \|_{\YYY_1} \\
&\lesssim \frac{1}{\eps} \| g^\eps \|_{\XXX}^2 \|  \widetilde\nabla_x {g^\eps} \|_{\YYY_1}
+ \frac{1}{\eps} \| g^\eps \|_{\XXX} \| (g^\eps)^\perp \|_{\YYY_2}  \|  \widetilde\nabla_x {g^\eps} \|_{\YYY_1}.
\end{aligned}
$$
Moreover, for the term $J_4$, estimate~\eqref{eq:DxBGamma-NL} in Proposition~\ref{prop:DxBGamma-NL} gives us
$$
\begin{aligned}
J_4
&\lesssim \frac{1}{\eps} \left(  \| g^\eps \|_{\XXX} \| \widetilde \nabla_x g^\eps \|_{\YYY_1} 
+ \| \widetilde \nabla_x  g^\eps \|_{\XXX} \| g^\eps \|_{\YYY_1} 
+ \| g^\eps \|_{\XXX} \| g^\eps \|_{\YYY_2}  \right) \| \widetilde\nabla_v g^\eps \|_{\YYY_1} \\
&\lesssim \frac{1}{\eps} \left(  \| g^\eps \|_{\XXX} \| \widetilde \nabla_x g^\eps \|_{\YYY_1} 
+ \| \widetilde \nabla_x  g^\eps \|_{\XXX} \| g^\eps \|_{\YYY_1} 
+ \| g^\eps \|_{\XXX} \| g^\eps \|_{\YYY_2}  \right) \| g^\eps \|_{\YYY_2} \\
&\lesssim \frac{1}{\eps} \bigg( \| g^\eps \|_{\XXX}^2 \| \widetilde \nabla_x g^\eps \|_{\YYY_1} 
+  \| g^\eps \|_{\XXX} \| \widetilde \nabla_x g^\eps \|_{\YYY_1} \| (g^\eps)^\perp \|_{\YYY_2}
+  \| g^\eps \|_{\XXX} \| \widetilde \nabla_x g^\eps \|_{\XXX} \| g^\eps \|_{\YYY_1} \\
&\quad
+  \| g^\eps \|_{\XXX} \| \widetilde \nabla_x g^\eps \|_{\XXX} \| (g^\eps)^\perp \|_{\YYY_2} 
+ \| (g^\eps)^\perp \|_{\YYY_1} \| \widetilde \nabla_x g^\eps \|_{\XXX} \| (g^\eps)^\perp \|_{\YYY_2}  \\
&\quad
+  \| g^\eps \|_{\XXX}^3
+  \| g^\eps \|_{\XXX}^2 \| (g^\eps)^\perp \|_{\YYY_2}
+  \| g^\eps \|_{\XXX} \| (g^\eps)^\perp \|_{\YYY_2}^2 \bigg).
\end{aligned}
$$
Putting together previous estimates,
we get 
\begin{equation}\label{eq:dtnablaxvgperpNL}
\begin{aligned}
&\frac{\d}{\dt} \la \widetilde \nabla_v g^\eps , \widetilde \nabla_x g^\eps \ra_{\XXX}\\
&\quad 
\le  -\frac{1}{\eps} \| \widetilde \nabla_x g^\eps \|_{\XXX}^2
+ \frac{C}{\eps^2}  \| \widetilde \nabla_x g^\eps \|_{\YYY_1} \| (g^\eps)^\perp \|_{\YYY_2}
+ \frac{C}{\eps}  \| (g^\eps)^\perp \|_{\YYY_1} \| \widetilde \nabla_x g^\eps \|_{\XXX} \| (g^\eps)^\perp \|_{\YYY_2} \\
&\quad\quad
+\frac{C}{\eps} \| g^\eps \|_{\XXX} \bigg( \| \widetilde \nabla_x g^\eps \|_{\XXX} \| (g^\eps)^\perp \|_{\YYY_2} 
+ \| g^\eps \|_{\YYY_1}^2 + \| g^\eps \|_{\YYY_1} \|  \widetilde\nabla_x {g^\eps} \|_{\YYY_1} \\
&\quad\quad
+ \| g^\eps \|_{\YYY_1}\| (g^\eps)^\perp \|_{\YYY_2} 
+  \|  \widetilde\nabla_x {g^\eps} \|_{\YYY_1} \| (g^\eps)^\perp \|_{\YYY_2} 
+ \| (g^\eps)^\perp \|_{\YYY_2}^2 \bigg) .
\end{aligned}
\end{equation}
\medskip\noindent
\textit{Step 5.} 
We compute 
\begin{align*}
&\frac12\frac{\d}{\dt}\left\{ \|\widetilde\nabla_x g^\eps \|_{\XXX}^2 + K\| \langle v \rangle^{\frac{\gamma}{2}} \nabla_x g^\eps \|_{\XXX}^2  \right\}\\ 
&\quad 
=\la  \widetilde \nabla_x (\Lambda_\eps g^\eps), \widetilde\nabla_x {g^\eps} \ra_{\XXX}+K\la  \langle v \rangle^{\frac{\gamma}{2}} \nabla_x (\Lambda_\eps g^\eps) , \langle v \rangle^{\frac{\gamma}{2}}  \nabla_x g^\eps  \ra_{\XXX}\\
&\quad\quad
+ \frac{1}{\eps}\la \widetilde\nabla_x \Gamma(g^\eps , g^\eps), \widetilde\nabla_x {g^\eps} \ra_{\XXX}
+\frac{K}{\eps} \la  \langle v \rangle^{\frac{\gamma}{2}}  \nabla_x \Gamma(g^\eps , g^\eps) , \langle v \rangle^{\frac{\gamma}{2}}  \nabla_x g^\eps  \ra_{\XXX}\\
&\quad=: R_1 + R_2 + R_3 + R_4.
\end{align*}
From \eqref{eq:dtnablaxfXXX} in the proof of  Proposition~\ref{prop:regLambda}, we already know that
$$
R_1 + R_2 \le - \frac{\kappa_2}{\eps^2} \| \widetilde \nabla_x g^\eps \|_{\YYY_1}^2
+ \frac{C}{\eps^2} \| \widetilde \nabla_x g^\eps \|_{\XXX}^2
$$
for some constants $\kappa_2,C >0$.
Thanks to \eqref{eq:DxBGamma-NL} in Proposition~\ref{prop:DxBGamma-NL} for the term $R_3$ (and a slight adaptation of it for the term $R_4$), we obtain
$$
\begin{aligned}
R_3 +R_4
&\lesssim \frac{1}{\eps} \left(  \| g^\eps \|_{\XXX} \| \widetilde \nabla_x g^\eps \|_{\YYY_1} 
+ \| \widetilde \nabla_x  g^\eps \|_{\XXX} \| g^\eps \|_{\YYY_1} 
+ \| g^\eps \|_{\XXX} \| g^\eps \|_{\YYY_2}  \right) \| \widetilde \nabla_x g^\eps \|_{\YYY_1} \\
&\lesssim \frac{1}{\eps} \bigg(  \| g^\eps \|_{\XXX} \| \widetilde \nabla_x g^\eps \|_{\YYY_1} 
+ \| \widetilde \nabla_x  g^\eps \|_{\XXX} \| g^\eps \|_{\XXX} 
+ \| \widetilde \nabla_x  g^\eps \|_{\XXX} \| (g^\eps)^\perp \|_{\YYY_1} \\
&\qquad \qquad \qquad
+ \| g^\eps \|_{\XXX}^2
+ \| g^\eps \|_{\XXX} \| (g^\eps)^\perp \|_{\YYY_2}  \bigg) \| \widetilde \nabla_x g^\eps \|_{\YYY_1} .
\end{aligned}
$$

We therefore obtain
\begin{equation}\label{eq:dtnablaxgperpNL}
\begin{aligned}
&\frac12\frac{\d}{\dt}\left\{ \|\widetilde\nabla_x g^\eps \|_{\XXX}^2 + K\| \langle v \rangle^{\frac{\gamma}{2}} \nabla_x g^\eps \|_{\XXX}^2  \right\}\\
& 
\le  - \frac{\kappa_2}{\eps^2} \| \widetilde \nabla_x g^\eps \|_{\YYY_1}^2
+ \frac{C}{\eps^2} \| \widetilde \nabla_x g^\eps \|_{\XXX}^2 
+\frac{C}{\eps} \| (g^\eps)^\perp \|_{\YYY_1} \| \widetilde \nabla_x g^\eps \|_\XXX \| \widetilde \nabla_x g^\eps \|_{\YYY_1}\\
&\quad
+\frac{C}{\eps} \| g^\eps \|_\XXX \left( \| \widetilde \nabla_xg^\eps \|_{\XXX} \| \widetilde \nabla_x g^\eps \|_{\YYY_1}  
+ \| \widetilde \nabla_x g^\eps \|_{\YYY_1}^2 
+ \| \widetilde \nabla_x g^\eps \|_{\YYY_1} \| (g^\eps)^\perp \|_{\YYY_2} \right).
\end{aligned}
\end{equation}

\medskip\noindent
\textit{Step 6. Conclusion.} Gathering estimates \eqref{eq:dtgperpNL}--\eqref{eq:dtm+nablavgperpNL}--\eqref{eq:dtnablaxvgperpNL}--\eqref{eq:dtnablaxgperpNL} and using that $\underset{t \ge 0}{\sup}\hspace{0.1cm}\| g^\eps(t) \|_{\XXX} \lesssim \eta_0$ from Theorem~\ref{theo:main1}-(i), we then obtain
$$
\begin{aligned}
&\frac{\d}{\dt} \UUU_\eps (t,g^\eps)
 \le - \left(\kappa - C \eta_0^2  - C \alpha_2 t^2 \eta_0 \right)  \| g^\eps \|_{\YYY_1}^2  
- \frac{1}{\eps^2} \left(\kappa - C\alpha_1 t - C \alpha_1 \right)\| (g^\eps)^\perp \|_{\YYY_1}^2\\ 
&\quad
- \frac{t}{\eps^2}\left( \kappa_1 \alpha_1   - C \eps^2 \alpha_2 t \eta_0  - C\eps \alpha_1   \eta_0\right)\|  (g^\eps)^\perp \|_{\YYY_2}^2 
+ \frac{C\alpha_1 t }{\eps}\|  (g^\eps)^\perp \|_{\YYY_1} \| \widetilde \nabla_x  g^\eps \|_{\XXX} \\
&\quad 
+\frac{C\alpha_1 t }{\eps} \eta_0 \| g^\eps \|_{\YYY_1} \|  (g^\eps)^\perp \|_{\YYY_2}  
+ 2 \eps \alpha_2 t \la \widetilde \nabla_v g^\eps , \widetilde \nabla_x g^\eps \ra_{\XXX} \\
&\quad 
-t^2 \left( \alpha_2  - C \eps^2 \alpha_3  -C \alpha_3 t  \right)\| \widetilde \nabla_x g^\eps \|_{\XXX}^2 
+ \frac{C\alpha_2 t^2}{\eps}  \| \widetilde \nabla_x g^\eps \|_{\YYY_1} \| (g^\eps)^\perp \|_{\YYY_2} \\
&\quad
+ C\alpha_2 t^2 \| (g^\eps)^\perp \|_{\YYY_1} \| \widetilde \nabla_x g^\eps \|_{\XXX} \| (g^\eps)^\perp \|_{\YYY_2}\\
&\quad
+C \alpha_2 t^2 \eta_0 \bigg( \| \widetilde \nabla_x g^\eps \|_{\XXX} \| (g^\eps)^\perp \|_{\YYY_2} 
+ \| g^\eps \|_{\YYY_1} \|  \widetilde\nabla_x {g^\eps} \|_{\YYY_1} \\
&\qquad\qquad\qquad\qquad\qquad\qquad
+ \| g^\eps \|_{\YYY_1}\| (g^\eps)^\perp \|_{\YYY_2}
+  \|  \widetilde\nabla_x {g^\eps} \|_{\YYY_1} \| (g^\eps)^\perp \|_{\YYY_2}  \bigg)\\
&\quad
- t^3 \left( \kappa_2  \alpha_3   - C \eps \alpha_3 \eta_0  \right)\| \widetilde \nabla_x g^\eps \|_{\YYY_1}^2
+C\eps \alpha_3 t^3 \| (g^\eps)^\perp \|_{\YYY_1} \| \widetilde \nabla_x g^\eps \|_\XXX \| \widetilde \nabla_x g^\eps \|_{\YYY_1}\\
&\quad
+C\eps \alpha_3 t^3 \eta_0 \left( \|\widetilde \nabla_x g^\eps \|_\XXX \| \widetilde \nabla_x g^\eps \|_{\YYY_1}
+ \| g^\eps \|_{\YYY_1} \| \widetilde \nabla_x g^\eps \|_{\YYY_1}  
+ \| \widetilde \nabla_x g^\eps \|_{\YYY_1} \| (g^\eps)^\perp \|_{\YYY_2} \right).
\end{aligned}
$$
We now use Young's inequality to write
$$
\begin{aligned}
 \frac{C\alpha_1 t}{\eps}   \| \widetilde\nabla_xg^\eps \|_{\XXX} \|(g^\eps)^\perp \|_{\YYY_1} 
& \le \frac{\alpha_2}{4}{t}^2 \| \widetilde\nabla_xg^\eps \|_{\XXX}^2 + C \frac{\alpha_1^2}{\alpha_2} \frac{1}{\eps^2} \|(g^\eps)^\perp \|^2_{\YYY_1}
\\
\frac{C\alpha_1 t }{\eps} \eta_0 \| g^\eps \|_{\YYY_1} \|  (g^\eps)^\perp \|_{\YYY_2}
&\le \frac{\kappa_1 \alpha_1 t}{4 \eps^2} \| (g^\eps)^\perp \|_{\YYY_2}^2 + C \alpha_1 t \eta_0^2 \| g^\eps \|_{\YYY_1}^2
\\
\frac{C\alpha_2 t^2}{\eps} \| \widetilde\nabla_x g^\eps \|_{\YYY_1}\| (g^\eps)^\perp \|_{\YYY_2}
&\le \frac{\kappa_2 \alpha_3}{12}t^3 \| \widetilde\nabla_x g^\eps \|^2_{\YYY_1} +C\frac{\alpha_2^2}{\alpha_3} \frac{t}{\eps^2}\| (g^\eps)^\perp \|_{\YYY_2}^2
\\
C\eta_0 \alpha_2 {t}^2 
\| g^\eps \|_{\YYY_1} \|\widetilde\nabla_x g^\eps \|_{\YYY_1} 
&\le \frac{\kappa_2 \alpha_3}{12} {t}^3\|\widetilde\nabla_x g^\eps \|^2_{\YYY_1}+C\frac{\alpha_2^2}{\alpha_3} \eta_0^2  t \| g^\eps \|^2_{\YYY_1}
\\
C\eta_0 \alpha_2 {t}^2 
\| g^\eps \|_{\YYY_1} \| (g^\eps)^\perp \|_{\YYY_2} 
&\le \frac{\kappa_1 \alpha_1 t}{4 \eps^2} \| (g^\eps)^\perp \|_{\YYY_2}^2 + C \eps \eta_0^2 t^{3/2} \frac{\alpha_2^2}{\alpha_1} \| g^\eps \|_{\YYY_1}^2 
\\
C\eta_0 \alpha_2 {t}^2 
\| \widetilde \nabla_x g^\eps \|_{\YYY_1} \| (g^\eps)^\perp \|_{\YYY_2} 
&\le \frac{\kappa_2 \alpha_3 t^3}{12} \| \widetilde \nabla_x g^\eps \|^2_{\YYY_1} + C \eta_0^2 t \frac{\alpha_2^2}{\alpha_3}\| (g^\eps)^\perp \|_{\YYY_2}^2  \\
C\eps \eta_0 \alpha_3 {t}^3 
\| g^\eps \|_{\YYY_1} \|\widetilde\nabla_x g^\eps \|_{\YYY_1} 
&\le \frac{\kappa_2 \alpha_3}{12} {t}^3\|\widetilde\nabla_x g^\eps \|^2_{\YYY_1}+C\eps^2 \eta_0^2 \alpha_3 t^3 \| g^\eps \|^2_{\YYY_1}
\\
C\eps \eta_0 \alpha_3 {t}^3  
\| \widetilde \nabla_x g^\eps \|_{\YYY_1} \| (g^\eps)^\perp \|_{\YYY_2} 
&\le \frac{\kappa_2 \alpha_3 t^3}{12} \| \widetilde \nabla_x g^\eps \|^{2}_{\YYY_1} + C\eps^2\eta_0^2 \alpha_3 t^3 \| (g^\eps)^\perp \|_{\YYY_2}^2  \\
C \eps \alpha_3 \eta_0 t^3 \|\widetilde \nabla_x g^\eps\|_\XXX \|\widetilde \nabla_x g^\eps\|_{\YYY_1}
&
\le C \alpha_3 \eta_0 t^3 \|\widetilde \nabla_x g^\eps\|_\XXX^2 
+ \alpha_3 \eta_0 t^3 \|\widetilde \nabla_x g^\eps\|_{\YYY_1}^2 \\
2\eps\alpha_{2} t \left\vert\la  \widetilde{\nabla}_x g^\eps ,  \widetilde{\nabla}_v g^\eps\ra_\XXX\right\vert
&\le \frac{\alpha_2}{4}{t^2} \|  \widetilde \nabla_x g^\eps\|_{\XXX}^2 + C\alpha_2\eps^2 \| g^\eps\|_{\YYY_1}^2, 
\end{aligned}
$$
as well as 
$$
\begin{aligned}
C\alpha_2 {t}^2 \|\widetilde\nabla_x g^\eps \|_{\XXX}\| (g^\eps)^\perp \|_{\YYY_1}\| (g^\eps)^\perp \|_{\YYY_2}&\leq C\alpha_2 {t}^3\|\widetilde\nabla_x g^\eps \|^2_{\XXX}\| (g^\eps)^\perp \|^2_{\YYY_1}+C\alpha_2 {t}\| (g^\eps)^\perp \|^2_{\YYY_2}\\
C\eps\alpha_3 {t}^3 \|\widetilde\nabla_x g^\eps \|_{\XXX}\| (g^\eps)^\perp \|_{\YYY_1}\| \widetilde\nabla_x{g^\eps} \|_{\YYY_1}&\leq C\eps^2\alpha_3 {t}^3\|\widetilde\nabla_x g^\eps \|^2_{\XXX}\| (g^\eps)^\perp \|^2_{\YYY_1}+\frac{\kappa_2 \alpha_3}{12} {t}^3\| \widetilde\nabla_x{g^\eps} \|^2_{\YYY_1}.
\end{aligned}
$$
We thus obtain
$$
\begin{aligned}
\frac{\d}{\dt} \UUU_\eps (t,g^\eps)
& \le - \bigg(\kappa - C \eta_0^2  - C \alpha_2 t^2 \eta_0 - C \alpha_2 \eps^2 - C \alpha_1 t \eta_0^2 \\
&\qquad\qquad\qquad\qquad\qquad
- C\frac{\alpha_2^2}{\alpha_3} \eta_0^2  t 
- C \eps \eta_0^2 t^{3/2} \frac{\alpha_2^2}{\alpha_1} - C\eps^2 \eta_0^2 \alpha_3 t^3\bigg)  \| g^\eps \|_{\YYY_1}^2 \\
&\quad 
- \frac{1}{\eps^2} \left(\kappa - C\alpha_1 t - C \frac{\alpha_1^2}{\alpha_2} \right)\| (g^\eps)^\perp \|_{\YYY_1}^2\\ 
&\quad
- \frac{t}{\eps^2}\bigg( \frac{\kappa_1 \alpha_1}{2}  - C \eps^2 \alpha_2 t \eta_0  - C\eps \alpha_1   \eta_0 - C \frac{\alpha_2^2 }{\alpha_3} \\
&\qquad\qquad\qquad\qquad\qquad\qquad
- C \eta_0^2 \frac{\alpha_2^2}{\alpha_3} 
- C\eps^4 \eta_0^2 \alpha_3 t^2 - C \eps^2 \alpha_2 \bigg)\|  (g^\eps)^\perp \|_{\YYY_2}^2 \\
&\quad 
-t^2 \left( \frac{\alpha_2}{2} - C \eps^2 \alpha_3 -C \alpha_3 t  - C \alpha_3 \eta_0 t \right)\| \widetilde \nabla_x g^\eps \|_{\XXX}^2 \\
&\quad
- t^3\left( \frac{\kappa_2  \alpha_3}{2}  - C \eps \alpha_3 \eta_0  - C \alpha_3 \eta_0 \right)\| \widetilde \nabla_x g^\eps \|_{\YYY_1}^2 \\
&\quad 
+Ct^3 \left( \alpha_2 +\eps^2\alpha_3  \right) \| (g^\eps)^\perp \|^2_{\YYY_1} \|\widetilde\nabla_x g^\eps \|^2_{\XXX} .
\end{aligned}
$$
We now choose $\alpha_{1}=\eta$, $\alpha_{2}=\eta^{\frac{3}{2}}$, and $\alpha_{3}=\eta^{\frac{5}{3}}$, with $\eta \in (0,1)$ small enough as well as $\eta_0$ small enough, so that we deduce, for all $t \in [0,1]$,
$$
\begin{aligned} 
\frac{\d}{\dt}\UUU_\eps(t,g^\eps) 
&\le
-\frac{\kappa'}{\eps^2}  \| (g^\eps)^\perp \|_{\YYY_1}^2 -\kappa' \| {g^\eps} \|_{\YYY_1}^2- \kappa' \frac{t}{\eps^2}\| (g^\eps)^\perp \|_{\YYY_2}^2- \kappa' t^2\| \widetilde{\nabla}_x{g^\eps} \|_{\XXX}^2- \kappa' t^3\| \widetilde{\nabla}_x{g^\eps} \|_{\YYY_1}^2\\
 &+C{t}^3\|\widetilde\nabla_x g^\eps \|^2_{\XXX}\| (g^\eps)^\perp \|^2_{\YYY_1},
\end{aligned}
$$
for some constants $ \kappa', C >0$. Integrating in time the last inequality, we hence obtain that for any $t \in [0, 1]$, there holds
$$
\begin{aligned}
\UUU_\eps(t,g^\eps)  
&+ \frac{ \kappa'}{\eps^2} \int_0^t \| (g^\eps)^\perp \|_{\YYY_1}^2 \, \d s
+ \kappa'\int_0^t  \| {g^\eps} \|_{\YYY_1}^2 \, \d s \\
& 
+ \frac{ \kappa'}{\eps^2} \int_0^t s  \|(g^\eps)^\perp \|_{\YYY_2}^2 \, \d s
+  \kappa' \int_0^t s^2  \|\widetilde{\nabla}_x{g^\eps} \|_{\XXX}^2\, \d s  
+  \kappa' \int_0^t s^3  \|\widetilde{\nabla}_x{g^\eps} \|_{\YYY_1}^2\, \d s\\
&\le \UUU_\eps(0) + {C}\int_0^t s^3 \|\widetilde\nabla_x g^\eps \|^2_{\XXX}\| (g^\eps)^\perp \|^2_{\YYY_1} \, \d s
\\
& 
\le\UUU_\eps(0) 
+ C \eps^2 \left(\sup_{t \in [0,1]} t^3 \|\widetilde\nabla_x g^\eps \|^2_{\XXX} \right) \, \frac{1}{\eps^2}\int_0^t
\| (g^\eps)^\perp \|^2_{\YYY_1} \, \d s
\\
& 
\le \UUU_\eps(0) 
+ C \eta_0^2 \eps^2 \left(\sup_{t \in [0,1]} t^3 \|\widetilde\nabla_x g^\eps \|^2_{\XXX} \right)
\end{aligned}
$$
where we have used that $\frac{1}{\eps^2}\int_0^t
\| (g^\eps)^\perp \|^2_{\YYY_1} \, \d s \lesssim \eta_0^2$ from Theorem~\ref{theo:main1}-(i) and we have denoted~$\UUU_\eps(0) = \UUU_\eps(0,g_{\rm in}^\eps)$.
Since 
$$
\left(\sup_{t \in [0,1]}  \| g^\eps(t) \|_{\XXX}^2  \right)
+ \left(\sup_{t \in [0,1]}  t\|  (g^\eps(t))^{\perp} \|_{\YYY_1}^2 \right)+
\eps^2 \left(\sup_{t \in [0,1]} t^3 \| \widetilde \nabla_x g^\eps(t)\|_{\XXX}^2 \right) \lesssim \sup_{t \in [0,1]} \UUU_\eps(t,g^\eps)
$$
if $\eta_0>0$ is small enough (independently of $\eps$), we finally obtain
$$
{t}\| g^\eps (t) \|^2_{\YYY_1} \lesssim \UUU_\eps(0) 
= \Nt g^\eps_{\mathrm{in}} \Nt^2_\XXX
\lesssim  \| g^\eps_{\mathrm{in}} \|_{\XXX}^2 \quad \forall \, t \in [0,1]
$$
and
$$
{t}^3 \| g^\eps (t) \|^2_{\ZZZ_1^\eps} \lesssim \UUU_\eps(0) 
= \Nt g^\eps_{\mathrm{in}} \Nt^2_\XXX
\lesssim  \| g^\eps_{\mathrm{in}} \|_{\XXX}^2 \quad \forall \, t \in [0,1],
$$
which gives estimates \eqref{eq:regNL} and concludes the proof.
\end{proof}

\black

\section{Refined semigroup estimates on the linearized operator} \label{sec:linear2}

In this part, we go back to the linearized problem and give some new and refined estimates on it. 
The first and second subsections are dedicated to the introduction of a nice splitting of the linearized operator~$\Lambda_\eps$ coming from~\cite{Carrapatoso,CTW} and to the proof of dissipativity, regularization and boundedness estimates. Roughly speaking, the linearized operator~$\Lambda_\eps$ splits into two parts: $\Lambda_\eps=\AA_\eps+\BB_\eps$, the first part $\AA_\eps$ having some good regularizing properties, the second one $\BB_\eps$ having some nice dissipativity (and also regularizing) properties. It is worth mentioning that the regularization estimates on $\BB_\eps$ that we develop are sharp (it was not the case in~\cite{CTW} where the authors did not intend to obtain optimal regularization estimates on $\BB_\eps$) and only the case $\eps=1$ was treated in~\cite{CTW}. 
As already mentioned, from those properties and Duhamel formula, we can give a new proof of Theorem~\ref{theo:mainlinear}. More importantly, Duhamel formula applied with this splitting will be used to prove our hydrodynamical limit theorem in Section~\ref{sec:hydro}.

In the last three subsections, we study the semigroup $U^\eps(t)$ associated with $\Lambda_\eps$ from another point of view which is based on a careful spectral analysis carried out in Fourier in~$x$ of $\Lambda_1$ in~\cite{Yang-Yu}, it in particular allows us to give a decomposition of the semigroup~$U^\eps(t)$, study its limit as $\eps$ goes to $0$ and give another type of estimates on it.

\subsection{Splitting of the operator}

We now introduce a splitting of the full linearized operator $\Lambda_\eps$. Let $\chi \in \mathcal{C}^\infty_c (\R)$ be a smooth cutoff function such that $0 \le \chi \le 1$, $\chi \equiv 1$ on $[-1/2,1/2]$ and $\chi \equiv 0$ on $\R \setminus [-1,1]$, consider positive constants $R , \bar R >0$ and define $\chi_{\bar R}(v) = \chi ( |v| / \bar R)$ as well as 
\begin{equation} \label{def:m}
m^2(v)
:=  \frac14 |\mathbf{B}(v) v |^2 
-\frac12 \nabla_{v}\cdot\left[ \mathbf{B}^{\top}(v)\mathbf{B}(v) v \right] + R \chi_{\bar R}(v)
\end{equation}
where ${\bf B}(v)$ is defined in~\eqref{def:B}. 
Recalling the formulation of $L$ in \eqref{eq:Lbis}, we then decompose~$\Lambda_\eps = \frac{1}{\eps^2} L - \frac{1}{\eps} v \cdot \nabla_x$ as
\begin{equation}\label{eq:Lambdaeps-splitting}
\Lambda_\eps = \AA_\eps + \BB_\eps
\quad\text{with}\quad
\AA_\eps := \frac{1}{\eps^2} \AA
\quad\text{and}\quad
\BB_\eps := \frac{1}{\eps^2} \BB - \frac{1}{\eps} v \cdot \nabla_x
\end{equation}
where 
\begin{equation}\label{eq:Aeps}
\begin{aligned} 
\AA f&:= - \Big\{ \! \left( a_{ij} * \sqrt M f \right) v_i v_j 
- \left( a_{ii} * \sqrt M f \right)  
+ \left( c * \sqrt M f \right) \! \Big\} \sqrt M +R \chi_{\bar R}f \\
&=(L_2 + R \chi_{\bar R})f
\end{aligned}
\end{equation}
and
\begin{equation}\label{eq:Beps}
\BB f := -  \widetilde \nabla_v^* \widetilde \nabla_v f - m^2(v) f  
=  (L_1 - R \chi_{\bar R})f .
\end{equation}

Let us now give a lemma providing estimates on $m$ and its derivatives. We fix 
\begin{equation} \label{def:sigma1}
\sigma_1 := 
\begin{cases}
+ \infty \quad&\text{if}\quad -2 < \gamma \le 1 \\
 \frac12 \quad&\text{if}\quad \,\, \gamma = - 2.
\end{cases}
\end{equation}

\begin{lem}\label{lem:malpha}
Let $0 < \sigma < \sigma_1$.

\medskip\noindent
(i) There are $R_0$ and $\bar R_0$ large enough so that for any $R \ge R_0$ and $\bar R \ge \bar R_0$, one has
			$$
			m^2(v) 
			\ge \sigma + \kappa{\langle v \rangle}^{\gamma+2} , \quad \forall \, v\in \R^{3},
			$$
		for some $0 < \kappa<  \sigma_1 - \sigma$.

\medskip\noindent
(ii) For any $\alpha \in \R$, we define 
			$$
			m_{\alpha}^2(v) 
			:=  m^2(v) - \langle v \rangle^{-2\alpha}|\widetilde \nabla_v \langle v \rangle^\alpha |^2.
			$$
		There are $R_0$ and $\bar R_0$ large enough so that for any $R \ge R_0$ and $\bar R \ge \bar R_0$, one has
			$$
			m_{\alpha}^2 (v) \ge \sigma + \kappa {\langle v \rangle}^{\gamma+2} ,  \quad \forall \, v\in \R^{3},
			$$ 
		for some $0 < \kappa<  \sigma_1 - \sigma$.
	
\medskip\noindent
(iii) For any multi-index $ \alpha \in \N^{3}$ we have, for all $v\in \R^{3}$,
		$$
		|\partial_v^{\alpha} m(v)| \lesssim \langle v \rangle^{\frac{\gamma}{2}+1 - |\alpha|}.
        $$
\end{lem}

Hereafter in the paper, for any $0 < \sigma < \sigma_1$ we then fix $R , \bar R >0$ large enough so that the results of Lemma~\ref{lem:malpha} are in force.

\begin{proof}[Proof of Lemma~\ref{lem:malpha}]
(i) See \cite[Lemma 2.6]{CTW}.

\medskip\noindent
(ii) We have 
	$$
	m_{\alpha}^2(v) 
	=  \frac14 |\mathbf{B}(v) v |^2 
-\frac12 \nabla_{v}\cdot\left[ \mathbf{B}^{\top}(v)\mathbf{B}(v) v \right]    
	- \alpha^2 |\mathbf B(v) v|^2 \langle v \rangle^{-4} + R \chi_{\bar R}(v)  
	$$
therefore
	$$
	m_{\alpha}^2(v) 
	=  \frac14 \ell_1(v)|v|^2  - \frac12 \nabla_{v}\cdot\Big[ \ell_1(v) v \Big]  
	- \alpha^2 \ell_1(v)|v|^2 \langle v \rangle^{-4} + R \chi_{\bar R}(v) .
	$$
Now observe that 
	\begin{multline*}
		\frac14 \ell_1(v)|v|^2  - \frac12 \nabla_{v}\cdot\Big[ \ell_1(v) v \Big]  - \alpha^2 \ell_1(v)|v|^2 \langle v \rangle^{-4} \\
		= \frac14 \ell_1(v)|v|^2  - \frac12 \nabla_{v} \ell_1(v) \cdot v - \frac32 \ell_1(v)
		- \alpha^2 \ell_1(v)|v|^2 \langle v \rangle^{-4} ,
	\end{multline*}
and from~\eqref{eq:vp}, $\ell_1(v)|v|^2 \langle v \rangle^{-4} \lesssim \langle v \rangle^{\gamma-2}$,
hence if $-2 < \gamma \le 1$ we have 
	$$
	\frac14 \ell_1(v)|v|^2  - \frac12 \nabla_{v}\cdot\Big[ \ell_1(v) v \Big]  
	- \alpha^2 \ell_1(v)|v|^2 \langle v \rangle^{-4} \xrightarrow[|v| \to \infty]{} + \infty,
	$$
and if $\gamma = -2$, one has 
	$$
	\frac14 \ell_1(v)|v|^2  - \frac12 \nabla_{v}\cdot\Big[ \ell_1(v) v \Big]  
	- \alpha^2 \ell_1(v)|v|^2 \langle v \rangle^{-4} \xrightarrow[|v| \to \infty]{} \frac12 >0.
	$$
We then conclude the proof by taking $R_0, \bar R_0 >0$ large enough.

\medskip\noindent
(iii) Direct consequence of \eqref{eq:nablaBij}.
\end{proof}

Arguing as in Lemma~\ref{lem:commutatorL} and recalling that $\big[\widetilde \nabla_{v_i},v \cdot \nabla_x\big] = \widetilde \nabla_{x_i}$ and $\big[\widetilde \nabla_{x_i},v \cdot \nabla_x\big] = 0$, we also obtain
\begin{lem}\label{lem:commutatorBeps}
There holds 

\medskip\noindent
(i) For any $1 \le i \le 3$, one has
$$
\begin{aligned}{}
[\widetilde \nabla_{v_i} , \BB_\eps] f
&= - \frac{1}{\eps^2} \widetilde \nabla_{v_j}^* [\widetilde \nabla_{v_i},  \widetilde \nabla_{v_j}] f
- \frac{1}{\eps^2} \widetilde \nabla_{v_j} [\widetilde \nabla_{v_i},  \widetilde \nabla_{v_j}^*]  f  \\
&\quad
- \frac{1}{\eps^2} (\widetilde \nabla_{v_i} m^2) f 
- \frac{1}{\eps} \widetilde \nabla_{x_i} f 
- \frac{1}{\eps^2} \left[ [\widetilde \nabla_{v_i},  \widetilde \nabla_{v_j}^*], \widetilde \nabla_{v_j} \right] f.
\end{aligned}
$$

\medskip\noindent
(ii) For any $1 \le i \le 3$, one has
$$
\begin{aligned}{}
[\widetilde \nabla_{x_i} , \BB_\eps] f
&= - \frac{1}{\eps^2} \widetilde \nabla_{v_j}^* [\widetilde \nabla_{x_i},  \widetilde \nabla_{v_j}] f
- \frac{1}{\eps^2} \widetilde \nabla_{v_j} [\widetilde \nabla_{x_i},  \widetilde \nabla_{v_j}^*]  f
- \frac{1}{\eps^2} \left[ [\widetilde \nabla_{x_i},  \widetilde \nabla_{v_j}^*], \widetilde \nabla_{v_j} \right] f .
\end{aligned}
$$

\end{lem}

Let us recall that the spaces $\XXX$, $\YYY_i$, $(\YYY_i)'$, $\ZZZ_i^\eps$ and $(\ZZZ_i^\eps)'$ are respectively defined in~\eqref{def:normXXX}, \eqref{def:normYYY1}, \eqref{def:normYYY2}, \eqref{def:YYYi'}, \eqref{def:normZZZ1eps}, \eqref{def:normZZZ2eps} and~\eqref{def:ZZZieps'}. 

We also have the following bounds on the operator $\AA$:
\begin{lem} \label{lem:Areg}
For any $\alpha \in \R$ one has
\begin{align*}
\| \langle v \rangle^\alpha \AA f \|_{L^2_{x,v}} 
&\lesssim  \| M^{\frac14} f \|_{L^2_{x,v}}\\
\| \langle v \rangle^\alpha \widetilde \nabla_v \AA f \|_{L^2_{x,v}} 
&\lesssim \| M^{\frac14} f \|_{L^2_{x,v}}+\| M^{\frac14} \widetilde \nabla_vf \|_{L^2_{x,v}}\\
\| \langle v \rangle^\alpha \widetilde \nabla_v \widetilde \nabla_v \AA f \|_{L^2_{x,v}}
& 
\lesssim \| M^{\frac14} f \|_{L^2_{x,v}}+\| M^{\frac14} \widetilde \nabla_vf \|_{L^2_{x,v}} + \| M^{\frac14} \widetilde \nabla_v \widetilde \nabla_v  f \|_{L^2_{x,v}}\\
\| \langle v \rangle^\alpha \widetilde \nabla_x \AA f \|_{L^2_{x,v}} 
&\lesssim  \| M^{\frac14} \widetilde \nabla_x f \|_{L^2_{x,v}} \\
\| \langle v \rangle^\alpha \widetilde \nabla_x \widetilde \nabla_x \AA f \|_{L^2_{x,v}} 
&
\lesssim  \| M^{\frac14} \widetilde \nabla_x \widetilde \nabla_x f \|_{L^2_{x,v}}. 
\end{align*}
In particular, one has $\AA \in \BBB(\XXX)$, $\AA \in \BBB(\YYY_i)$ and $\AA \in \BBB(\ZZZ_i^\eps)$ for $i=1,2$ and since $\AA$ is self-adjoint in $L^2_{x,v}$, we also have $\AA \in \BBB(\YYY_i')$ and $\AA \in \BBB((\ZZZ_i^\eps)')$ for $i=1,2$. 
\end{lem}
\begin{proof}
The terms coming from $R \chi_{\bar R}$ are easily treated because $\chi_{\bar R}$ has compact support. The first, second and fourth estimates thus directly come from Lemma~\ref{lem:boundL2}. The proofs of the other estimates are completely similar and rely on Lemmas~\ref{lem:conv} and~\ref{lem:boundL20}.
\end{proof}

\subsection{Decay and regularization estimates for $S_{\BB_\eps}$}  \label{subsec:propB}
In this section we provide several results on the dissipatition and regularization properties of the operator $\BB_\eps$. 
We start with the dissipative ones. 

\begin{lem}\label{lem:dissipativeBeps}
Let $\sigma \in (0, \sigma_1)$ (where $\sigma_1$ is defined in~\eqref{def:sigma1}). Then for any $\alpha \in \R$, one has
\begin{equation}\label{eq:Beps-dissipative}
\begin{aligned}
\la \langle v \rangle^{\alpha} f ,  \langle v \rangle^{\alpha}\BB_\eps f \ra_{L^2_{x,v}}
&\le -  \frac{\sigma}{\eps^2}\| \langle v \rangle^{\alpha} f \|_{L^2_{x,v}}^2 
- \frac{\kappa}{\eps^2} \| \langle v \rangle^{\alpha} f \|_{L^2_{x}(H^1_{v,*})}^2,
\end{aligned}
\end{equation}
for some constant $\kappa >0$.
As a consequence, one has for all $t \ge 0$,
$$
\| S_{\BB_\eps} (t) \|_{\XXX \to \XXX} \le e^{-\sigma t/\eps^2} .
$$
\end{lem}

\begin{proof}
We compute: 
$$
\begin{aligned}
\la \langle v \rangle^{2\alpha} f , \BB_\eps f \ra_{L^2_{x,v}}
&= \la \langle v \rangle^{2\alpha} f , \left\{ - \frac{1}{\eps^2} \widetilde \nabla_v^* \widetilde \nabla_v f -  \frac{1}{\eps^2} m^2 f - \frac{1}{\eps} v \cdot \nabla_x f   \right\} \ra_{L^2_{x,v}} \\
&= -  \frac{1}{\eps^2} \| \langle v \rangle^\alpha m f \|_{L^2_{x,v}}^2
- \frac{1}{\eps^2} \la \widetilde \nabla_v( \langle v \rangle^{2\alpha} f) , \widetilde \nabla_v f \ra_{L^2_{x,v}}.
\end{aligned}
$$
Observing that 
$$
\begin{aligned}
&\la \widetilde \nabla_v( \langle v \rangle^{2\alpha} f) , \widetilde \nabla_v f \ra_{L^2_{x,v}}
= \la \widetilde \nabla_v( \langle v \rangle^\alpha f) , \langle v \rangle^\alpha \widetilde \nabla_v f \ra_{L^2_{x,v}}
+ \la (\widetilde \nabla_v \langle v \rangle^\alpha) f , \langle v \rangle^\alpha \widetilde \nabla_v f \ra_{L^2_{x,v}} \\
&\qquad  
= \| \widetilde \nabla_v( \langle v \rangle^\alpha f) \|_{L^2_{x,v}}^2
- \la \widetilde \nabla_v( \langle v \rangle^\alpha f) , (\widetilde \nabla_v \langle v \rangle^\alpha) f \ra_{L^2_{x,v}}
+ \la (\widetilde \nabla_v \langle v \rangle^\alpha) f , \langle v \rangle^\alpha \widetilde \nabla_v f \ra_{L^2_{x,v}} \\
&\qquad
= \| \widetilde \nabla_v( \langle v \rangle^\alpha f) \|_{L^2_{x,v}}^2
- \| (\widetilde \nabla_v \langle v \rangle^\alpha) f \|_{L^2_{x,v}}^2
\end{aligned}
$$
and recalling that $m_{\alpha}$ is defined in Lemma~\ref{lem:malpha} so that $\langle v \rangle^{2\alpha} m_{\alpha}^2 = \langle v \rangle^{2\alpha} m^2 - |\widetilde \nabla_v \langle v \rangle^\alpha|^2$,
we obtain the estimate 
$$
\begin{aligned}
\la \langle v \rangle^{2\alpha} f , \BB_\eps f \ra_{L^2_{x,v}}
&= -  \frac{1}{\eps^2}\| \langle v \rangle^{\alpha} m_\alpha f \|_{L^2_{x,v}}^2 
- \frac{1}{\eps^2}\|  \widetilde \nabla_v ( \langle v \rangle^{\alpha} f) \|_{L^2_{x,v}}^2 .
\end{aligned}
$$
We then conclude the proof of \eqref{eq:Beps-dissipative} by using the bound by below on $m_\alpha$ of Lemma~\ref{lem:malpha}.

\smallskip

The exponential decay estimate of $S_{\BB_\eps}$ on $\XXX$ is then a consequence of \eqref{eq:Beps-dissipative} together with the fact that $\nabla_x$ commutes with $\BB_\eps$.
\end{proof}

In what follows, we prove regularization results for the semigroup $S_{\BB_\eps}$. 
\begin{lem}\label{lem:regSBeps}
Let $\sigma \in (0, \sigma_1)$ (where $\sigma_1$ is defined in~\eqref{def:sigma1}). For any $t >0$, one has
\begin{equation} \label{eq:regSBeps0}
\|S_{\BB_\eps}(t)\|_{\XXX \to \YYY_1} \lesssim \frac{\eps}{\sqrt{t}} \, e^{-\sigma t/\eps^2} 
\quad \text{and} \quad 
\|S_{\BB_\eps}(t)\|_{\XXX \to \ZZZ_1^\eps} \lesssim \frac{\eps^3}{t^{3/2}} \, e^{-\sigma t/\eps^2}. 
\end{equation}
as well as the dual estimates
\begin{equation} \label{eq:regSBeps0bis}
\|S_{\BB_\eps}(t)\|_{\YYY_1' \to \XXX} \lesssim \frac{\eps}{\sqrt{t}} \, e^{-\sigma t/\eps^2} 
\quad \text{and} \quad 
\|S_{\BB_\eps}(t)\|_{(\ZZZ_1^\eps)' \to \XXX} \lesssim \frac{\eps^3}{t^{3/2}} \, e^{-\sigma t/\eps^2}. 
\end{equation}
\end{lem}

The proof follows similar ideas as the proof of Proposition~\ref{prop:regLambda} but is somewhat simpler because, the operator $\BB_\eps$ provides $1/\eps^2$-dissipativity and regularity on both macroscopic and microscopic parts of the solution whereas the operator $\Lambda_\eps$ only provided a gain of~$1/\eps^2$ on the microscopic part. We thus do not need to separate microscopic and macroscopic parts when defining our functional that will be a Lyapunov functional for our equation $\partial_t f = \BB_\eps f$ (see~\eqref{eq:Lyapunov2}). For sake of completeness, the proof is given in Appendix~\ref{app:Beps}.

Using the same method, we can push our previous result up to the next notch of regularity (we only mention the dual results because they are the only ones that will be used in the sequel) and the proof is also postponed to Appendix~\ref{app:Beps}:
\begin{lem}  \label{lem:regSBeps2}
Let $\sigma \in (0,\sigma_1)$ (where $\sigma_1$ is defined in~\eqref{def:sigma1}). For any $t > 0$, one has
	$$
	\|S_{\BB_\eps}(t)\|_{\YYY_2' \to \XXX} \lesssim \frac{\eps^2}{t} \, e^{-\sigma t/\eps^2} 
	\quad \text{and} \quad 
	\|S_{\BB_\eps}(t)\|_{(\ZZZ_2^\eps)' \to \XXX} \lesssim \frac{\eps^6}{t^3} \, e^{-\sigma t/\eps^2}. 
	$$
\end{lem}

We also have the following dissipativity properties, the proof of which relies on the same line of proof as the above regularization results. Indeed, the idea is to use the same functionals without the weights in time. We thus skip the proof since the computations are the same. As previously, the result is only given in the dual framework which will be the only one useful in the sequel. 
\begin{lem}  \label{lem:dissipativeBeps2}
Let $\sigma \in (0,\sigma_1)$ (where $\sigma_1$ is defined in~\eqref{def:sigma1}). For any $t \ge 0$ and $i=1,2$, one has
$$
\|S_{\BB_\eps}(t)\|_{(\ZZZ_i^\eps)' \to (\ZZZ_i^\eps)'} \lesssim e^{-\sigma t/\eps^2}.
$$
\end{lem}

\color{black}

In our forthcoming analysis, we will use an iterated Duhamel formula based on the splitting 
$\Lambda_\eps = \AA_\eps + \BB_\eps$. We introduce the following definition of convolution of semigroups: If $S_1$ and $S_2$ are two semigroups, their convolution product is defined by
	$$
	S_1 * S_2(t) := \int_0^t S_1(s) S_2(t-s) \, \d s.
	$$
We also introduce the semigroups $V_j^\eps(t)$ defined through:
 $$
 V_0^\eps(t) := S_{\BB_\eps}(t), \quad V_{j+1}^\eps(t) := (S_{\BB_\eps} * \AA_\eps V_j^\eps)(t) = (V_j^\eps* \AA_\eps S_{\BB_\eps})(t), \, \, j \in \N
 $$
 so that for any $n \in \N$ and any $t \geq 0$, we have:
\begin{equation} \label{eq:splitUeps}
 U^\eps(t) = \sum_{j=0}^n V_j^\eps(t) + (V_n^\eps*\AA_\eps U^\eps)(t). 
\end{equation}

As a consequence of the previous results on $\AA_\eps$ and $\BB_\eps$, we obtain:
\begin{cor} \label{cor:Vjepsdecay}
Let $\sigma \in (0,\sigma_1)$ (where $\sigma_1$ is defined in~\eqref{def:sigma1}). For any $t > 0$, any $j \in \N$ and $i=1,2$, we have:
	$$
	\|V^\eps_j(t)\|_{\XXX \to \XXX}  \lesssim e^{-\sigma t/\eps^2} 
	\quad \text{and} \quad 
	\|V^\eps_j(t)\|_{(\ZZZ_i^\eps)' \to (\ZZZ_i^\eps)'} \lesssim e^{-\sigma t/\eps^2}.
	$$
\end{cor}

\begin{proof}
Fix $\sigma \in (0,\sigma_1)$. The first estimate for $j=0$ is given by Lemma~\ref{lem:dissipativeBeps}. We then argue by induction and thus suppose that the property is satisfies for some $j \in \N$. Then, we consider $\sigma' \in (\sigma,\sigma_1)$. From Lemmas~\ref{lem:Areg} and~\ref{lem:dissipativeBeps} and the induction hypothesis, we have:
	\begin{align*}
	\|V^\eps_{j+1}(t)\|_{\XXX \to \XXX} 
	&\lesssim 
	\frac{1}{\eps^2} \int_0^t 
	\|S_{\BB_\eps}(s)\|_{\XXX \to \XXX} \|\AA\|_{\XXX \to \XXX} \|V_j^\eps(t-s)\|_{\XXX \to \XXX} \, \d s \\
	&\lesssim \frac{1}{\eps^2} \int_0^t e^{- \sigma' s/\eps^2} e^{-\sigma(t-s)/\eps^2} \, \d s
	\lesssim \frac{1}{\eps^2} e^{-\sigma t/\eps^2}  \int_0^t e^{- (\sigma'-\sigma) s/\eps^2} \, \d s 
	\lesssim e^{-\sigma t/\eps^2}.
	\end{align*}
The second estimate can be proven in a similar way by using Lemmas~\ref{lem:Areg} and~\ref{lem:dissipativeBeps2}. 
\end{proof}

\begin{cor} \label{cor:Vjepsreg}
Let $\sigma \in (0,\sigma_1)$ (where $\sigma_1$ is defined in~\eqref{def:sigma1}). For any $t > 0$ and any~$j \in \N$, we have:
	$$
	\|V_j^\eps(t)\|_{(\ZZZ_1^\eps)' \to \XXX} \lesssim \frac{\eps^{3-2j}}{{t}^{\frac{3-2j}{2}}} e^{-\sigma t/\eps^2}
		\quad \text{and} \quad 
	\|V_j^\eps(t)\|_{(\ZZZ_2^\eps)' \to \XXX} \lesssim \frac{\eps^{6-2j}}{t^{3-j}} e^{-\sigma t/\eps^2}. 
	$$
\end{cor}

\begin{proof}
Fix $\sigma \in (0,\sigma_1)$. 
We focus on the proof of the first estimate, the second one is treated in a similar way. We proceed by induction. The case~$j=0$ is given by Lemma~\ref{lem:regSBeps}. Suppose then that the estimate holds for some~$j \in \N$ and consider $\sigma' \in (\sigma,\sigma_1)$. From Lemmas~\ref{lem:Areg}-\ref{lem:dissipativeBeps}, Corollary~\ref{cor:Vjepsdecay} and the induction hypothesis, we have:
	\begin{align*}
	&\|V_{j+1}^\eps(t)\|_{(\ZZZ_1^\eps)' \to \XXX} 
	\lesssim
	\int_0^{t/2} \|S_{\BB_\eps}(s)\|_{\XXX \to \XXX} \|\AA_\eps\|_{\XXX \to \XXX} \|V_j^\eps(t-s)\|_{(\ZZZ_1^\eps)' \to \XXX} \, \d s \\
	&\qquad \qquad 
	+ \int_{t/2}^t \|V_j^\eps(s)\|_{(\ZZZ_1^\eps)' \to \XXX} \|\AA_\eps\|_{(\ZZZ_1^\eps)'  \to (\ZZZ_1^\eps)' } \|S_{\BB_\eps}(t-s)\|_{(\ZZZ_1^\eps)' \to (\ZZZ_1^\eps)' } \, \d s \\
	&\quad \lesssim \frac{1}{\eps^2} \int_0^{t/2} e^{-\sigma' s/\eps^2} \frac{\eps^{3-2j}}{(t-s)^{\frac{3-2j}{2}}} e^{-\sigma (t-s)/\eps^2} \, \d s
	+ \frac{1}{\eps^2} \int_{t/2}^{t} \frac{\eps^{3-2j}}{s^{\frac{3-2j}{2}}} e^{-\sigma' s/\eps^2}  e^{-\sigma (t-s)/\eps^2} \, \d s.
	\end{align*}
From this, we deduce that 
	\begin{align*}
	\|V_{j+1}^\eps(t)\|_{(\ZZZ_1^\eps)' \to \XXX} 
	&\lesssim e^{-\sigma t/\eps^2} \eps^{1-2j} \left(\int_0^{t/2} \frac{\d s}{(t-s)^{\frac{3-2j}{2}}} 
	+  \int_{t/2}^t \frac{\d s}{s^{\frac{3-2j}{2}}}\right)  
	\lesssim \frac{\eps^{1-2j}}{t^{\frac{1-2j}{2}}} e^{-\sigma t/\eps^2},
	\end{align*}
which yields the wanted result. 
\end{proof}

\subsection{Spectral study in Fourier space}

We denote by $\mathcal{F}_x$ the Fourier transform in~$x \in \T^3$ with $\xi \in \Z^3$ its dual variable. Since we will only be working with Fourier transform in~$x$, we will also interchangeably use the classical ``hat'' notation. Moreover, to lighten the reading, for any operator that acts only on velocity, with a little abuse of notation, we will omit the ``hat'' in the notation for its $x$-wise Fourier transform. 

In this part, we are going to look at the Fourier transform in~$x \in \T^3$ of the operator $\Lambda_1$:
	$$
	\widehat \Lambda_1(\xi):=-i \, \xi \cdot v + L
	$$
and study the spectrum of $\widehat \Lambda_1(\xi)$ for $\xi \in \Z^3$. 
This type of analysis was initiated in~\cite{Nicolaenko,CIP,Ellis-Pinsky} for the Boltzmann equation for hard spheres and then with hard cutoff potentials (see also~\cite{Ukai-Yang}). 
In~\cite{Yang-Yu}, Yang and Yu were then able to adapt it to more general kinetic equations including the linearized Landau one 
for hard and moderately soft potentials. 

Roughly speaking, for small frequencies, the spectrum of $\widehat \Lambda_1(\xi)$ is a perturbation of the one of the homogeneous collision operator $L$ (which acts only on velocity). As already mentioned, in the case of hard and moderately soft potentials ($\gamma \geq -2$), the operator $L$ has a spectral gap and in~\cite{Yang-Yu}, the authors then prove that for small frequencies $\xi$, the spectrum of $\widehat \Lambda_1(\xi)$ is made of ``small'' eigenvalues around $0$ in the right part of the plan. They also provide Taylor expansions of those eigenvalues as well as for their associated projectors. For large frequencies, they prove that the operator $\widehat \Lambda_1(\xi)$ has no spectrum in some suitable right part of the plan. All those spectral results provide a decomposition of the semigroup which is given in Lemma~\ref{lem:YangYu}.

In what follows, we write 
	$$
	\widehat U^\eps(t) = \mathcal{F}_x  U^\eps(t ) \mathcal{F}_x^{-1}
	$$
so that $\widehat U^\eps$ is the semigroup associated with the operator 
	$$
	\widehat \Lambda_\eps(\xi):=\frac{1}{\eps^2}(-i\eps \, \xi \cdot v + L) . 
	$$
We also introduce the bilinear operator $\Psi^\eps(t)$ defined by
	\begin{equation}\label{def:Psiepsf1f2}
		\Psi^\eps(t)  (f_1,f_2) :=\frac{1}{\eps} \int_0^t U^\eps (t-s) \Gamma\big(f_1(s),f_2(s)\big) \, \d s  .
	\end{equation}
We recall that~$\chi$ is a fixed, compactly supported function of the interval~$(-1,1)$, equal to one on~$[-1/2 , 1/2]$. 

\begin{lem}\label{lem:YangYu}  
There exists $\kappa>0$ such that one can write
	\begin{equation*} \label{eq:decompUeps}
	\begin{aligned}
		U^\eps(t) 
		&=\sum_{j=1}^4   U^\eps_j(t)  +   U^{\eps\sharp}(t)  \\
		& \text{with} \quad \widehat U_j^\eps(t,\xi): = \widehat U _j\Big(\frac{t}{\eps^2},\eps \xi\Big)
		 	\quad \text{and}\quad \widehat U^{\eps\sharp}(t,\xi):= \widehat U^{\sharp}\Big(\frac{t}{\eps^2},\eps \xi\Big),
	\end{aligned}
	\end{equation*}
where for $1 \le j \le 4$,  
	$$
	\widehat U _j(t,\xi) := \chi\left(\frac{|\xi|}\kappa\right) \, e^{t \lambda_j(\xi)} P_j(\xi)
	$$
with $\lambda_j$ satisfying
	\begin{equation} \label{eq:estimatelambdaj}
	\begin{aligned}
		&\lambda_j(\xi) = i \alpha_j |\xi| - \beta_j|\xi|^2 + \gamma_j(|\xi|), \\
		&\alpha_1>0\, , \quad \alpha_2>0\, , \quad  \alpha_3=\alpha_4=0\,  , \quad \beta_j>0, \\
		&\gamma_j(|\xi|) =_{|\xi| \to 0} O(|\xi|^3) \quad \text{and} \quad \gamma_j(|\xi|) \le \beta_j |\xi|^2/2
		\quad \text{for} \quad |\xi|\le \kappa,
	\end{aligned}
	\end{equation}
and 
	$$
	P_j(\xi) = P_{j}^0\left(\frac{\xi}{|\xi|}\right) +|\xi| P_{j}^1\left(\frac{\xi}{|\xi|}\right) + |\xi|^2 P_{j}^2(\xi),
	$$
with $P_{j}^n$ bounded linear operators on $L^2_v$ with operator norms uniform for $|\xi| \le \kappa$. 

We also have that the orthogonal  projector~$\pi$ onto $\operatorname{Ker} L$ (see~\eqref{def:pi}) satisfies
	$$
	\pi  = \sum_{j=1}^4 P_{j}^0\left(\frac{\xi}{|\xi|}\right)
	$$
and is independent of $\xi/|\xi|$.

Moreover, $P_{j}^0(\xi/|\xi|)$, $P_{j}^1(\xi/|\xi|)$ and $P_{j}^2(\xi)$ are bounded from
$L^2_v$ into $L^2_v(\langle v \rangle^\ell)$ uniformly in $|\xi| \leq \kappa$ for any $\ell \geq 0$.

Finally, for any $\ell \geq 0$, $\widehat U^\sharp$ satisfies
	\begin{equation} \label{eq:rateUsharp}
		\| \widehat U^\sharp\|_{L^2_v(\la v \ra^\ell) \to L^2_v(\la v \ra^\ell)} \le C e^{-\alpha t}
	\end{equation}
for some positive constants $C$ and $\alpha$ independent of $t$ and $\xi$.
\end{lem}

\begin{proof}
The decomposition of~$\widehat U^\eps(t)$ follows that of~$\widehat U^1(t)$: 
We recall that according to~\cite[Theorem~3.2 and Remark~5.2]{Yang-Yu}, one can write
	$$
	\widehat U^1(t,\xi) = \sum_{j=1}^4 \widehat U _j(t,\xi)  + \widehat U^{\sharp}(t,\xi), 
	$$
where for $1 \le j \le 4$,  
	$$
	\widehat U _j(t,\xi) := \chi\left(\frac{|\xi|}\kappa\right) \, e^{t \lambda_j(\xi)} P_j(\xi)
	$$
and~$\lambda_j(\xi) \in \C$ are the eigenvalues of $\widehat \Lambda_1(\xi)$ with associated eigenprojections $P_j(\xi)$ on~$L^2_v$, satisfying the expansions stated in the lemma. The fact that $\pi  = \sum_{j=1}^4 P_{j}^0\left(\frac{\xi}{|\xi|}\right)$ also comes from~\cite[Theorem~3.2]{Yang-Yu}. 

Let us now prove that $P_{j}^0(\xi/|\xi|)$, $P_{j}^1(\xi/|\xi|)$ and $P_{j}^2(\xi)$ are bounded from $L^2_v$ into $L^2_v(\langle v \rangle^\ell)$ uniformly in $|\xi| \leq \kappa$ for any $\ell \geq 0$. We first prove that this property is satisfied for $P_j(\xi)$. Recall that for $|\xi| \leq \kappa$,
	$$
	\widehat \Lambda_1(\xi) P_j(\xi) = \lambda_j(\xi) P_j(\xi).
	$$
Thanks to the splitting $\Lambda_1=\AA_1+\BB_1$ introduced in Section~\ref{sec:linear2}, denoting $\widehat \BB_1(\xi) := - i \xi \cdot v + \BB_1$, we have for $|\xi| \leq \kappa$:
	$$
	P_j(\xi) = (\lambda_j(\xi)-\widehat \BB_1(\xi))^{-1} \AA_1 P_j(\xi). 
	$$
The dissipative properties of $\BB_1$ in $L^2_v(\la v \ra^\ell)$ and the regularization properties of $\AA_1$ (from~$L^2_v$ into~$L^2_v(\la v \ra^\ell)$) established respectively in Lemmas~\ref{lem:dissipativeBeps} and~\ref{lem:Areg} and the fact that from~\cite[Theorem~3.2]{Yang-Yu}, we already know that $P_j(\xi)$ is uniformly bounded in~$|\xi| \leq \kappa$ from $L^2_v$ into itself, imply that $P_j(\xi)$ is bounded from $L^2_v$ into $L^2_v(\la v \ra^\ell)$ for any $\ell \geq 0$ and uniformly in~$|\xi| \leq \kappa$. To conclude that the same properties hold for $P_{j}^0(\xi/|\xi|)$, $P_{j}^1(\xi/|\xi|)$ and $P_{j}^2(\xi)$, we notice that $P_{j}^0(\xi/|\xi|)$, $P_{j}^1(\xi/|\xi|)$ are given by explicit formula (see the proof of Theorem~3.2 in~\cite{Yang-Yu}) that clearly define bounded operators from~$L^2_v$ into~$L^2_v(\la v \ra^\ell)$ for $|\xi| \leq \kappa$.

Finally, the estimate~\eqref{eq:rateUsharp} comes from~\cite[Remark~5.2]{Yang-Yu} for $\ell=0$. 
We can also prove it for any $\ell \geq 0$ thanks to Duhamel formula applied with the splitting $\Lambda_1=\AA_1+\BB_1$ introduced in Section~\ref{sec:linear2}. We write that 
	\begin{align*}
	U^\sharp(t) &= U^1(t) \FF_x^{-1} \left(\Id - \chi\left(\frac{|\xi|}\kappa\right) \, e^{t \lambda_j(\xi)} P_j(\xi)\right) \FF_x \\
	&= \left(S_{\BB_1}(t) + (S_{\BB_1}* \AA_1 U^1)(t) \right) \FF_x^{-1}
	\left(\Id - \chi\left(\frac{|\xi|}\kappa\right) \, e^{t \lambda_j(\xi)} P_j(\xi)\right) \FF_x.
	\end{align*}
We are able to get the wanted estimate in $L^2_v(\la v \ra^\ell)$ thanks to the uniform boundedness in~$|\xi| \leq \kappa$ of the projectors $P_j(\xi)$ in $L^2_v$, the dissipativity properties of $\BB_1$ in $L^2_v(\la v \ra^\ell)$ and the regularization properties of $\AA_1$ (from $L^2_v$ to $L^2_v(\la v \ra^\ell)$) established respectively in Lemmas~\ref{lem:dissipativeBeps} and~\ref{lem:Areg}.
\end{proof}

\begin{rem} \label{rem:defU} 
Denoting
	\begin{equation} \label{def:tildeP1j}
		\widetilde P_j \left(\xi, \frac{\xi}{|\xi|}\right) := P_{j}^1\left(\frac{\xi}{|\xi|}\right) + |\xi| P_{j}^2(\xi)
	\end{equation}
for $1 \le j \le 4$, we can further split $\widehat U^\eps_j(t)$ into four parts (a main part and three remainder terms):
	$$
  	U^\eps_j =   U^\eps_{j0} +  U^{\eps\sharp}_{j0} +   U^\eps_{j1} +   U^\eps_{j2}
	$$
where 
	\begin{align*}
		&\widehat U^\eps_{j0}(t,\xi) := e^{i\alpha_j |\xi|\frac{t}{\eps} - \beta_j t |\xi|^2} P_{j}^0\left(\frac{\xi}{|\xi|}\right), \\
		&\widehat U^{\eps\sharp}_{j0}(t,\xi) := 
			\left(\chi \left(\frac{\eps|\xi|}\kappa\right)-1\right) e^{i\alpha_j |\xi|\frac{t}{\eps} - \beta_j t |\xi|^2}
			P_j^0\left(\frac{\xi}{|\xi|}\right),\\
		&\widehat U^\eps_{j1}(t,\xi) := 
			\chi \left(\frac{\eps|\xi|}\kappa\right) e^{i\alpha_j |\xi|\frac{t}{\eps} - \beta_j t |\xi|^2} 
			\left(e^{t \frac{\gamma_j(\eps|\xi|)}{\eps^2}}-1 \right) P_{j}^0\left(\frac{\xi}{|\xi|}\right), \\
		&\widehat U^\eps_{j2}(t,\xi) := 
			\chi \left(\frac{\eps|\xi|}\kappa\right) e^{i\alpha_j |\xi|\frac{t}{\eps} - \beta_j t |\xi|^2+ t \frac{\gamma_j(\eps|\xi|)}{\eps^2}} 
			\eps |\xi| \widetilde P_{j} \left(\eps \xi, \frac{\xi}{|\xi|}\right).
	\end{align*}
One can notice that $U_{30}:=U^\eps_{30}$ and $U_{40}:=U^\eps_{40}$ do not depend on $\eps$ since $\alpha_3=\alpha_4=0$. We set 
	$$
	U := U_{30} + U_{40}. 
	$$
We shall see that the operator~$U(t)$ is in some sense the limit of~$U^\eps(t)$ (see Lemma~\ref{lem:UepstoU}). 
\end{rem}


The decomposition of the semigroup $U^\eps(t)$ also gives us a decomposition of the operator~$\Psi_\eps(t)$ defined in~\eqref{def:Psiepsf1f2} (see~\cite[Lemma~A.4]{Gallagher-Tristani} and its proof).
\begin{lem}\label{lem:estimatePsiepslimitPsi0}
The following decomposition holds
	$$    
	\Psi^\eps   = \sum_{j=1}^4 \Psi^\eps_j +  \Psi^{\eps \sharp}  
	$$
with
	$$
	\widehat \Psi^{\eps \sharp}(t)(f_1,f_2):= \frac{1}{\eps}  \int_0^t \widehat U^{\eps\sharp}(t-s) 
	\widehat  \Gamma \big(f_1(s),f_2(s)\big) \, \d s
	$$
and
	\begin{equation*}\label{eq:decompositionPsi}
		\Psi^\eps_j = \Psi^\eps_{j0}  + \Psi^{\eps \sharp}_{j0}  + \Psi^\eps_{j1}  + \Psi^\eps_{j2} 
	\end{equation*}
where 
	\begin{align*}
		&\mathcal{F}_x \left( \Psi^\eps_{j0}(t)(f,f)\right)(\xi) := \int_0^t e^{i\alpha_j |\xi|\frac{t-s}{\eps} - \beta_j (t-s) |\xi|^2} |\xi|
			P_{j}^1\left(\frac{\xi}{|\xi|}\right) \widehat \Gamma \big(f_1,f_2\big)(s) \, \d s , \\
		&\mathcal{F}_x \left( \Psi^{\eps\sharp}_{j0}(t)(f,f)\right)(\xi):= \left(\chi \left(\frac{\eps|\xi|}\kappa\right)-1\right) 
			\int_0^te^{i\alpha_j |\xi|\frac{t-s}{\eps} - \beta_j (t-s) |\xi|^2}|\xi|P_{j}^1\left(\frac{\xi}{|\xi|}\right) \widehat \Gamma \big(f_1,f_2\big)(s) \, \d s  ,\\
		&\mathcal{F}_x \left( \Psi^\eps_{j1}(t)(f,f)\right)(\xi) \\
		&\quad := \chi \left(\frac{\eps|\xi|}\kappa\right) \int_0^te^{i\alpha_j |\xi|\frac{t-s}{\eps} - \beta_j (t-s) |\xi|^2} 
			\left(e^{(t-s) \frac{\gamma_j(\eps|\xi|)}{\eps^2}}-1 \right)|\xi|P_{j}^1\left(\frac{\xi}{|\xi|}\right) \widehat \Gamma \big(f_1,f_2\big)(s) \, \d s , \\
		&\mathcal{F}_x \left( \Psi^\eps_{j2}(t)(f,f)\right)(\xi) \\
		&\quad := \chi \left(\frac{\eps|\xi|}\kappa\right) \int_0^te^{i\alpha_j |\xi|\frac{t-s}{\eps} - \beta_j (t-s) |\xi|^2
			+ (t-s) \frac{\gamma_j(\eps|\xi|)}{\eps^2}} \eps |\xi|^2P^2_j(\eps\xi) \widehat \Gamma \big(f_1,f_2\big)(s) \, \d s .
	\end{align*}
\end{lem}

\begin{rem}\label{rem:defPsi}
Let us notice that as in Remark~\ref{rem:defU}, there holds
	$$
	 \Psi^\eps_{30} =  \Psi_{30} \quad \mbox{and}\quad  \Psi^\eps_{40}  =  \Psi_{40} 
	$$
and we set
	$$
	\Psi := \Psi_{30}+ \Psi_{40}.
	$$
\end{rem}

\subsection{Limit operators $U(t)$ and $\Psi(t)$} \label{subsec:UandPsi}

The following lemma studies the limit of $U^\eps(t)$ as~$\eps$ goes to $0$, its proof is completely similar as the one of~\cite[Lemma~3.5]{Gallagher-Tristani}, the only difference being that we use that the projectors are bounded from~$L^2_v$ into~$L^2_v\left(\la v\ra^{3(\frac{\gamma}{2}+1)}\right)$ (see Lemma~\ref{lem:YangYu}), we thus skip the proof. 

\begin{lem} \label{lem:UepstoU}
Let $f$ be a well-prepared data as defined in~\eqref{def:wp}. Then, we have that 
	\begin{equation} \label{eq:Ueps-U_1}
		\|(U^\eps(t)-U(t))f\|_{L^\infty_t(\XXX)}\lesssim \|f\|_{H^3_xL^2_v}
	\end{equation}
and
	\begin{equation} \label{eq:Ueps-U_2}
		\|(U^\eps(t)-U(t)) f\|_{L^\infty_t(\XXX)}\lesssim \eps \|f\|_{H^4_x L^2_v}. 
	\end{equation}
\end{lem}

In the following lemma, we study the convergence of $\Psi^\eps(t) (f,f)$ towards $\Psi(t)(f,f)$ when $f$ is a well-prepared data and its associated macroscopic quantities solve the limit system~\eqref{eq:NSF}. The proof is similar to the one of~\cite[Lemma~4.1]{Gallagher-Tristani}, we thus omit the proof (notice that this result relies on refined estimates on quantities related to $f$ that can be found in~\cite[Lemmas~B.6 and~B.7]{Gallagher-Tristani}).

\begin{lem} \label{lem:PsiepstoPsi}
Consider~$f$ a well-prepared data as defined in~\eqref{def:wp} with associated macroscopic quantities solving the limit system~\eqref{eq:NSF} on~$\R^+$ and with mean free initial data $(\rho_0,u_0,\theta_0) \in H^3_x$ and associated kinetic distribution $f_0 \in \XXX$ (as in~\eqref{def:g0}) satisfying $\|f_0\|_\XXX \leq \eta_1$ (so that $f$ is defined globally in time), then  
	$$
	\| \Psi^\eps(t) (f,f)- \Psi(t) (f,f)\|_{L^\infty_t(\XXX)} \lesssim \eps \, C(\|f_0\|_{H^3_xL^2_v}),
	$$
where $C(\|f_0\|_{H^3_xL^2_v})$ is a constant only depending on $\|f_0\|_{H^3_xL^2_v}$. 
\end{lem}

\subsection{Decay estimates on the linearized Landau semigroup} \label{subsec:UepsEP}

We recall that $\pi$ is the projector onto the kernel of $L$ and is given in~\eqref{def:pi} and that the spaces $\YYY_1$ and $(\ZZZ_i^\eps)'$ are respectively defined in~\eqref{def:normYYY1} and~\eqref{def:ZZZieps'}. 
From Lemma~\ref{lem:YangYu}, as in~\cite[Lemma~3.2]{Gallagher-Tristani}, we can prove some new decay estimates on the linearized Landau semigroup:

\begin{lem} \label{lem:EP}
Let $\sigma_2:=\min(\alpha,\beta_1, \dots, \beta_4)$ (where $\alpha$ and $\beta_j$ for $j=1,\dots 4$ are defined in Lemma~\ref{lem:YangYu}). Then, for any $\sigma \in (0,\sigma_2)$, we have
	$$
	\|U^\eps(t)(\Id-\pi)\|_{\XXX \to \XXX} \lesssim \eps \, \frac{e^{-\sigma t}}{\sqrt{t}}, \quad \forall \, t >0. 
	$$
\end{lem}

Combining this with Corollaries~\ref{cor:Vjepsdecay}-\ref{cor:Vjepsreg}, one can deduce the following result which mixes decay and regularization estimates:
\begin{cor} \label{cor:EPreg}
For any $\sigma \in (0, \min(\sigma_1,\sigma_2))$, there holds
	$$
	\|U^\eps(t)(\Id-\pi)\|_{\XXX \to \YYY_1} \lesssim \eps \, \frac{e^{-\sigma t}}{\sqrt{t}}, \quad \forall \, t > 0,
	$$
where $\sigma_1$ is defined in~\eqref{def:sigma1} and $\sigma_2$ in Lemma~\ref{lem:EP}. 
\end{cor} 

\begin{proof}
Let $\sigma \in (0, \min(\sigma_1,\sigma_2))$. From Duhamel formula, we have:
	$$
	U^\eps(t) (\Id - \pi) = S_{\BB_\eps}(t) (\Id - \pi) + (S_{\BB_\eps}*\AA_\eps U^\eps)(t) (\Id - \pi).
	$$
From Lemma~\ref{lem:regSBeps}, since $\sigma < \sigma_1$, we have:
	$$
	\|S_{\BB_\eps}(t) (\Id - \pi)\|_{\XXX \to \YYY_1} 
	\lesssim \varepsilon\frac{e^{-\sigma t}}{\sqrt{t}} \|\Id - \pi\|_{\XXX \to \XXX} 
	\lesssim \varepsilon\frac{e^{-\sigma t}}{\sqrt{t}}.
	$$
For the second term, we use Lemmas~\ref{lem:Areg} and~\ref{lem:regSBeps}:
	\begin{align*}
	\| (S_{\BB_\eps}*\AA_\eps U^\eps)(t) (\Id - \pi) \|_{\XXX \to \YYY_1} 
	&\lesssim \int_0^t \left\|S_{\BB_\eps}(t-s) \AA_\eps U^\eps(s) (\Id - \pi)\right\|_{\XXX \to \YYY_1} \, \d s \\
	&\lesssim \int_0^t \frac{e^{-\sigma (t-s)}}{\sqrt{t-s}} \left\|U^\eps(s) (\Id - \pi)\right\|_{\XXX \to \XXX} \, \d s.  
	\end{align*}
Finally, from Lemma~\ref{lem:EP},
	$$
	\| (S_{\BB_\eps}*\AA_\eps U^\eps)(t) (\Id - \pi) \|_{\XXX \to \YYY_1} 
	\lesssim \eps \int_0^t \frac{e^{-\sigma (t-s)}}{\sqrt{t-s}} \frac{e^{-\sigma s}}{\sqrt{s}} \, \d s 
	\lesssim \eps \, e^{-\sigma t},
	$$
which yields the conclusion. 
\end{proof}

By using an interpolation argument, one can deduce the following result:
\begin{lem} \label{lem:EPbis}
For any $\sigma \in (0,\min(\sigma_0,\sigma_1,\sigma_2))$, we have:
	$$
	\|U^\eps(t)(\Id-\pi)\|_{(\ZZZ_1^\eps)' \to (\ZZZ_1^\eps)'} 
	\lesssim \sqrt{\eps}\, \frac{e^{-\sigma t}}{t^{1/4}}, \quad \forall \, t >0,
	$$
where $\sigma_0$, $\sigma_1$ and $\sigma_2$ are respectively defined in Proposition~\ref{prop:hypoL2}, in~\eqref{def:sigma1} and Lemma~\ref{lem:EP}. 
\end{lem}

\begin{proof}
\noindent {\it Step 1.} First, by using an enlargement argument (from~\cite{GMM}), we prove that 
	\begin{equation} \label{eq:UepsZ1eps'}
	\|U^\eps(t)(\Id-\Pi)\|_{(\ZZZ_2^\eps)' \to (\ZZZ_2^\eps)'} \lesssim e^{-\sigma t}.
	\end{equation}
For sake of completeness and in order to carefully handle the $\eps$-dependencies, we write the proof.  
From Duhamel formula, we have that 
	$$
	U^\eps(t) = \sum_{j=0}^3 V_j^\eps(t) + (U^\eps*V_3^\eps)(t).
	$$
Moreover, we have $U^\eps(t)(\Id - \Pi) = (\Id - \Pi) U^\eps(t)$. Then, from Corollary~\ref{cor:Vjepsdecay} and the fact that $\Pi \in \BBB((\ZZZ_2^\eps)')$, for any $j = 0, \dots, 3$, we have:
	\begin{equation} \label{eq:(I-Pi)Vjeps}
	\|(\Id - \Pi) V_j^\eps(t)\|_{(\ZZZ_2^\eps)' \to (\ZZZ_2^\eps)'} \lesssim e^{-\sigma t/\eps^2}. 
	\end{equation}
For the last term, using that $\XXX \hookrightarrow (\ZZZ_2^\eps)'$ (independently of $\eps$) and Theorem~\ref{theo:mainlinear}, we have:
	\begin{align*}
	\|(\Id - \Pi) (U^\eps*V_3^\eps)(t)\|_{(\ZZZ_2^\eps)' \to (\ZZZ_2^\eps)'} 
	&\lesssim \int_0^t \left\|(\Id - \Pi) U^\eps(t-s) V_3^\eps(s)\right\|_{(\ZZZ_2^\eps)' \to \XXX} \, \d s \\
	&\lesssim \int_0^t e^{-\sigma(t-s)} \left\|V_3^\eps(s)\right\|_{(\ZZZ_2^\eps)' \to \XXX} \, \d s .
	\end{align*}
Corollary~\ref{cor:Vjepsreg} allows us to conclude that 
	\begin{equation} \label{eq:UepsV3}
	\|(\Id - \Pi) (U^\eps*V_3^\eps)(t)\|_{(\ZZZ_2^\eps)' \to (\ZZZ_2^\eps)'} 
	\lesssim \int_0^t e^{-\sigma(t-s)} e^{-\sigma s/\eps^2} \, \d s 
	\lesssim e^{-\sigma t}. 
	\end{equation}
From estimates~\eqref{eq:(I-Pi)Vjeps} and~\eqref{eq:UepsV3}, we can conclude that~\eqref{eq:UepsZ1eps'} holds. 

\medskip
\noindent {\it Step 2.} From~\eqref{eq:UepsZ1eps'} and the fact that $U^\eps(t) \Pi = \Pi \in \BBB((\ZZZ_2^\eps)')$, we deduce
	$$
	\|U^\eps(t)\|_{(\ZZZ_2^\eps)' \to (\ZZZ_2^\eps)'} \lesssim 1.
	$$
Since $\pi \in \BBB((\ZZZ_2^\eps)')$, it implies that 
	$$
	\|U^\eps(t)(\Id-\pi)\|_{(\ZZZ_2^\eps)' \to (\ZZZ_2^\eps)'} \lesssim 1.
	$$
Consequently, combining this with Lemma~\ref{lem:EP}, we can conclude the proof by interpolation because from~\eqref{eq:interp}, we have $(\ZZZ_1^\eps)'=[\XXX,(\ZZZ_2^\eps)']_{1/2,2}$. 
\end{proof}

\section{Hydrodynamical limit} \label{sec:hydro}

We first state a quantitative result which provides estimates on the difference on the solution $g^\eps$ to the Landau equation constructed in Theorem~\ref{theo:main1} and the solution $g$ defined in~\eqref{def:g} whose first macroscopic quantities are solution to the fluid system~\eqref{eq:NSF}. As explained in~Subsection~\ref{subsec:proofmain2}, this theorem combined with a density argument allows to prove Theorem~\ref{theo:main2}. It is important to notice that thanks to the estimates obtained on the kinetic equation in Theorem~\ref{theo:main1}, under some suitable smallness assumptions on the initial data of both kinetic and fluid equations, only extra-regularity in $x$ on the initial data of the fluid system is needed to obtain a quantitative rate of convergence in $\eps$, as can be seen in~\eqref{eq:cv1precise} and~\eqref{eq:cv2precise}. 

\begin{theo} \label{theo:main2precise}
Let $g_{\rm in}^\eps \in \XXX \cap (\operatorname{Ker} \Lambda_\eps)^\perp$ for $\eps \in (0,1)$ such that $\|g_{\rm in}^\eps\|_\XXX \leq \eta_0$ (where $\eta_0$ is defined in Theorem~\ref{theo:main1}) and~$g^\eps \in L^\infty_t(\XXX)$ being the associated solutions of~\eqref{eq:geps} with initial data $g_{\rm in}^\eps$ constructed in Theorem~\ref{theo:main1}.
Consider also $g_0 \in H^{3+\delta}_xL^2_v \cap (\operatorname{Ker} \Lambda_\eps)^\perp$ for some $\delta \in [0,1/2]$ such that $\|g_0\|_\XXX \leq \eta_1$ and $g$ defined respectively in~\eqref{def:g0} and~\eqref{def:g}. 

There exists $\eta_2 \in (0,\min(\eta_0,\eta_1))$ such that if $\max\left(\|g_{\rm in}^\eps\|_\XXX,\|g_0\|_\XXX\right) \leq \eta_2$, then we have
	\begin{equation} \label{eq:cv1precise}
		\|g^\eps-g\|_{L^\infty_t(\XXX)} 
			\lesssim \eps^\delta C\left(\|g_0\|_{H^{3+\delta}_xL^2_v},\|g_{\rm in}^\eps\|_\XXX\right) 
			+ \|g_{\rm in}^\eps - g_0\|_\XXX
	\end{equation}
and 
	\begin{equation} \label{eq:cv2precise}
		\|g^\eps-g\|_{L^1_t(\YYY_1) + L^\infty_t(\XXX)} 
			\lesssim \eps^\delta C\left(\|g_0\|_{H^{3+\delta}_xL^2_v}, \|g_{\rm in}^\eps\|_\XXX\right)
			+ \|\pi g_{\rm in}^\eps - g_0\|_\XXX
	\end{equation}
where $C\left(\|g_0\|_{H^{3+\delta}_xL^2_v}, \|g_{\rm in}^\eps\|_\XXX\right)$ is a contant only depending on $\|g_0\|_{H^{3+\delta}_xL^2_v}$ and $\|g_{\rm in}^\eps\|_\XXX$. 
\end{theo}

\begin{rem}
Since $g_0 \in \operatorname{Ker} L$, it decays better than any polynomial in velocity at infinity, it explains the fact that we only use classical Sobolev spaces for $g_0$ in the RHS of the above inequalities, as already noticed, we have
	$$
	\|g_0\|_{H^3_xL^2_v} \lesssim \|g_0\|_\XXX \lesssim \|g_0\|_{H^3_xL^2_v}
	$$
and similar inequalities could be obtained for higher order Sobolev spaces. 
\end{rem}

\begin{rem}
We restrict ourselves to the case $\delta \in [0,1/2]$ in our estimates but one can of course suppose more regularity on the initial data~$g_0$, notice however that we will still have a rate of $\sqrt{\eps}$. It should be noted that we did not look for optimality in terms of rate in our estimates.  
\end{rem}

\subsection{Reformulation of the hydrodynamical problem}
 Before starting the proof of Theorem~\ref{theo:main2precise}, we reformulate the problem. 
Using the definition of the operator $\Psi^\eps(t)$ in~\eqref{def:Psiepsf1f2}, we have that the solution $g^\eps$ of~\eqref{eq:geps} constructed in Theorem~\ref{theo:main1} writes
	$$
	g^\eps(t) = U^\eps(t) g_{\rm in}^\eps + \Psi^\eps(t) (g^\eps,g^\eps).
	$$
It also follows from~\cite{Bardos-Ukai} that given a well-prepared data~$g_0 \in \XXX$ of the form~\eqref{def:g0}, 
the function~$g$ defined in~\eqref{def:g} satisfies
	\begin{equation}\label{eq:kineticformulationfluid}
		g(t) =U(t) g_0 + \Psi(t) (g,g),
	\end{equation}
where, as explained in Subsection~\ref{subsec:UandPsi}, the operators~$U(t)$ and~$\Psi(t)$ 
(defined respectively in Remarks~\ref{rem:defU} and~\ref{rem:defPsi}) 
are in some sense the limiting operators of~$U^\eps (t)$ and~$ \Psi^\eps(t)$. 
Formulation~\eqref{eq:kineticformulationfluid} is thus a way to reformulate the fluid equation in a kinetic fashion.

\smallskip
\subsubsection*{$L^\infty_t$-estimate}
We first reformulate the problem in order to prove the estimate~\eqref{eq:cv1precise}. 
To this end, we write the relation satisfied by $h^\eps := g^{\eps}-g$:
	\begin{equation} \label{eq:hepseta}
	\begin{aligned}
	h^\eps 
		&= U^\eps(t) g_{\rm in}^{\eps} + \Psi^\eps(t)(g^{\eps},g^{\eps}) - U(t) g_{0} - \Psi(t)(g,g) \\
		&= (U^\eps(t)-U(t)) g_{0} +U^\eps(t) (g_{\rm in}^{\eps} - g_{0})
		+ (\Psi^\eps(t) - \Psi(t))(g,g)\\
		&\quad + \Psi^\eps(t)((g^{\eps})^\perp,(g^{\eps})^\perp) + \Psi^\eps(t) ((g^{\eps})^\perp, \pi g^{\eps}) + \Psi^\eps(t) (\pi g^{\eps},(g^{\eps})^\perp) \\
		&\quad + \Psi^\eps(t)(\pi h^\eps, \pi g^{\eps}) + \Psi^\eps(t)(g, \pi h^\eps). 
	\end{aligned}
	\end{equation}
In the next subsection, we are going to study each term in the RHS of the above equality. Some terms are going to vanish in the limit $\eps \to 0$ and other ones will be absorbed in the LHS under suitable smallness assumptions on the initial data. 
Let us underline that the singularity in $\eps$ in the definition of $\Psi^\eps$ is going to be handled thanks to Lemma~\ref{lem:EP} which provides a gain of $\eps$ when the semigroup $U^\eps(t)$ acts on microscopic quantities. Using that~$\pi \Gamma (f_1,f_2)=0$ for any suitable functions $f_1$, $f_2$, we are thus going to be able to remove the singularity in~$\eps$ in the operator $\Psi^\eps$. 
In what follows, we shall prove that:
\begin{enumerate}[itemsep=5pt,leftmargin=*]
\item[-] The three first terms tend to $0$ as $\eps$ go to $0$ (see Lemmas~\ref{lem:source1} and~\ref{lem:source2}): For this purpose, we use that the limits of~$U^\eps(t)$ and $\Psi^\eps(t)$ as $\eps \to 0$ are $U(t)$ and $\Psi(t)$ (see Lemmas~\ref{lem:UepstoU} and~\ref{lem:PsiepstoPsi}). 
\item[-] The third, fourth and fifth terms tend to $0$ as $\eps \to 0$  (see Lemma~\ref{lem:sourcemicro}) because those three terms involve the microscopic part of the kinetic solution $g^\eps$, which provides us some extra smallness in $\eps$ in $L^2_t(\YYY_1)$ thanks to Theorem~\ref{theo:main1}.
\item[-] The last two terms are bounded by some quantity that involves the norms of the kinetic and fluid initial data multiplied by the $L^\infty_t(\XXX)$-norm of $h^\eps$ (see Lemma~\ref{lem:macro}). It will thus be absorbed in the LHS of the equality if initial data $g_{\rm in}^\eps$ and~$g_0$ are chosen to be small enough. Notice that one can not hope smallness in $\eps$ for those terms because they only involve macroscopic quantities.  
\end{enumerate} 
	
\smallskip
\subsubsection*{$L^1_t+L^\infty_t$-estimate} We now reformulate the problem in order to prove~\eqref{eq:cv2precise}. To this end, we introduce $R^\eps(t) := U^\eps(t) (g_{\rm in}^{\eps})^\perp$ and we write that 
	$$
	U^\eps(t) g_{\rm in}^{\eps} - U(t) g_{0} 	
	= (U^\eps(t)-U(t)) g_{0} +U^\eps(t) (\pi g_{\rm in}^{\eps} - g_{0}) + R^\eps(t).
	$$
From Corollary~\ref{cor:EPreg}, we have that for $\sigma \in (0, \min(\sigma_1,\sigma_2))$,
	\begin{equation} \label{eq:Reps}
	\|R^\eps(t)\|_{\YYY_1} \lesssim \frac{\eps}{\sqrt{t}} e^{-\sigma t} \|g_{\rm in}^{\eps}\|_\XXX
	\end{equation}
and thus 
	\begin{equation} \label{eq:Reps2}
	\|R^\eps\|_{L^1_t(\YYY_1)} \lesssim \eps \|g_{\rm in}^{\eps}\|_\XXX.	
	\end{equation}
To conclude, it is thus enough to prove that for $\delta \in [0,1/2]$,
	$$
	\|g^\eps-g - R^\eps\|_{L^\infty_t(\XXX)} 
	\lesssim \eps^\delta C\left(\|g_0\|_{H^{3+\delta}_xL^2_v}, \|g_{\rm in}^\eps\|_\XXX\right)
			+ \|\pi g_{\rm in}^\eps - g_0\|_\XXX.
	$$
We then write the relation satisfied by $\widetilde h^\eps := g^\eps-g - R^\eps$:  
	\begin{equation} \label{eq:tildehepseta}
	\begin{aligned}
	\widetilde h^\eps 
		&= (U^\eps(t)-U(t)) g_{0} +U^\eps(t) (\pi g_{\rm in}^{\eps} - g_{0})
		+ (\Psi^\eps(t) - \Psi(t))(g,g)\\
		&\quad + \Psi^\eps(t)((g^\eps)^\perp,(g^\eps)^\perp) + \Psi^\eps(t) ((g^\eps)^\perp, \pi g^\eps) + \Psi^\eps(t) (\pi g^\eps,(g^\eps)^\perp) \\
		&\quad + \Psi^\eps(t)(\pi R^\eps, \pi g^\eps) + \Psi^\eps(t)(g,\pi R^\eps) \\
		&\quad + \Psi^\eps(t)(\pi \widetilde h^\eps, \pi g^\eps) + \Psi^\eps(t)(g, \pi \widetilde h^\eps). 
	\end{aligned}
	\end{equation}
In what follows, we shall study each term in the RHS of this equality, the ideas between this decomposition being the same as the ones explained after~\eqref{eq:hepseta}. Notice furthermore that thanks to estimate~\eqref{eq:Reps}, we are also going to be able to prove that the seventh and eighth terms tend to $0$ when $\eps \to 0$ (see Lemma~\ref{lem:sourceill}). 

\subsection{Proofs of Theorems~\ref{theo:main2precise} and~\ref{theo:main2}} \label{subsec:proofmain2}

From now on, assumptions of Theorem~\ref{theo:main2precise} are supposed to hold. 
As explained above, in order to prove Theorem~\ref{theo:main2precise}, we have to estimate the $L^\infty_t(\XXX)$-norm of each term of the decompositions~\eqref{eq:hepseta} and~\eqref{eq:tildehepseta}. 

In both decompositions, concerning the first and third terms, Lemmas~\ref{lem:UepstoU} and~\ref{lem:PsiepstoPsi} immediately give by interpolation the following lemma:
\begin{lem} \label{lem:source1}
We have: For any $t \geq 0$ and $\delta \in [0,1/2]$, 
	$$
	\|(U^\eps(t)-U(t)) g_{0}\|_\XXX + \|(\Psi^\eps(t) - \Psi(t))(g,g)\|_\XXX 
	\lesssim \eps^\delta \, C\!\left(\|g_{0}\|_{H^{3+\delta}_xL^2_v}\right) .
	$$
\end{lem}
 
Concerning the second terms of~\eqref{eq:hepseta} and~\eqref{eq:tildehepseta}, from Theorem~\ref{theo:mainlinear}, we have that $U^\eps(t)$ is bounded in $\XXX$ uniformly in time and $\eps$. As a consequence, we obtain:
\begin{lem} \label{lem:source2}
We have: 
	$$
	\| U^\eps(t) (g_{\rm in}^{\eps} - g_{0})\|_{L^\infty_t(\XXX)} \lesssim \|g_{\rm in}^\eps-g_0\|_\XXX
	\quad \text{and} \quad 
	\| U^\eps(t) (\pi g_{\rm in}^{\eps} - g_{0})\|_{L^\infty_t(\XXX)} \lesssim \|\pi g_{\rm in}^\eps - g_0\|_\XXX.
	$$
\end{lem} 

The third, fourth and fifth terms of decompositions~\eqref{eq:hepseta} and~\eqref{eq:tildehepseta} are the most difficult ones to estimate. Indeed, they involve microscopic quantities, we thus have to be sharp in terms of regularity in velocity in order to obtain the following lemma:
\begin{lem} \label{lem:sourcemicro}
We have: For any $t \geq 0$,
	\begin{align*}
	\|\Psi^\eps(t)((g^\eps)^\perp,(g^\eps)^\perp)\|_{\XXX}
	&+ \|\Psi^\eps(t) ((g^\eps)^\perp, \pi g^\eps)\|_{\XXX}   + \|\Psi^\eps(t) (\pi g^\eps,(g^\eps)^\perp)\|_{\XXX} 
	\lesssim \sqrt{\eps} \|g_{\rm in}^{\eps}\|_\XXX^2.
	\end{align*}
\end{lem}
\begin{proof}
In the whole proof, we fix $\sigma \in (0,\min(\sigma_0,\sigma_1,\sigma_2))$ where $\sigma_j$ for $j=0,1,2$ are respectively defined in Proposition~\ref{prop:hypoL2},~\eqref{def:sigma1} and Lemma~\ref{lem:EP}. 
We focus on the first term which is the most intricate. 
We are going to use~\eqref{eq:splitUeps} with $n=2$ to decompose $\Psi^\eps(t)$ into several parts, it yields
	\begin{align*}
	\Psi^\eps(t)((g^\eps)^\perp,(g^\eps)^\perp) 
	&= \sum_{j=0}^2 \frac{1}{\eps} \int_0^t V_j^\eps(t-s) \Gamma((g^\eps)^\perp,(g^\eps)^\perp) (s) \, \d s \\
	&\qquad + \frac{1}{\eps} \int_0^t (V_2^\eps * \AA_\eps U^\eps)(t-s) \Gamma((g^\eps)^\perp,(g^\eps)^\perp) (s) \, \d s \\
	&=: \sum_{j=0}^2 \Psi^\eps_j(t)((g^\eps)^\perp,(g^\eps)^\perp) 
	+ \Psi^\eps_\star(t)((g^\eps)^\perp,(g^\eps)^\perp).
	\end{align*}
We first estimate $\Psi^\eps_j(t)((g^\eps)^\perp,(g^\eps)^\perp)$ for $j=0,1,2$. From Propositions~\ref{prop:Gamma-NL} and Corollary~\ref{cor:Vjepsreg}, we obtain:
	$$
	\| \Psi^\eps_j(t)((g^\eps)^\perp,(g^\eps)^\perp) \|_\XXX 
		\lesssim \frac{1}{\eps} \int_0^t \frac{\eps}{\sqrt{t-s}} e^{-\sigma(t-s)/\eps^2} 
			\|(g^\eps)^\perp(s)\|_\XXX \|(g^\eps)^\perp(s)\|_{\YYY_1} \, \d s \, .
	$$
Then, from~\eqref{eq:theo:main1:Linftybound}-\eqref{eq:theo:main1:regbound} in Theorem~\ref{theo:main1}, we have:
	$$
	\| \Psi^\eps_j(t)((g^\eps)^\perp,(g^\eps)^\perp) \|_\XXX 
		\lesssim \int_0^t \frac{e^{-\sigma(t-s)/\eps^2}}{\sqrt{t-s}} 
		\frac{e^{-3 \sigma s/2}}{s^{1/4}}\|(g^\eps)^\perp(s)\|_{\YYY_1}^{1/2} \, \d s \, \|g_{\rm in}^{\eps}\|_\XXX^{3/2}.
	$$
Using H\"older inequality and~\eqref{eq:theo:main1:Linftybound}, we obtain
	\begin{align*}
	\| \Psi^\eps_j(t)((g^\eps)^\perp,(g^\eps)^\perp) \|_\XXX 
		&\lesssim \left\| \frac{1}{ (t-s)^{1/2} s^{1/4}} \right\|_{L^{4/3}_s([0,t])} 
			\left\| \|(g^\eps)^\perp(s)\|_{\YYY_1}^{1/2} \right\|_{L^4_s} \|g_{\rm in}^{\eps}\|_\XXX^{3/2} \\
		&\lesssim \left( \int_0^t \frac{1}{(t-s)^{2/3}} \frac{\d s}{s^{1/3}} \right)^{3/4}  
			\|(g^\eps)^\perp\|_{L^2_t(\YYY_1)}^{1/2} \|g_{\rm in}^{\eps}\|_\XXX^{3/2} \\
		&\lesssim \sqrt{\eps} \|g_{\rm in}^{\eps}\|_\XXX^{2}. 
	\end{align*}
Let us now deal with $\Psi^\eps_\star(t)((g^\eps)^\perp,(g^\eps)^\perp)$. Performing a change of variable and recalling that $\pi \Gamma((g^\eps)^\perp,(g^\eps)^\perp) = 0$, one can notice that 
	\begin{multline*}
	 \Psi^\eps_\star(t)((g^\eps)^\perp,(g^\eps)^\perp) 
	 = \frac{1}{\eps} \int_0^t \int_0^s V_2^\eps(t-s) \AA_\eps U^\eps(s-\tau) (\Id - \pi)
	 		\Gamma((g^\eps)^\perp,(g^\eps)^\perp) (\tau) \, \d \tau \, \d s.
	\end{multline*}
Remark then that Corollary~\ref{cor:Vjepsreg} for $j=2$ implies that for $\sigma' \in (\sigma,\sigma_1)$, we have 
	$$
	\|V_2^\eps(t)\|_{(\ZZZ_1^\eps)' \to \XXX} \lesssim \frac{\sqrt{t}}{\eps} e^{-\sigma't/\eps^2} \lesssim e^{-\sigma t/\eps^2}. 
	$$
Using now Lemma~\ref{lem:Areg}, it implies that 
	\begin{multline*}
	\|\Psi^\eps_\star(t)((g^\eps)^\perp,(g^\eps)^\perp)\|_\XXX \\
	\lesssim \frac{1}{\eps^3} \int_0^t \int_0^s e^{-\sigma(t-s)/\eps^2} 
	\|U^\eps(s-\tau) (\Id - \pi) \Gamma((g^\eps)^\perp,(g^\eps)^\perp) (\tau) \|_{(\ZZZ_1^\eps)'} \, \d \tau \, \d s. 
	\end{multline*}
Lemma~\ref{lem:EPbis} and the fact that $\YYY_1' \hookrightarrow (\ZZZ_1^\eps)'$ (independently of $\eps$) then imply that 
	\begin{multline*}
	\|\Psi^\eps_\star(t)((g^\eps)^\perp,(g^\eps)^\perp)\|_\XXX \\
	\lesssim \frac{1}{\eps^3} \int_0^t \int_0^s e^{-\sigma(t-s)/\eps^2} 
		\frac{\sqrt{\eps}}{(s-\tau)^{1/4}} e^{-\sigma(s-\tau)} 
		\|\Gamma((g^\eps)^\perp,(g^\eps)^\perp) (\tau) \|_{\YYY_1'} \, \d \tau \, \d s. 
	\end{multline*}
From Proposition~\ref{prop:Gamma-NL} and the fact that $\|(g^\eps)^\perp\|_\XXX \lesssim \|g^\eps\|_\XXX$, we deduce that 
	\begin{align*}
	&\|\Psi^\eps_\star(t)((g^\eps)^\perp,(g^\eps)^\perp)\|_\XXX \\
	&\qquad \lesssim \frac{1}{\eps^{5/2}} \int_0^t  \int_0^s e^{-\sigma(t-s)/\eps^2} 
		\frac{e^{-\sigma(s-\tau)} }{(s-\tau)^{1/4}} 
		\|g^\eps(\tau)\|_\XXX \| (g^\eps)^\perp (\tau) \|_{\YYY_1} \, \d \tau \, \d s. 
	\end{align*}
Using Cauchy-Schwarz inequality in the variable $\tau$ and~\eqref{eq:theo:main1:Linftybound} in Theorem~\ref{theo:main1}, we obtain 
	$$
	\|\Psi^\eps_\star(t)((g^\eps)^\perp,(g^\eps)^\perp)\|_\XXX 
	\lesssim  \frac{1}{\eps^{3/2}} \int_0^t e^{-\sigma(t-s)/\eps^2} \, \d s \, \|g_{\rm in}^{\eps}\|_\XXX^{2}	
	\lesssim \sqrt{\eps} \, \|g_{\rm in}^{\eps}\|_\XXX^{2},
	$$
which concludes the proof of the term $\Psi^\eps(t)((g^\eps)^\perp,(g^\eps)^\perp)$. 

The proof for the second and third terms $\Psi^\eps(t)((g^\eps)^\perp,\pi g^\eps)$ and $\Psi^\eps(t)(\pi g^\eps,(g^\eps)^\perp)$ is completely similar once one has noticed that from Proposition~\ref{prop:Gamma-NL},
	\begin{align*}
	&\|\Gamma ((g^\eps)^\perp,\pi g^\eps)\|_{\YYY_1'} 
	+ \|\Gamma(\pi g^\eps,(g^\eps)^\perp)\|_{\YYY_1'} \\
		&\qquad \lesssim \|(g^\eps)^\perp\|_{\YYY_1} \|\pi g^\eps\|_\XXX 
		+ \|(g^\eps)^\perp\|_{\XXX} \|\pi g^\eps\|_{\YYY_1} 
		\lesssim \|(g^\eps)^\perp\|_{\YYY_1} \|g^\eps\|_\XXX 
	\end{align*}
where we used the facts that $\YYY_1 \hookrightarrow \XXX$ and $\pi \in \BBB(\XXX,\YYY_1)$. 
\end{proof}

For the proof of~\eqref{eq:cv2precise}, we also need the following lemma:
\begin{lem} \label{lem:sourceill}
We have: For any $t \geq 0$, 
	$$
	\|\Psi^\eps(t)(\pi R^\eps, \pi g^\eps)\|_\XXX 
	+ \|\Psi^\eps(t)(g,\pi R^\eps)\|_\XXX
	\lesssim \eps \left(\|g_{\rm in}^{\eps}\|_\XXX + \|g_{0}\|_\XXX\right) \|g_{\rm in}^{\eps}\|_\XXX. 
	$$
\end{lem}

\begin{proof}
From Lemma~\ref{lem:EP}, Proposition~\ref{prop:Gamma-NL-XXX}, the fact that $\pi \in \BBB(\XXX,\YYY_2)$ and $g=\pi g$, we obtain:	
	\begin{multline*}
	\|\Psi^\eps(t)(\pi R^\eps, \pi g^\eps)\|_\XXX 
	+ \|\Psi^\eps(t)(g,\pi R^\eps)\|_\XXX \\
	\lesssim 
	\int_0^t \frac{e^{-\sigma(t-s)}}{\sqrt{t-s}} \|R^\eps(s)\|_\XXX \left(\|g(s)\|_\XXX + \|g^\eps(s)\|_\XXX\right) \d s.
	\end{multline*}
Using now~\eqref{eq:Reps},~\eqref{eq:theo:main1:Linftybound} from Theorem~\ref{theo:main1} and~\eqref{eq:expNSF}, we obtain:
	\begin{multline*}
	\|\Psi^\eps(t)(\pi R^\eps, \pi g^\eps)\|_\XXX 
	+ \|\Psi^\eps(t)(g,\pi R^\eps)\|_\XXX \\
	\lesssim 
	\eps \int_0^t \frac{e^{-\sigma(t-s)}}{\sqrt{t-s}} \frac{e^{-\sigma s}}{\sqrt{s}} \, \d s \, 
	\|g_{\rm in}^{\eps}\|^2_{\XXX} 
	+ \eps \int_0^t \frac{e^{-\sigma(t-s)}}{\sqrt{t-s}} \frac{e^{-\sigma s}}{\sqrt{s}} \, \d s \,
	\|g_{\rm in}^{\eps}\|_{\XXX} \|g_{0}\|_\XXX,
	\end{multline*}
which yields the final result. 
\end{proof}

Concerning the last two terms of~\eqref{eq:hepseta} and~\eqref{eq:tildehepseta}, we prove the following lemma:
\begin{lem} \label{lem:macro}
For any $f \in L^\infty_t(\XXX)$, we have: For any $t \geq 0$,
	$$
	\|\Psi^\eps(t)(\pi f, \pi g^\eps)\|_{\XXX}
	+ \|\Psi^\eps(t)(g, \pi f)\|_{\XXX}
	\lesssim \left(\|g_{\rm in}^{\eps}\|_{\XXX} + \|g_{0}\|_\XXX\right) \|f\|_{L^\infty_t(\XXX)}. 
	$$
\end{lem}
\begin{proof}
The proof is similar to the one of Lemma~\ref{lem:sourceill}. We use Lemma~\ref{lem:EP}, Proposition~\ref{prop:Gamma-NL-XXX}, the fact that $\pi \in \BBB(\XXX,\YYY_2)$ and $g=\pi g$, and we conclude thanks to~\eqref{eq:theo:main1:Linftybound} and~\eqref{eq:expNSF}.
\end{proof}

\smallskip
\subsubsection*{End of the proof of Theorem~\ref{theo:main2precise}}
Gathering results from Lemmas~\ref{lem:source1},~\ref{lem:source2},~\ref{lem:sourcemicro} and~\ref{lem:macro}, we can conclude the proof of~\eqref{eq:cv1precise} by taking $\|g_{\rm in}^\eps\|_\XXX$ and $\|g_0\|_\XXX$ small enough. Concerning~\eqref{eq:cv2precise}, Lemmas~\ref{lem:source1},~\ref{lem:source2},~\ref{lem:sourcemicro},~\ref{lem:sourceill}~and~\ref{lem:macro} imply that 
	$$
	\|g^\eps-g-R^\eps\|_{L^\infty_t(\XXX)} 
	\lesssim \eps^\delta C\left(\|g_0\|_{H^{3+\delta}_xL^2_v}, \|g_{\rm in}^\eps\|_\XXX\right)
			+ \|\pi g_{\rm in}^\eps - g_0\|_\XXX.  
	$$
The estimate~\eqref{eq:Reps2} then allows to conclude the proof of~\eqref{eq:cv2precise}. \qed

\smallskip
\subsubsection*{Proof of Theorem~\ref{theo:main2}} 
As mentioned above, in order to obtain results of convergence in Theorem~\ref{theo:main2}, we are going to use a density argument that is explained in what follows. 
We consider a smooth family $(g_{0,\eta})_{\eta \in (0,1)}$  such that 
	\begin{equation} \label{eq:approxg0}
	\|g_0 - g_{0,\eta}\|_\XXX \leq \eta, \quad \forall \, \eta \in (0,1).
	\end{equation}
We also have stability for the Navier-Stokes-Fourier system (see for example~\cite[Appendix~B.3]{Gallagher-Tristani}), we know that 
	$$
	g_\eta(t) := U(t) g_{0,\eta} + \Psi(t)(g_\eta,g_\eta)
	$$
satisfies 
	\begin{equation} \label{eq:stabNSF}
	\lim_{\eta \to 0} \|g_\eta-g\|_{L^\infty_t(\XXX)} = 0.
	\end{equation} 
Then, to study the convergence of $g^\eps$ towards $g$, we write 
	\begin{equation} \label{eq:geps-g}
	g^\eps - g = g^{\eps}-g_\eta + g_\eta-g.  
	\end{equation}
We then apply estimate~\eqref{eq:cv1precise} from Theorem~\ref{theo:main2precise} with $g_{0,\eta}$ and $g_\eta$ instead of $g_0$ and $g$. Notice that for $g_0$ and $\eta$ small enough, $g_{0,\eta}$ will also satisfy $\|g_{0,\eta}\|_\XXX \leq \eta_2$. Coming back to~\eqref{eq:geps-g}, we deduce that 
	$$
	\|g^\eps-g\|_{L^\infty_t(\XXX)} \lesssim
	\|g-g_\eta\|_{L^\infty_t(\XXX)} +
	\eps^\delta C(\|g_{\rm in}^\eps\|_\XXX, \|g_{0,\eta}\|_\XXX) + \|g_{\rm in}^\eps-g_0\|_\XXX + \|g_0-g_{0,\eta}\|_\XXX.
	$$
Using~\eqref{eq:stabNSF}, we can conclude the proof of~\eqref{eq:cv1}. The proof of ~\eqref{eq:cv2} is similar. \qed

\appendix
\section{Proof of Proposition~\ref{prop:hypoL2}} \label{app:hypo}

We start by recalling that for any suitable function $g$, we write
$$
g = g^\perp + \pi g, 
\quad 
\pi g (x,v) = \rho_g (x) \sqrt M(v) + u_g(x) \cdot v \sqrt M(v)
+ \theta_g(x) \frac{(|v|^2 - 3)}{2} \sqrt M(v)
$$
with $\rho_g$, $u_g$, $\theta_g$ defined in~\eqref{rhof},~\eqref{uf} and~\eqref{thetaf}. 
In what follows, we also use the following notations: $\rho[g] = \rho_g$, $u[g]=u_g$, $\theta[g]=\theta_g$.

For $g \in L^2_x(\T^3)$, we introduce the following notation: $\la g \ra := \int_{\T^3} g \, \dx$. 
Recall the following classical result. For any $\phi \in L^2_x(\T^3)$ with null mean (i.e.\ $\la \phi \ra = 0$), there is a unique solution $u \in H^2_x(\T^3)$ to the equation 
$$
-\Delta_x u = \phi \quad \text{in} \quad \T^3
\quad \text{with} \quad
\la u \ra = 0 .
$$
Denote by $(-\Delta_x)^{-1}$ the following bounded operator:
\begin{equation} \label{eq:Delta-1}
\begin{array}{ccrcl}
(-\Delta_x)^{-1} & : & L^2_x(\T^d) \cap \{ \la \cdot \ra =0\}  & \longrightarrow & H^2_x(\T^3) \cap \{ \la \cdot \ra =0\}  \\
 & & \phi & \longmapsto & u.
\end{array}
\end{equation}  
Notice that in particular, we have that the operator $(-\Delta_x)^{-1}$ is also bounded from $L^2_x(\T^d) \cap \{ \la \cdot \ra =0\}$ into $H^1_x(\T^3) \cap \{ \la \cdot \ra =0\}$.

Let $f \in \operatorname{Dom} \Lambda_\eps \cap (\operatorname{Ker} \Lambda_\eps)^\perp$. Remark that in particular we have for $k=1,2,3$:
$$
\la \rho[f] \ra = \la u_k[f] \ra = \la \theta[f] \ra = 0
$$
where we have denoted by $u_k[f]$ the $k$-th coordinate of $u[f]$. We thus deduce that
$(-\Delta_x)^{-1} \rho[f]$, $(-\Delta_x)^{-1} u_k[f] $, and $(-\Delta_x)^{-1} \theta[f]$ are well-defined, and more precisely we have 
$$
\begin{aligned}
\| (-\Delta_x)^{-1} \rho[f] \|_{H^2_x} 
&\lesssim \| \rho[f] \|_{L^2_x} \\
\| (-\Delta_x)^{-1} u_k[f] \|_{H^2_x} 
&\lesssim \| u_k[f] \|_{L^2_x} \\
\| (-\Delta_x)^{-1} \theta[f] \|_{H^2_x} 
&\lesssim \| \theta[f] \|_{L^2_x} .
\end{aligned}
$$

In what follows, we shall use without further mention the following explicit computations: for $k \in \{1,2,3\}$,
$$
\begin{aligned}
\int_{\R^3} v_k^2 M \, \dv = 1 , \quad 
\int_{\R^3} |v|^2 M \, \dv = 3 , \\ 
\int_{\R^3} v_k^2 |v|^2 M \, \dv = 5 , \quad 
\int_{\R^3} |v|^4 M \, \dv = 15 , \\ 
\int_{\R^3} v_k^2 |v|^4 M \, \dv = 35 , \quad 
\int_{\R^3} |v|^6 M \, \dv = 105 , 
\end{aligned}
$$
as well as that for any odd polynomial function $p=p(v)$, one has $\int_{\R^3} p(v) M \, \dv = 0$.

\smallskip

We split the proof into five steps.

\medskip
\noindent {\it Step 1. Microscopic part.}
From~\eqref{eq:spectralgap} and the skew-adjointness of the transport operator one has
\begin{equation}\label{micro}
\begin{aligned}
\la \Lambda_\eps f , f \ra_{L^2_{x,v}} 
&= \frac{1}{\eps^2} \la L f , f \ra_{L^2_{x,v}} - \frac{1}{\eps} \la  v \cdot \nabla_x f , f \ra_{L^2_{x,v}} \\
&\le - \frac{\sigma_L}{\eps^2} \| f^\perp \|_{L^2_{x}(H^1_{v,*})}^2.
\end{aligned}
\end{equation}

\medskip
\noindent{\it Step 2. Energy estimate.}
Define for any $k \in \{1,2,3\}$ and any suitable function $g$
\begin{equation}\label{def:psi}
\psi_k[g]  = \int_{\R^3} v_k \, \frac{(|v|^2 -5)}{3} \, g \sqrt M \, \dv
\end{equation}
and remark that
\begin{equation}\label{psifperp}
\psi_k[g] = \psi_k[g^\perp].
\end{equation}
We first compute
\begin{equation}\label{thetaLf}
\begin{aligned}
\theta [\Lambda_\eps f]  
&= \frac{1}{\eps^2} \theta [L f^\perp] - \frac{1}{\eps} \theta[v \cdot \nabla_x f] \\
&= - \frac{1}{\eps} \nabla_x \cdot \int_{\R^3} \frac{(|v|^2-3)}{3} \, v f \sqrt M \, \dv \\
&= - \frac{1}{\eps} \frac{2}{3} \nabla_x \cdot \int_{\R^3}  \, v f \sqrt M \, \dv 
- \frac{1}{\eps} \nabla_x \cdot \int_{\R^3} \frac{(|v|^2 - 5)}{3} \, v f \sqrt M \, \dv \\
&= -\frac{1}{\eps} \frac{2}{3} \, \nabla_x \cdot u[f]
- \frac{1}{\eps} \nabla_x \cdot \psi[f],
\end{aligned}
\end{equation}
and using some aforementioned classical results on the moments of~$M$, we also obtain
\begin{equation}\label{psiLf}
\begin{aligned}
\psi_k [\Lambda_\eps f] 
&= \frac{1}{\eps^2} \psi_k [L f^\perp] 
- \frac{1}{\eps} \psi_k [v \cdot \nabla_x f] \\
&= \frac{1}{\eps^2} \psi_k [L f^\perp] 
- \frac{1}{\eps} \partial_{x_\ell} \int_{\R^d} v_k v_\ell \, \frac{(|v|^2 - 5)}{3} \, f \sqrt M \, \dv \\
&= \frac{1}{\eps^2} \psi_k [L f^\perp] 
- \frac{1}{\eps} \partial_{x_\ell} \int_{\R^d} v_k v_\ell \, \frac{(|v|^2 - 5)}{3} \left( \rho[f] + u[f] \cdot v  + \theta[f] \frac{(|v|^2-3)}{2} \right)  M \, \dv \\
&\quad 
- \frac{1}{\eps} \partial_{x_\ell} \int_{\R^d} v_k v_\ell \, \frac{(|v|^2 - 5)}{3} \, f^\perp \sqrt M \, \dv\\
&= \frac{1}{\eps^2} \psi_k [L f^\perp] 
- \frac{1}{\eps} \frac{5}{3} \partial_{x_k} \theta[f]
- \frac{1}{\eps} \partial_{x_\ell} \int_{\R^d} v_k v_\ell \, \frac{(|v|^2 - 5)}{3} \, f^\perp \sqrt M \, \dv.
\end{aligned}
\end{equation}

We now estimate the term
$$
\begin{aligned}
\la \nabla_x (-\Delta_x)^{-1}\theta[\Lambda_\eps f] ,  \psi [f]  \ra_{L^2_x}
+ \la \nabla_x (-\Delta_x)^{-1}\theta[f] ,  \psi [\Lambda_\eps f]  \ra_{L^2_x}  =: I_1+I_2.
\end{aligned}
$$
We start by estimating $I_1$. 
First, using integrations by parts and~\eqref{thetaLf}, we have
\begin{align*}
\|\nabla_x (-\Delta_x)^{-1}\theta[\Lambda_\eps f]\|^2_{L^2_x}
&= \la \theta[\Lambda_\eps f] , (-\Delta_x)^{-1}\theta[\Lambda_\eps f] \ra_{L^2_x} \\
&= -\frac{1}{\eps} \la \frac{2}{3} \, \nabla_x \cdot u[f]
+ \nabla_x \cdot \psi[f], (-\Delta_x)^{-1}\theta[\Lambda_\eps f] \ra_{L^2_x} \\
&= -\frac{1}{\eps} \la \frac{2}{3} u[f] + \psi[f], \nabla_x (-\Delta_x)^{-1}\theta[\Lambda_\eps f] \ra_{L^2_x}.
\end{align*}
Using now Cauchy-Schwarz inequality in $x$ and~\eqref{psifperp}, we obtain:
$$
\|\nabla_x (-\Delta_x)^{-1}\theta[\Lambda_\eps f]\|_{L^2_x} 
\lesssim \frac1\eps \left(\|u[f]\|_{L^2_x}+\|\psi[f]\|_{L^2_x}\right).
$$
Then, from~\eqref{psifperp} and Cauchy-Schwarz inequality in velocity, we have:
$$
\|\psi[f]\|_{L^2_x} = \|\psi[f^\perp]\|_{L^2_x} \lesssim \|f^\perp\|_{L^2_{x,v}}.
$$
We deduce that 
$$
|I_1| \lesssim \frac{1}{\eps} \left(\|u[f]\|_{L^2_x}+\|f^\perp\|_{L^2_{x,v}}\right) \| f^\perp \|_{L^2_{x,v}}.
$$
Concerning $I_2$, from the computation~\eqref{psiLf}, we have:
$$
\begin{aligned}
I_2 
&= \frac{1}{\eps^2} \la \partial_{x_k} (-\Delta_x)^{-1}\theta[f] ,  \psi_k [L f^\perp]   \ra_{L^2_x}
- \frac{1}{\eps} \frac{5}{3} \la  \partial_{x_k}  (-\Delta_x)^{-1}\theta[f] ,  \partial_{x_k} \theta[f]  \ra_{L^2_x} \\
&\quad 
- \frac{1}{\eps} \la  \partial_{x_k}  (-\Delta_x)^{-1}\theta[f] ,  \partial_{x_\ell} \int_{\R^3} v_k v_\ell \, \frac{(|v|^2 - 5)}{3} \, f^\perp \sqrt M \, \dv  \ra_{L^2_x} \\
&=: I_{21}+I_{22}+ I_{23}.
\end{aligned}
$$
We treat each term of the above splitting in succession. For $I_{21}$, we notice first that from the self-adjointness of $L$ in $L^2_v$: 
	$$
	\|\psi[Lf^\perp]\|_{L^2_x}^2 \lesssim \int_{\T^3} \left(\int_{\R^3} L \left(v_k \frac{(|v|^2-5)}{3} \sqrt{M}\right) f^\perp \, \dv \right)^2 \, \dx
	$$
so that from Cauchy-Schwarz inequality:
	\begin{equation} \label{psiLfperp}
	\|\psi[Lf^\perp]\|_{L^2_x} \lesssim \|f^\perp\|_{L^2_{x,v}}. 
	\end{equation}
We thus have:
$$
\begin{aligned}
I_{21}
&\lesssim \frac{1}{\eps^2} \| \nabla_x (-\Delta_x)^{-1}\theta[f] \|_{L^2_x} \| f^\perp \|_{L^2_{x,v}}  
\lesssim \frac{1}{\eps^2} \| \theta[f] \|_{L^2_x} \| f^\perp \|_{L^2_{x,v}}
\end{aligned}
$$
where we used~\eqref{eq:Delta-1}. The term $I_{22}$ is computed exactly thanks to an integration by parts:
$$
\begin{aligned}
I_{22}
&=  \frac{1}{\eps} \frac{5}{3} \la \partial_{x_k} \partial_{x_k}  (-\Delta_x)^{-1}(\theta[f]) ,   \theta[f]  \ra_{L^2_x}
=  - \frac{1}{\eps} \frac{5}{3} \|  \theta[f] \|_{L^2_x}^2.
\end{aligned}
$$
We bound the last term $I_{23}$ using an integration by parts,~\eqref{eq:Delta-1} and Cauchy-Schwarz inequality:
$$
\begin{aligned}
I_{23}
&= \frac{1}{\eps} \la \partial_{x_\ell} \partial_{x_k}  (-\Delta_x)^{-1}(\theta[f]) ,   \int_{\R^d} v_k v_\ell \, \frac{(|v|^2 - 5)}{3} \, f^\perp \sqrt M \, \dv  \ra_{L^2_x} \\
&\lesssim \frac{1}{\eps} \| \nabla_x^2  (-\Delta_x)^{-1} \theta[f] \|_{L^2_x} \| f^\perp \|_{L^2_{x,v}} \\
&\lesssim \frac{1}{\eps} \| \theta[f] \|_{L^2_x} \| f^\perp \|_{L^2_{x,v}}.
\end{aligned}
$$
Finally, we have obtained for some $\kappa,C>0$:
\begin{equation}\label{energy}
\begin{aligned}
&\la \nabla_x (-\Delta_x)^{-1}\theta[\Lambda_\eps f] ,  \psi [f]  \ra_{L^2_x}
+ \la \nabla_x (-\Delta_x)^{-1}\theta[f] ,  \psi [\Lambda_\eps f]  \ra_{L^2_x} \\
&\qquad 
\le -\frac{\kappa}{\eps}  \| \theta[f] \|_{L^2_x}^2
+ \frac{C}{\eps} \| u[f] \|_{L^2_x} \| f^\perp \|_{L^2_{x,v}}
+ \frac{C}{\eps^{3}} \| f^\perp \|_{L^2_{x,v}}^2.
\end{aligned}
\end{equation}

\medskip
\noindent {\it Step 3. Momentum estimate.}
Define, for any $k,\ell \in \{1 , 2 , 3 \}$ and any suitable function $g$
\begin{equation}\label{def:Theta}
\Theta_{k \ell}[g] 
= \left\{ 
\begin{aligned}
& \frac{1}{7} \int_{\R^3} v_k v_\ell |v|^2 g \sqrt M \, \dv \quad &\text{if} \quad k \neq \ell ,\\
& \int_{\R^3} \left( \frac{1}{2} +  v_k^2 - \frac{1}{2} |v|^2 \right) g \sqrt M \, \dv 
\quad &\text{if} \quad k = \ell,
\end{aligned}
\right.
\end{equation}
and remark that, using some classical aforementioned results on the moments of~$M$, we have
\begin{equation}\label{Thetafperp0}
\Theta_{k\ell}[g] = \Theta_{k\ell}[g^\perp] \quad \text{if} \quad k \neq \ell
\end{equation}
and 
\begin{equation}\label{Thetafperp1}
\begin{aligned}
\Theta_{kk}[g] 
&= \Theta_{kk}[g^\perp] 
+ \int_{\R^3} \left( \frac{1}{2} +  v_k^2 - \frac{1}{2} |v|^2 \right) \left( \rho[g] + u[g] \cdot v  + \theta[g] \frac{(|v|^2-3)}{2} \right) M \,  \dv \\
&= \Theta_{kk}[g^\perp] 
-\frac{1}{2} \theta[g].
\end{aligned}
\end{equation}

We now compute for any $k \in \{1,2,3 \}$, writing $f = \pi f + f^\perp$,
\begin{equation}\label{jLf}
\begin{aligned}
u_k [\Lambda_\eps f] 
&=  \frac{1}{\eps^2} u_k [L f^\perp] - \frac{1}{\eps} u_k [v \cdot \nabla_x f] \\
&=  - \frac{1}{\eps} \partial_{x_\ell} \int_{\R^3} v_k v_{\ell} \left( \rho[f] + u[f] \cdot v  + \theta[f] \frac{(|v|^2-3)}{2} \right)  M \, \dv \\
&\qquad - \frac{1}{\eps} \partial_{x_\ell} \int_{\R^3} v_k v_{\ell} f^\perp \sqrt M \, \dv \\
&=  - \frac{1}{\eps} \partial_{x_k} \rho[f]
- \frac{1}{\eps} \partial_{x_k} \theta[f]
- \frac{1}{\eps} \partial_{x_\ell} \int_{\R^3} v_k v_\ell  f^\perp \sqrt M \, \dv.
\end{aligned}
\end{equation}
We also obtain
\begin{equation}\label{ThetaLf}
\begin{aligned}
\Theta_{k \ell} [\Lambda_\eps f] 
&=  \frac{1}{\eps^2} \Theta_{k \ell} [L f^\perp] 
- \frac{1}{\eps} \Theta_{k \ell} [v \cdot \nabla_x \pi f]
- \frac{1}{\eps} \Theta_{k \ell} [v \cdot \nabla_x f^\perp] \\
&= \frac{1}{\eps^2} \Theta_{k \ell} [L f^\perp] 
- \frac{1}{\eps} \Theta_{k \ell} \left[v \cdot \nabla_x \left(u[f] \cdot v \sqrt M\right)\right]
- \frac{1}{\eps} \Theta_{k \ell} [v \cdot \nabla_x f^\perp]  .
\end{aligned}
\end{equation}
If $k \neq \ell$, then
\begin{equation} \label{eq:Thetakl}
\begin{aligned}
&\Theta_{k \ell} \left[v \cdot \nabla_x \left(u[f] \cdot v \sqrt M\right)\right] \\
&\quad 
= \partial_{x_p} u_q [f] \frac{1}{7} \int_{\R^3} v_k v_\ell v_p v_q  |v|^2 M \, \dv  \\
&\quad 
= \partial_{x_k} u_\ell [f] \frac{1}{7} \int_{\R^3} v_k^2 v_\ell^2   |v|^2 M \, \dv
+ \partial_{x_\ell} u_k [f] \frac{1}{7} \int_{\R^3} v_k^2 v_\ell^2  |v|^2 M \, \dv  \\
&\quad 
= \partial_{x_k} u_\ell[f] + \partial_{x_\ell} u_k[f]
\end{aligned}
\end{equation}
because
$$
\int_{\R^3} v_k^2 v_\ell^2 |v|^2 M \, \dv  
= \int_{\R^3} v_k^2 v_\ell^2 (v_k^2 + v_\ell^2 + v_m^2) M \, \dv
=7.
$$
If $k = \ell$, then
\begin{equation} \label{eq:Thetakk}
\begin{aligned}
&\Theta_{kk}\left[v \cdot \nabla_x \left(u[f] \cdot v \sqrt M\right)\right]\\
&\quad 
= \partial_{x_p} u_q [f] \int_{\R^3} \left( \frac{1}{2} +  v_k^2 - \frac{1}{2} |v|^2 \right) v_p v_q M \, \dv \\
&\quad 
= \partial_{x_k} u_k [f] \int_{\R^3} \left( \frac{1}{2} +  v_k^2 - \frac{1}{2} |v|^2 \right) v_k^2 M \, \dv + \sum_{p \neq k} \partial_{x_p} u_p [f] \int_{\R^3} \left( \frac{1}{2} +  v_k^2 - \frac{1}{2} |v|^2 \right) v_p^2  M \, \dv \\
&\quad = \partial_{x_k} u_k[f]
- \sum_{p \neq k} \partial_{x_p} u_p[f],
\end{aligned}
\end{equation}
because
$$
\int_{\R^3} \left( \frac{1}{2} +  v_k^2 - \frac{1}{2} |v|^2 \right) v_k^2 M \, \dv = 1
\quad \text{and} \quad  
\int_{\R^3} \left( \frac{1}{2} +  v_k^2 - \frac{1}{2} |v|^2 \right) v_p^2  M \, \dv=-1.
$$

We now estimate the term
$$
\begin{aligned}
\la \partial_{x_\ell} (-\Delta_x)^{-1} u_k[\Lambda_\eps f] , \Theta_{k\ell} [f] \ra_{L^2_x}
+ \la \partial_{x_\ell} (-\Delta_x)^{-1} u_k[f] , \Theta_{k\ell} [\Lambda_\eps f] \ra_{L^2_x}  =: J_1 + J_2.
\end{aligned}
$$
For $J_1$, we first notice that using~\eqref{Thetafperp0} and~\eqref{Thetafperp1} and Cauchy-Schwarz inequality, one can prove that 
$$
\| \Theta_{k \ell} [f] \|_{L^2_x} \lesssim \| \theta[f] \|_{L^2_x} + \| f^\perp \|_{L^2_{x,v}} .
$$
Performing integrations by parts and using~\eqref{jLf}, we also have that 
\begin{align*}
&\|\nabla_x (-\Delta_x)^{-1} u_k[\Lambda_\eps f]\|^2_{L^2_x} 
= \la u_k[\Lambda_\eps f], (-\Delta_x)^{-1} u_k[\Lambda_\eps f]\ra_{L^2_x} \\
&\quad = - \frac{1}{\eps}\la \partial_{x_k} \rho[f]
+  \partial_{x_k} \theta[f]
+ \partial_{x_\ell} \int_{\R^3} v_k v_\ell  f^\perp \sqrt M \, \dv, 
(-\Delta_x)^{-1} u_k[\Lambda_\eps f]\ra_{L^2_x} \\
&\quad 
= \frac{1}{\eps}\la \rho[f] +  \theta[f] , 
 \partial_{x_k} (-\Delta_x)^{-1} u_k[\Lambda_\eps f]\ra_{L^2_x} \\
&\quad\quad 
+ \frac{1}{\eps}\la  \int_{\R^3} v_k v_\ell  f^\perp \sqrt M \, \dv, 
\partial_{x_\ell}(-\Delta_x)^{-1} u_k[\Lambda_\eps f]\ra_{L^2_x}.
\end{align*}
From Cauchy-Schwarz inequality (in $x$ and also in $v$ for the third term), we obtain:
$$
\|\nabla_x (-\Delta_x)^{-1} u_k[\Lambda_\eps f]\|_{L^2_x}  
\lesssim \frac1\eps \left(\|\rho[f]\|_{L^2_x} + \|\theta[f]\|_{L^2_x} + \|f^\perp\|_{L^2_{x,v}}\right). 
$$
Gathering the two previous estimates, by Cauchy-Schwarz inequality, we deduce
\begin{align*}
|J_1| &\lesssim \| (-\Delta_x)^{-1} u_k[\Lambda_\eps f]  \|_{H^1_x} \| \Theta_{k \ell} [f] \|_{L^2_x} \\
&\lesssim \frac{1}{\eps}  \left(\|\rho[f]\|_{L^2_x} + \|\theta[f]\|_{L^2_x} + \|f^\perp\|_{L^2_{x,v}}\right) 
\left( \| \theta[f] \|_{L^2_x} + \| f^\perp \|_{L^2_{x,v}} \right).
\end{align*}
For the term $J_2$, the computation made in~\eqref{ThetaLf} yields
$$
\begin{aligned}
J_2
&= \frac{1}{\eps^2} 
\la \partial_{x_\ell} (-\Delta_x)^{-1} u_k[f] , \Theta_{k \ell} [L f^\perp]\ra_{L^2_x} \\
&\quad 
-\frac{1}{\eps} 
\la \partial_{x_\ell} (-\Delta_x)^{-1} u_k[f] ,  \Theta_{k \ell} \left[v \cdot \nabla_x (u[f] \cdot v \sqrt M)\right] \ra_{L^2_x} \\
&\quad
-\frac{1}{\eps} 
\la \partial_{x_\ell} (-\Delta_x)^{-1} u_k[f] , \Theta_{k \ell} [v \cdot \nabla_x f^\perp] \ra_{L^2_x} \\
&=: J_{21} + J_{22} + J_{23}.
\end{aligned}
$$
To bound $J_{21}$, we first notice that as in~\eqref{psiLfperp}, one can prove that 
$$
\|\Theta_{k\ell} [Lf^\perp]\|_{L^2_x} \lesssim \|f^\perp\|_{L^2_{x,v}}.
$$
Then,~\eqref{eq:Delta-1} and Cauchy-Schwarz inequality give 
$$
\begin{aligned}
J_{21}
&\lesssim \frac{1}{\eps^2} \| \partial_{x_\ell} (-\Delta_x)^{-1}u_k[f] \|_{L^2_x} \| f^\perp \|_{L^2_{x,v}}  
\lesssim \frac{1}{\eps^2} \| u [f] \|_{L^2_x} \| f^\perp \|_{L^2_{x,v}}.  
\end{aligned}
$$
The term $J_{22}$ is computed explicitly thanks to~\eqref{eq:Thetakl} and~\eqref{eq:Thetakk} and integrations by parts:
$$
\begin{aligned}
J_{22}
&=  -\frac{1}{\eps} \sum_{k} \sum_{\ell \neq k}
\la \partial_{x_\ell} (-\Delta_x)^{-1} u_k[f] ,  \partial_{x_k} u_\ell[f] + \partial_{x_\ell} u_k[f] \ra_{L^2_x} \\
&\quad 
-\frac{1}{\eps} \sum_{k} 
\la \partial_{x_k} (-\Delta_x)^{-1} u_k[f] ,  \partial_{x_k} u_k[f] - \sum_{p \neq k} \partial_{x_p} u_p[f] \ra_{L^2_x} \\
&=  -\frac{1}{\eps} \sum_{k} \sum_{\ell \neq k}
\la -\partial_{x_k} \partial_{x_\ell} (-\Delta_x)^{-1} u_k[f] ,   u_\ell[f] \ra_{L^2_x} \\
&\quad -\frac{1}{\eps} \sum_{k} \sum_{\ell \neq k}
\la -\partial_{x_\ell} \partial_{x_\ell} (-\Delta_x)^{-1} u_k[f] ,   u_k[f] \ra_{L^2_x} \\
&\quad 
-\frac{1}{\eps} \sum_{k} 
\la - \partial_{x_k}\partial_{x_k} (-\Delta_x)^{-1} u_k[f] ,   u_k[f] \ra_{L^2_x} \\
&\quad +\frac{1}{\eps} \sum_{k} \sum_{p \neq k}
\la -\partial_{x_p}\partial_{x_k} (-\Delta_x)^{-1} u_k[f] ,  u_p[f] \ra_{L^2_x} 
\end{aligned}
$$
so that 
$$
\begin{aligned}
J_{22} &=  -\frac{1}{\eps} \sum_{k} \la - \Delta_x (-\Delta_x)^{-1} u_k[f] ,   u_k[f] \ra_{L^2_x} 
=  -\frac{1}{\eps} \sum_{k} \| u_k[f] \|_{L^2_x}^2 
= -\frac{1}{\eps}  \| u[f] \|_{L^2_x}^2 .
\end{aligned}
$$
The term $J_{23}$ is treated thanks to an integration by parts,~\eqref{eq:Delta-1} and Cauchy-Schwarz inequality:
$$
\begin{aligned}
J_{23}
&\lesssim \frac{1}{\eps} \| \nabla_x^2  (-\Delta_x)^{-1} u[f] \|_{L^2_x} \| f^\perp \|_{L^2_{x,v}} \\
&\lesssim \frac{1}{\eps} \| u[f] \|_{L^2_x} \| f^\perp \|_{L^2_{x,v}} .
\end{aligned}
$$
Finally, we have obtained for some constants $\kappa , C>0$
\begin{equation}\label{momentum}
\begin{aligned}
&\la \partial_{x_\ell} (-\Delta_x)^{-1} u_k[\Lambda_\eps f] , \Theta_{k\ell} [f] \ra_{L^2_x}
+ \la \partial_{x_\ell} (-\Delta_x)^{-1} u_k[f] , \Theta_{k\ell} [\Lambda_\eps f] \ra_{L^2_x} \\
&\qquad 
\le -\frac{\kappa}{\eps}  \| u[f] \|_{L^2_x}^2 
+ \frac{C}{\eps} \| \rho[f] \|_{L^2_x} \| \theta[f] \|_{L^2_x}  
+ \frac{C}{\eps} \| \rho[f] \|_{L^2_x} \| f^\perp\|_{L^2_{x,v}} \\
&\qquad\quad
+ \frac{C}{\eps} \| \theta[f] \|^2_{L^2_x}  
+ \frac{C}{\eps^3} \| f^\perp \|_{L^2_{x,v}}^2.
\end{aligned}
\end{equation}

\smallskip
\noindent{\it Step 4. Mass estimate.}
We first compute
\begin{equation}\label{rhoLf}
\begin{aligned}
\rho [\Lambda_\eps  f] 
&= \frac{1}{\eps^2} \rho [ L f^\perp] - \frac{1}{\eps} \rho[ v \cdot \nabla_x f] \\
&= - \frac{1}{\eps} \nabla_x \cdot \int_{\R^d} v f \sqrt M \, \dv \\
&= -  \frac{1}{\eps} \nabla_x \cdot u[f].
\end{aligned}
\end{equation}

We now estimate the term
$$
\begin{aligned}
\la \partial_{x_k} (-\Delta_x)^{-1}\rho[\Lambda_\eps f] ,  u_k [f]  \ra_{L^2_x}
+\la \partial_{x_k} (-\Delta_x)^{-1}\rho[f] ,  u_k [\Lambda_\eps f]  \ra_{L^2_x}  
=: R_1 + R_2.
\end{aligned}
$$
For $R_1$, from the computation~\eqref{rhoLf} and~\eqref{eq:Delta-1}, we have
$$
R_1 \lesssim \| (-\Delta_x)^{-1} \rho [\Lambda_\eps f]  \|_{H^1_x} \| u [f] \|_{L^2_x}
\lesssim \frac{1}{\eps}  \| u [f] \|_{L^2_x}^2.
$$
From~\eqref{jLf} we rewrite the term $R_2$ as
$$
\begin{aligned} 
R_2
&= - \frac{1}{\eps}\la \partial_{x_k} (-\Delta_x)^{-1}\rho[f] ,  \partial_{x_k} \rho[f]  \ra_{L^2_x}
- \frac{1}{\eps} \frac23 \la \partial_{x_k} (-\Delta_x)^{-1}\rho[f] ,  \partial_{x_k} \theta[f]  \ra_{L^2_x}\\
&\quad
- \frac{1}{\eps} \la \partial_{x_k} (-\Delta_x)^{-1}\rho[f] ,  \partial_{x_\ell} \int_{\R^d} v_k v_\ell  f^\perp \sqrt M \, \dv  \ra_{L^2_x} \\
&=: R_{21} + R_{22} + R_{23}.
\end{aligned}
$$
The first term $R_{21}$ is computed exactly thanks to an integration by parts:
$$
\begin{aligned} 
R_{21}
&= - \frac{1}{\eps}\la - \partial_{x_k} \partial_{x_k} (-\Delta_x)^{-1}\rho[f] ,   \rho[f]  \ra_{L^2_x}
= - \frac{1}{\eps} \| \rho[f] \|_{L^2_x}^2.
\end{aligned}
$$
The second one is estimated thanks to and integration by parts, Cauchy-Schwarz inequality and~\eqref{eq:Delta-1}:
$$
\begin{aligned} 
R_{22}
&\lesssim \frac{1}{\eps} \| \nabla_x^2 (-\Delta_x)^{-1}\rho_f \|_{L^2_x}  \|  \theta[f]  \|_{L^2_x} \\
&\lesssim \frac{1}{\eps} \| \rho[f] \|_{L^2_x}  \|  \theta[f]  \|_{L^2_x}.
\end{aligned}
$$
Similarly, we obtain 
$$
\begin{aligned} 
R_{23}
&\lesssim \frac{1}{\eps} \| \nabla_x^2 (-\Delta_x)^{-1}\rho[f] \|_{L^2_x}  \left\| \int_{\R^d} v_k v_\ell f^\perp \sqrt M \, \dv  \right\|_{L^2_x} \\
&\lesssim \frac{1}{\eps} \| \rho[f] \|_{L^2_x}  \|  f^\perp  \|_{L^2_{x,v}}.
\end{aligned}
$$
Gathering the previous estimates, we obtain for some constants $\kappa, C >0$
\begin{equation}\label{mass}
\begin{aligned}
&\la \partial_{x_k} (-\Delta_x)^{-1}\rho[\Lambda_\eps f] ,  u_k [f]  \ra_{L^2_x}
+\la \partial_{x_k} (-\Delta_x)^{-1}\rho[f] ,  u_k [\Lambda_\eps f]  \ra_{L^2_x} \\
&\qquad 
\le -\frac{\kappa}{\eps}  \| \rho[f] \|_{L^2_x}^2
+ \frac{C}{\eps} \| u[f] \|_{L^2_x}^2
+ \frac{C}{\eps} \| \theta[f] \|_{L^2_x}^2
+ \frac{C}{\eps} \| f^\perp \|_{L^2_{x,v}}^2.
\end{aligned}
\end{equation}

\medskip
\noindent{\it Step 5. Conclusion.}
We introduce the following inner product on $L^2_{x,v}$: 
\begin{equation}\label{def:Nt}
\begin{aligned}
\la \! \la f,g \ra \! \ra_{L^2_{x,v}}
&:= \la f, g \ra_{L^2_{x,v}}
+ \eta_1 \eps \la \partial_{x_k} (-\Delta_x)^{-1}\theta[f] ,  \psi_k [g]  \ra_{L^2_x}
+ \eta_1 \eps \la \partial_{x_k} (-\Delta_x)^{-1}\theta[g] ,  \psi_k [f]  \ra_{L^2_x} \\
&\quad
+ \eta_2 \eps \la \partial_{x_\ell} (-\Delta_x)^{-1}u_k[f] ,  \Theta_{k \ell} [g] \ra_{L^2_x}
+ \eta_2 \eps \la \partial_{x_\ell} (-\Delta_x)^{-1}u_k[g] ,  \Theta_{k \ell} [f] \ra_{L^2_x} \\
&\quad
+ \eta_3 \eps \la \partial_{x_k} (-\Delta_x)^{-1}\rho[f] ,  u_k [g]  \ra_{L^2_x}
+ \eta_3 \eps \la \partial_{x_k} (-\Delta_x)^{-1}\rho[g] ,  u_k [f]  \ra_{L^2_x}
\end{aligned}
\end{equation}
with constants $0 < \eta_3 \ll \eta_2 \ll \eta_1 \ll 1$ to be chosen below, and denote by $\Nt g \Nt_{L^2_{x,v}}^2 = \la\!\la g,g \ra\!\ra_{L^2_{x,v}}$ the associated norm.
We observe that 
$$
\| g \|_{L^2_{x,v}} \lesssim \Nt g \Nt_{L^2_{x,v}} \lesssim \| g \|_{L^2_{x,v}}
$$
where the multiplicative constants are uniform in $\eps \in (0,1]$. The norms $\| \cdot \|_{L^2_{x,v}}$ and~$\Nt \cdot \Nt_{L^2_{x,v}}$ are thus equivalent independently of $\eps \in (0,1]$. 
Gathering estimates \eqref{micro}--\eqref{energy}--\eqref{momentum}--\eqref{mass} and using Young's inequality, we obtain: 
\begin{align*}
\la\!\la \Lambda_\eps f , f \ra\!\ra_{L^2_{x,v}}
&\leq - \frac{1}{\eps^2} \left(\frac{\sigma_L}{2} - C \eta_1 - C \eta_2 - C \eta_3\right) \| f^\perp \|_{L^2_{x}(H^1_{v,*})}^2 \\
&\quad
 -\left(\frac{\kappa \eta_1}{2} - C \eta_2 - C \eta_3\right)  \| \theta[f] \|_{L^2_x}^2 \\
&\quad
- \left(\kappa \eta_2 - C \eta_1^2 -C \eta_3 \right) \| u[f] \|_{L^2_x}^2 \\
&\quad
-\left(\kappa \eta_3  - C \eta_2^2 - C \frac{\eta_2^2}{\eta_1}\right) \| \rho[f] \|_{L^2_x}^2. 
\end{align*}
By choosing 
$\eta_1 := \eta$, $\eta_2 := \eta^{\frac{3}{2}}$, $\eta_3 := \eta^{\frac{7}{4}}$ with $\eta >0$ small enough and recalling that $\| f \|_{L^2_{x,v}}^2 = \| f^\perp \|_{L^2_{x,v}}^2 + \| \rho[f] \|_{L^2_x}^2 + \| u[f] \|_{L^2_x}^2 + \| \theta[f] \|_{L^2_x}^2$, we thus obtain
\begin{align*}
\la\!\la \Lambda_\eps f , f \ra\!\ra_{L^2_{x,v}}
&\leq 
- \sigma \Nt f \Nt_{L^2_{x,v}}^2
- \frac{\kappa}{\eps^2} \| f^\perp \|_{L^2_{x}(H^1_{v,*})}^2 
\end{align*}
for some constants $\sigma, \kappa>0$, which completes the proof.
\qed

\section{Proofs of Lemmas~\ref{lem:regSBeps} and~\ref{lem:regSBeps2}} \label{app:Beps}

\begin{proof}[Proof of Lemma~\ref{lem:regSBeps}]
Let us first recall that $\BB_\eps=\eps^{-2} \BB - \eps^{-1} v \cdot \nabla_x$ with $\BB$ and $v \cdot \nabla_x$ that are respectively self-adjoint and skew-adjoint operators in $L^2_{x,v}$. Proving that~$S_{\BB_\eps}$ regularizes from $\YYY_1'$ to $\XXX$ and from $(\ZZZ_1^\eps)'$ to $\XXX$ is similar to prove that~$S_{\BB_\eps}$ regularizes from~$\XXX$ to $\YYY_1$ and from~$\XXX$ to $\ZZZ_1^\eps$ with same rates. We will only focus on the proof of~\eqref{eq:regSBeps0} and explain the adaptation to make to prove~\eqref{eq:regSBeps0bis} in the first step of the proof.

Moreover, to prove \eqref{eq:regSBeps0}, it is sufficient to prove that for any $\alpha \in \R$ and $\langle v \rangle^\alpha f_{\rm in} \in L^2_{x,v}$, one has: 
For any $t \in (0,\eps^2]$,
\begin{equation}\label{eq:regSBeps1}
\|\langle v \rangle^\alpha S_{\BB_\eps}(t) f_{\rm in}\|_{L^2_{x}(H^1_{v,*})}  \lesssim \frac{\eps}{\sqrt{t}}  \| \langle v \rangle^\alpha f_{\rm in}\|_{L^2_{x,v}}
\end{equation}
and 
\begin{multline}\label{eq:regSBeps1bis}
\|\langle v \rangle^\alpha S_{\BB_\eps}(t) f_{\rm in}\|_{L^2_{x,v}} 
+ \|\langle v \rangle^\alpha S_{\BB_\eps}(t) f_{\rm in}\|_{L^2_{x}(H^1_{v,*})}  
+ \eps \| \widetilde \nabla_x ( \langle v \rangle^\alpha S_{\BB_\eps}(t) f_{\rm in})\|_{L^2_{x,v}} \\
\lesssim \frac{\eps^3}{t^{3/2}} \| \langle v \rangle^\alpha f_{\rm in}\|_{L^2_{x,v}}.
\end{multline}
Indeed, since $\nabla_x$ commutes with $\BB_\eps$, from estimate \eqref{eq:regSBeps1} (resp.\ \eqref{eq:regSBeps1bis}), we already obtain the first (resp.\ second) estimate of \eqref{eq:regSBeps0} for $t \in (0, \eps^2]$. Then, fix $\sigma \in (0,\sigma_1)$. We obtain the estimates given in \eqref{eq:regSBeps0} for all~$t >0$ by using the exponential decay of $S_{\BB_\eps}$ in $\XXX$ given in Lemma~\ref{lem:dissipativeBeps}. More precisely, using that for $t \geq \eps^2$, $S_{\BB_\eps}(t) = S_{\BB_\eps}(\eps^2) S_{\BB_\eps}(t-\eps^2)$, we obtain that for $\sigma' \in (\sigma,\sigma_1)$ and for any $t>0$, 
	$$
	\|S_{\BB_\eps}(t)\|_{\XXX \to \YYY_1} \lesssim \frac{\eps}{\min(\eps,\sqrt{t})} \, e^{-\sigma' t/\eps^2} 
	\quad \text{and} \quad 
	\|S_{\BB_\eps}(t)\|_{\XXX \to \ZZZ_1^\eps} \lesssim \frac{\eps^3}{\min(\eps^3,t^{3/2})} \, e^{-\sigma' t/\eps^2}. 
	$$
It implies the wanted conclusion using that $\sigma'>\sigma$.

\smallskip

For the remainder of the proof, we let $\alpha \in \R$ and $f_{\rm in}$ be such that $\langle v \rangle^\alpha f_{\rm in} \in L^2_{x,v}$, and consider the solution $f(t) = S_{\BB_\eps}(t) f_{\rm in}$ to the equation $\partial_t f = \BB_\eps f$ with initial data $f_{\rm in}$, and we shall prove \eqref{eq:regSBeps1} and \eqref{eq:regSBeps1bis}. 


\medskip
\noindent {\it Step 1.} Define the functional 
\begin{multline} \label{eq:Lyapunov2}
\mathcal E_{\varepsilon}(t) 
:= \| \langle v \rangle^\alpha f \|_{L^2_{x,v}}^2 + \alpha_1\frac{t}{\varepsilon^2}  \Big( K \|\langle v \rangle^\alpha \langle v \rangle^{\frac{\gamma}{2}+1} f  \|_{L^2_{x,v}}^2 + \| \langle v \rangle^\alpha \widetilde{\nabla}_v f \|_{L^2_{x,v}}^2  \Big) \\
+ \varepsilon\alpha_2 \left(\frac{t}{\varepsilon^2}\right)^2 \la \langle v \rangle^\alpha \widetilde{\nabla}_v f , \langle v \rangle^\alpha \widetilde{\nabla}_x f \ra_{L^2_{x,v}} \\
+ \varepsilon^2\alpha_3 \left(\frac{t}{\varepsilon^2}\right)^3 \Big( \| \langle v \rangle^\alpha \widetilde{\nabla}_x f \|_{L^2_{x,v}}^2 
+ K\| \langle v \rangle^\alpha \langle v \rangle^{\frac{\gamma}{2}} \nabla_x f \|_{L^2_{x,v}}^2 \Big),
\end{multline}
where $\alpha_1, \alpha_2 , \alpha_3, K >0$ are positive constants  such that $0<\alpha_{3}\ll \alpha_{2}\ll \alpha_{1}\ll 1$ and $\alpha_{2}\leq \sqrt{\alpha_{1}\alpha_{3}}$. Notice that in order to prove~\eqref{eq:regSBeps0bis}, one just has to change the sign in front of the term which mixes derivatives in $x$ and $v$ in the definition of the functional $\EE_\eps$. It will allow to conclude that~\eqref{eq:regSBeps0} holds for the adjoint of $\BB_\eps$ instead of $\BB_\eps$ and thus imply~\eqref{eq:regSBeps0bis}. The constants $\alpha_i$ will be chosen small enough in Step 6, and $K$ will be chosen large enough in Step 3 and Step 5. We remark that 
$$
\| \langle v \rangle^\alpha f \|_{L^2_{x}(H^1_{v,*})}^2 \lesssim
K\|\langle v \rangle^\alpha \langle v \rangle^{\frac{\gamma}{2}+1}  f  \|_{L^2_{x,v}}^2 + \| \langle v \rangle^\alpha \widetilde{\nabla}_v f \|_{L^2_{x,v}}^2 
\lesssim \| \langle v \rangle^\alpha f \|_{L^2_{x}(H^1_{v,*})}^2
$$
and 
$$
\| \langle v \rangle^\alpha \widetilde{\nabla}_x f \|_{L^2_{x,v}}^2
\lesssim
\| \langle v \rangle^\alpha \widetilde{\nabla}_x f \|_{L^2_{x,v}}^2 
+ K\| \langle v \rangle^\alpha \langle v \rangle^{\frac{\gamma}{2}} \nabla_x f \|_{L^2_{x,v}}^2
\lesssim \| \langle v \rangle^\alpha \widetilde{\nabla}_x f \|_{L^2_{x,v}}^2 .
$$
Therefore, we can already observe that for any $t \in (0,\eps^2]$, one has the following lower bounds
$$
\frac{t}{\eps^2} \| \langle v \rangle^\alpha f \|_{L^2_{x}(H^1_{v,*})}^2 \lesssim \EE_\eps(t) 
$$
and 
$$
\left(\frac{t}{\eps^2}\right)^3  \left( \| \langle v \rangle^\alpha f \|_{L^2_{x,v}}^2 + \| \langle v \rangle^\alpha f \|_{L^2_{x}(H^1_{v,*})}^2 
+ \eps^2 \| \langle v \rangle^\alpha \widetilde{\nabla}_x f \|_{L^2_{x,v}}^2  \right) \lesssim \EE_\eps(t) .
$$
Therefore, in order to prove \eqref{eq:regSBeps1} and \eqref{eq:regSBeps1bis}, it is sufficient to prove that $\frac{\mathrm{d}}{\mathrm{d}t}\mathcal E_{\varepsilon}(t) \le 0$ for all $t \in (0,\eps^2]$.
We then compute 
\begin{align*}
\frac{\mathrm{d}}{\mathrm{d}t}\mathcal E_{\varepsilon}(t)&=\frac{\mathrm{d}}{\mathrm{d}t}\Vert \langle v \rangle^\alpha f \Vert_X^2 
+\frac{\alpha_{1}}{\varepsilon^2}\Big( K\|\langle v \rangle^\alpha \langle v \rangle^{\frac{\gamma}{2}+1}  f  \|_{L^2_{x,v}}^2 
+ \| \langle v \rangle^\alpha \widetilde{\nabla}_v f \|_{L^2_{x,v}}^2 \Big) \\
&+\frac{\alpha_{1}}{\varepsilon^2}t \frac{\mathrm{d}}{\mathrm{d}t}\Big( K\|\langle v \rangle^\alpha \langle v \rangle^{\frac{\gamma}{2}+1}  f  \|_{L^2_{x,v}}^2 + \| \langle v \rangle^\alpha \widetilde{\nabla}_v f \|_{L^2_{x,v}}^2  \Big)\\
&+2\frac{\alpha_{2}}{\varepsilon^3} t\la \langle v \rangle^\alpha \widetilde{\nabla}_v f , \langle v \rangle^\alpha \widetilde{\nabla}_x f \ra_{L^2_{x,v}}+\frac{\alpha_{2}}{\varepsilon^3}t^2  \frac{\mathrm{d}}{\mathrm{d}t}\la \langle v \rangle^\alpha \widetilde{\nabla}_v f , \langle v \rangle^\alpha \widetilde{\nabla}_x f \ra_{L^2_{x,v}}\\
&+3\frac{\alpha_{3}}{\varepsilon^4}t^{2} \Big(\| \langle v \rangle^\alpha \widetilde{\nabla}_x f \|_{L^2_{x,v}}^2
+ K\| \langle v \rangle^\alpha \langle v \rangle^{\frac{\gamma}{2}} \nabla_x f \|_{L^2_{x,v}}^2 \Big) \\
&+\frac{\alpha_{3}}{\varepsilon^4}t^{3}\frac{\mathrm{d}}{\mathrm{d}t} \Big(\| \langle v \rangle^\alpha \widetilde{\nabla}_x f \|_{L^2_{x,v}}^2
+ K\| \langle v \rangle^\alpha \langle v \rangle^{\frac{\gamma}{2}} \nabla_x f \|_{L^2_{x,v}}^2 \Big)
\end{align*}
and we shall estimate each term separately in the sequel.

\medskip
\noindent {\it Step 2.} From Lemma~\ref{lem:dissipativeBeps}, we already have
\begin{equation}\label{eq:f}
\begin{aligned}
\frac12\frac{\d}{\dt} \| \langle v \rangle^\alpha f \|_{L^2_{x,v}}^2
&\le - \frac{\kappa}{\eps^2} \| \langle v \rangle^\alpha f \|_{L^2_{x}(H^1_{v,*})}^2
\end{aligned}
\end{equation}
for some constant $\kappa>0$.

\medskip
\noindent {\it Step 3.} We prove in this step that, choosing $K>0$ large enough, we have
\begin{equation}\label{eq:DvB}
\begin{aligned}
&\frac12 \frac{\d}{\dt} \Big\{ K \|\langle v \rangle^\alpha \langle v \rangle^{\frac{\gamma}{2}+1}  f  \|_{L^2_{x,v}}^2 + \| \langle v \rangle^\alpha \widetilde{\nabla}_v f \|_{L^2_{x,v}}^2 \Big\} \\
&\qquad 
\le -\frac{\kappa}{2\eps^2} \| \langle v \rangle^\alpha f \|_{L^2_{x}(H^2_{v,*})}^2 
+ \frac{C}{\eps} \| \langle v \rangle^\alpha f  \|_{L^2_{x}(H^1_{v,*})} \| \langle v \rangle^\alpha \widetilde \nabla_x f\|_X ,
\end{aligned}
\end{equation}
for some constants $\kappa,C >0$.
From Lemma~\ref{lem:dissipativeBeps}, we already have
$$
\frac12\frac{\d}{\dt} \| \langle v \rangle^\alpha \langle v \rangle^{\frac{\gamma}{2}+1}  f \|_{L^2_{x,v}}^2
\le - \frac{\kappa}{\eps^2} \| \langle v \rangle^\alpha \langle v \rangle^{\frac{\gamma}{2}+1}  f \|_{L^2_{x}(H^1_{v,*})}^2
$$
for some $\kappa >0$.
Moreover, we have 
$$
\begin{aligned}
\frac12 \frac{\d}{\dt} \| \langle v \rangle^\alpha \widetilde \nabla_v f \|_{L^2_{x,v}}^2
&= \la \langle v \rangle^{2\alpha} \widetilde \nabla_{v_i} f , \widetilde \nabla_{v_i} (\BB_\eps f) \ra_{L^2_{x,v}} \\
&= \la \langle v \rangle^{2\alpha} \widetilde \nabla_{v_i} f , \BB_\eps (\widetilde \nabla_{v_i} f) \ra_{L^2_{x,v}}
+\la \langle v \rangle^{2\alpha} \widetilde \nabla_{v_i} f , [\widetilde \nabla_{v_i}, \BB_\eps] f \ra_{L^2_{x,v}}.
\end{aligned}
$$
Thanks to Lemma~\ref{lem:dissipativeBeps}, there holds 
$$
\begin{aligned}
\la \langle v \rangle^{2\alpha} \widetilde \nabla_{v_i} f , \BB_\eps (\widetilde \nabla_{v_i} f) \ra_{L^2_{x,v}}
&\le - \frac{\kappa}{\eps^2} \|\langle v \rangle^\alpha \widetilde \nabla_v f \|_{L^2_{x}(H^1_{v,*})}^2,
\end{aligned}
$$
for some constant $\kappa>0$. From Lemma~\ref{lem:commutatorBeps}-(i), we have
$$
\begin{aligned}
&\la \langle v \rangle^{2\alpha} \widetilde \nabla_{v_i} f , [\widetilde \nabla_{v_i}, \BB_\eps] f \ra_{L^2_{x,v}}\\
&\quad 
= - \frac{1}{\eps^2} \la \widetilde \nabla_{v_j}(\langle v \rangle^{2\alpha}  \widetilde \nabla_{v_i} f) ,   [\widetilde \nabla_{v_i},  \widetilde \nabla_{v_j}] f \ra_{L^2_{x,v}}
- \frac{1}{\eps^2} \la \widetilde \nabla_{v_j}^* (\langle v \rangle^{2\alpha} \widetilde \nabla_{v_i} f) , 
[\widetilde \nabla_{v_i},  \widetilde \nabla_{v_j}^*]  f \ra_{L^2_{x,v}}
\\
&\quad\quad
- \frac{1}{\eps^2}  \la  \langle v \rangle^{2\alpha} \widetilde \nabla_{v_i} f ,  (\widetilde \nabla_{v_i} m^2) f  \ra_{L^2_{x,v}}
- \frac{1}{\eps^2} \la \langle v \rangle^{2\alpha} \widetilde \nabla_{v_i} f , 
\left[ [\widetilde \nabla_{v_i},  \widetilde \nabla_{v_j}^*] , \widetilde \nabla_{v_j} \right] f \ra_{L^2_{x,v}}\\
&\quad\quad
- \frac{1}{\eps} \la \langle v \rangle^{2\alpha} \widetilde \nabla_{v_i} f , \widetilde \nabla_{x_i} f  \ra_{L^2_{x,v}}
\\
&\quad
=: -\frac{1}{\eps^2} (T_1 + T_2 + T_3 + T_4) 
- \frac{1}{\eps} \la \langle v \rangle^{2\alpha} \widetilde \nabla_{v_i} f , \widetilde \nabla_{x_i} f  \ra_{L^2_{x,v}}.
\end{aligned} 
$$
Writing $
\widetilde \nabla_{v_j}(\langle v \rangle^{2\alpha}  \widetilde \nabla_{v_i} f)
= \langle v \rangle^\alpha \widetilde \nabla_{v_j}(\langle v \rangle^\alpha  \widetilde \nabla_{v_i} f) + (\widetilde \nabla_{v_j} \langle v \rangle^\alpha) \langle v \rangle^\alpha \widetilde \nabla_{v_i} f
$, using that $| \widetilde \nabla_{v_j} \langle v \rangle^\alpha| \lesssim \langle v \rangle^{\frac{\gamma}{2}-1} \langle v \rangle^\alpha$ and thanks to Lemma~\ref{lem:commutator}-(iii), we obtain 
$$
\begin{aligned}
& T_1
\le \frac{\kappa}{8} \| \langle v \rangle^\alpha  \widetilde \nabla_v f \|_{L^2_{x}(H^1_{v,*})}^2   + C \|\langle v \rangle^\alpha \langle v \rangle^{\gamma+1} \nabla_v f \|_{L^2_{x,v}}^2 .
\end{aligned}
$$
In a similar way, observing now that 
$$
\widetilde \nabla_{v_j}^* (\langle v \rangle^{2\alpha}  \widetilde \nabla_{v_i} f)
= -\langle v \rangle^\alpha \widetilde \nabla_{v_j}(\langle v \rangle^\alpha  \widetilde \nabla_{v_i} f) 
- \left[ (\partial_{v_\ell} B_{j\ell}) \langle v \rangle^\alpha + (\widetilde \nabla_{v_j} \langle v \rangle^\alpha) \right] \langle v \rangle^\alpha \widetilde \nabla_{v_i} f
$$
and using Lemma~\ref{lem:commutator}-(iv), we get
$$
\begin{aligned}
& T_2
\le \frac{\kappa}{8} \| \langle v \rangle^\alpha  \widetilde \nabla_v f \|_{L^2_{x}(H^1_{v,*})}^2   
+C \|\langle v \rangle^\alpha \langle v \rangle^{\gamma} f \|_{L^2_{x,v}}^2
+ C \|\langle v \rangle^\alpha \langle v \rangle^{\gamma+1} \nabla_v f \|_{L^2_{x,v}}^2 .
\end{aligned}
$$
Using that 
$|\widetilde \nabla_{v_i} m^2 | \lesssim \langle v \rangle^{\frac{3\gamma}{2}+2}$ from Lemma~\ref{lem:malpha}, we also get  
$$
\begin{aligned}
T_3
&\le \frac{\kappa}{8} \| \langle v \rangle^\alpha  \widetilde \nabla_v f \|_{L^2_{x}(H^1_{v,*})}^2   + C \|\langle v \rangle^\alpha \langle v \rangle^{\gamma+1} f \|_{L^2_{x,v}}^2 .
\end{aligned}
$$
For the term $T_4$, we use Lemma~\ref{lem:commutator}-(vii) to obtain
$$
T_4 \le \frac{\kappa}{8} \| \langle v \rangle^\alpha  \widetilde \nabla_v f \|_{L^2_{x}(H^1_{v,*})}^2   
+C \|\langle v \rangle^\alpha \langle v \rangle^{\gamma-1} f \|_{L^2_{x,v}}^2
+ C \|\langle v \rangle^\alpha \langle v \rangle^{\gamma} \nabla_v f \|_{L^2_{x,v}}^2.
$$
Hence, we obtain
$$
\begin{aligned}
\frac12 \frac{\d}{\dt} \| \langle v \rangle^\alpha \widetilde \nabla_v f \|_{L^2_{x,v}}^2
&\le - \frac{\kappa}{2\eps^2} \|\langle v \rangle^\alpha \widetilde \nabla_v f \|_{L^2_{x}(H^1_{v,*})}^2 
+ \frac{C}{\eps^2} \|\langle v \rangle^\alpha \langle v \rangle^{\gamma+1} f \|_{L^2_{x,v}}^2\\
&\quad
+ \frac{C}{\eps^2} \|\langle v \rangle^\alpha \langle v \rangle^{\gamma+1} \nabla_v f \|_{L^2_{x,v}}^2 
+ \frac{C}{\eps} \| \langle v \rangle^\alpha \widetilde \nabla_v f \|_{L^2_{x,v}} \| \langle v \rangle^\alpha \widetilde \nabla_x f \|_{L^2_{x,v}}.
\end{aligned}
$$
We conclude to \eqref{eq:DvB} by gathering previous estimates, observing that $\|\langle v \rangle^\alpha \langle v \rangle^{\gamma+1} f \|_{L^2_{x,v}} + \| \langle v \rangle^\alpha \langle v \rangle^{\gamma+1} \nabla_v f \|_{L^2_{x,v}} \lesssim \| \langle v \rangle^\alpha \langle v \rangle^{\frac{\gamma}{2}+1}  f \|_{L^2_{x}(H^1_{v,*})}$ and taking $K>0$ large enough.

\medskip
\noindent {\it Step 4.} 
In this step, we show that 
\begin{equation}\label{eq:DvBDxB}
\frac{\d}{\dt} \la \langle v \rangle^\alpha \widetilde \nabla_v f , \langle v \rangle^\alpha \widetilde \nabla_x f \ra_{L^2_{x,v}}
\le -\frac{1}{\eps} \| \langle v \rangle^\alpha \widetilde \nabla_{x} f \|_{L^2_{x,v}}^2
+ \frac{C}{\eps^2} \| \langle v \rangle^\alpha f \|_{L^2_{x}(H^2_{v,*})} \| \langle v \rangle^\alpha \widetilde \nabla_x f \|_{L^2_{x}(H^1_{v,*})},
\end{equation}
for some constant $C>0$. We compute, using \eqref{eq:Beps},
$$
\begin{aligned}
&\frac{\d}{\dt} \la \langle v \rangle^\alpha \widetilde \nabla_v f , \langle v \rangle^\alpha \widetilde \nabla_x f \ra_{L^2_{x,v}} 
= \la \langle v \rangle^{2\alpha} \widetilde \nabla_{v_i} f , \widetilde \nabla_{x_i} \left\{ \frac{1}{\eps^2} \BB f - \frac{1}{\eps} v \cdot \nabla_x f   \right\} \ra_{L^2_{x,v}} \\
&\qquad 
+\la \langle v \rangle^{2\alpha} \widetilde \nabla_{x_i} f , \widetilde \nabla_{v_i} \left\{ \frac{1}{\eps^2}\BB f - \frac{1}{\eps} v \cdot \nabla_x f   \right\} \ra_{L^2_{x,v}} \\
&\quad 
= -\frac{1}{\eps^2}\la \widetilde \nabla_{x_i} (\langle v \rangle^{2\alpha} \widetilde \nabla_{v_i} f) ,   \BB f \ra_{L^2_{x,v}}
+\frac{1}{\eps^2} \la \widetilde \nabla_{v_i}^*(\langle v \rangle^{2\alpha} \widetilde \nabla_{x_i} f ),  \BB  f  \ra_{L^2_{x,v}}
-\frac{1}{\eps} \| \langle v \rangle^\alpha \widetilde \nabla_{x} f \|_{L^2_{x,v}}^2,
\end{aligned}
$$
where we have used Lemma~\ref{lem:commutator}-(i).
We hence conclude to \eqref{eq:DvBDxB} by using the fact that $\| \langle v \rangle^\alpha \BB f  \|_{L^2_{x,v}} \lesssim \| \langle v \rangle^\alpha f \|_{L^2_{x}(H^2_{v,*})}$ (see~\eqref{eq:boundL} for a similar estimate) together with estimates~\eqref{nablaomega2h} and \eqref{nablaxomega2}.

\medskip
\noindent {\it Step 5.} We prove in this step that, choosing $K>0$ large enough, we have
\begin{equation}\label{eq:DxB}
\frac12 \frac{\d}{\dt} \Big\{ \| \langle v \rangle^\alpha \widetilde \nabla_x f \|_{L^2_{x,v}}^2 + K\| \langle v \rangle^\alpha \langle v \rangle^{\frac{\gamma}{2}} \nabla_x f \|_{L^2_{x,v}}^2  \Big\}  
\le -\frac{\kappa}{2\eps^2} \| \langle v \rangle^\alpha \widetilde \nabla_x f \|_{L^2_{x}(H^1_{v,*})}^2 ,
\end{equation}
for some constants $\kappa,C >0$. We first remark that, since $\nabla_x $ commutes with $\BB_\eps$, we already have from Lemma~\ref{lem:dissipativeBeps} that 
$$
\frac12 \frac{\d}{\dt} \| \langle v \rangle^\alpha \langle v \rangle^{\frac{\gamma}{2}} \nabla_x f \|_{L^2_{x,v}}^2
\le - \frac{\kappa}{\eps^2} \| \langle v \rangle^\alpha \langle v \rangle^{\frac{\gamma}{2}} \nabla_x f \|_{L^2_{x}(H^1_{v,*})}^2
$$
for some constant $\kappa>0$.

We then write
$$
\begin{aligned}
\frac12 \frac{\d}{\dt} \| \langle v \rangle^\alpha \widetilde \nabla_x f \|_{L^2_{x,v}}^2
&= \la \langle v \rangle^{2\alpha} \widetilde \nabla_{x_i} f , \widetilde \nabla_{x_i} (\BB_\eps f) \ra_{L^2_{x,v}} \\
&= \la \langle v \rangle^{2\alpha} \widetilde \nabla_{x_i} f , \BB_\eps (\widetilde \nabla_{x_i} f) \ra_{L^2_{x,v}}
+\la \langle v \rangle^{2\alpha} \widetilde \nabla_{x_i} f , [\widetilde \nabla_{x_i}, \BB_\eps] f \ra_{L^2_{x,v}}.
\end{aligned}
$$
Thanks to Lemma~\ref{lem:dissipativeBeps}, there holds 
$$
\begin{aligned}
\la \langle v \rangle^{2\alpha} \widetilde \nabla_{x_i} f , \BB_\eps (\widetilde \nabla_{x_i} f) \ra_{L^2_{x,v}}
&\le - \frac{\kappa}{\eps^2} \|\langle v \rangle^\alpha \widetilde \nabla_x f \|_{L^2_{x}(H^1_{v,*})}^2,
\end{aligned}
$$
for some constant $\kappa>0$. From Lemma~\ref{lem:commutatorBeps}-(ii), we have
$$
\begin{aligned}
&\la \langle v \rangle^{2\alpha} \widetilde \nabla_{x_i} f , [\widetilde \nabla_{x_i}, \BB_\eps] f \ra_{L^2_{x,v}}\\
&\quad 
= - \frac{1}{\eps^2} \la \widetilde \nabla_{v_j}(\langle v \rangle^{2\alpha}  \widetilde \nabla_{x_i} f) ,   [\widetilde \nabla_{x_i},  \widetilde \nabla_{v_j}] f \ra_{L^2_{x,v}}
- \frac{1}{\eps^2} \la \widetilde \nabla_{v_j}^* (\langle v \rangle^{2\alpha} \widetilde \nabla_{x_i} f) , 
[\widetilde \nabla_{x_i},  \widetilde \nabla_{v_j}^*] f \ra_{L^2_{x,v}}\\
&\quad\quad
- \frac{1}{\eps^2} \la \langle v \rangle^{2\alpha} \widetilde \nabla_{x_i} f , \left[ [\widetilde \nabla_{x_i},  \widetilde \nabla_{v_j}^*] ,\widetilde \nabla_{v_j} \right]f  \ra_{L^2_{x,v}}\\
&\quad 
=: - \frac{1}{\eps^2} (R_1+R_2 + R_3).
\end{aligned}
$$
Writing $
\widetilde \nabla_{v_j}(\langle v \rangle^{2\alpha}  \widetilde \nabla_{x_i} f)
= \langle v \rangle^\alpha \widetilde \nabla_{v_j}(\langle v \rangle^\alpha  \widetilde \nabla_{x_i} f) + (\widetilde \nabla_{v_j} \langle v \rangle^\alpha) \langle v \rangle^\alpha \widetilde \nabla_{x_i} f
$, using the fact that $| \widetilde \nabla_{v_j} \langle v \rangle^\alpha| \lesssim \langle v \rangle^{\frac{\gamma}{2}-1} \langle v \rangle^\alpha$ and thanks to Lemma~\ref{lem:commutator}-(v), we obtain 
$$
\begin{aligned}
& R_1
\le \frac{\kappa}{6} \| \langle v \rangle^\alpha  \widetilde \nabla_x f \|_{L^2_{x}(H^1_{v,*})}^2   + C \|\langle v \rangle^\alpha \langle v \rangle^{\gamma+1} \nabla_x f \|_{L^2_{x,v}}^2 .
\end{aligned}
$$
In a similar way, observing now that 
$$
\widetilde \nabla_{v_j}^* (\langle v \rangle^{2\alpha}  \widetilde \nabla_{x_i} f)
= -\langle v \rangle^\alpha \widetilde \nabla_{v_j}(\langle v \rangle^\alpha  \widetilde \nabla_{x_i} f) 
- \left[ (\partial_{v_\ell} B_{j\ell}) \langle v \rangle^\alpha + (\widetilde \nabla_{v_j} \langle v \rangle^\alpha) \right] \langle v \rangle^\alpha \widetilde \nabla_{x_i} f,
$$
we get
$$
\begin{aligned}
& R_2
\le \frac{\kappa}{6} \| \langle v \rangle^\alpha  \widetilde \nabla_x f \|_{L^2_{x}(H^1_{v,*})}^2   
+ C \|\langle v \rangle^\alpha \langle v \rangle^{\gamma+1} \nabla_x f \|_{L^2_{x,v}}^2 .
\end{aligned}
$$
For the term $R_3$, we use Lemma~\ref{lem:commutator}-(viii) to obtain
$$
R_3 \le \frac{\kappa}{6} \| \langle v \rangle^\alpha  \widetilde \nabla_x f \|_{L^2_{x}(H^1_{v,*})}^2   
+ C \|\langle v \rangle^\alpha \langle v \rangle^{\gamma} \nabla_x f \|_{L^2_{x,v}}^2.
$$
Hence, we obtain
$$
\begin{aligned}
\frac12 \frac{\d}{\dt} \| \langle v \rangle^\alpha \widetilde \nabla_x f \|_{L^2_{x,v}}^2
&\le - \frac{\kappa}{2\eps^2} \|\langle v \rangle^\alpha \widetilde \nabla_x f \|_{L^2_{x}(H^1_{v,*})}^2 
+ \frac{C}{\eps^2} \|\langle v \rangle^\alpha \langle v \rangle^{\gamma+1} \nabla_x f \|_{L^2_{x,v}}^2 .
\end{aligned}
$$

We conclude to \eqref{eq:DxB} by gathering previous estimates as well as noticing the fact that $\| \langle v \rangle^\alpha \langle v \rangle^{\gamma+1} \nabla_x f \|_{L^2_{x,v}} \lesssim \| \langle v \rangle^\alpha \langle v \rangle^{\frac{\gamma}{2}} \nabla_x f \|_{L^2_{x}(H^1_{v,*})}$  and taking $K>0$ large enough.

\medskip
\noindent {\it Step 6. Proof of \eqref{eq:regSBeps1} and \eqref{eq:regSBeps1bis}.} Gathering \eqref{eq:f}--\eqref{eq:DvB}--\eqref{eq:DvBDxB}--\eqref{eq:DxB}, we obtain
\begin{align*}
\frac{\mathrm{d}}{\mathrm{d}t}\mathcal E_{\varepsilon}(t)
&\le - \frac{2\kappa}{\eps^2} \| \langle v \rangle^\alpha f \|_{L^2_{x}(H^1_{v,*})}^2
+\frac{\alpha_{1}C}{\varepsilon^2} \| \langle v \rangle^\alpha f \|_{L^2_{x}(H^1_{v,*})}^2 \\
&\quad
+\frac{\alpha_{1}}{\varepsilon^2}t \left( - \frac{\kappa}{\eps^2}\| \langle v \rangle^\alpha f \|_{L^2_{x}(H^2_{v,*})}^2 
+ \frac{C}{\eps} \| \langle v \rangle^\alpha f \|_{L^2_{x}(H^1_{v,*})} \| \langle v \rangle^\alpha \widetilde \nabla_x f \|_X \right)\\
&\quad
+2\frac{\alpha_{2}}{\varepsilon^3} t\la \langle v \rangle^\alpha \widetilde{\nabla}_v f , \langle v \rangle^\alpha \widetilde{\nabla}_x f \ra_{L^2_{x,v}}\\
&\quad
+\frac{\alpha_{2}}{\varepsilon^3}t^2  \left( -\frac{1}{\eps} \| \langle v \rangle^\alpha \widetilde \nabla_x f \|_{L^2_{x,v}}^2 + \frac{C}{\eps^2} \| \langle v \rangle^\alpha f \|_{L^2_{x}(H^2_{v,*})} \| \langle v \rangle^\alpha \widetilde \nabla_x f \|_{L^2_{x}(H^1_{v,*})}  \right) \\
&\quad
+\frac{\alpha_{3}C}{\varepsilon^4}t^{2} \| \langle v \rangle^\alpha \widetilde{\nabla}_x f \|_{L^2_{x,v}}^2 
-\frac{\alpha_{3}}{\varepsilon^4}t^{3} \frac{\kappa}{\eps^2}\| \langle v \rangle^\alpha \widetilde \nabla_x f \|_{L^2_{x}(H^1_{v,*})}^2 .
\end{align*}
Using Young's inequality, we have
$$
\begin{aligned}
C \alpha_1 \frac{t}{\eps^3} \| \langle v \rangle^\alpha f \|_{L^2_{x}(H^1_{v,*})} \| \langle v \rangle^\alpha \widetilde \nabla_x f \|_{L^2_{x,v}} 
&\le \frac{\alpha_2}{4} \frac{t^2}{\eps^4} \| \langle v \rangle^\alpha \widetilde \nabla_x f \|_{L^2_{x,v}}^2 + C \frac{\alpha_1^2}{\alpha_2} \frac{1}{\eps^2} \| \langle v \rangle^\alpha f \|_{L^2_{x}(H^1_{v,*})}^2 \\
2\frac{\alpha_{2}}{\varepsilon^3} t\la \langle v \rangle^\alpha \widetilde{\nabla}_v f , \langle v \rangle^\alpha \widetilde{\nabla}_x f \ra_{L^2_{x,v}}
&\le \frac{\alpha_2}{4} \frac{t^2}{\eps^4} \| \langle v \rangle^\alpha \widetilde \nabla_x f \|_{L^2_{x,v}}^2 + C \frac{\alpha_2}{\eps^2} \| \langle v \rangle^\alpha f \|_{L^2_{x}(H^1_{v,*})}^2 \\
C \alpha_2 \frac{t^2}{\eps^5} \| \langle v \rangle^\alpha f \|_{L^2_{x}(H^2_{v,*})} \| \langle v \rangle^\alpha \widetilde \nabla_x f \|_{L^2_{x}(H^1_{v,*})}
&\le \frac{\alpha_3 \kappa}{2} \frac{t^3}{\eps^6} \| \langle v \rangle^\alpha \widetilde \nabla_x f \|_{L^2_{x}(H^1_{v,*})}^2 \\
&\qquad \qquad \qquad 
+ C \frac{\alpha_2^2}{\alpha_3 } \frac{t}{\eps^4} \| \langle v \rangle^\alpha f \|^{2}_{L^2_{x}(H^2_{v,*})}.
\end{aligned}
$$
We thus deduce, for any $t \in (0,\eps^2]$, that
\begin{align*}
\frac{\mathrm{d}}{\mathrm{d}t}\mathcal E_{\varepsilon}(t)
&\le - \frac{1}{\eps^2}\left(2\kappa
-C \alpha_{1}
-C \alpha_2
-C \frac{\alpha_1^2}{\alpha_2} \right)\| \langle v \rangle^\alpha f \|_{L^2_{x}(H^1_{v,*})}^2  \\
&\quad
- \frac{t}{\eps^4} \left( \alpha_1 \kappa - C \frac{\alpha_2^2}{\alpha_3 }   \right)\| \langle v \rangle^\alpha f \|_{L^2_{x}(H^2_{v,*})}^2  
- \frac{t^2}{\eps^4}\left(\frac{\alpha_2}{2} - C \alpha_3  \right)\| \langle v \rangle^\alpha \widetilde \nabla_x f \|_{L^2_{x,v}}^2 \\
&\quad
- \frac{\alpha_3 \kappa}{2} \frac{t^3}{\eps^6} \| \langle v \rangle^\alpha \widetilde \nabla_x f \|_{L^2_{x}(H^1_{v,*})}^2 .
\end{align*}
We now choose $ \alpha_{1}=\eta$, $\alpha_{2}=\eta^{3/2}$, and $\alpha_{3}=\eta^{5/3}$ with $\eta \in (0,1)$ small enough such that each quantity appearing inside the parentheses in above inequality is positive. 
Therefore one obtains that 
$
\frac{\mathrm{d}}{\mathrm{d}t}\mathcal E_{\varepsilon}(t) \le 0,
$ 
for any $t \in (0,\eps^2] $, which concludes the proof as explained in Step 1. 
\end{proof}

\begin{proof}[Proof of Lemma~\ref{lem:regSBeps2}]
The proof follows the same lines as the proof of Lemma~\ref{lem:regSBeps}. More precisely, as explained in the Step 1 of the proof of Lemma~\ref{lem:regSBeps}, in order to obtain the desired result it is sufficient to prove that $\frac{\mathrm{d}}{\mathrm{d}t}\mathcal E_{\varepsilon}^1(t) \le 0$ for all $t \in (0,\eps^2]$ for some well-chosen functional $\mathcal E_{\varepsilon}^1$.

Let $\alpha \in \R$ and $f_{\rm in}$ be such that $\langle v \rangle^\alpha f_{\rm in} \in L^2_{x,v}$.
We then consider the solution $f(t) = S_{\BB_\eps}(t) f_{\rm in}$ to the equation $\partial_t f = \BB_\eps f$ with initial data $f_{\rm in}$.
Recalling that $\mathcal E_\eps$ is defined in~\eqref{eq:Lyapunov2}, we then define the functional 
\begin{align*}
&\mathcal E^1_{\varepsilon}(t) 
:= \mathcal E_{\eps}(t)\\
&\qquad
+\beta_1 \left(\frac{t}{\varepsilon^2}\right)^2  \Big(
K_1 \|\langle v \rangle^\alpha \langle v \rangle^{\gamma+2} f  \|_{L^2_{x,v}}^2 
+ K_2\|\langle v \rangle^\alpha  \langle v \rangle^{\frac{\gamma}{2}+1}\widetilde{\nabla}_v f  \|_{L^2_{x,v}}^2
+ \| \langle v \rangle^\alpha \widetilde{\nabla}_v \widetilde{\nabla}_v f \|_{L^2_{x,v}}^2  \Big)\\
&\qquad
+\varepsilon^2 \beta_2 \left(\frac{t}{\varepsilon^2}\right)^4 \Big( 
K_1\| \langle v \rangle^\alpha \langle v \rangle^{\gamma+1}  \nabla_x f \|_{L^2_{x,v}}^2
+ K_2\|\langle v \rangle^\alpha \langle v \rangle^{\frac{\gamma}{2}+1}\widetilde{\nabla}_xf  \|_{L^2_{x,v}}^2 \\
&\hspace{11cm}
+ \| \langle v \rangle^\alpha \widetilde{\nabla}_v\widetilde{\nabla}_x f \|_{L^2_{x,v}}^2 \Big) \\
&\qquad
+\varepsilon^{3} \beta_3 \left(\frac{t}{\varepsilon^2}\right)^5 \la \langle v \rangle^\alpha \widetilde{\nabla}_v\widetilde{\nabla}_x f , \langle v \rangle^\alpha \widetilde{\nabla}_x\widetilde{\nabla}_x f \ra_{L^2_{x,v}}\\
&\qquad
+\varepsilon^4 \beta_4 \left(\frac{t}{\varepsilon^2}\right)^6 \Big( 
K_1\| \langle v \rangle^\alpha \langle v \rangle^{\gamma} \nabla_x {\nabla}_xf \|_{L^2_{x,v}}^2 
+ K_2\| \langle v \rangle^\alpha \langle v \rangle^{\frac{\gamma}{2}} \nabla_x \widetilde{\nabla}_xf \|_{L^2_{x,v}}^2 \\
&\hspace{11cm}
+ \| \langle v \rangle^\alpha \widetilde{\nabla}_x \widetilde{\nabla}_x f \|_{L^2_{x,v}}^2 
\Big)
\end{align*}
where $\beta_1, \beta_2, \beta_3, \beta_4, K_1,K_2 >0$ are positive constants to be chosen later such that $0< \beta_4 \ll \beta_3 \ll \beta_2 \ll \beta_1 \ll 1$, $\beta_4\leq \sqrt{\beta_3 \beta_2}$ and $K_1 \gg K_2 \gg 1$.

\medskip\noindent
\textit{Step 1.}
From the proof of Lemma~\ref{lem:regSBeps} (Step 6), we already have that for any $t \in (0,\eps^2]$, there holds
\begin{equation}\label{dtEE}
\begin{aligned}
\frac{\mathrm{d}}{\mathrm{d}t}\mathcal E_{\varepsilon}(t)
&\le - \frac{\kappa}{2\eps^2}\| \langle v \rangle^\alpha f \|_{L^2_{x}(H^1_{v,*})}^2  
- \frac{\eta \kappa t}{2 \eps^4} \| \langle v \rangle^\alpha f \|_{L^2_{x}(H^2_{v,*})}^2 \\
&\quad 
- \frac{\eta^{3/2} t^2 }{4\eps^4} \| \langle v \rangle^\alpha \widetilde \nabla_x f \|_{L^2_{x,v}}^2 
- \frac{\eta^{5/3} \kappa}{2} \frac{t^3}{\eps^6} \| \langle v \rangle^\alpha \widetilde \nabla_x f \|_{L^2_{x}(H^1_{v,*})}^2 .
\end{aligned}
\end{equation}
for some constant $\kappa>0$ and where $\eta \in (0,1)$ is small enough.

\medskip\noindent
\textit{Step 2.}
From Lemma~\ref{lem:dissipativeBeps}, we already have
\begin{equation}\label{dtfY1-0}
\frac{\d}{\dt} \|\langle v \rangle^\alpha \langle v \rangle^{\gamma+2} f  \|_{L^2_{x,v}}^2 
\le - \frac{\kappa}{\eps^2} \|\langle v \rangle^\alpha \langle v \rangle^{\gamma+2} f  \|_{L^2_{x} (H^1_{v,*})}^2
\end{equation}
for some $\kappa >0$. Moreover, from the proof of Lemma~\ref{lem:regSBeps} (Step 3), we already know that 
\begin{equation}\label{dtfY1-1}
\begin{aligned}
&\frac{\d}{\dt} \| \langle v \rangle^{\alpha} \langle v \rangle^{\frac{\gamma}{2}+1} \widetilde \nabla_v f \|_{L^2_{x,v}}^2 \\
&\quad \le 
- \frac{\kappa}{2\eps^2} \| \langle v \rangle^{\alpha} \langle v \rangle^{\frac{\gamma}{2}+1} \widetilde \nabla_v f \|_{L^2_{x}(H^1_{v,*})}^2 
+ \frac{C}{\eps^2} \| \langle v \rangle^{\alpha} \langle v \rangle^{\frac{3\gamma}{2}+2} f \|_{L^2_{x,v}}^2 \\
&\quad \quad
+ \frac{C}{\eps^2} \| \langle v \rangle^{\alpha} \langle v \rangle^{\frac{3\gamma}{2}+2} \nabla_v f \|_{L^2_{x,v}}^2
+ \frac{C}{\eps} \| \langle v \rangle^{\alpha} \langle v \rangle^{\frac{\gamma}{2}+1} \widetilde \nabla_v f \|_{L^2_{x,v}} \| \langle v \rangle^{\alpha} \langle v \rangle^{\frac{\gamma}{2}+1} \widetilde \nabla_x f \|_{L^2_{x,v}} 
\end{aligned}
\end{equation}
for some constants $\kappa , C >0$.
We now compute 
$$
\begin{aligned}
&\frac12\frac{\d}{\dt} \| \langle v \rangle^{\alpha} \widetilde \nabla_v \widetilde \nabla_v f \|_{L^2_{x,v}}^2 \\
&\quad 
= \la \langle v \rangle^{\alpha} \widetilde \nabla_{v_i} \widetilde \nabla_{v_j} f ,  \langle v \rangle^{\alpha} \widetilde \nabla_{v_i} \widetilde \nabla_{v_j} \BB_\eps f\ra_{L^2_{x,v}} \\
&\quad 
= \la \langle v \rangle^{\alpha} \widetilde \nabla_{v_i} \widetilde \nabla_{v_j} f ,  \langle v \rangle^{\alpha}  \widetilde \nabla_{v_i} \BB_\eps (\widetilde \nabla_{v_j} f) \ra_{L^2_{x,v}}
+ \la \langle v \rangle^{\alpha} \widetilde \nabla_{v_i} \widetilde \nabla_{v_j} f ,  \langle v \rangle^{\alpha} \widetilde \nabla_{v_i} [\widetilde \nabla_{v_j}, \BB_\eps] f \ra_{L^2_{x,v}} \\
&\quad 
= \la \langle v \rangle^{\alpha} \widetilde \nabla_{v_i} \widetilde \nabla_{v_j} f ,  \langle v \rangle^{\alpha}  \widetilde \nabla_{v_i} \BB_\eps (\widetilde \nabla_{v_j} f) \ra_{L^2_{x,v}}
+ \la \langle v \rangle^{\alpha} \widetilde \nabla_{v_i} \widetilde \nabla_{v_j} f ,  \langle v \rangle^{\alpha}  [\widetilde \nabla_{v_j}, \BB_\eps] \widetilde \nabla_{v_i} f \ra_{L^2_{x,v}} \\
&\quad\quad 
+ \la \langle v \rangle^{\alpha} \widetilde \nabla_{v_i} \widetilde \nabla_{v_j} f ,  \langle v \rangle^{\alpha} [\widetilde \nabla_{v_i} , [\widetilde \nabla_{v_j}, \BB_\eps]] f \ra_{L^2_{x,v}} \\
&\quad
=: I_1 + I_2 + I_3.
\end{aligned}
$$
From the proof of Lemma~\ref{lem:regSBeps} (Step 3), we already know that
$$
\begin{aligned}
I_1 + I_2
&\le -\frac{\kappa}{2\eps^2} \| \langle v \rangle^{\alpha} \widetilde \nabla_v \widetilde \nabla_v f \|_{L^2_{x}(H^1_{v,*})}^2 
+ \frac{C}{\eps^2} \| \langle v \rangle^{\alpha} \langle v \rangle^{\gamma+1} \widetilde \nabla_v f \|_{L^2_{x,v}}^2 \\
&\quad
+ \frac{C}{\eps^2} \| \langle v \rangle^{\alpha} \langle v \rangle^{\gamma+1} \nabla_v \widetilde \nabla_v f \|_{L^2_{x,v}}^2
+ \frac{C}{\eps} \| \langle v \rangle^{\alpha} \widetilde \nabla_v \widetilde \nabla_v f \|_{L^2_{x,v}} \| \langle v \rangle^{\alpha}  \widetilde \nabla_x \widetilde \nabla_v f \|_{L^2_{x,v}}.
\end{aligned} 
$$
For the term $I_3$, we use Lemma~\ref{lem:commutator} to get
$$
\begin{aligned}{}
\left[\widetilde \nabla_{v_i} , [\widetilde \nabla_{v_j}, \BB_\eps] \right] f 
&= - \frac{1}{\eps^2} \left[\widetilde \nabla_{v_i} , \widetilde \nabla_{v_k}^* [\widetilde \nabla_{v_j} , \widetilde \nabla_{v_k}] \right] f
- \frac{1}{\eps^2} \left[\widetilde \nabla_{v_i} , \widetilde \nabla_{v_k} [\widetilde \nabla_{v_j} , \widetilde \nabla_{v_k}^*] \right] f \\
&\quad
-\frac{1}{\eps^2} \left[ \widetilde \nabla_{v_i} , [ [\widetilde \nabla_{v_j} , \widetilde \nabla_{v_k}^*] , \widetilde \nabla_{v_k} ] \right]f
-\frac{1}{\eps^2} (\widetilde \nabla_{v_i} \widetilde \nabla_{v_j} m^2) f \\
&\quad 
-\frac{1}{\eps} [\widetilde \nabla_{v_i} , \widetilde \nabla_{x_j}]f .
\end{aligned}
$$
Expanding the first two terms, we observe that 
$$
\begin{aligned}
&\left[\widetilde \nabla_{v_i} , \widetilde \nabla_{v_k}^* [\widetilde \nabla_{v_j} , \widetilde \nabla_{v_k}] \right] 
+ \left[\widetilde \nabla_{v_i} , \widetilde \nabla_{v_k} [\widetilde \nabla_{v_j} , \widetilde \nabla_{v_k}^*] \right]  \\
&\quad
= \widetilde \nabla_{v_i}  \widetilde \nabla_{v_k}^* [\widetilde \nabla_{v_j} , \widetilde \nabla_{v_k}]  
- \widetilde \nabla_{v_k}^* [\widetilde \nabla_{v_j} , \widetilde \nabla_{v_k}]  \widetilde \nabla_{v_i}  
+\widetilde \nabla_{v_i}  \widetilde \nabla_{v_k} [\widetilde \nabla_{v_j} , \widetilde \nabla_{v_k}^*]  
-  \widetilde \nabla_{v_k} [\widetilde \nabla_{v_j} , \widetilde \nabla_{v_k}^*] \widetilde \nabla_{v_i}  \\
&\quad 
=  \widetilde \nabla_{v_k}^* \left[ \widetilde \nabla_{v_i}, [ \widetilde \nabla_{v_j} , \widetilde \nabla_{v_k}]   \right] 
+ \widetilde \nabla_{v_k} \left[ \widetilde \nabla_{v_i}, [ \widetilde \nabla_{v_j} , \widetilde \nabla_{v_k}^*]   \right] 
+ [\widetilde \nabla_{v_i} , \widetilde \nabla_{v_k}^*] [\widetilde \nabla_{v_j} , \widetilde \nabla_{v_k}] \\
&\hspace{11cm}
+[\widetilde \nabla_{v_i} , \widetilde \nabla_{v_k}] [\widetilde \nabla_{v_j} , \widetilde \nabla_{v_k}^*],
\end{aligned}
$$
whence
$$
\begin{aligned}
I_3
&= -\frac{1}{\eps^2} \la \widetilde \nabla_{v_k} \left(\langle v \rangle^{2\alpha} \widetilde \nabla_{v_i} \widetilde \nabla_{v_j} f\right)  ,  \left[ \widetilde \nabla_{v_i}, [ \widetilde \nabla_{v_j} , \widetilde \nabla_{v_k}]   \right] f \ra_{L^2_{x,v}} \\
&\quad
-\frac{1}{\eps^2} \la \widetilde \nabla_{v_k}^* \left(\langle v \rangle^{2\alpha} \widetilde \nabla_{v_i} \widetilde \nabla_{v_j} f\right)  ,  \left[ \widetilde \nabla_{v_i}, [ \widetilde \nabla_{v_j} , \widetilde \nabla_{v_k}^*]   \right] f \ra_{L^2_{x,v}}
\\
&\quad
-\frac{1}{\eps^2} \la \langle v \rangle^{\alpha} \widetilde \nabla_{v_i} \widetilde \nabla_{v_j} f  , \langle v \rangle^{\alpha}  [\widetilde \nabla_{v_i} , \widetilde \nabla_{v_k}^*] [\widetilde \nabla_{v_j} , \widetilde \nabla_{v_k}]f \ra_{L^2_{x,v}} \\
&\quad
-\frac{1}{\eps^2} \la \langle v \rangle^{\alpha} \widetilde \nabla_{v_i} \widetilde \nabla_{v_j} f  , \langle v \rangle^{\alpha}  [\widetilde \nabla_{v_i} , \widetilde \nabla_{v_k}] [\widetilde \nabla_{v_j} , \widetilde \nabla_{v_k}^*]f \ra_{L^2_{x,v}} \\
&\quad
-\frac{1}{\eps^2} \la \langle v \rangle^{\alpha} \widetilde \nabla_{v_i} \widetilde \nabla_{v_j} f  , \langle v \rangle^{\alpha}  \left[ \widetilde \nabla_{v_i} , [ [\widetilde \nabla_{v_j} , \widetilde \nabla_{v_k}^*] , \widetilde \nabla_{v_k} ] \right]f \ra_{L^2_{x,v}} 
\\
&\quad
-\frac{1}{\eps^2} \la  \langle v \rangle^{\alpha} \widetilde \nabla_{v_i} \widetilde \nabla_{v_j} f  , \langle v \rangle^{\alpha} (\widetilde \nabla_{v_i} \widetilde \nabla_{v_j} m^2) f \ra_{L^2_{x,v}} \\
&\quad
-\frac{1}{\eps} \la  \langle v \rangle^{\alpha} \widetilde \nabla_{v_i} \widetilde \nabla_{v_j} f , \langle v \rangle^{\alpha} [\widetilde \nabla_{v_i} , \widetilde \nabla_{x_j}]f \ra_{L^2_{x,v}} \\
&\quad
=: - \frac{1}{\eps^2} (I_{31}+I_{32} + I_{33} + I_{34} + I_{35} + I_{36}) - \frac{1}{\eps}\la  \langle v \rangle^{\alpha} \widetilde \nabla_{v_i} \widetilde \nabla_{v_j} f , \langle v \rangle^{\alpha} [\widetilde \nabla_{v_i} ,\widetilde \nabla_{x_j}] f \ra_{L^2_{x,v}}.
\end{aligned}
$$
Now remark that 
$$
\|  \langle v \rangle^{-\alpha} \widetilde \nabla_{v_k}^* ( \langle v \rangle^{2\alpha} \widetilde \nabla_{v_i} \widetilde \nabla_{v_j} f )  \|_{L^2_{x,v}}
+ \|  \langle v \rangle^{-\alpha} \widetilde \nabla_{v_k} ( \langle v \rangle^{2\alpha} \widetilde \nabla_{v_i} \widetilde \nabla_{v_j} f )  \|_{L^2_{x,v}} \lesssim \| \langle v \rangle^{\alpha} \widetilde \nabla_v \widetilde \nabla_v f \|_{L^2_{x}(H^1_{v,*})}
$$
and using Lemma~\ref{lem:commutator}, we obtain 
$$
I_{31}+I_{32} 
\lesssim \| \langle v \rangle^{\alpha} \widetilde \nabla_v \widetilde \nabla_v f \|_{L^2_{x}(H^1_{v,*})} \left( \| \langle v \rangle^{\alpha} \langle v \rangle^{\frac{3\gamma}{2}+1} \nabla_v f \|_{L^2_{x,v}} + \| \langle v \rangle^{\alpha} \langle v \rangle^{\frac{3\gamma}{2}}  f \|_{L^2_{x,v}} \right).
$$
Thanks to Lemma~\ref{lem:commutator}, we observe that 
$$
|[\widetilde \nabla_{v_i} , \widetilde \nabla_{v_k}^*] [\widetilde \nabla_{v_j} , \widetilde \nabla_{v_k}]f|
+ |[\widetilde \nabla_{v_i} , \widetilde \nabla_{v_k}] [\widetilde \nabla_{v_j} , \widetilde \nabla_{v_k}^*]f|
\lesssim \langle v \rangle^{2\gamma+2} |\nabla_v \nabla_v f| + \langle v \rangle^{2\gamma+1} | \nabla_v f| 
$$
and thus we obtain
$$
\begin{aligned}
I_{33}+I_{34} 
&\lesssim \| \langle v \rangle^{\alpha} \langle v \rangle^{\frac{\gamma}{2}+1}  \widetilde \nabla_v \widetilde \nabla_v f \|_{L^2_{x,v}} \left( \| \langle v \rangle^{\alpha} \langle v \rangle^{\frac{3\gamma}{2}+1} \nabla_v \nabla_v f \|_{L^2_{x,v}}   + \| \langle v \rangle^{\alpha} \langle v \rangle^{\frac{3\gamma}{2}}  \nabla_v f \|_{L^2_{x,v}}  \right)  \\
&\lesssim \| \langle v \rangle^{\alpha} \widetilde \nabla_v \widetilde \nabla_v f \|_{L^2_{x}(H^1_{v,*})} \left( \| \langle v \rangle^{\alpha} \langle v \rangle^{\frac{3\gamma}{2}+1} \nabla_v \nabla_v f \|_{L^2_{x,v}}   + \| \langle v \rangle^{\alpha} \langle v \rangle^{\frac{3\gamma}{2}}  \nabla_v f \|_{L^2_{x,v}}  \right) .
\end{aligned}
$$
For the term $I_{35}$, we write 
\begin{multline*}
I_{35}
= \la \widetilde \nabla_{v_i}^* (\langle v \rangle^{2\alpha} \widetilde \nabla_{v_i} \widetilde \nabla_{v_j} f)  ,   [ [\widetilde \nabla_{v_j} , \widetilde \nabla_{v_k}^*] , \widetilde \nabla_{v_k} ] f \ra_{L^2_{x,v}} \\
- \la \langle v \rangle^{\alpha} \widetilde \nabla_{v_i} \widetilde \nabla_{v_j} f  , \langle v \rangle^{\alpha}     [ [\widetilde \nabla_{v_j} , \widetilde \nabla_{v_k}^*] , \widetilde \nabla_{v_k} ] \widetilde \nabla_{v_i} f \ra_{L^2_{x,v}} 
\end{multline*}
and using Lemma~\ref{lem:commutator}, we thus get 
\begin{multline*}
I_{35}
\lesssim \| \langle v \rangle^{\alpha} \widetilde \nabla_v \widetilde \nabla_v f \|_{L^2_{x}(H^1_{v,*})} \Big( \| \langle v \rangle^{\alpha} \langle v \rangle^{\frac{3\gamma}{2}+1} \nabla_v f \|_{L^2_{x,v}}
+ \| \langle v \rangle^{\alpha} \langle v \rangle^{\frac{3\gamma}{2}} f \|_{L^2_{x,v}} \\
+ \| \langle v \rangle^{\alpha} \langle v \rangle^{\gamma} \nabla_v \widetilde \nabla_v f \|_{L^2_{x,v}} \Big).
\end{multline*}
Using that $|\widetilde \nabla_{v_i} \widetilde \nabla_{v_j} m^2| \lesssim \langle v \rangle^{2\gamma+2}$, we also obtain
$$
I_{36} \lesssim \| \langle v \rangle^{\alpha} \widetilde \nabla_v \widetilde \nabla_v f \|_{L^2_{x}(H^1_{v,*})} \| \langle v \rangle^{\alpha} \langle v \rangle^{\frac{3\gamma}{2}+1} f \|_{L^2_{x,v}}.
$$
Gathering previous estimates and using Young's inequality, we thus get
\begin{equation}\label{dtfY1-2}
\begin{aligned}
\frac{\d}{\dt} \| \langle v \rangle^{\alpha} \widetilde \nabla_v \widetilde \nabla_v f \|_{L^2_{x,v}}^2 
& 
\le -\frac{\kappa}{4\eps^2} \| \langle v \rangle^{\alpha} \widetilde \nabla_v \widetilde \nabla_v f \|_{L^2_{x}(H^1_{v,*})}^2 
+ \frac{C}{\eps^2} \| \langle v \rangle^{\alpha} \langle v \rangle^{\gamma+1} \widetilde \nabla_v f \|_{L^2_{x,v}}^2 \\
&\quad
+ \frac{C}{\eps^2} \| \langle v \rangle^{\alpha} \langle v \rangle^{\gamma+1} \nabla_v \widetilde \nabla_v f \|_{L^2_{x,v}}^2 
+ \frac{C}{\eps^2} \| \langle v \rangle^{\alpha} \langle v \rangle^{\frac{3\gamma}{2}+1} \nabla_v \nabla_v f \|_{L^2_{x,v}}^2\\
&\quad 
+ \frac{C}{\eps^2} \| \langle v \rangle^{\alpha} \langle v \rangle^{\frac{3\gamma}{2}+1} \nabla_v  f \|_{L^2_{x,v}}^2
+ \frac{C}{\eps^2} \| \langle v \rangle^{\alpha} \langle v \rangle^{\frac{3\gamma}{2}+2}  f \|_{L^2_{x,v}}^2 \\
&\quad
+ \frac{C}{\eps} \| \langle v \rangle^{\alpha} \widetilde \nabla_v \widetilde \nabla_v f \|_{L^2_{x,v}} \| \langle v \rangle^{\alpha}  \widetilde \nabla_x  f \|_{L^2_{x}(H^1_{v,*})}.
\end{aligned}
\end{equation}

Observing that 
\begin{multline*}
\| \langle v \rangle^{\alpha} \langle v \rangle^{\frac{3\gamma}{2}+2} f \|_{L^2_{x,v}} 
+ \| \langle v \rangle^{\alpha} \langle v \rangle^{\frac{3\gamma}{2}+2} \nabla_v f \|_{L^2_{x,v}} 
+ \| \langle v \rangle^{\alpha} \langle v \rangle^{\gamma+1} \widetilde \nabla_v f \|_{L^2_{x,v}} \\
\lesssim \| \langle v \rangle^{\alpha} \langle v \rangle^{\gamma +2} f \|_{L^2_{x}(H^1_{v,*})}
\end{multline*}
and
$$
\| \langle v \rangle^{\alpha} \langle v \rangle^{\gamma+1} \nabla_v \widetilde \nabla_v f \|_{L^2_{x,v}}
+ \| \langle v \rangle^{\alpha} \langle v \rangle^{\frac{3\gamma}{2}+1} \nabla_v \nabla_v f \|_{L^2_{x,v}}
\lesssim \| \langle v \rangle^{\alpha} \langle v \rangle^{\frac{\gamma}{2}+1} \widetilde \nabla_v f \|_{L^2_x(H^1_{v,*})},
$$
we then gather \eqref{dtfY1-0}--\eqref{dtfY1-1}--\eqref{dtfY1-2} choosing $K_1 \gg K_2 \gg 1$ large enough, which yields
\begin{equation}\label{dtfY1}
\begin{aligned}
&\frac{\d}{\dt} \Big(
K_1 \|\langle v \rangle^\alpha \langle v \rangle^{\gamma+2} f  \|_{L^2_{x,v}}^2 
+ K_2\|\langle v \rangle^\alpha  \langle v \rangle^{\frac{\gamma}{2}+1}\widetilde{\nabla}_v f  \|_{L^2_{x,v}}^2
+ \| \langle v \rangle^\alpha \widetilde{\nabla}_v \widetilde{\nabla}_v f \|_{L^2_{x,v}}^2  \Big) \\
&\quad 
\le - \frac{\kappa}{\eps^2} \left( \| \langle v \rangle^{\alpha} \langle v \rangle^{\gamma+2} f \|_{L^2_x(H^1_{v,*})}^2
+ \| \langle v \rangle^{\alpha} \langle v \rangle^{\frac{\gamma}{2}+1} \widetilde \nabla_v f \|_{L^2_x(H^1_{v,*})}^2
+ \| \langle v \rangle^{\alpha} \widetilde \nabla_v \widetilde \nabla_v   f \|_{L^2_x(H^1_{v,*})}^2   \right) \\
&\quad\quad 
+ \frac{C}{\eps} \| \langle v \rangle^{\alpha} f \|_{L^2_{x}(H^2_{v,*})} \| \langle v \rangle^{\alpha}  \widetilde \nabla_x  f \|_{L^2_{x}(H^1_{v,*})}.
\end{aligned}
\end{equation}

\medskip\noindent
\textit{Step 3.}
From the proof of Lemma~\ref{lem:regSBeps} (Step 5) we already have, with $K_1$ large enough with respect to $K_2$, 
\begin{equation}\label{dtDvDxf-1}
\begin{aligned}
&\frac{\d}{\dt} \left\{ K_1\| \langle v \rangle^{\alpha} \langle v \rangle^{\gamma+1} \nabla_x f \|_{L^2_{x,v}}^2 
+ K_2 \| \langle v \rangle^{\alpha} \langle v \rangle^{\frac{\gamma}{2}+1} \widetilde \nabla_x f \|_{L^2_{x,v}}^2  \right\} \\
&\quad 
\le - \frac{\kappa}{\eps^2} \| \langle v \rangle^{\alpha} \langle v \rangle^{\frac{\gamma}{2}+1} \widetilde \nabla_x f \|_{L^2_{x}(H^1_{v,*})}^2.
\end{aligned}
\end{equation}

We now compute
$$
\begin{aligned}
&\frac12\frac{\d}{\dt} \| \langle v \rangle^{\alpha} \widetilde \nabla_v \widetilde \nabla_x f \|_{L^2_{x,v}}^2 \\
&\quad 
= \la \langle v \rangle^{\alpha} \widetilde \nabla_{v_i} \widetilde \nabla_{x_j} f ,  \langle v \rangle^{\alpha} \widetilde \nabla_{v_i} \widetilde \nabla_{x_j} \BB_\eps f\ra_{L^2_{x,v}} \\
&\quad 
= \la \langle v \rangle^{\alpha} \widetilde \nabla_{v_i} \widetilde \nabla_{x_j} f ,  \langle v \rangle^{\alpha}  \widetilde \nabla_{v_i} \BB_\eps (\widetilde \nabla_{x_j} f) \ra_{L^2_{x,v}}
+ \la \langle v \rangle^{\alpha} \widetilde \nabla_{v_i} \widetilde \nabla_{x_j} f ,  \langle v \rangle^{\alpha} \widetilde \nabla_{v_i} [\widetilde \nabla_{x_j}, \BB_\eps] f \ra_{L^2_{x,v}} \\
&\quad 
= \la \langle v \rangle^{\alpha} \widetilde \nabla_{v_i} \widetilde \nabla_{x_j} f ,  \langle v \rangle^{\alpha}  \widetilde \nabla_{v_i} \BB_\eps (\widetilde \nabla_{x_j} f) \ra_{L^2_{x,v}}
+ \la \langle v \rangle^{\alpha} \widetilde \nabla_{v_i} \widetilde \nabla_{x_j} f ,  \langle v \rangle^{\alpha}  [\widetilde \nabla_{x_j}, \BB_\eps] \widetilde \nabla_{v_i} f \ra_{L^2_{x,v}} \\
&\quad\quad 
+ \la \langle v \rangle^{\alpha} \widetilde \nabla_{v_i} \widetilde \nabla_{x_j} f ,  \langle v \rangle^{\alpha} [\widetilde \nabla_{v_i} , [\widetilde \nabla_{x_j}, \BB_\eps]] f \ra_{L^2_{x,v}} \\
&\quad
=: J_1 + J_2 + J_3.
\end{aligned}
$$
From the proof of Lemma~\ref{lem:regSBeps} (Step 3), we already know that
$$
\begin{aligned}
J_1 + J_2
&\le -\frac{\kappa}{2\eps^2} \| \langle v \rangle^{\alpha} \widetilde \nabla_v \widetilde \nabla_x f \|_{L^2_{x}(H^1_{v,*})}^2 
+ \frac{C}{\eps^2} \| \langle v \rangle^{\alpha} \langle v \rangle^{\gamma+1} \widetilde \nabla_x f \|_{L^2_{x,v}}^2 \\
&\quad
+ \frac{C}{\eps^2} \| \langle v \rangle^{\alpha} \langle v \rangle^{\gamma+1} \nabla_v \widetilde \nabla_x f \|_{L^2_{x,v}}^2
+ \frac{C}{\eps} \| \langle v \rangle^{\alpha} \widetilde \nabla_v \widetilde \nabla_x f \|_{L^2_{x,v}} \| \langle v \rangle^{\alpha}  \widetilde \nabla_x \widetilde \nabla_x f \|_{L^2_{x,v}}.
\end{aligned} 
$$
The term $J_3$ can be estimated as the term $I_3$ in Step 2 above. Therefore, we obtain, gathering these estimates and using Young's inequality,
\begin{equation}\label{dtDvDxf-2}
\begin{aligned}
&\frac{\d}{\dt} \| \langle v \rangle^{\alpha} \widetilde \nabla_v \widetilde \nabla_x f \|_{L^2_{x,v}}^2 
\le -\frac{\kappa}{4\eps^2} \| \langle v \rangle^{\alpha} \widetilde \nabla_v \widetilde \nabla_x f \|_{L^2_{x}(H^1_{v,*})}^2 
+ \frac{C}{\eps^2} \| \langle v \rangle^{\alpha} \langle v \rangle^{\gamma+1} \widetilde \nabla_x f \|_{L^2_{x,v}}^2 \\
&\qquad
+ \frac{C}{\eps^2} \| \langle v \rangle^{\alpha} \langle v \rangle^{\gamma+1} \nabla_v \widetilde \nabla_x f \|_{L^2_{x,v}}^2 
+ \frac{C}{\eps^2} \| \langle v \rangle^{\alpha} \langle v \rangle^{\frac{3\gamma}{2}+1} \nabla_v \nabla_x f \|_{L^2_{x,v}}^2\\
&\qquad 
+ \frac{C}{\eps^2} \| \langle v \rangle^{\alpha} \langle v \rangle^{\frac{3\gamma}{2}+1} \nabla_x  f \|_{L^2_{x,v}}^2 
+ \frac{C}{\eps} \| \langle v \rangle^{\alpha} \widetilde \nabla_v \widetilde \nabla_x f \|_{L^2_{x,v}} \| \langle v \rangle^{\alpha}  \widetilde \nabla_x \widetilde \nabla_x  f \|_{L^2_{x,v}}.
\end{aligned}
\end{equation}

Observing that 
$$
\begin{aligned}
&\| \langle v \rangle^{\alpha} \langle v \rangle^{\gamma+1}  \widetilde \nabla_x f \|_{L^2_{x,v}}
+ \| \langle v \rangle^{\alpha} \langle v \rangle^{\gamma+1} \nabla_v \widetilde \nabla_x f \|_{L^2_{x,v}} \\
&+ \| \langle v \rangle^{\alpha} \langle v \rangle^{\frac{3\gamma}{2}+1} \nabla_v \nabla_x f \|_{L^2_{x,v}}
+ \| \langle v \rangle^{\alpha} \langle v \rangle^{\frac{3\gamma}{2}+1}  \nabla_x f \|_{L^2_{x,v}}
\lesssim \| \langle v \rangle^{\alpha} \langle v \rangle^{\frac{\gamma}{2}+1} \widetilde \nabla_x f \|_{L^2_x(H^1_{v,*})}
\end{aligned}
$$
we then gather \eqref{dtDvDxf-1}--\eqref{dtDvDxf-2} and choose $K_1 \gg K_2 \gg 1$ large enough, which yields
\begin{equation}\label{dtDvDxf}
\begin{aligned}
&\frac{\d}{\dt} \Big(
K_1\| \langle v \rangle^{\alpha} \langle v \rangle^{\gamma+1} \nabla_x f \|_{L^2_{x,v}}^2 
+ K_2 \| \langle v \rangle^{\alpha} \langle v \rangle^{\frac{\gamma}{2}+1} \widetilde \nabla_x f \|_{L^2_{x,v}}^2
+ \| \langle v \rangle^{\alpha} \widetilde \nabla_v \widetilde \nabla_x f \|_{L^2_{x,v}}^2   \Big) \\
&\quad 
\le - \frac{\kappa}{\eps^2} \left( \| \langle v \rangle^{\alpha} \langle v \rangle^{\frac{\gamma}{2}+1} \widetilde \nabla_x f \|_{L^2_{x}(H^1_{v,*})}^2  + \| \langle v \rangle^{\alpha} \widetilde \nabla_v \widetilde \nabla_x f \|_{L^2_{x}(H^1_{v,*})}^2 \right) \\
&\quad\quad 
+ \frac{C}{\eps} \| \langle v \rangle^{\alpha} \widetilde \nabla_x f \|_{L^2_{x}(H^1_{v,*})} \| \langle v \rangle^{\alpha}  \widetilde \nabla_x \widetilde \nabla_x  f \|_{L^2_{x,v}}.
\end{aligned}
\end{equation}

\medskip\noindent
\textit{Step 4.} Arguing as in the proof of Lemma~\ref{lem:regSBeps} (Step 4), we obtain
\begin{equation}\label{dtDvDxDxDxf}
\begin{aligned}
&\frac{\d}{\dt} \la \langle v \rangle^\alpha \widetilde{\nabla}_v\widetilde{\nabla}_x f , \langle v \rangle^\alpha \widetilde{\nabla}_x\widetilde{\nabla}_x f\ra_{L^2_{x,v}} \\
&\quad
\le - \frac{1}{\eps} \| \langle v \rangle^{\alpha} \widetilde \nabla_x \widetilde{\nabla}_x f \|_{L^2_{x,v}}^2
+ \frac{C}{\eps^2} \| \langle v \rangle^{\alpha} \widetilde \nabla_x f \|_{L^2_x(H^2_{v,*})} \| \langle v \rangle^{\alpha} \widetilde \nabla_x \widetilde \nabla_x f \|_{L^2_x(H^1_{v,*})}.
\end{aligned}
\end{equation}

\medskip\noindent
\textit{Step 5.} 
Since $\nabla_x$ commutes with $\widetilde \nabla_x$ and $\BB_\eps$, we already know from the proof of Lemma~\ref{lem:regSBeps} (Step 5) that, with $K_1$ large enough with respect to $K_2$,
\begin{equation}\label{dtDxDxf-1}
\begin{aligned}
&\frac{\d}{\dt} \left\{ K_1 \| \langle v \rangle^{\alpha} \langle v \rangle^{\gamma} \nabla_x \nabla_x f \|_{L^2_{x,v}}^2
+ K_2 \| \langle v \rangle^{\alpha} \langle v \rangle^{\frac{\gamma}{2}} \nabla_x \widetilde \nabla_x f \|_{L^2_{x,v}}^2  \right\} \\
&\quad 
\le -\frac{\kappa}{\eps^2} \| \langle v \rangle^{\alpha} \langle v \rangle^{\frac{\gamma}{2}} \nabla_x \widetilde \nabla_x f \|_{L^2_{x}(H^1_{v,*})}^2
\end{aligned}
\end{equation}
for some constant $\kappa >0$.
We now compute
$$
\begin{aligned}
&\frac12\frac{\d}{\dt} \| \langle v \rangle^{\alpha} \widetilde \nabla_x \widetilde \nabla_x f \|_{L^2_{x,v}}^2 \\
&\quad 
= \la \langle v \rangle^{\alpha} \widetilde \nabla_{x_i} \widetilde \nabla_{x_j} f ,  \langle v \rangle^{\alpha} \widetilde \nabla_{x_i} \widetilde \nabla_{x_j} \BB_\eps f\ra_{L^2_{x,v}} \\
&\quad 
= \la \langle v \rangle^{\alpha} \widetilde \nabla_{x_i} \widetilde \nabla_{x_j} f ,  \langle v \rangle^{\alpha}  \widetilde \nabla_{x_i} \BB_\eps (\widetilde \nabla_{x_j} f) \ra_{L^2_{x,v}}
+ \la \langle v \rangle^{\alpha} \widetilde \nabla_{x_i} \widetilde \nabla_{x_j} f ,  \langle v \rangle^{\alpha} \widetilde \nabla_{x_i} [\widetilde \nabla_{x_j}, \BB_\eps] f \ra_{L^2_{x,v}} \\
&\quad 
= \la \langle v \rangle^{\alpha} \widetilde \nabla_{x_i} \widetilde \nabla_{x_j} f ,  \langle v \rangle^{\alpha}  \widetilde \nabla_{x_i} \BB_\eps (\widetilde \nabla_{x_j} f) \ra_{L^2_{x,v}}
+ \la \langle v \rangle^{\alpha} \widetilde \nabla_{x_i} \widetilde \nabla_{x_j} f ,  \langle v \rangle^{\alpha}  [\widetilde \nabla_{x_j}, \BB_\eps] \widetilde \nabla_{x_i} f \ra_{L^2_{x,v}} \\
&\quad\quad 
+ \la \langle v \rangle^{\alpha} \widetilde \nabla_{x_i} \widetilde \nabla_{x_j} f ,  \langle v \rangle^{\alpha} [\widetilde \nabla_{x_i} , [\widetilde \nabla_{x_j}, \BB_\eps]] f \ra_{L^2_{x,v}} \\
&\quad
=: R_1 + R_2 + R_3.
\end{aligned}
$$
From the proof of Lemma~\ref{lem:regSBeps} (Step 5), we already know that
$$
\begin{aligned}
R_1 + R_2
&\le -\frac{\kappa}{2\eps^2} \| \langle v \rangle^{\alpha} \widetilde \nabla_x \widetilde \nabla_x f \|_{L^2_{x}(H^1_{v,*})}^2 
+ \frac{C}{\eps^2} \| \langle v \rangle^{\alpha} \langle v \rangle^{\gamma+1} \nabla_x \widetilde \nabla_x f \|_{L^2_{x,v}}^2 .
\end{aligned} 
$$
The term $R_3$ can be estimated in a similar way as the term $I_3$ in Step 2 above. Therefore we obtain, using Young's inequality,
\begin{equation}\label{dtDxDxf-2}
\begin{aligned}
\frac{\d}{\dt} \| \langle v \rangle^{\alpha} \widetilde \nabla_x \widetilde \nabla_x f \|_{L^2_{x,v}}^2 
& 
\le -\frac{\kappa}{2\eps^2} \| \langle v \rangle^{\alpha} \widetilde \nabla_x \widetilde \nabla_x f \|_{L^2_{x}(H^1_{v,*})}^2 
+ \frac{C}{\eps^2} \| \langle v \rangle^{\alpha} \langle v \rangle^{\gamma+1} \nabla_x \widetilde \nabla_x f \|_{L^2_{x,v}}^2 \\
&\quad
+ \frac{C}{\eps^2} \| \langle v \rangle^{\alpha} \langle v \rangle^{\frac{3\gamma}{2}+1} \nabla_x  \nabla_x f \|_{L^2_{x,v}}^2.
\end{aligned}
\end{equation}
Observing that 
$$
\begin{aligned}
&\| \langle v \rangle^{\alpha} \langle v \rangle^{\gamma+1}  \nabla_x \widetilde \nabla_x f \|_{L^2_{x,v}}
+ \| \langle v \rangle^{\alpha} \langle v \rangle^{\frac{3\gamma}{2}+1} \nabla_x \nabla_x f \|_{L^2_{x,v}}
\lesssim \| \langle v \rangle^{\alpha} \langle v \rangle^{\frac{\gamma}{2}} \nabla_x \widetilde \nabla_x f \|_{L^2_x(H^1_{v,*})}
\end{aligned}
$$
we then gather \eqref{dtDxDxf-1}--\eqref{dtDxDxf-2} and choose $K_1 \gg K_2 \gg 1$ large enough, which yields
\begin{equation}\label{dtDxDxf}
\begin{aligned}
&\frac{\d}{\dt} \Big(
K_1\| \langle v \rangle^{\alpha} \langle v \rangle^{\gamma} \nabla_x \nabla_x f \|_{L^2_{x,v}}^2 
+ K_2 \| \langle v \rangle^{\alpha} \langle v \rangle^{\frac{\gamma}{2}} \nabla_x \widetilde \nabla_x f \|_{L^2_{x,v}}^2
+ \| \langle v \rangle^{\alpha} \widetilde \nabla_x \widetilde \nabla_x f \|_{L^2_{x,v}}^2   \Big) \\
&\quad 
\le - \frac{\kappa}{\eps^2}  \| \langle v \rangle^{\alpha} \widetilde \nabla_x \widetilde \nabla_x f \|_{L^2_{x}(H^1_{v,*})}^2  .
\end{aligned}
\end{equation}

\medskip\noindent
\textit{Step 6.} We can then conclude the proof as in Step 6 of the proof of Lemma~\ref{lem:regSBeps}.
\color{black}
\end{proof}

\bigskip
\bibliographystyle{acm}

\end{document}